\documentclass{amsart}

\usepackage{pstricks}
\usepackage{pst-node}
\usepackage{graphicx}
\usepackage{amsmath,amsfonts,amssymb}

\newtheorem{theorem}{Theorem}[section]
\newtheorem{lemma}[theorem]{Lemma}
\newtheorem{proposition}[theorem]{Proposition}
\newtheorem{corollary}[theorem]{Corollary}

\theoremstyle{definition}
\newtheorem{definition}[theorem]{Definition}

\theoremstyle{remark}
\newtheorem{remark}[theorem]{Remark}

\numberwithin{equation}{section}

\newcommand{\iph}{\ensuremath i\!+\!3}
\newcommand{\ipt}{\ensuremath i\!+\!2}
\newcommand{\ipo}{\ensuremath i\!+\!1}
\newcommand{\imo}{\ensuremath i\!-\!1}
\newcommand{\imt}{\ensuremath i\!-\!2}
\newcommand{\imh}{\ensuremath i\!-\!3}
\newcommand{\imf}{\ensuremath i\!-\!4}
\newcommand{\triple}{\ensuremath i\!-\!1,i,i\!+\!1}

\newcommand{\npo}{\ensuremath n\!+\!1}
\newcommand{\nmo}{\ensuremath n\!-\!1}
\newcommand{\nmt}{\ensuremath n\!-\!2}
\newcommand{\nmh}{\ensuremath n\!-\!3}

\newcommand{\G}{\ensuremath\mathcal{G}}
\newcommand{\C}{\ensuremath\mathcal{C}}

\newcommand{\LSP}{\ensuremath\mathrm{LSP}}

\newcommand{\B}{\bullet}
\newcommand{\cs}[1]{c@{\hskip #1ex}}

\newlength\cellsize 

\savebox2{%
\begin{picture}(12,12)
\put(0,0){\line(1,0){12}}
\put(0,0){\line(0,1){12}}
\put(12,0){\line(0,1){12}}
\put(0,12){\line(1,0){12}}
\end{picture}}

\newcommand\cellify[1]{\def\thearg{#1}\def\nothing{}%
\ifx\thearg\nothing
\vrule width0pt height\cellsize depth0pt\else
\hbox to 0pt{\usebox2\hss}\fi%
\vbox to 12\unitlength{
\vss
\hbox to 12\unitlength{\hss$#1$\hss}
\vss}}

\newcommand\tableau[1]{\vtop{\let\\=\cr
\setlength\baselineskip{-12000pt}
\setlength\lineskiplimit{12000pt}
\setlength\lineskip{0pt}
\halign{&\cellify{##}\cr#1\crcr}}}

\newlength\smcellsize 

\savebox3{%
\begin{picture}(8,8)
\put(0,0){\line(1,0){8}}
\put(0,0){\line(0,1){8}}
\put(8,0){\line(0,1){8}}
\put(0,8){\line(1,0){8}}
\end{picture}}

\newcommand\smcellify[1]{\def\thearg{#1}\def\nothing{}%
\ifx\thearg\nothing
\vrule width0pt height\smcellsize depth0pt\else
\hbox to 0pt{\usebox3\hss}\fi%
\vbox to 8\unitlength{
\vss
\hbox to 8\unitlength{\hss$#1$\hss}
\vss}}

\newcommand\smtableau[1]{\vtop{\let\\=\cr
\setlength\baselineskip{-8000pt}
\setlength\lineskiplimit{8000pt}
\setlength\lineskip{0pt}
\halign{&\smcellify{##}\cr#1\crcr}}}

\newcommand{\e}{\mbox{}}
\definecolor{boxgray}{gray}{.7}
\newcommand{\cb}{\color{boxgray}\rule{1\cellsize}{1\cellsize}\hspace{-\cellsize}\usebox2}

\newcommand{\stab}[3]{\begin{array}{c}\rnode{#1}{\tableau{#2}}\\\rnode{#1#1}{_{#3}}\end{array}}
\newcommand{\smstab}[3]{\begin{array}{c}\rnode{#1}{\smtableau{#2}}\\\rnode{#1#1}{_{#3}}\end{array}}

\newcommand{\sbull}[2]{\begin{array}{c}\rnode{#1}{\B}\\[-1ex]\rnode{#1#1}{\makebox[0pt]{$_{#2}$}}\end{array}}

\definecolor{lightgray}{gray}{.85}

\begin{document}


\title[LLT and Macdonald positivity]{Dual equivalence graphs and \\ 
  a combinatorial proof of LLT and Macdonald positivity}  

\author[S. Assaf]{Sami H. Assaf}
\address{Department of Mathematics, University of Southern California, Los Angeles, CA 90089-2532}
\email{shassaf@usc.edu}
\thanks{Work supported in part by NSF MSPRF DMS-0703567.}

\subjclass[2000]{Primary 05E10; Secondary 05A30, 33D52}



\keywords{LLT polynomials, Macdonald polynomials, dual equivalence
  graphs, quasisymmetric functions, Schur positivity}

\begin{abstract}
  We make a systematic study of a new combinatorial construction
  called a dual equivalence graph. We axiomatize these graphs and
  prove that their generating functions are symmetric and Schur
  positive. By constructing a graph on ribbon tableaux which we
  transform into a dual equivalence graph, we give a combinatorial
  proof of the symmetry and Schur positivity of the ribbon tableaux
  generating functions introduced by Lascoux, Leclerc and
  Thibon. Using Haglund's formula for the transformed Macdonald
  polynomials, this also gives a combinatorial formula for the Schur
  expansion of Macdonald polynomials.
\end{abstract}

\maketitle

%
\section{Introduction}
%
\label{sec:introduction}

The immediate purpose of this paper is to give a combinatorial formula
for the Schur coefficients of LLT polynomials which, as a corollary,
yields a combinatorial formula for the Schur coefficients of Macdonald
polynomials. Our real purpose, however, is not only to obtain these
results, but also to introduce a new combinatorial construction,
called a \emph{dual equivalence graph}, by which one can establish the
symmetry and Schur positivity of functions expressed in terms of
monomials.

The transformed Macdonald polynomials, $\widetilde{H}_{\mu}(x;q,t)$, a
transformation of the polynomials introduced by Macdonald
\cite{Macdonald1988} in 1988, are defined to be the unique symmetric
functions satisfying certain triangularity and orthogonality
conditions. The existence of functions satisfying these conditions is
a theorem, from which it follows that the $\widetilde{H}_{\mu}(x;q,t)$
form a basis for symmetric functions in two additional parameters. The
\emph{Kostka-Macdonald coefficients}, denoted
$\widetilde{K}_{\lambda,\mu}(q,t)$, give the change of basis from
Macdonald polynomials to Schur functions, namely,
\begin{displaymath}
  \widetilde{H}_{\mu}(x;q,t) = \sum_{\lambda}
  \widetilde{K}_{\lambda,\mu}(q,t) s_{\lambda}(x) .
\end{displaymath}
A priori, $\widetilde{K}_{\lambda,\mu}(q,t)$ is a rational function in
$q$ and $t$ with rational coefficients,
i.e. $\widetilde{K}_{\lambda,\mu}(q,t) \in \mathbb{Q}(q,t)$.

The Macdonald Positivity Theorem \cite{Haiman2001}, first conjectured
by Macdonald in 1988 \cite{Macdonald1988}, states that
$\widetilde{K}_{\lambda,\mu} (q,t)$ is in fact a polynomial in $q$ and
$t$ with nonnegative integer coefficients,
i.e. $\widetilde{K}_{\lambda,\mu} (q,t) \in \mathbb{N}[q,t]$. Garsia
and Haiman \cite{GaHa1993} conjectured that the transformed Macdonald
polynomials $\widetilde{H}_{\mu}(x;q,t)$ could be realized as the
bi-graded characters of certain modules for the diagonal action of the
symmetric group $S_n$ on two sets of variables. Once resolved, this
conjecture gives a representation theoretic interpretation of
Kostka-Macdonald coefficients as the graded multiplicity of an
irreducible representation in the Garsia-Haiman module, and hence
$\widetilde{K}_{\lambda,\mu} (q,t) \in \mathbb{N}[q,t]$. Following an
idea outlined by Procesi, Haiman \cite{Haiman2001} proved this
conjecture by analyzing the algebraic geometry of the isospectral
Hilbert scheme of $n$ points in the plane, consequently establishing
Macdonald Positivity. This proof, however, is purely geometric and
does not offer a combinatorial interpretation for
$\widetilde{K}_{\lambda,\mu} (q,t)$.

The LLT polynomial $\widetilde{G}_{\mu}^{(k)}(x;q)$, originally
defined by Lascoux, Leclerc and Thibon \cite{LLT1997} in 1997, is the
$q$-generating function of $k$-ribbon tableaux of shape $\mu$ weighted
by a statistic called cospin. By the Stanton-White correspondence
\cite{StWh1985}, $k$-ribbon tableaux are in bijection with certain
$k$-tuples of tableaux, from which it follows that LLT polynomials are
$q$-analogs of products of Schur functions. More recently, an
alternative definition of
$\widetilde{G}_{\boldsymbol{\mu}}^{(k)}(x;q)$ as the $q$-generating
function of $k$-tuples of semi-standard tableaux of shapes
$\boldsymbol{\mu} = (\mu^{(0)}, \ldots, \mu^{(k-1)})$ weighted by a
statistic called $k$-inversions is given in \cite{HHLRU2005}.

Using Fock space representations of quantum affine Lie algebras
constructed by Kashiwara, Miwa and Stern \cite{KMS1995}, Lascoux,
Leclerc and Thibon \cite{LLT1997} proved that
$\widetilde{G}^{(k)}_{\mu}(x;q)$ is a symmetric function.  Thus we may
define the Schur coefficients, $\widetilde{K}^{(k)}_{\lambda,\mu}(q)$,
by
\begin{displaymath}
  \widetilde{G}^{(k)}_{\mu}(x;q) = \sum_{\lambda}
  \widetilde{K}^{(k)}_{\lambda,\mu}(q) s_{\lambda}(x) .
\end{displaymath}
Using Kazhdan-Lusztig theory, Leclerc and Thibon \cite{LeTh2000}
proved that $\widetilde{K}^{(k)}_{\lambda,\mu}(q) \in \mathbb{N}[q]$
for straight shapes $\mu$. Grojnowski and Haiman \cite{GrHa2007}
report to have extended this to skew shapes. The proof of
positivity is by a geometric argument, and as such offers no
combinatorial description for $\widetilde{K}^{(k)}_{\lambda,\mu}(q)$.

In 2004, Haglund \cite{Haglund2004} conjectured a combinatorial
formula for the monomial expansion of
$\widetilde{H}_{\mu}(x;q,t)$. Haglund, Haiman and Loehr \cite{HHL2005}
proved this formula using an elegant combinatorial argument, but this
does not prove that $\widetilde{K}_{\lambda,\mu} (q,t) \in
\mathbb{N}[q,t]$ since monomials are not Schur positive.  Combining
Theorem 2.3, Proposition 3.4 and equation (23) from \cite{HHL2005},
Haglund's formula expresses $\widetilde{H}_{\mu}(x;q,t)$ as a positive
sum of LLT polynomials $\widetilde{G}^{(\mu_1)}_{\nu}(x;q)$ for
certain skew shapes $\nu$ depending on $\mu$. Therefore a proof of LLT
positivity for skew shapes would also provide a proof of Macdonald
positivity. One of the main purposes of this paper is to give a
combinatorial proof of LLT positivity for arbitrary shapes, thereby
completing the combinatorial proof of Macdonald positivity from
Haglund's formula.

Combinatorial formulas for $\widetilde{K}^{(k)}_{\lambda,\mu}(q)$ and
$\widetilde{K}_{\lambda,\mu}(q,t)$ have been found for certain special
cases.  In 1995, Carr\'{e} and Leclerc \cite{CaLe1995} gave a
combinatorial interpretation of $\widetilde{K}^{(2)}_{\lambda,\mu}(q)$
in their study of $2$-ribbon tableaux, though a complete proof of
their result wasn't found until 2005 by van Leeuwen
\cite{vanLeeuwen2005} using the theory of crystal graphs. Also in
1995, Fishel \cite{Fishel1995} gave the first combinatorial
interpretation for $\widetilde{K}_{\lambda,\mu}(q,t)$ when $\mu$ is a
partition with $2$ columns using rigged configurations. Other
techniques have also led to formulas for the $2$ column Macdonald
polynomials \cite{Zabrocki1999,LaMo2003,Haglund2004}, but in all
cases, finding extensions for these formulas has proven elusive.

In this paper, we consider the dual equivalence relation on standard
tableaux defined in \cite{Haiman1992}. From this relation, Haiman
suggested defining an edge-colored graph on standard tableaux and
investigating how this graph may be related to the crystal graph on
semi-standard tableaux. The result of this idea is a new combinatorial
method for establishing the Schur positivity of a function expressed
in terms of monomials. We apply this method to LLT polynomials to
obtain a combinatorial proof that
$\widetilde{K}^{(k)}_{\lambda,\mu}(q)$ and
$\widetilde{K}_{\lambda,\mu}(q,t)$ are nonnegative integer
polynomials.

This paper is organized as follows. In
Section~\ref{sec:preliminaries}, we review symmetric functions and the
associated tableaux combinatorics.  The theory of dual equivalence
graphs is developed in Section~\ref{sec:deg}, beginning in
Section~\ref{sec:deg-standard} with a review of dual equivalence and
the construction of the graphs suggested by Haiman. In
Section~\ref{sec:deg-general}, we define a \emph{dual equivalence
  graph} and present the structure theorem stating that every dual
equivalence graph is isomorphic to one of the graphs from
Section~\ref{sec:deg-standard}. On the symmetric function level, this
shows that the generating function of a dual equivalence graph is
symmetric and Schur positive and gives a combinatorial interpretation
for the Schur coefficients. The proof of the theorem is left to
Section~\ref{sec:deg-proof}.

The remainder of this paper contains the first application of this
theory, beginning in Section~\ref{sec:llt} with the construction of a
graph on $k$-tuples of tableaux.  We present a reformulation of LLT
polynomials in Section~\ref{sec:llt-words}, and use it to describe the
vertices and signatures of the graph. The edges are constructed in
Section~\ref{sec:llt-edges} using a natural analog of dual
equivalence. While these graphs are not, in general, dual equivalence
graphs, we show in Section~\ref{sec:Dgraphs} that they can be
transformed into dual equivalence graphs in a natural way that
preserves the generating function. In particular, connected components
of these graphs are Schur positive. The main consequence of this is a
purely combinatorial proof of the symmetry and Schur positivity of LLT
and Macdonald polynomials.

Examples of the graphs and transformations introduced in this paper are
given in the appendices.

\newpage
\begin{center}
{\sc Acknowledgments}
\end{center}

The author is grateful to Mark Haiman for inspiring and helping to
develop many of the ideas contained in this paper and in its precursor
\cite{Assaf2007}. The author also thanks A. Garsia and G. Musiker for
helping to implement the algorithms described in
Section~\ref{sec:Dgraphs} in Maple. Finally, the author is indebted to
M. Haiman, J. Haglund, S. Billey, N. Bergeron and the referees for
carefully reading earlier drafts and providing feedback that greatly
improved the exposition.

%
\section{Preliminaries}
%
\label{sec:preliminaries}

\subsection{Partitions and tableaux}
\label{sec:pre-partitions}

We represent an integer \emph{partition} $\lambda$ by the decreasing
sequence of its (nonzero) parts
$$
\lambda = (\lambda_1,\lambda_2, \ldots, \lambda_l), \;\;\;\;\;
\lambda_1 \geq \lambda_2 \geq \cdots \geq \lambda_l > 0 .
$$
We denote the size of $\lambda$ by $|\lambda| = \sum_i \lambda_i$ and
the length of $\lambda$ by $l(\lambda) = \max \{i : \lambda_i >
0\}$. If $|\lambda| = n$, we say that $\lambda$ is a \emph{partition of
  $n$}. Let $\geq$ denote the \emph{dominance partial ordering} on
partitions of $n$, defined by
\begin{eqnarray}
  \lambda \geq \mu & \Leftrightarrow & \lambda_1 + \lambda_2 + \cdots +
  \lambda_i \geq \mu_1 + \mu_2 + \cdots + \mu_i \;\;\; \forall \; i .
\label{eqn:dominance}
\end{eqnarray}

A \emph{composition} $\pi$ is a finite sequence of non-negative
integers $\pi = (\pi_1, \pi_2, \ldots, \pi_m), \pi_i \geq 0$.

The \emph{Young diagram} of a partition $\lambda$ is the set of points
$(i,j)$ in the $\mathbb{Z} \times\mathbb{Z}$ lattice such that $1 \leq
i \leq \lambda_j$. We draw the diagram so that each point $(i,j)$ is
represented by the unit cell southwest of the point; see
Figure~\ref{fig:5441}.  Abusing notation, we write $\lambda$ for both
the partition and its diagram.

\begin{figure}[ht]
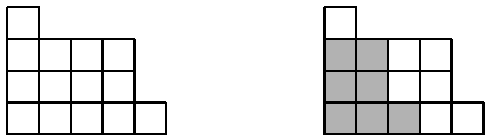

  \begin{displaymath}
    \tableau{ \e  \\
      \e & \e & \e & \e \\
      \e & \e & \e & \e \\
      \e & \e & \e & \e & \e} 
    \hspace{5\cellsize}
    \tableau{ \e  \\
      \cb & \cb & \e & \e \\
      \cb & \cb & \e & \e \\
      \cb & \cb & \cb & \e & \e} 
  \end{displaymath}
  \caption{\label{fig:5441} The Young diagram for $(5,4,4,1)$ and the
    skew diagram for $(5,4,4,1)/(3,2,2)$.}
\end{figure}

For partitions $\lambda,\mu$, we write $\mu \subset \lambda$ whenever
the diagram of $\mu$ is contained within the diagram of $\lambda$;
equivalently $\mu_i \leq \lambda_i$ for all $i$. In this case, we
define the \emph{skew diagram} $\lambda / \mu$ to be the set theoretic
difference $\lambda - \mu$, e.g. see Figure~\ref{fig:5441}. For our
purposes, we depart from the norm by \emph{not} identifying skew
shapes that are translates of one another.  A \emph{connected skew
  diagram} is one where exactly one cell has no cell immediately north
or west of it, and exactly one cell has no cell immediately south or
east of it.  A \emph{ribbon}, also called a \emph{rim hook}, is a
connected skew diagram containing no $2\times 2$ block.

A \emph{filling} of a (skew) diagram $\lambda$ is a map $S : \lambda
\rightarrow \mathbb{Z}_+$. A \emph{semi-standard Young tableau} is a
filling which is weakly increasing along each row and strictly
increasing along each column. A semi-standard Young tableau is
\emph{standard} if it is a bijection from $\lambda$ to $[n]$, where
$[n] = \{1,2,\ldots,n\}$. For $\lambda$ a diagram of size $n$, define
\begin{eqnarray*}
  \mathrm{SSYT}(\lambda) & = & \{\mbox{semi-standard tableaux}\; 
  T : \lambda \rightarrow \mathbb{Z}_+ \}, \\
  \mathrm{SYT}(\lambda) & = & \{\mbox{standard tableaux}\; 
  T : \lambda \tilde{\rightarrow} [n]\}. 
\end{eqnarray*}
For $T \in \mathrm{SSYT}(\lambda)$, we say that $T$ has \emph{shape}
$\lambda$.  If $T$ contains entries $1^{\pi_1}, 2^{\pi_2}, \ldots$ for
some composition $\pi$, then we say $T$ has \emph{weight} $\pi$.

\begin{figure}[ht]
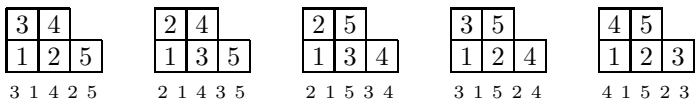

  \begin{displaymath}
    \begin{array}{ccccc}
          \stab{a}{3 & 4 \\ 1 & 2 & 5}{3 \ 1 \ 4 \ 2 \ 5} &
          \stab{b}{2 & 4 \\ 1 & 3 & 5}{2 \ 1 \ 4 \ 3 \ 5} &
          \stab{c}{2 & 5 \\ 1 & 3 & 4}{2 \ 1 \ 5 \ 3 \ 4} &
          \stab{d}{3 & 5 \\ 1 & 2 & 4}{3 \ 1 \ 5 \ 2 \ 4} &
          \stab{e}{4 & 5 \\ 1 & 2 & 3}{4 \ 1 \ 5 \ 2 \ 3}
    \end{array}
  \end{displaymath}
  \caption{\label{fig:SYT32} The standard Young tableaux of shape
    $(3,2)$ with their content reading words.}
\end{figure}

The \emph{content} of a cell of a diagram indexes the diagonal on
which it occurs, i.e. $c(x) = i-j$ when the cell $x$ lies in position
$(i,j) \in \mathbb{Z}_+ \times \mathbb{Z}_+$. The \emph{content
  reading word} of a semi-standard tableaux is obtained by reading the
entries in increasing order of content, going southwest to northeast
along each diagonal (on which the content is constant). For examples,
see Figure~\ref{fig:SYT32}.

\subsection{Symmetric functions}
\label{sec:pre-functions}

We have the familiar integral bases for $\Lambda$, the ring of
symmetric functions, from \cite{Macdonald1995}: the monomial symmetric
functions $m_{\lambda}$, the elementary symmetric functions
$e_{\lambda}$, the complete homogeneous symmetric functions
$h_{\lambda}$, and, most importantly, the \emph{Schur functions},
$s_{\lambda}$, which may be defined in several ways. For the purposes
of this paper, we take the tableau approach:
\begin{equation}
  s_{\lambda}(x) = \sum_{T \in \mathrm{SSYT}(\lambda)} x^{T} ,
\label{eqn:s}
\end{equation}
where $x^T$ is the monomial $x_{1}^{\pi_1} x_{2}^{\pi_2} \cdots$ when
$T$ has weight $\pi$. This formula also defines the \emph{skew Schur
  functions}, $s_{\lambda/\mu}$, by taking the sum over semi-standard
tableaux of shape $\lambda/\mu$. 

The \emph{Kostka numbers}, $K_{\lambda,\mu}$, give the change of basis
from the complete homogeneous symmetric functions to the Schur
functions and, dually, the change of basis from Schur functions to
monomial symmetric functions, i.e.
$$
h_{\mu} = \sum_{\lambda} K_{\lambda,\mu} s_{\lambda} ; \;\;\;
s_{\lambda} = \sum_{\mu} K_{\lambda,\mu} m_{\mu} .
$$
In particular, $K_{\lambda,\mu}$ is the number of semi-standard Young
tableaux of shape $\lambda$ and weight $\mu$. For example,
$K_{(3,2),(1^5)}=5$ corresponding to the five standard Young tableaux
of shape $(3,2)$ in Figure~\ref{fig:SYT32}. Since the Schur functions
are the characters of the irreducible representations of
$\mathrm{GL}_n$, the Kostka numbers also give weight multiplicities
for $\mathrm{GL}_n$ modules. Throughout this paper, we are interested
in certain one- and two-parameter generalizations of the Kostka
numbers.

As we shall see in Section~\ref{sec:deg}, it will often be useful to
express a function in terms of Gessel's fundamental quasi-symmetric
functions \cite{Gessel1984} rather than monomials. For $\sigma \in
\{\pm 1\}^{\nmo}$, the \emph{fundamental quasi-symmetric function}
$Q_{\sigma}(x)$ is defined by
\begin{equation}
  Q_{\sigma}(x) = \sum_{\substack{i_1 \leq \cdots \leq i_n \\ i_j =
      i_{j+1} \Rightarrow \sigma_j = +1}} x_{i_1} \cdots x_{i_n} .
\label{eqn:quasisym}
\end{equation}
We have indexed quasi-symmetric functions by sequences of $+1$'s and
$-1$'s, though by setting $D(\sigma) = \{ i | \sigma_i = -1\}$, we may
change the indexing to subsets of $[\nmo]$. Similarly, letting
$\pi(\sigma)$ be the composition defined by setting $\pi_1 + \cdots +
\pi_i$ to be the position of the $i$th $-1$, where here we regard
$\sigma_n=-1$ as the final $-1$, we may change the indexing to
compositions of $n$.

To connect quasi-symmetric functions with Schur functions, for $T$ a
standard tableau on $[n]$ with content reading word $w_{T}$, define
the \emph{descent signature} $\sigma(T) \in \{\pm1\}^{\nmo}$ by
\begin{equation}
  \sigma(T)_{i} \; = \; \left\{ 
    \begin{array}{ll}
      +1 & \; \mbox{if $i$ appears to the left of $\ipo$ in $w_T$} \\
      -1 & \; \mbox{if $\ipo$ appears to the left of $i$ in $w_T$}
    \end{array} \right. .
\label{eqn:sigma}
\end{equation}
For example, the descent signatures for the tableaux in
Figure~\ref{fig:SYT32} are $+-++, \ -+-+, \ -++-, \ +-+-, \ ++-+$,
from left to right.  Note that if we replace the content reading word
with either the row or column reading word, the resulting sequence in
\eqref{eqn:sigma} remains unchanged.

\begin{proposition}[\cite{Gessel1984}]
  The Schur function $s_{\lambda}$ is expressed in terms of
  quasi-symmetric functions by
  \begin{equation}
    s_{\lambda}(x) = \sum_{T \in \mathrm{SYT}(\lambda)} Q_{\sigma(T)}(x) .
  \label{eqn:quasi-s}
  \end{equation}
\label{prop:quasisym}
\end{proposition}

Comparing \eqref{eqn:s} with \eqref{eqn:quasi-s}, using
quasi-symmetric functions instead of monomials allows us to work with
standard tableaux rather than semi-standard tableaux. One advantage of
this formula is that unlike \eqref{eqn:s}, the right hand side of
\eqref{eqn:quasi-s} is finite. Continuing with the example in
Figure~\ref{fig:SYT32},
\begin{displaymath}
  s_{(3,2)}(x) = Q_{+-++}(x) + Q_{-+-+}(x) + Q_{-++-}(x) + Q_{+-+-}(x)
  + Q_{++-+}(x).
\end{displaymath}

\subsection{LLT polynomials}
\label{sec:pre-llt}

Lascoux, Leclerc and Thibon \cite{LLT1997} originally defined
$\widetilde{G}_{\mu}^{(k)}(x;q)$ to be the $q$-generating function of
$k$-ribbon tableaux of shape $\mu$ weighted by cospin. Below we give
an alternative definition of
$\widetilde{G}_{\boldsymbol{\mu}}^{(k)}(x;q)$ as the $q$-generating
function of $k$-tuples of semi-standard tableaux of shapes
$\boldsymbol{\mu} = (\mu^{(0)}, \ldots, \mu^{(k-1)})$ weighted by
$k$-inversions first presented in \cite{HHLRU2005}. For a detailed
account of the equivalence of these definitions (actually
$q^{a}\widetilde{G}_{\mu}^{(k)}(x;q) =
\widetilde{G}_{\boldsymbol{\mu}}^{(k)}(x;q)$ for a constant $a \geq 0$
depending on $\mu$), see \cite{HHLRU2005,Assaf2007}.

Extending prior notation, define
\begin{eqnarray*}
  \mathrm{SSYT}_k(\boldsymbol{\lambda}) & = & 
  \{\mbox{semi-standard $k$-tuples of tableaux of shapes
    $(\lambda^{(0)}, \ldots, \lambda^{(k-1)})$} \} , \\  
  \mathrm{SYT}_k(\boldsymbol{\lambda}) & = & 
  \{\mbox{standard $k$-tuples of tableaux of shapes $(\lambda^{(0)},
    \ldots, \lambda^{(k-1)})$} \}.  
\end{eqnarray*}
As with tableaux, if $\mathbf{T}=(T^{(0)}, \ldots, T^{(k-1)}) \in
\mathrm{SSYT}_k(\boldsymbol{\lambda})$ has entries $1^{\pi_1},
2^{\pi_2}, \ldots$, then we say that $\mathbf{T}$ has \emph{shape}
$\boldsymbol{\lambda}$ and \emph{weight} $\pi$. Note that a standard
$k$-tuple of tableaux has weight $(1^n)$, e.g. see
Figure~\ref{fig:LLT-inv}, and this is not the same as a $k$-tuple of
standard tableaux, which has weight $(1^{m_1},2^{m_2}, \ldots)$ where
$m_i$ is the number of shapes of size at least $i$.

\begin{figure}[ht]
  \begin{displaymath}
      \tableau{\\ 7 & 11 \\ 2 & 6 & 10} \hspace{2\cellsize}
      \tableau{\\ 8 \\ 1 & 12}          \hspace{2\cellsize}
      \makebox[0pt]{\rule[-2\cellsize]{.5pt}{\cellsize}}
      \rule[-2\cellsize]{\cellsize}{.5pt}  
      \hspace{2\cellsize}
      \tableau{9 \\ 3 & 5 \\ \cb & 4}
  \end{displaymath}
  \caption{\label{fig:LLT-inv}A standard $4$-tuple of shape $( \ (3,2),
    \ (2,1), \ \varnothing, \ (2,2,1)/(1) \ )$}
\end{figure}

For a $k$-tuple of (skew) shapes $(\lambda^{(0)}, \ldots,
\lambda^{(k-1)})$, define the \emph{shifted content} of a cell $x$ by
\begin{equation}
  \widetilde{c}(x) \; = \; k \cdot c(x) + i
\label{eqn:shifted-content}
\end{equation}
when $x$ is a cell of $\lambda^{(i)}$, where $c(x)$ is the usual
content of $x$ regarded as a cell of $\lambda^{(i)}$. For $\mathbf{T}
\in \mathrm{SSYT}_k$, let $\mathbf{T}(x)$ denote the entry of the cell
$x$ in $\mathbf{T}$. Define the \emph{set of $k$-inversions of
  $\mathbf{T}$} by
\begin{equation}
  \mathrm{Inv}_k(\mathbf{T}) = \{ (x,y) \; | \; k > \widetilde{c}(y) -
  \widetilde{c}(x) > 0 \; \mbox{and} \; \mathbf{T}(x) > \mathbf{T}(y) \}.
\label{eqn:Invk-T}
\end{equation}
Then the \emph{$k$-inversion number of $\mathbf{T}$} is given by
\begin{equation}
  \mathrm{inv}_k(\mathbf{T}) = \left| \mathrm{Inv}_k(\mathbf{T}) \right| .
\label{eqn:invk-T}
\end{equation}

For example, suppose $\mathbf{T}$ is the $4$-tuple of tableaux in
Figure~\ref{fig:LLT-inv}. Since $\mathbf{T}$ is standard, let us abuse
notation by representing a cell of $\mathbf{T}$ by the entry it
contains. Then the set of $4$-inversions is
\begin{displaymath}
  \mathrm{Inv}_4(\mathbf{T}) = \left\{ \begin{array}{c}
      (9,7), \ (9,8), \ ( 7,3), \ (8,3), \ (8,2), \ (3,2), \ ( 3,1), \\
      ( 2,1), \ (11,1), \ (11,5), \ ( 6,4), \ (12,4), \ (12,10)
    \end{array} \right\} ,
\end{displaymath}
and so $\mathrm{inv}_4(\mathbf{T}) = 13$.  

By \cite{HHLRU2005}, the LLT polynomial
$\widetilde{G}^{(k)}_{\boldsymbol{\mu}}(x;q)$ is given by
\begin{equation}
  \widetilde{G}^{(k)}_{\boldsymbol{\mu}}(x;q) \; = \; 
  \sum_{\mathbf{T} \in \mathrm{SSYT}_{k}(\boldsymbol{\mu})}
  q^{\mathrm{inv}_k(\mathbf{T})} x^{\mathbf{T}} ,
\label{eqn:llt}
\end{equation}
where $x^{\mathbf{T}}$ is the monomial $x_1^{\pi_1}
x_{2}^{\pi_2}\cdots$ when $\mathbf{T}$ has weight $\pi$.

Notice that when $q=1$, \eqref{eqn:llt} reduces to a product of Schur
functions:
\begin{equation}
  \sum_{\mathbf{T} \in \mathrm{SSYT}_{k}(\boldsymbol{\lambda})} x^{\mathbf{T}} \; = \; \prod_{i=0}^{k-1}
  \sum_{T^{(i)} \in  \mathrm{SSYT}(\lambda^{(i)})} x^{T^{(i)}} \; = \;
  \prod_{i=0}^{k-1} s_{\lambda^{(i)}}(x) .
\label{eqn:schurprod} 
\end{equation}

Define the \emph{content reading word} of a $k$-tuple of tableaux to
be the word obtained by reading entries in increasing order of shifted
content and reading diagonals southwest to northeast. For the example
in Figure~\ref{fig:LLT-inv}, the content reading word is
$(9,7,8,3,2,11,1,5,6,12,4,10)$.

For $\mathbf{T}$ a standard $k$-tuple of tableaux, define
$\sigma(\mathbf{T})$ analogously to \eqref{eqn:sigma} using the
content reading word. Expressed in terms of quasi-symmetric functions,
\eqref{eqn:llt} becomes
\begin{equation}
  \widetilde{G}^{(k)}_{\boldsymbol{\mu}}(x;q) \; = \; \sum_{\mathbf{T} \in \mathrm{SYT}_{k}(\boldsymbol{\mu})}
  q^{\mathrm{inv}_k(\mathbf{T})} Q_{\sigma(\mathbf{T})}(x).
\label{eqn:llt-quasi}
\end{equation}
One of the main goals of this paper is to understand the Schur
coefficients of $\widetilde{G}^{(k)}_{\boldsymbol{\mu}}(x;q)$ defined by
\begin{displaymath}
  \widetilde{G}^{(k)}_{\boldsymbol{\mu}}(x;q) = \sum_{\lambda}
  \widetilde{K}^{(k)}_{\lambda,\boldsymbol{\mu}}(q) s_{\lambda}(x) .
\end{displaymath}
In particular, we will show that
$\widetilde{K}^{(k)}_{\lambda,\boldsymbol{\mu}}(q)$ is a polynomial in
$q$ with nonnegative integer coefficients.

\subsection{Macdonald polynomials}
\label{sec:pre-mac}

The transformed Macdonald polynomials $\widetilde{H}_{\mu}(x;q,t)$
were originally defined by Macdonald \cite{Macdonald1988} to be the
unique symmetric functions satisfying certain orthogonality and
triangularity conditions. Haglund's monomial expansion for Macdonald
polynomials \cite{Haglund2004,HHL2005} gives an alternative
combinatorial definition of $\widetilde{H}_{\mu}(x;q,t)$ as the
$q,t$-generating functions for fillings of the diagram of $\mu$,
e.g. see Figure~\ref{fig:filling}. Since the proof of the equivalence
of these two definitions is purely combinatorial \cite{HHL2005}, we
will use the latter characterization.

For a cell $x$ in the diagram of $\lambda$, define the \emph{arm of
  $x$} to be the set of cells east of $x$, and the \emph{leg of $x$}
to be the set of cells north of $x$. Denote the sizes of the arm and
leg of $x$ by $a(x)$ and $l(x)$, respectively. For example, letting
$x$ denote the cell with entry $3$ in the filling in
Figure~\ref{fig:filling}, the arm of $x$ consists of the cells with
entries $4$ and $10$ and the leg of $x$ consists of the cell with
entry $14$, and so we have $a(x)=2$ and $l(x)=1$.

\begin{figure}[ht]
  \begin{displaymath}
    \tableau{ 5 \\
      11 & 14 &  9 &  2 \\
      6 &  3 &  4 & 10 \\
      8 &  1 & 13 &  7 & 12}
  \end{displaymath}
  \caption{\label{fig:filling}A standard filling of shape $(5,4,4,1)$.}
\end{figure}

Let $S$ be a filling of a partition $\lambda$. A \emph{descent} of $S$
is a cell $c$ of $\lambda$, not in the first row, such that the entry
in $c$ is greater than the entry in the cell immediately south of
$c$. Denote by $\mathrm{Des}(S)$ the set of all descents of $S$, i.e.
\begin{equation}
  \mathrm{Des}(S) = \{ (i,j) \in \lambda \; | \; j>1 \; \mbox{and} \; S(i,j) > S(i,j-1) \}.
\label{eqn:Des}
\end{equation}
Define the \emph{major index} of $S$, denoted $\mathrm{maj}(S)$, by
\begin{equation}
  \mathrm{maj}(S) \; \stackrel{\mbox{\scriptsize def}}{=} \; 
  \left| \mathrm{Des}(S) \right| + \sum_{c \in \mathrm{Des}(S)} l(c) . 
\label{eqn:maj}
\end{equation}
Note that for $\mu=(1^n)$, this gives the usual major index for the
reading word of the filling. 

For example, let $S$ be the filling in Figure~\ref{fig:filling}. As
before, let us abuse notation by representing a cell of $S$ by the
entry which it contains. Then the descents of $S$ are $\mathrm{Des}(S)
= \left\{11, \ 14, \ 9, \ 3, \ 10 \right\}$, and so the major index of
$S$ is $\mathrm{maj}(S) = 5 + (1 + 0 + 0 + 1 + 1) = 8$.

An ordered pair of cells $(c,d)$ is called \emph{attacking} if $c$ and
$d$ lie in the same row with $c$ to the west of $d$, or if $c$ is in
the row immediately north of $d$ and $c$ lies strictly east of $d$.
An \emph{inversion pair} of $S$ is an attacking pair $(c,d)$ such that
the entry in $c$ is greater than the entry in $d$. Denote by
$\mathrm{Inv}(S)$ the set of inversion pairs of $S$, i.e.
\begin{equation}
  \mathrm{Inv}(S) = \left\{ \left((i,j), (g,h)\right) \in \lambda \; \left| \; 
    \begin{array}{c}
      j=h \;\mbox{and}\; i<g \;\mbox{or}\; j=h+1 \;\mbox{and}\; g<i, \\
      \mbox{and}\; S(i,j) > S(g,h) 
    \end{array} \right. \right\}.
\label{eqn:Inv}
\end{equation}
Define the \emph{inversion number} of $S$, denoted $\mathrm{inv}(S)$, by
\begin{equation}
  \mathrm{inv}(S) \; \stackrel{\mbox{\scriptsize def}}{=} \; 
  \left| \mathrm{Inv}(S) \right| - \sum_{c \in \mathrm{Des}(S)} a(c) .
\label{eqn:inv}
\end{equation}
Note that for $\mu=(n)$, this gives the usual inversion number for the
reading word of the filling.

For our running example, the inversion pairs of $S$ are given by
\begin{displaymath}
  \mathrm{Inv}(S) = \left\{ \begin{array}{rrrrrr}
      (11,9), & (14,2), & (9,6), & ( 6,4), & (10,1), & (13,7), \\
      (11,2), & (14,6), & (9,3), & ( 4,1), & ( 8,1), & (13,12) \\
      (14,9), & ( 9,2), & (6,3), & (10,8), & ( 8,7), &
    \end{array} \right\} ,
\end{displaymath}
and so the inversion number of $S$ is $\mathrm{inv}(S) = 17 - (3 + 2 +
1 + 2 + 0) = 9$.

\begin{remark}
  If $c \in \mathrm{Des}(S)$, say with $d$ the cell of $S$ immediately
  south of $c$, then for every cell $e$ of the arm of $c$, the entry
  in $e$ is either bigger than the entry in $d$ or smaller than the
  entry in $c$ (or both). In the former case, $(e,d)$ will form an
  inversion pair, and in the latter case, $(c,e)$ will form an
  inversion pair. Thus every triple of cells $(c,e,d)$ with $d$
  immediately south of $c$ and $e$ in the arm of $c$ contributes at
  least one inversion to $\mathrm{inv}(S)$, and so $\mathrm{inv}(S)$
  is a non-negative integer.
\label{rmk:inv-pos}
\end{remark}

By \cite{HHL2005}, the transformed Macdonald polynomial
$\widetilde{H}_{\mu}(x;q,t)$ is given by
\begin{equation}
  \widetilde{H}_{\mu}(x;q,t) = \sum_{S : \mu \rightarrow \mathbb{Z}_+}
  q^{\mathrm{inv}(S)} t^{\mathrm{maj}(S)} x^{S}  = \sum_{S : \mu
    \stackrel{\sim}{\rightarrow} [n]} q^{\mathrm{inv}(S)} t^{\mathrm{maj}(S)}
  Q_{\sigma(S)}, 
  \label{eqn:haglund}
\end{equation}
where $\sigma(S)$ is defined analogously to \eqref{eqn:sigma} using
the \emph{row reading word} of a standard filling $S$. For example,
the row reading word for the standard filling in
Figure~\ref{fig:filling} is $(5,11,14,9,2,6,3,4,10,8,1,13,7,12)$.
Again, our main objective is to understand the Schur coefficients
defined by
\begin{equation}
  \widetilde{H}_{\mu}(x;q,t) = \sum_{\lambda} \widetilde{K}_{\lambda,\mu}(q,t)
  s_{\lambda}(x) .
  \label{eqn:qt-kostka}
\end{equation}
In this paper, we give a combinatorial proof that
$\widetilde{K}_{\lambda,\mu}(q,t)$ is a polynomial in $q$ and $t$ with
nonnegative integer coefficients. This proof is a corollary to the
proof for $\widetilde{K}^{(k)}_{\lambda,\boldsymbol{\mu}}(q)$ as we
now explain.

The expression in \eqref{eqn:haglund} is related to LLT polynomials as
follows. Let $D$ be a possible descent set for $\mu$, i.e. $D$ is a
collection cells of $\mu/(\mu_1)$. For $i=1,\ldots,\mu_1$, let
$\mu_D^{(i-1)}$ be the ribbon obtained from the $i$th column of $\mu$
by putting the cell $(i,j)$ immediately south of $(i,j+1)$ if $(i,j+1)
\in D$ and immediately east of $(i,j+1)$ otherwise. Translate each
$\mu_D^{(i)}$ so that the southeastern most cell has shifted content
$n+i$ for some (any) fixed integer $n$. Then each filling $S$ of shape
$\mu$ with $\mathrm{Des}(S)=D$ may be regarded as a semi-standard
$\mu_1$-tuple of tableaux of shape $\boldsymbol{\mu}_{D}$, denoted
$\mathbf{S}$.  For example, the filling $S$ of shape $(5,4,4,1)$ in
Figure~\ref{fig:filling} corresponds to the $5$-tuple of ribbons of
shapes $(3,3,3,2)/(3,3,1), \ (1,1,1), \ (2,2,1)/(2), \ (2,2,1)/(2,1),
\ (1)$; see Figure~\ref{fig:ribbons}.

\begin{figure}[ht]
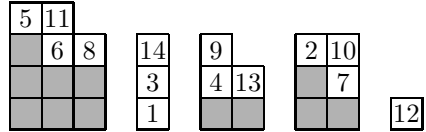

  \begin{displaymath}
    \tableau{5 & 11 \\ \cb & 6 & 8 \\ \cb & \cb & \cb \\ \cb & \cb & \cb} \hspace{\cellsize}
    \tableau{\\ 14 \\ 3 \\ 1} \hspace{\cellsize}
    \tableau{\\ 9 \\ 4 & 13 \\ \cb & \cb} \hspace{\cellsize}
    \tableau{\\ 2 & 10 \\ \cb & 7 \\ \cb & \cb} \hspace{\cellsize}
    \tableau{\\ \\  \\ 12} 
  \end{displaymath}
  \caption{\label{fig:ribbons} A standard filling of shape $(5,4,4,1)$
    transformed into a $5$-tuple of ribbons of shapes
    $(3,3,3,2)/(3,3,1), \ (1,1,1), \ (2,2,1)/(2), \ (2,2,2)/(2,1), \
    (1)$.}
\end{figure}

For this correspondence, it is crucial that we do not identify skew
shapes that are translates of one another. For example, the row
reading word of the filling in Figure~\ref{fig:filling} is precisely
the content reading word of $5$-tuple in Figure~\ref{fig:ribbons}, but
this is not the case if the first tableau is instead considered to
have shape $(3,2)/(1)$. Furthermore, the inversion pairs of $S$ as
defined in \eqref{eqn:Inv} correspond precisely with the
$\mu_1$-inversions of $\mathbf{S}$ as defined in
\eqref{eqn:Invk-T}. Since the major index statistic depends only on
the descent set, for a given descent set $D$ we may define
$\mathrm{maj}(D)$ by $\mathrm{maj}(D) = \mathrm{maj}(S)$ for any
filling $S$ with $\mathrm{Des}(S) = D$. Similarly, define $a(D) =
\sum_{c \in D} a(c)$. Then we may rewrite \eqref{eqn:haglund} in terms
of LLT polynomials as
\begin{equation}
  \widetilde{H}_{\mu}(x;q,t) = \sum_{D \subseteq \mu/(\mu_1)}
  q^{-a(D)} t^{\mathrm{maj}(D)}
  \widetilde{G}_{\boldsymbol{\mu}_D}^{(\mu_1)}(x;q) . 
\label{eqn:hag-llt}
\end{equation}
Note that each term of
$\widetilde{G}_{\boldsymbol{\mu}_D}^{(\mu_1)}(x;q)$ contains a factor
of $q^{a}$ for some $a \geq a(D)$ (in fact, this is the same constant
mentioned in Section~\ref{sec:pre-llt}). In terms of Schur expansions,
\eqref{eqn:hag-llt} may also be expressed as
\begin{equation}
  \widetilde{K}_{\lambda,\mu}(q,t) = \sum_{D \subseteq \mu/(\mu_1)} 
  q^{-a(D)} t^{\mathrm{maj}(D)}
  \widetilde{K}_{\lambda,\boldsymbol{\mu}_D}^{(\mu_1)}(q) . 
\label{eqn:hag-llt-kostka}
\end{equation}
By the previous remark, proving
$\widetilde{K}_{\lambda,\boldsymbol{\mu}_D}^{(\mu_1)}(q) \in
\mathbb{N}[q]$ consequently proves $\widetilde{K}_{\lambda,\mu}(q,t)
\in \mathbb{N}[q,t]$.

%
\section{Dual equivalence graphs}
%
\label{sec:deg}

\subsection{The standard dual equivalence graph}
\label{sec:deg-standard}

Dual equivalence was first defined by Haiman \cite{Haiman1992} as a
relation on tableaux dual to \emph{jeu de taquin} equivalence under
the Schensted correspondence.  We use this relation to construct a
graph whose vertices are standard tableaux and edges are elementary
dual equivalence relations. Using quasi-symmetric functions, we define
the generating function on the vertices of these graphs, thereby
providing the connection with symmetric functions.

We begin by recalling the definition of dual equivalence on
permutations regarded as words on $[n]$, which we extend to standard
Young tableaux via the content reading word.

\begin{definition}[\cite{Haiman1992}]
  An \emph{elementary dual equivalence} on three consecutive letters
  $\triple$ of a permutation is given by switching the outer two
  letters whenever the middle letter is not $i$:
  $$
  \cdots\; i \;\cdots\;i\pm 1\;\cdots\;i\mp 1\;\cdots \equiv^{*}
  \cdots\;i\mp 1\;\cdots\;i\pm 1\;\cdots\; i \;\cdots
  $$
  Two permutations are \emph{dual equivalent} if they differ by some
  sequence of elementary dual equivalences. Two standard tableaux of
  the same shape are \emph{dual equivalent} if their content reading
  words are.
\label{defn:ede}
\end{definition}

Construct an edge-colored graph on standard tableaux of partition
shape from the dual equivalence relation in the following way.
Whenever two standard tableaux $T,U$ have content reading words that
differ by an elementary dual equivalence for $\triple$, connect $T$
and $U$ with an edge colored by $i$. Recall the definition of the
content reading word $w_{T}$ and the \emph{descent signature} of a
standard tableau $T$ from \eqref{eqn:sigma}:
\begin{displaymath}
  \sigma(T)_{i} \; = \; \left\{ 
    \begin{array}{ll}
      +1 & \; \mbox{if $i$ appears to the left of $\ipo$ in $w_{T}$} \\
      -1 & \; \mbox{if $\ipo$ appears to the left of $i$ in $w_{T}$}
    \end{array} \right. .
\end{displaymath}
We associate to each tableau $T$ the signature $\sigma(T)$. Several
examples are given in Figure~\ref{fig:G5}, and several more in
Appendix~\ref{app:DEGs}.

\begin{figure}[ht]
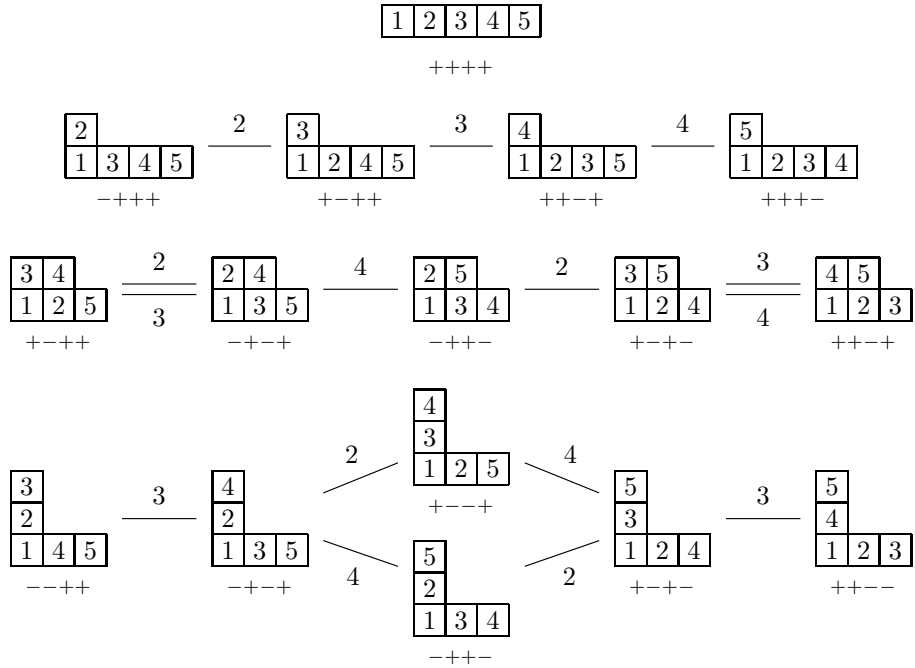

  \begin{center}
    \begin{displaymath}
      \begin{array}{c}
        \begin{array}{c}
          \stab{h}{1 & 2 & 3 & 4 & 5}{++++}
        \end{array} \\[5ex] 
        \begin{array}{\cs{6} \cs{6} \cs{6} c}
          \stab{h}{2 \\ 1 & 3 & 4 & 5}{-+++} &
          \stab{i}{3 \\ 1 & 2 & 4 & 5}{+-++} &
          \stab{j}{4 \\ 1 & 2 & 3 & 5}{++-+} &
          \stab{k}{5 \\ 1 & 2 & 3 & 4}{+++-}
        \end{array} \\[7ex]
        \begin{array}{\cs{7} \cs{7} \cs{7} \cs{7} c}
          \stab{a}{3 & 4 \\ 1 & 2 & 5}{+-++} &
          \stab{b}{2 & 4 \\ 1 & 3 & 5}{-+-+} &
          \stab{c}{2 & 5 \\ 1 & 3 & 4}{-++-} &
          \stab{d}{3 & 5 \\ 1 & 2 & 4}{+-+-} &
          \stab{e}{4 & 5 \\ 1 & 2 & 3}{++-+}
        \end{array} \\[6ex]
        \begin{array}{\cs{7} \cs{7} \cs{7} \cs{7} c}
          & & \stab{w}{4 \\ 3 \\ 1 & 2 & 5}{+--+} & & \\[-4ex]
              \stab{u}{3 \\ 2 \\ 1 & 4 & 5}{--++} &
              \stab{v}{4 \\ 2 \\ 1 & 3 & 5}{-+-+} & &
              \stab{y}{5 \\ 3 \\ 1 & 2 & 4}{+-+-} &
              \stab{z}{5 \\ 4 \\ 1 & 2 & 3}{++--} \\[-5ex]
          & & \stab{x}{5 \\ 2 \\ 1 & 3 & 4}{-++-} & &
        \end{array}
        \psset{nodesep=6pt,linewidth=.1ex}
        \ncline  {h}{i} \naput{2}
        \ncline  {i}{j} \naput{3}
        \ncline  {j}{k} \naput{4}
        \ncline[offset=2pt] {a}{b} \naput{2}
        \ncline[offset=2pt] {b}{a} \naput{3}
        \ncline             {b}{c} \naput{4}
        \ncline             {c}{d} \naput{2}
        \ncline[offset=2pt] {d}{e} \naput{3}
        \ncline[offset=2pt] {e}{d} \naput{4}
        \ncline {u}{v}  \naput{3}
        \ncline {v}{w}  \naput{2}
        \ncline {v}{x}  \nbput{4}
        \ncline {w}{y}  \naput{4}
        \ncline {x}{y}  \nbput{2}
        \ncline {y}{z}  \naput{3}
      \end{array}
    \end{displaymath}
    \caption{\label{fig:G5}The standard dual equivalence graphs
      $\G_{5}, \G_{4,1}, \G_{3,2}$ and $\G_{3,1,1}$.}
  \end{center}
\end{figure}

The connected components of the graph so constructed are the dual
equivalence classes of standard tableaux. Let $\G_{\lambda}$ denote
the subgraph on tableaux of shape $\lambda$. The following proposition
tells us that the $\G_{\lambda}$ exactly give the connected components
of the graph.

\begin{proposition}[\cite{Haiman1992}]
  Two standard tableaux on partition shapes are dual equivalent if and
  only if they have the same shape.
\label{prop:deshape}
\end{proposition}

Define the generating function associated to $\G_{\lambda}$ by
\begin{equation}
  \sum_{v \in V(\G_{\lambda})} Q_{\sigma(v)}(x) = s_{\lambda}(x).
\label{eqn:glamschur}
\end{equation}
By Proposition~\ref{prop:quasisym}, this is Gessel's quasi-symmetric
function expansion for a Schur function. In particular, the generating
function of any vertex-signed graph whose connected components are
isomorphic to the graphs $\G_{\lambda}$ is automatically Schur
positive. This observation is the main idea behind the following
method for establishing the symmetry and Schur positivity of a
function expressed in terms of fundamental quasi-symmetric
functions. We will realize the given function as the generating
function for a vertex-signed, edge-colored graph such that connected
components of the graph are isomorphic to the graphs
$\G_{\lambda}$. Therefore the connected components of the graph will
correspond precisely to terms in the Schur expansion of the given
function.

\subsection{Axiomatization of dual equivalence}
\label{sec:deg-general}

In this section, we characterize $\G_{\lambda}$ in terms of edges and
signatures so that we can readily identify those graphs that are
isomorphic to some $\G_{\lambda}$.

\begin{definition}
  A \emph{signed, colored graph of type $(n,N)$} consists of the
  following data:
  \begin{itemize}
    \item a finite vertex set $V$;
    \item a signature function $\sigma : V \rightarrow \{\pm
      1\}^{N-1}$;
    \item for each $1 < i < n$, a collection $E_i$ of pairs of
      distinct vertices of $V$.
  \end{itemize}
  We denote such a graph by $\G = (V,\sigma,E_{2} \cup\cdots\cup E_{\nmo})$ or
  simply $(V,\sigma,E)$.
\label{defn:scg}
\end{definition}

\begin{definition}
  A signed, colored graph $\G = (V,\sigma,E)$ of type $(n,N)$ is a
  \emph{dual equivalence graph of type $(n,N)$} if $n \leq N$ and the
  following hold:
  \begin{itemize}
    
  \item[(ax$1$)] For $w \in V$ and $1<i<n$, $\sigma(w)_{\imo} =
    -\sigma(w)_{i}$ if and only if there exists $x \in V$ such that
    $\{w,x\} \in E_{i}$. Moreover, $x$ is unique when it exists.

  \item[(ax$2$)] For $\{w,x\} \in E_{i}$, $\sigma(w)_j = -\sigma(x)_j$
    for $j=\imo,i$, and $\sigma(w)_h = \hspace{1ex}\sigma(x)_h$ for $h
    < \imt$ and $h > \ipo$.
      
  \item[(ax$3$)] For $\{w,x\} \in E_{i}$, if $\sigma(w)_{\imt} =
    -\sigma(x)_{\imt}$, then $\sigma(w)_{\imt} = -\sigma(w)_{\imo}$,
    and if $\sigma(w)_{\ipo} = -\sigma(x)_{\ipo}$, then
    $\sigma(w)_{\ipo} = -\sigma(w)_{i}$.
    
  \item[(ax$4$)] Every connected component of $(V,\sigma, E_{\imo}
    \cup E_{i})$ appears in Figure~\ref{fig:lambda4} and every
    connected component of $(V,\sigma,E_{\imt} \cup E_{\imo} \cup
    E_{i})$ appears in Figure~\ref{fig:lambda5}.

  \item[(ax$5$)] If $\{w,x\} \in E_i$ and $\{x,y\} \in E_j$ for $|i-j|
    \geq 3$, then $\{w,v\} \in E_j$ and $\{v,y\} \in E_i$ for some $v
    \in V$.

  \item[(ax$6$)] Any two vertices of a connected component of
    $(V,\sigma,E_2 \cup\cdots\cup E_i)$ may be connected by a path
    crossing at most one $E_i$ edge.

  \end{itemize}
\label{defn:deg}
\end{definition}

Note that if $n>4$, then the allowed structure for connected
components of $(V,\sigma,E_{\imt} \cup E_{\imo} \cup E_{i})$ dictates
that every connected component of $(V,\sigma, E_{\imo} \cup E_{i})$
appears in Figure~\ref{fig:lambda4}.

\begin{figure}[ht]
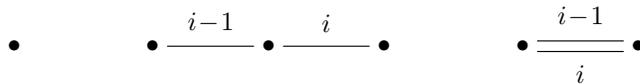

  \begin{displaymath}
    \begin{array}{\cs{11}\cs{9}\cs{9}\cs{11}\cs{9}c}
      \B & \rnode{a}{\B} & \rnode{b}{\B} & \rnode{c}{\B} & \rnode{d}{\B} & \rnode{e}{\B}
    \end{array}
    \psset{nodesep=3pt,linewidth=.1ex}
    \ncline            {a}{b} \naput{\imo}
    \ncline            {b}{c} \naput{i}
    \ncline[offset=2pt]{d}{e} \naput{\imo}
    \ncline[offset=2pt]{e}{d} \naput{i}
  \end{displaymath}
  \caption{\label{fig:lambda4} Allowed $2$-color connected
    components of a dual equivalence graph.}
\end{figure}

\begin{figure}[ht]
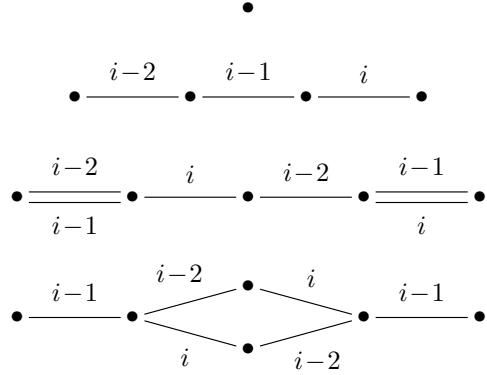

  \begin{displaymath}
    \begin{array}{c}
        \begin{array}{c}
          \B
        \end{array}\\[5ex]
        \begin{array}{\cs{9} \cs{9} \cs{9} c}
          \rnode{h}{\B} & \rnode{i}{\B} & \rnode{j}{\B} & \rnode{k}{\B}
        \end{array} \\[6ex]
        \begin{array}{\cs{9} \cs{9} \cs{9} \cs{9} c}
          \rnode{a}{\B} & \rnode{b}{\B} & \rnode{c}{\B} & \rnode{d}{\B} & \rnode{e}{\B}
        \end{array} \\[5ex]
        \begin{array}{\cs{9} \cs{9} \cs{9} \cs{9} c}
          & & \rnode{w}{\B} & & \\
          \rnode{u}{\B} & \rnode{v}{\B} & & \rnode{y}{\B} & \rnode{z}{\B} \\
          & & \rnode{x}{\B} & &
        \end{array}
        \psset{nodesep=2pt,linewidth=.1ex}
        \ncline  {h}{i} \naput{\imt}
        \ncline  {i}{j} \naput{\imo}
        \ncline  {j}{k} \naput{i}
        \ncline[offset=2pt] {a}{b} \naput{\imt}
        \ncline[offset=2pt] {b}{a} \naput{\imo}
        \ncline             {b}{c} \naput{i}
        \ncline             {c}{d} \naput{\imt}
        \ncline[offset=2pt] {d}{e} \naput{\imo}
        \ncline[offset=2pt] {e}{d} \naput{i}
        \ncline {u}{v}  \naput{\imo}
        \ncline {v}{w}  \naput{\imt}
        \ncline {v}{x}  \nbput{i}
        \ncline {w}{y}  \naput{i}
        \ncline {x}{y}  \nbput{\imt}
        \ncline {y}{z}  \naput{\imo}
      \end{array}
    \end{displaymath}
  \caption{\label{fig:lambda5} Allowed $3$-color connected
    components of a dual equivalence graph.}
\end{figure}

Every connected component of a dual equivalence graph of type $(n,N)$
is again a dual equivalence graph of type $(n,N)$. 

It is often useful to consider a restricted set of edges of a signed,
colored graph. To be precise, for $m \leq n$ and $M \leq N$, the
\emph{$(m,M)$-restriction} of a signed, colored graph $\G$ of type
$(n,N)$ consists of the vertex set $V$, signature function $\sigma: V
\rightarrow \{\pm 1\}^{M-1}$ obtained by truncating $\sigma$ at $M-1$,
and the edge set $E_{2} \cup\cdots\cup E_{m-1}$. For $m \leq n,M \leq
N$, the $(m,M)$-restriction of a dual equivalence graph of type
$(n,N)$ is a dual equivalence graph of type $(m,M)$.

The graph for $\G_{\lambda'}$ is obtained from $\G_{\lambda}$ by
conjugating each standard tableau and multiplying the signatures
coordinate-wise by $-1$. Therefore the structure of $\G_{(2,1,1,1)},
\G_{(2,2,1)}$ and $\G_{(1,1,1,1,1)}$ is also indicated by
Figure~\ref{fig:G5}. Comparing this with Figure~\ref{fig:lambda5},
axiom $4$ stipulates that the restricted components of a dual
equivalence graph are exactly the graphs for $\G_{\lambda}$ when
$\lambda$ is a partition of $5$.

\begin{proposition}
  For $\lambda$ a partition of $n$, $\G_{\lambda}$ is a dual
  equivalence graph of type $(n,n)$.
\label{prop:good-defn}
\end{proposition}

\begin{proof}
  For $T \in \mathrm{SYT}(\lambda)$, $\sigma(T)_{\imo} = -\sigma(T)_i$
  if and only if $i$ does not lie between $\imo$ and $\ipo$ in the
  content reading word of $T$. In this case, there exists $U \in
  \mathrm{SYT}(\lambda)$ such that $T$ and $U$ differ by an elementary
  dual equivalence for $\triple$. Therefore $U$ is obtained from $T$
  by swapping $i$ with $\imo$ or $\ipo$, whichever lies further away,
  with the result that $\sigma(T)_j = -\sigma(U)_j$ for $j=\imo,i$ and
  also $\sigma(T)_h = \sigma(U)_h$ for $h<\imt$ and $\ipo<h$. This
  verifies axioms $1$ and $2$.
  
  For axiom $3$, if $\sigma(T)_{\imt} = -\sigma(U)_{\imt}$, then $i$
  and $\imo$ have interchanged positions with $\imt$ lying between, so
  that $T$ and $U$ also differ by an elementary dual equivalence for
  $i\!-\!2,i\!-\!1,i$, and similarly for $\ipo$. From this, we obtain
  an explicit description of double edges, and so axiom $4$ becomes a
  straightforward, finite check. If $|i-j| \geq 3$, then $\{\triple\}
  \cap \{j\!-\!1,j,j\!+\!1\} = \emptyset$, so the elementary dual
  equivalences for $\triple$ and for $j\!-\!1,j,j\!+\!1$ commute,
  thereby demonstrating axiom $5$.

  Finally, for $T,U \in \mathrm{SYT}(\lambda)$, $|\lambda| = \ipo$, we
  must show that there exists a path from $T$ to $U$ crossing at most
  one $E_i$ edge. Let $\C_T$ (resp. $\C_U$) denote the connected
  component of the $(i,i)$-restriction of $\G_{\lambda}$ containing
  $T$ (resp.  $U$).  Let $\mu$ (resp. $\nu$) be the shape of $T$
  (resp. $U$) with the cell containing $\ipo$ removed. Then $\C_T
  \cong \G_{\mu}$ and $\C_U \cong \G_{\nu}$. If $\mu = \nu$, then, by
  Proposition~\ref{prop:deshape}, $\C_T = \C_U$ and axiom $6$
  holds. Assume, then, that $\mu \neq \nu$. Since $\mu,\nu \subset
  \lambda$ and $|\mu| = |\nu| = |\lambda|-1$, both cells $\lambda/\mu$
  and $\lambda/\nu$ must be northeastern corners of
  $\lambda$. Therefore there exists $T' \in \mathrm{SYT}(\lambda)$
  with $i$ in position $\lambda/\nu$, $\ipo$ in position
  $\lambda/\mu$, and $\imo$ between $i$ and $\ipo$ in the content
  reading word of $T'$.  Let $U'$ be the result of swapping $i$ and
  $\ipo$ in $T'$, in particular, $\{T',U'\} \in E_{i}$. By
  Proposition~\ref{prop:deshape}, $T'$ is in $\C_T$ and $U'$ is in
  $\C_U$, hence there exists a path from $T$ to $T'$ and a path from
  $U'$ to $U$ each crossing only edges $E_h$, $h < i$. This
  establishes axiom $6$.
\end{proof}

\begin{remark}
  For partitions $\lambda \subset \rho$, with $|\lambda| = n$ and
  $|\rho| = N$, choose a tableau $A$ of shape $\rho/\lambda$ with
  entries $n+1,\ldots,N$. Define the set of standard Young tableaux of
  shape $\lambda$ augmented by $A$, denoted
  $\mathrm{ASYT}(\lambda,A)$, to be those $T \in \mathrm{SYT}(\rho)$
  such that $T$ restricted to $\rho/\lambda$ is $A$. Let
  $\G_{\lambda,A}$ be the signed, colored graph of type $(n,N)$
  constructed on $\mathrm{ASYT}(\lambda,A)$ with $i$-edges given by
  elementary dual equivalences for $\triple$ with $i<n$. Then
  $\G_{\lambda,A}$ is a dual equivalence graph of type $(n,N)$, and
  the $(n,n)$-restriction of $\G_{\lambda,A}$ is $\G_{\lambda}$.
\label{rmk:augment}
\end{remark}

Proposition~\ref{prop:good-defn} is the first step towards justifying
Definition~\ref{defn:deg}, and also allows us to refer to
$\G_{\lambda}$ as the \emph{standard dual equivalence graph
  corresponding to $\lambda$}. In order to show the converse, we first
need the notion of a morphism between two signed, colored graphs.

\begin{definition}
  A \emph{morphism} between two signed, colored graphs of type
  $(n,N)$, say $\G=(V,\sigma,E)$ and $\mathcal{H}=(W,\tau,F)$, is a
  map $\phi : V \rightarrow W$ such that for every $u,v \in V$
  \begin{itemize}
  \item for every $1 \leq i < N$, we have $\sigma(v)_i =
    \tau(\phi(v))_i$, and 
  \item for every $1 < i < n$, if $\{u,v\} \in E_i$, then
    $\{\phi(u),\phi(v)\} \in F_i$.
  \end{itemize}
  A morphism is an \emph{isomorphism} if it is a bijection on vertex
  sets.
\label{defn:isomorphism}
\end{definition}

When two graphs satisfy axiom 1, as all graphs in this paper do, an
isomorphism between them is a sign-preserving bijection on vertex sets
that respects color-adjacency.

\begin{remark}
  If $\phi$ is a morphism from a signed, colored graph $\G$ of type
  $(n,N)$ satisfying axiom $1$ to an augmented standard dual
  equivalence graph $\G_{\lambda,A}$, then $\phi$ is
  surjective. Indeed, suppose $T = \phi(v)$ for some $T \in
  \mathrm{ASYT}(\lambda,A)$ and some vertex $v$ of $\G$. Then for
  every $1 < i < n$, if $\{T,U\} \in E_{i}$, then since $\sigma(v) =
  \sigma(T)$, by axiom $1$ there exists a unique vertex $w$ of $\G$
  such that $\{v,w\} \in E_i$ in $\G$. Since $\phi$ is a morphism, we
  must have $\{T,\phi(w)\} \in E_i$ in $\G_{\lambda,A}$. Thus by the
  uniqueness condition of axiom $1$, $\phi(w) = U$, and so $U$ also
  lies in the image of $\phi$. Therefore the $i$-neighbor of any
  vertex in the image of $\phi$ also lies in the image since $\phi$
  preserves $i$-edges. Since $\G_{\lambda,A}$ is connected, $\phi$ is
  surjective.
\label{rmk:onto}
\end{remark}

The final justification of this axiomatization is the following
converse of Proposition~\ref{prop:good-defn}.

\begin{theorem}
  Every connected component of a dual equivalence graph of type
  $(n,n)$ is isomorphic to $\G_{\lambda}$ for a unique partition
  $\lambda$ of $n$.
\label{thm:isomorphic}
\end{theorem}

The proof of Theorem~\ref{thm:isomorphic} is postponed until
Section~\ref{sec:deg-proof}.  We conclude this section by interpreting
Theorem~\ref{thm:isomorphic} in terms of symmetric functions.

\begin{corollary}
  Let $\G$ be a dual equivalence graph of type $(n,n)$ such that every
  vertex is assigned some additional statistic $\alpha$. Let
  $C(\lambda)$ denote the set of connected components of $\G$ that
  are isomorphic to $\G_{\lambda}$. If $\alpha$ is constant on
  connected components of $\G$, then
  \begin{equation}
    \sum_{v \in V(\G)} q^{\alpha(v)} Q_{\sigma(v)}(X) = 
    \sum_{\lambda} \sum_{\C \in C(\lambda)} q^{\alpha(\C)} s_{\lambda}(X) .
    \label{eqn:schurpos}
  \end{equation}
  In particular, the generating function for $\G$ so defined is
  symmetric and Schur positive.
\label{cor:schurpos}
\end{corollary}

We can, of course, include multivariate statistics in
\eqref{eqn:schurpos}, but as our immediate purpose is to apply this
theory to LLT polynomials, a single parameter suffices.

Equation \ref{eqn:schurpos} appears to be difficult to work with
since, in general, it is difficult to determine when two signed,
colored graphs are isomorphic. However, this problem simplifies for
dual equivalence graphs. For each vertex $v$ of a dual equivalence
graph, let $\pi(v)$ be the composition formed by the lengths of the
runs of the $+1$'s in $\sigma(v)$. As shown in
Proposition~\ref{prop:noauto-noniso}, each $\G_{\lambda}$ contains a
unique vertex $T_{\lambda}$ with the property that $\pi(T_{\lambda})$
forms a partition and, if $\pi(T)$ also forms a partition for some $T
\in \mathrm{SYT}(\lambda)$, then $\pi(T) \leq \pi(T_{\lambda})$ in
dominance order. Therefore if we know which vertices occur on a given
connected component of a dual equivalence graph, determining the
$\G_{\lambda}$ to which the component is isomorphic is simply a matter
of comparing $\pi(v)$ for each vertex of the component.

\subsection{The structure of dual equivalence graphs}
\label{sec:deg-proof}

We begin the proof of Theorem~\ref{thm:isomorphic} by showing that the
standard dual equivalence graphs are non-redundant in the sense that
they are mutually non-isomorphic and have no nontrivial
automorphisms. Both results stem from the observation that
$\G_{\lambda}$ contains a unique vertex such that the composition
formed by the lengths of the runs of $+1$'s in the signature gives a
maximal partition.

\begin{proposition} 
  If $\phi: \G_{\lambda} \rightarrow \G_{\mu}$ is an isomorphism, then
  $\lambda = \mu$ and $\phi=\mathrm{id}$.
\label{prop:noauto-noniso}
\end{proposition}

\begin{proof}
  Let $T_{\lambda}$ be the tableau obtained by filling the numbers 1
  through $n$ into the rows of $\lambda$ from left to right, bottom to
  top, in which case $\sigma(T_{\lambda}) = +^{\lambda_1\!-\!1}, -,
  +^{\lambda_2\!-\!1}, -, \cdots$. For any standard tableau $T$ such
  that $\sigma(T) = \sigma(T_{\lambda})$, the numbers $1$ through
  $\lambda_1$, and also $\lambda_1 + 1$ through $\lambda_1 +
  \lambda_2$, and so on, must form horizontal strips. In particular,
  if $\sigma(T) = \sigma(T_{\lambda})$ for some $T$ of shape $\mu$,
  then $\lambda \leq \mu$ with equality if and only if $T =
  T_{\lambda}$.

  Suppose $\phi : \G_{\lambda} \rightarrow \G_{\mu}$ is an
  isomorphism. Let $T_{\lambda}$ be as above for $\lambda$, and let
  $T_{\mu}$ be the corresponding tableau for $\mu$. Then since
  $\sigma(\phi(T_{\lambda})) = \sigma(T_{\lambda})$, $\lambda \leq
  \mu$. Conversely, since $\sigma(\phi^{-1}(T_{\mu})) =
  \sigma(T_{\mu})$, $\mu \leq \lambda$.  Therefore $\lambda=\mu$.
  Furthermore, $\phi(T_{\lambda}) = T_{\lambda}$. For $T \in
  \mathrm{SYT}(\lambda)$ such that $\{T_{\lambda},T\} \in E_i$, we
  have $\{T_{\lambda},\phi(T)\} \in E_i$, so $\phi(T) = T$ by dual
  equivalence axiom $1$.  Extending this, every tableau connected to a
  fixed point by some sequence of edges is also a fixed point for
  $\phi$, hence $\phi = \mathrm{id}$ on each $\G_{\lambda}$ by
  Proposition~\ref{prop:deshape}.
\end{proof}

In order to avoid cumbersome notation, as we investigate the
connection between an arbitrary dual equivalence graph and the
standard one, we will often abuse notation by simultaneously referring
to $\sigma$ and $E$ as the signature function and edge set for both
graphs.

\begin{definition}
  Let $\G = (V,\sigma,E)$ be a signed, colored graph of type $(n,N)$
  satisfying axiom $1$. For $1 < i < N$, we say that a vertex $w \in
  V$ \emph{admits an $i$-neighbor} if $\sigma(w)_{\imo} =
  -\sigma(w)_i$.
\label{defn:neighbor}
\end{definition}

For $1 < i < n$, if $\sigma(w)_{\imo} = -\sigma(w)_i$ for some $w \in
V$, then axiom $1$ implies the existence of $x \in V$ such that
$\{w,x\} \in E_i$. That is, if $w$ admits an $i$-neighbor for some
$1<i<n$, then $w$ has an $i$-neighbor in $\G$. For $n \leq i < N$,
though $i$-edges do not exist in $\G$, if $\G$ were the restriction of
a graph of type $(\ipo,N)$ also satisfying axiom $1$, then the
condition $\sigma(w)_{\imo} = -\sigma(w)_i$ would imply the existence
of a vertex $x$ such that $\{w,x\} \in E_{i}$ in the type $(\ipo,N)$
graph. When convenient, $E_{i}$ may be regarded as an involution on
vertices admitting an $i$-neighbor, i.e. if $w$ admits an
$i$-neighbor, then $E_i(w) = x$ where $\{w,x\} \in E_{i}$.

Recall the notion of augmenting a partition $\lambda$ by a skew
tableau $A$ and the resulting dual equivalence graph $\G_{\lambda,A}$
from Remark~\ref{rmk:augment}.

\begin{lemma}
  Let $\G = (V,\sigma,E)$ be a connected dual equivalence graph of
  type $(n,N)$, and let $\phi$ be an isomorphism from the
  $(n,n)$-restriction of $\G$ to $\G_{\lambda}$ for some partition
  $\lambda$ of $n$.  Then there exists a semi-standard tableau $A$ of
  shape $\rho/\lambda$, $|\rho| = N$, with entries $n+1,\ldots,N$ such
  that $\phi$ gives an isomorphism from $\G$ to $\G_{\lambda,A}$.
  Moreover, the position of the cell of $A$ containing $n+1$ is
  unique.
\label{lem:extend-signs}
\end{lemma}

\begin{proof}
  By axiom $2$ and the fact that $\G$ is connected, $\sigma_h$ is
  constant on $\G$ for $h \geq \npo$. Therefore once a suitable cell
  for $\npo$ has been chosen, the cells for $n+2, \cdots, N$ may be
  chosen in any way that gives the correct signature. One solution is
  to place $j$ north of the first column if $\sigma_{j-1} = -1$ or
  east of the first row if $\sigma_{j-1} = +1$ for
  $j=n+2,\cdots,N$. Assume, then, that $N = \npo$.
  
  By dual equivalence axiom $2$, $\sigma_n$ is constant on connected
  components of the $(\nmo,\npo)$-restriction of $\G$. By
  Proposition~\ref{prop:deshape}, a connected component of the
  $(\nmo,\nmo)$-restriction of $\G_{\lambda}$ consists of all standard
  Young tableaux where $n$ lies in a particular northeastern cell of
  $\lambda$. Therefore, for each connected component of the
  $(\nmo,\npo)$-restriction of $\G$, we may identify its image under
  $\phi$ with $\G_{\mu}$ for some partition $\mu \subset \lambda$,
  $|\mu| = \nmo$, with $n$ lying in position $\lambda/\mu$. We will
  show that $\sigma_n$ has the monotonicity property on connected
  components of the $(\nmo,\npo)$-restriction of $\G$ depicted in
  Figure~\ref{fig:monotonic}, i.e., there is an inner corner above
  which $\sigma_n = +1$ and below which $\sigma_n = -1$.

  \begin{figure}[ht]
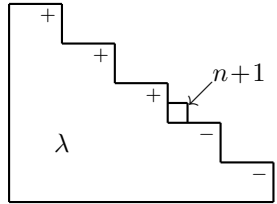

    \begin{center}
      \psset{xunit=2em}
      \psset{yunit=1.5em}
      \pspicture(0,0)(5,5)
      \psline(0,0)(0,5)
      \psline(1,4)(1,5)
      \psline(2,3)(2,4)
      \psline(3,2)(3,3)
      \psline(4,1)(4,2)
      \psline(5,0)(5,1)
      \psline(0,0)(5,0)
      \psline(4,1)(5,1)
      \psline(3,2)(4,2)
      \psline(2,3)(3,3)
      \psline(1,4)(2,4)
      \psline(0,5)(1,5)
      \rput(0.75,4.7){$_+$}
      \rput(1.75,3.7){$_+$}
      \rput(2.75,2.7){$_+$}
      \rput(3.75,1.7){$_-$}
      \rput(4.75,0.7){$_-$}
      \psline(3,2.5)(3.375,2.5)
      \psline(3.375,2)(3.375,2.5)
      \rput(3.625,2.75){$\swarrow$}
      \rput(4.35,3.25){$\npo$}
      \rput(1,1.5){$\lambda$}
      \endpspicture
      \caption{\label{fig:monotonic} Identifying the unique position
        for $\npo$ based on $\sigma_n$.}
    \end{center}
  \end{figure}

  Let $\C$ and $\mathcal{D}$ be two distinct connected components of
  the $(\nmo,\npo)$-restriction of $\G$ such that there exist vertices
  $v$ of $\C$ and $u$ of $\mathcal{D}$ with $\{v,u\} \in
  E_{\nmo}$. Let $\phi(\C) \cong \G_{\mu}$, and let $\phi(\mathcal{D})
  \cong \G_{\nu}$. Since $\{v,u\} \in E_{\nmo}$, $\phi(v)$ must have
  $\nmo$ in position $\lambda/\mu$ with $\nmt$ lying between $\nmo$
  and $n$ in the content reading word. Since $\phi$ preserves
  $E_{\nmo}$ edges, $\phi(u)$ must be the result of an elementary dual
  equivalence on $\phi(v)$ for $\nmt, \nmo, n$, which will necessarily
  interchange $\nmo$ and $n$. Since $\phi$ preserves signatures,
  $\lambda/\nu$ lies northwest of the position of $\lambda/\mu$ if and
  only if $\sigma(v)_{\nmt,\nmo} = + -$ and $\sigma(u)_{\nmt,\nmo} = -
  +$. If $\lambda/\nu$ lies northwest of the position of $\lambda/\mu$
  and $\sigma(v)_{n} = -1$, then that $\sigma(v)_{n} =
  \sigma(v)_{\nmo}$. Thus, by axiom 3, $\sigma(u)_{n} = \sigma(v)_{n}
  = -1$. Similarly, if $\lambda/\nu$ lies northwest of the position of
  $\lambda/\mu$ and $\sigma(u)_{n} = +1$, then $\sigma(u)_{n} =
  \sigma(u)_{\nmo}$. Thus, by axiom 3, $\sigma(v)_{n} = \sigma(u)_{n}
  = +1$.

  Abusing notation and terminology, we have shown that if
  $\sigma_n(\C) = +1$ and $\mathcal{D}$ is any component connected to
  $\C$ by an $\nmo$-edge such that $\phi(\mathcal{D})$ lies northwest
  of $\phi(\C)$, then $\sigma_n(\mathcal{D}) = +1$ as well. Similarly,
  if $\sigma_n(\C) = -1$ and $\mathcal{D}$ is any component connected
  to $\C$ by an $\nmo$-edge such that $\phi(\mathcal{D})$ lies
  southeast of $\phi(\C)$, then $\sigma_n(\mathcal{D}) = -1$ as
  well. By dual equivalence graph axiom 6, for any two distinct
  connected components $\C$ and $\mathcal{D}$ of the
  $(\nmo,\npo)$-restriction of $\G$ and any pair of vertices $w$ on
  $\C$ and $x$ on $\mathcal{D}$, there is a path from $w$ to $x$
  crossing at most one, and hence exactly one, $\nmo$ edge. Therefore
  for any $\C$ and $\mathcal{D}$, there exist vertices $v$ of $\C$ and
  $u$ of $\mathcal{D}$ such that $\{v,u\} \in E_{\nmo}$. Hence every
  two connected components of the $(\nmo,\npo)$-restriction of $\G$
  are connected by an $\nmo$-edge, thus establishing the monotonicity
  depicted in Figure~\ref{fig:monotonic}.

  This established, it follows that there exists a unique row such
  that $\sigma(\C)_n = -1$ whenever the $\phi(\C)$ has $n$ south of
  this row and $\sigma(\C)_n = +1$ whenever the $\phi(\C)$ has $n$
  north of this row. In this case, the cell containing $\npo$ must be
  placed at the eastern end of this pivotal row, and doing so extends
  $\phi$ to an isomorphism between $(n,\npo)$ graphs.
\end{proof}

Once Theorem~\ref{thm:isomorphic} has been proved,
Lemma~\ref{lem:extend-signs} may be used to obtain the
following generalization of Theorem~\ref{thm:isomorphic} for dual
equivalence graphs of type $(n,N)$: Every connected component of a
dual equivalence graph of type $(n,N)$ is isomorphic to
$\G_{\lambda,A}$ for a unique partition $\lambda$ and some skew
tableau $A$ of shape $\rho/\lambda$, $|\rho| = N$, with entries $\npo,
\ldots, N$.

Finally we have all of the ingredients necessary to prove the main
result of this section.

\begin{theorem}
  Let $\G$ be a connected signed, colored graph of type $(\npo,\npo)$
  satisfying axioms $1$ through $5$ such that each connected component
  of the $(n,n)$-restriction of $\G$ is isomorphic to a standard dual
  equivalence graph. Then there exists a morphism $\phi$ from $\G$ to
  $\G_{\lambda}$ for some unique partition $\lambda$ of
  $\npo$. 
\label{thm:cover}
\end{theorem}

\begin{proof}
  When $\npo=2$ or, more generally, when $\G$ has no $n$-edges, the
  result follows immediately from
  Lemma~\ref{lem:extend-signs}. Therefore we proceed by induction,
  assuming that $\G$ has at least one $n$-edge and assuming the result
  for graphs of type $(n,n)$.

  By induction, for every connected component $\C$ of the
  $(n,\npo)$-restriction of $\G$, we have an isomorphism from the
  $(n,n)$-restriction of $\C$ to $\G_{\mu}$ for a unique partition
  $\mu$ of $n$. By Lemma~\ref{lem:extend-signs}, this isomorphism
  extends to an isomorphism from $\C$ to $\G_{\mu,A}$ for a unique
  augmenting tableau $A$, say with shape $\lambda/\mu$. We will show
  that for any $\C$ the shape of $\mu$ augmented with $A$ is the
  same and that we may glue these isomorphisms together to obtain a
  morphism from $\G$ to $\G_{\lambda}$.

  \begin{figure}[ht]
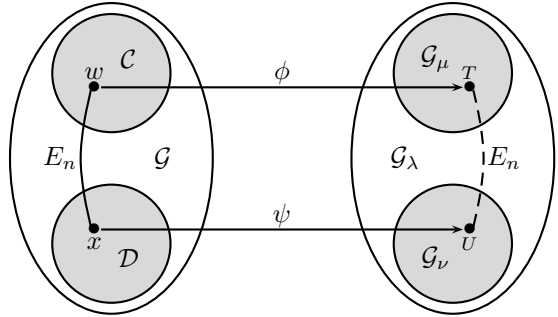

    \begin{center}
      \psset{xunit=3ex}
      \psset{yunit=2.5ex}
      \pspicture(0,0)(14,10)
      \pscircle[fillstyle=solid,fillcolor=lightgray](12,2){.8}
      \rput(12.5,2.5){$\bullet$}
      \rput(12.5,2){$_U$}
      \rput(11.5,1.5){$\G_{\nu}$}
      \pscircle[fillstyle=solid,fillcolor=lightgray](12,8){.8}
      \rput(12.5,7.5){$\bullet$}
      \rput(12.5,8){$_T$}
      \rput(11.5,8.5){$\G_{\mu}$}
      \psellipse(12,5)(3,5.5)
      \rput(10.6,5){$\G_{\lambda}$}
      \pscurve[linestyle=dashed](12.6,2.6)(12.8,3.75)(12.9,5)(12.8,6.25)(12.6,7.4)
      \rput(13.5,5){$E_n$}
      \pscircle[fillstyle=solid,fillcolor=lightgray](2,2){.8}
      \rput(1.5,2.5){$\bullet$}
      \rput(1.5,2){$x$}
      \rput(2.5,1.5){$\mathcal{D}$}
      \pscircle[fillstyle=solid,fillcolor=lightgray](2,8){.8}
      \rput(1.5,7.5){$\bullet$}
      \rput(1.5,8){$w$}
      \rput(2.5,8.5){$\C$}
      \psellipse(2,5)(3,5.5)
      \rput(3.5,5){$\G$}
      \pscurve(1.4,2.6)(1.2,3.75)(1.1,5)(1.2,6.25)(1.4,7.4)
      \rput(0.5,5){$E_n$}
      \psline{->}(1.7,7.5)(12.3,7.5)
      \rput(7,8){$\phi$}
      \psline{->}(1.7,2.5)(12.3,2.5)
      \rput(7,3){$\psi$}
      \endpspicture
      \caption{\label{fig:extend} An illustration of the gluing
        process.}
    \end{center}
  \end{figure}

  Suppose $\{w,x\} \in E_{n}$. Let $\C$ (resp. $\mathcal{D}$) denote
  the connected component of the $(n,\npo)$-restriction of $\G$
  containing $w$ (resp. $x$). Let $\phi$ (resp. $\psi$) be the
  isomorphism from $\C$ (resp. $\mathcal{D}$) to $\G_{\mu,A}$
  (resp. $\G_{\nu,B}$), and set $T = \phi(w)$; see
  Figure~\ref{fig:extend}. We will show that $\psi(x) = E_n(T)$, and
  hence if $\mu,A$ has shape $\lambda$, then so does $\nu,B$ and the
  maps $\phi$ and $\psi$ glue together to give an morphism from $\C
  \cup \mathcal{D}$ to $\G_{\lambda}$ that preserves $n$-edges. There
  are two cases to consider, based on the relative positions of
  $\nmo,n,\npo$ in $T$, regarded as a tableau of shape $\lambda$.

  First suppose that $\npo$ lies between $n$ and $\nmo$ in the reading
  word of $T$. We will show that, in this case, $\C =
  \mathcal{D}$. Since $\npo$ lies between $n$ and $\nmo$ in the
  reading word of $T$, both $\nmo$ and $n$ must be northeastern
  corners of $\mu$, and so there is a cell with entry less than $\nmo$
  that also lies between them. By Proposition~\ref{prop:deshape},
  there exists a tableau $T'$ with $\nmo,n,\npo$ in the same positions
  as in $T$, but now with $\nmt$ lying between $n$ and $\nmo$ in the
  reading word of $T'$. Furthermore, since both $T$ and $T'$ lie on
  the $(\nmt,\npo)$-restriction of $\G_{\mu,A}$, there is a path from
  $T$ to $T'$ in $\G_{\mu,A}$ using only edges $E_h$ with $h \leq
  \nmh$. Let $U' = E_n(T')$. Then since $\nmt$ lies between $n$ and
  $\nmo$ in $U'$, we have $U' = E_{\nmo}(T')$ as well. By axioms 2 and
  5, all edges in the path from $T$ to $T'$ commute with $E_n$, and so
  the same path takes $U = E_n(T)$ to $U'$, and each pair of
  corresponding tableaux on the two paths is connected by an $E_n$
  edge; see Figure~\ref{fig:double}.

  \begin{figure}[ht]
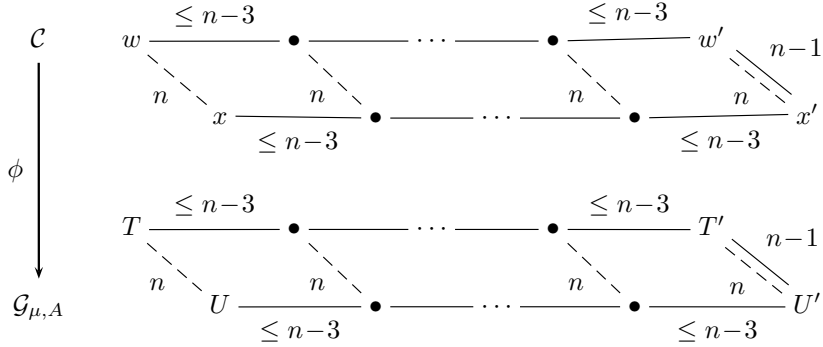

    \begin{center}
      \begin{displaymath}
        \begin{array}{\cs{5}\cs{6}\cs{5}\cs{6}\cs{3}\cs{3}\cs{3}\cs{6}\cs{5}\cs{6}c}
          \rnode{C}{\C} & \rnode{w}{w} & & \rnode{wl}{\B} & & \rnode{wm}{\cdots} 
          & & \rnode{wr}{\B} & & \rnode{wd}{w'} & \\[4ex]
          & & \Rnode{x}{x} & & \rnode{xl}{\B} & & \rnode{xm}{\cdots} 
          & & \rnode{xr}{\B} & & \rnode{xd}{x'} \\[7ex]
          & \Rnode[vref=0.5ex]{T}{T} & & \rnode{Tl}{\B} & & \rnode{Tm}{\cdots} 
          & & \rnode{Tr}{\B} & & \Rnode[vref=0.5ex]{Td}{T'} & \\[4ex]
          \rnode{G}{\G_{\mu,A}} & & \Rnode[vref=0.5ex]{U}{U} & & \rnode{Ul}{\B} 
          & & \rnode{Um}{\cdots} & & \rnode{Ur}{\B} & & \Rnode[vref=0.5ex]{Ud}{U'} 
        \end{array}
        \psset{nodesep=3pt,linewidth=.1ex}
        \ncline {w}{wl}   \naput{\leq \nmh}
        \ncline {wl}{wm}  
        \ncline {wm}{wr}  
        \ncline {wr}{wd}  \naput{\leq \nmh}
        \ncline {x}{xl}   \nbput{\leq \nmh}
        \ncline {xl}{xm}  
        \ncline {xm}{xr}  
        \ncline {xr}{xd}  \nbput{\leq \nmh}
        \ncline[linestyle=dashed] {w}{x} \nbput{n}
        \ncline[linestyle=dashed] {wl}{xl} \nbput{n}
        \ncline[linestyle=dashed] {wr}{xr} \nbput{n}
        \ncline[offset=2pt] {wd}{xd} \naput {\nmo}
        \ncline[linestyle=dashed,offset=2pt] {xd}{wd} \naput{n}
        \ncline[linewidth=.2ex,nodesep=6pt]{->} {C}{G}   \nbput{\phi}
        \ncline {T}{Tl}   \naput{\leq \nmh}
        \ncline {Tl}{Tm}  
        \ncline {Tm}{Tr}  
        \ncline {Tr}{Td}  \naput{\leq \nmh}
        \ncline {U}{Ul}   \nbput{\leq \nmh}
        \ncline {Ul}{Um}  
        \ncline {Um}{Ur}  
        \ncline {Ur}{Ud}  \nbput{\leq \nmh}
        \ncline[linestyle=dashed] {T}{U} \nbput{n}
        \ncline[linestyle=dashed] {Tl}{Ul} \nbput{n}
        \ncline[linestyle=dashed] {Tr}{Ur} \nbput{n}
        \ncline[offset=2pt] {Td}{Ud} \naput {\nmo}
        \ncline[linestyle=dashed,offset=2pt] {Ud}{Td} \naput{n}
      \end{displaymath}
    \end{center}
    \caption{\label{fig:double} Illustration of the path from $T$ to
      $U$ in $\G_{\mu,A}$ and its lift in $\C$.}
  \end{figure}

  Since the path from $T$ to $T'$ to $U'$ to $U$ uses only edges from
  $\G_{\mu,A}$, this path lifts via the isomorphism $\phi$ to a path
  in $\C$. Let $w' = \phi^{-1}(T')$ and $x' = \phi^{-1}(U')$. We will
  show that $x = \phi^{-1}(U)$ and so lies on $\C$. Since $\phi$
  preserves signatures, both $w'$ and $x'$ must admit an $n$-edge in
  $\G$. As summarized in Figure~\ref{fig:lambda4}, axioms 3 and 4
  dictate that the only way for two vertices connected by an
  $\nmo$-edge both to admit an $n$-edge is for $\{w',x'\} \in E_n$ in
  $\G$. By axioms 2 and 5, the path from $w'$ to $w$ gives an
  identical path from $x'$ to $\phi^{-1}(U)$. Since each corresponding
  pair along the two paths must be paired by an $n$-edge, we have
  $\phi^{-1}(U) = E_n(w) = x$, as desired.  Therefore $x$ lies on
  $\C$, and $\phi$ respects the $n$-edge between $w$ and $x$. In this
  case $\C=\mathcal{D}$ and, by Proposition~\ref{prop:noauto-noniso},
  $\psi=\phi$.

  For the second case, suppose that $\nmo$ lies between $n$ and $\npo$
  in $T$. Consider the subset of tableaux in $\G_{\mu,A}$ with $n$ and
  $\npo$ fixed in the same position as in $T$ and $\nmo$ lying
  anywhere between them in the reading word. In terms of the graph
  structure, these are all tableaux reachable from $T$ using edges
  $E_h$ with $h \leq \nmh$ and a certain subset of the $E_{\nmt}$
  edges. We will return soon to the question of which $E_{\nmt}$ edges
  these are. For now, let $\mathcal{T}$ denote the union of the graphs
  $\G_{\rho,R}$, where $\rho$ is a partition of $\nmt$ with augmenting
  tableau $R$ consisting of a single cell containing $\nmo$ such that
  $\rho,R$ is the shape of $T$ with $n$ and $\npo$ removed and the
  augmented cell of $R$ lies strictly between the positions of $n$ and
  $\npo$ in $T$. Clearly the set of $\rho,R$ uniquely determines the
  cells containing $n$ and $\npo$, and so uniquely determines
  $\lambda$. Furthermore, which of $n,\npo$ occupies which cell is
  determined by $\sigma_n$, which is constant on this subset by axioms
  $2$ and $3$. Lifting $\mathcal{T}$ to $\C$ using $\phi^{-1}$ gives
  rise to an induced subgraph of $\C$ that completely determines
  $\lambda$ as well as the positions of $n$ and $\npo$ in the image of
  this subgraph under $\phi$. We will show that the corresponding
  induced subgraph for $\mathcal{D}$ is isomorphic but with the
  opposite sign for $\sigma_n$.
  
  \begin{figure}[ht]
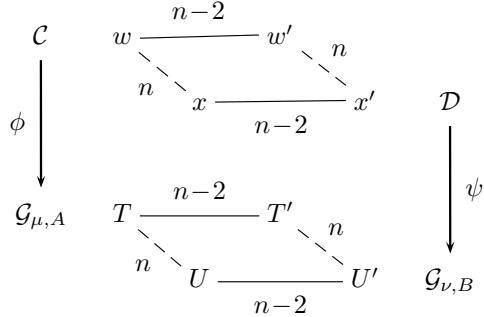

    \begin{center}
      \begin{displaymath}
        \begin{array}{\cs{4}\cs{5}\cs{5}\cs{5}\cs{4}c}
          \rnode{C}{\C} & \rnode{w}{w} & & \rnode{wd}{w'} & & \\[3ex]
          & & \Rnode{x}{x} & & \rnode{xd}{x'} & \rnode{D}{\mathcal{D}} \\[7ex]
          \rnode{A}{\G_{\mu,A}} & \Rnode[vref=0.5ex]{T}{T} & & \Rnode[vref=0.5ex]{Td}{T'} & & \\[3ex]
          & & \Rnode[vref=0.5ex]{U}{U} & & \Rnode[vref=0.5ex]{Ud}{U'} & \rnode{B}{\G_{\nu,B}} 
        \end{array}
        \psset{nodesep=3pt,linewidth=.1ex}
        \ncline {w}{wd}   \naput{\nmt}
        \ncline {x}{xd}   \nbput{\nmt}
        \ncline {T}{Td}   \naput{\nmt}
        \ncline {U}{Ud}   \nbput{\nmt}
        \ncline[linestyle=dashed] {w}{x} \nbput{n}
        \ncline[linestyle=dashed] {wd}{xd} \naput{n}
        \ncline[linestyle=dashed] {T}{U} \nbput{n}
        \ncline[linestyle=dashed] {Td}{Ud} \naput{n}
        \ncline[linewidth=.2ex,nodesep=6pt]{->} {C}{A}   \nbput{\phi}
        \ncline[linewidth=.2ex,nodesep=6pt]{->} {D}{B}   \naput{\psi}
      \end{displaymath}
    \end{center}
    \caption{\label{fig:typeC} Illustration of $E_{\nmt}$ edges on
      $\mathcal{T} \cup \mathcal{U}$ and their lifts in $\C \cup
      \mathcal{D}$.}
  \end{figure}

  To prove the assertion, we return to the question of which
  $E_{\nmt}$ edges are allowed in generating $\mathcal{T}$. Any
  $E_{\nmt}$ edge that keeps $\nmo$ between $n$ and $\npo$ clearly
  does not change $\sigma_{\nmo}$ or $\sigma_{n}$. Therefore such
  $E_{\nmt}$ edges must pair vertices both of which admit an
  $n$-neighbor. Further, neither of these vertices may have $E_n$ as a
  double edge with $E_{\nmo}$ since $\nmo$ lies between $n$ and
  $\npo$. By axiom 4, the $E_{\nmt}$ edges that meet these conditions
  are precisely those in the lower component of
  Figure~\ref{fig:lambda5}. In particular, these $E_{\nmt}$ edges
  commute with $E_n$ edges as depicted in Figure~\ref{fig:typeC}. By
  axioms 2 and 5, $E_h$ also commutes with $E_n$ for $h \leq
  \nmh$. Therefore all edges on the induced subgraph of $\C$
  containing $\phi^{-1}(\mathcal{T})$ commute with $E_n$. Therefore
  $E_n$ may be regarded as an isomorphism from this subgraph to
  $\mathcal{X} = E_n(\phi^{-1}(\mathcal{T}))$. Since $\{w,x\} \in E_n$
  and $w \in \phi^{-1}(\mathcal{T})$, we have $x \in
  \mathcal{X}$. Since all edges of the induced subgraph have color at
  most $\nmt$, it follows that $\mathcal{X} \subset \mathcal{D}$.

  Let $U = \psi(x)$, and, more generally, let $\mathcal{U} =
  \psi(\mathcal{X})$. Since $\phi,\psi$ and $E_n$ are isomorphisms,
  $\mathcal{U}$ together with its induced edges is isomorphic to
  $\mathcal{T}$ together with its induced edges, though, by axiom 1,
  the signs for $\sigma_n$ and $\sigma_{\npo}$ are reversed. By the
  earlier characterization of $\mathcal{T}$, this implies that the
  tableaux in $\mathcal{U}$ have shape $\lambda$, with the cells
  containing $n$ and $\npo$ reversed from that in $\mathcal{T}$. In
  particular, $\mathcal{U} = E_n(\mathcal{T})$, that is to say, $\phi$
  and $\psi$ glue to give a morphism from $\C \cup \mathcal{D} \subset
  \G$ to $\G_{\mu,A} \cup \G_{\nu,B} \subset \G_{\lambda}$ that
  respects $E_n$ edges of the induced subgraphs.

  Since $T$ admits an $n$-neighbor, $n$ cannot lie between $\nmo$ and
  $\npo$, so these two are the only cases. Thus we now have a
  well-defined morphism from the $(n,\npo)$-restriction of $\G$ to the
  $(n,\npo)$-restriction of $\G_{\lambda}$ that respects $n$-edges. As
  such, this map lifts to a morphism from $\G$ to $\G_{\lambda}$.
\end{proof}

By Remark~\ref{rmk:onto}, the morphism of Theorem~\ref{thm:cover} is
necessarily surjective, though in general it need not be injective.
The smallest example where injectivity fails was first observed by
Gregg Musiker in a graph of type $(6,6)$ with generating function $2
s_{(3,2,1)}(X)$; see Figure~\ref{fig:gregg} in
Appendix~\ref{app:axiom6}. 


\begin{corollary}
  Let $\G$ satisfy the hypotheses of Theorem~\ref{thm:cover}. Then the
  fiber over each vertex of $\G_{\lambda}$ in the morphism from $\G$
  to $\G_{\lambda}$ has the same cardinality.
  \label{cor:fibers}
\end{corollary}

\begin{proof}
  Let $\phi$ be the morphism from $\G$ to $\G_{\lambda}$. We show that
  for any connected component $\C$ of the $(n,n)$-restriction of $\G$,
  say with $\phi(\C) = \G_{\mu}$, and any partition $\nu \subset
  \lambda$ of size $n$, there is a unique connected component
  $\mathcal{D}$ of the $(n,n)$-restriction of $\G$ with
  $\phi(\mathcal{D}) = \G_{\nu}$ that can be reached from $\C$ by
  crossing at most one $E_n$ edge. Once established, this gives a
  bijective correspondence between connected components of
  $\phi^{-1}(\G_{\mu})$ and connected components of
  $\phi^{-1}(\G_{\nu})$, thus proving the result.

  To prove existence, if $\nu \neq \mu$, let $T$ be a tableau of shape
  $\lambda$ with $\npo$ in position $\lambda/\mu$, $n$ in position
  $\lambda/\nu$, and $\nmo$ lying between in the reading
  word. Otherwise let $T$ be a tableau with $\npo$ in position
  $\lambda/\mu$ and $n$ and $\nmo$ lying on opposite sides in the
  reading word. Let $w$ be the unique element in $\phi^{-1}(T) \cap
  \C$. Then $w$ admits an $n$-neighbor, and, since $\phi$ is a
  morphism, $\phi(E_n(w)) = E_n(\phi(w)) \in \G_{\nu}$.

  To prove uniqueness, let $\{w,x\} \in E_n$ with $w \in \C \cong
  \G_{\mu}$ and $x \in \mathcal{D} \cong \G_{\nu}$. If $\npo$ lies
  between $n$ and $\nmo$ in $\phi(w)$, then $\mu = \nu$, and just as
  in the proof of Theorem~\ref{thm:cover}, we concluded that
  $\mathcal{D} = \C$ as desired. Alternately, assume $\nmo$ lies
  between $n$ and $\npo$ in $\phi(w)$, and suppose $\{w',x'\} \in
  E_{\nmo}$ with $w' \in \C$ and $x' \in \mathcal{D}' \cong
  G_{\nu}$. Since $\phi(w)$ and $\phi(w')$ have the same shape, and
  $E_n(\phi(w)) = \phi(E_n(w)) = \phi(x)$ and $E_n(\phi(w')) =
  \phi(E_n(w')) = \phi(x')$ have the same shape, just as in the proof
  of Theorem~\ref{thm:cover}, there must be a path from $\phi(w)$ to
  $\phi(w')$ in $\G_{\nu}$ using only edges $E_h$ with $h \leq \nmh$
  and those $E_{\nmt}$ that commute with $E_n$. Therefore this path
  gives rise to the same path from $\phi(x)$ to $\phi(x')$ in
  $\G_{\mu}$. The former path lifts to a path from $w$ to $w'$ in
  $\C$, and so the latter lifts to a path from $E_n(w)=x$ to
  $E_n(w')=x'$ in $\mathcal{D} = \mathcal{D}'$, which is as desired.
\end{proof}

In order to ensure that the morphism in the conclusion of
Theorem~\ref{thm:cover} is an isomorphism, and thereby complete the
proof of Theorem~\ref{thm:isomorphic}, we need only invoke the
heretofore uninvoked axiom 6.

\begin{proof}[Proof of Theorem~\ref{thm:isomorphic}]
  Let $\G$ be a dual equivalence graph of type $(\npo,\npo)$. We aim
  to show that $\G$ is isomorphic to $\G_{\lambda}$ for a unique
  partition $\lambda$ of $\npo$. We proceed by induction on $\npo$,
  noting that the result is trivial for $\npo = 2$.  Every connected
  component of the $(n,n)$-restriction of $\G$ is a dual equivalence
  graph, and so, by induction, is isomorphic to a standard dual
  equivalence graph. Thus, by Theorem~\ref{thm:cover}, there exists a
  morphism, say $\phi$, from $\G$ to $\G_{\lambda}$ for a unique
  partition $\lambda$ of $\npo$. By Corollary~\ref{cor:fibers}, for
  any connected component $\C$ of the $(n,n)$-restriction of $\G$ and
  any partition $\nu \subset \lambda$ of size $n$, there is a unique
  connected component $\mathcal{D}$ of the $(n,n)$-restriction of $\G$
  that can be reached from $\C$ by crossing at most one $E_n$ edge
  such that $\phi(\mathcal{D}) = \G_{\nu}$. By dual equivalence axiom
  $6$, any two connected components of the $(n,n)$-restriction of $\G$
  can be connected by a path using at most one $E_n$ edge. Therefore
  the connected components of the $(n,n)$-restriction of $\G$ are
  pairwise non-isomorphic. Hence the morphism from $\G$ to
  $\G_{\lambda}$ is injective on the $(n,\npo)$-restrictions, and so
  it is injective on all of $\G$. Surjectivity follows from
  Remark~\ref{rmk:onto}, thus $\phi$ is an isomorphism.
\end{proof}

%
\section{A graph for LLT polynomials}
%
\label{sec:llt}

\subsection{Words with content}
\label{sec:llt-words}

In this section we describe a modified characterization of LLT
polynomials as the generating function of \emph{$k$-ribbon words}. As
Proposition~\ref{prop:words} shows, these are precisely the content
reading words of semi-standard $k$-tuples of tableaux (which
correspond to \emph{ribbon} tableaux).

Given a word $w$ and a non-decreasing sequence of integers $c$ of the
same length, define the \emph{$k$-descent set of the pair $(w,c)$},
denoted $\mathrm{Des}_k(w,c)$, by
\begin{equation}
  \mathrm{Des}_k(w,c) = \left\{ (i,j) \; | \; w_{i} > w_{j} \; \mbox{and}
    \; c_{j} - c_{i} = k \right\}.
  \label{eqn:Desk-w}
\end{equation}

\begin{definition}
  A \emph{$k$-ribbon word} is a pair $(w,c)$ consisting of a word $w$
  and a non-decreasing sequence of integers $c$ of the same length
  such that if $c_{i} = c_{\ipo}$, then there exist integers $h$ and
  $j$ such that $(h,i), (\ipo,j) \in \mathrm{Des}_{k}(w,c)$ and
  $(i,j),(h,\ipo) \not\in \mathrm{Des}_k(w,c)$. In other words, $c_h =
  c_i-k$ and $w_i < w_h \leq w_{i+1}$ while $c_j = c_i+k$ and $w_i
  \leq w_j < w_{i+1}$.
\label{defn:ktableau}
\end{definition}

\begin{proposition}
  The pair $(w,c)$ is a $k$-ribbon word if and only if there exists a
  $k$-tuple of (skew) semi-standard tableaux such that $w$ is the
  content reading word of the $k$-tuple and $c$ gives the
  corresponding contents.
\label{prop:words}
\end{proposition}

\begin{proof}
  Suppose first that $w$ is the content reading word of some $k$-tuple
  of semi-standard tableaux with corresponding shifted contents given
  by $c$. If $c_{i} = c_{\ipo}$, then in the $k$-tuple there must
  exist entries $w_h$ and $w_j$ as shown in
  Figure~\ref{fig:samecontent}. The semi-standard condition ensures
  that $w_i < w_h \leq w_{\ipo}$ and $w_i \leq w_j <
  w_{\ipo}$. Therefore the conditions of
  Definition~\ref{defn:ktableau} are met.

  \begin{figure}[ht]
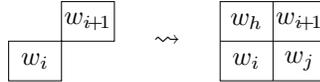

    \begin{center}
      \psset{xunit=2em,yunit=1.5em}
      \psset{linewidth=.3pt}      
      \pspicture(0,0)(6,2)
      \pspolygon(0,0)(0,1)(2,1)(2,2)(1,2)(1,0)
      \rput(0.5,0.5){$w_{i}$}
      \rput(1.5,1.5){$w_{\ipo}$}
      \rput(3,1){$\leadsto$}
      \pspolygon(4,0)(4,2)(6,2)(6,0)
      \psline(4,1)(6,1) \psline(5,0)(5,2)
      \rput(4.5,0.5){$w_{i}$}
      \rput(5.5,1.5){$w_{\ipo}$}
      \rput(4.5,1.5){$w_{h}$}
      \rput(5.5,0.5){$w_{j}$}
      \endpspicture
      \caption{\label{fig:samecontent} Situation when $c_{i} =
        c_{\ipo}$ for a $k$-tuple of semi-standard Young tableaux.}
    \end{center}
  \end{figure}
  
  Now suppose that $(w,c)$ is a $k$-ribbon word. For each $j$, arrange
  all $w_i$ such that $c_i = j$ into cells along a southwest to
  northeast diagonal in increasing order. Align the southwest corner
  of the diagonal for $j-k$ immediately north (resp. west) of the
  southwest corner of the diagonal for $j$ whenever the smallest
  letter with content $j-k$ is greater than (resp. less than or equal
  to) the smallest letter with content $j$.

  We must show that the result is a $k$-tuple of (skew) shapes whose
  entries satisfy the semi-standard condition. Consider two adjacent
  diagonals $j-k$ and $j$. By construction, the southwestern most
  cells of the diagonals form a partition shape and satisfy the
  semi-standard condition. By induction, assume that the entries in
  diagonal $j-k$ through $w_h$ and the entries in diagonal $j$ through
  $w_i$ belong to a semi-standard tableau of skew shape, with $w_h$
  immediately west or immediately north of $w_i$.

  Suppose that $c_{\ipo} = c_{i}$, noting that the case when
  $c_{h\!+\!1} = c_{h}$ may be solved similarly. If $w_h > w_i$, then
  we must show that $w_h \leq w_{\ipo}$. By
  Definition~\ref{defn:ktableau}, there exists an integer $l$ such
  that $(l,\ipo) \not\in \mathrm{Des}_k(w,c)$, and therefore $w_{l}
  \leq w_{\ipo}$. Since $c_{l}=j-k$, we have $w_{h} \leq w_{l} \leq
  w_{\ipo}$. If $w_h \leq w_i$, then we must show that $c_{h\!+\!1} =
  j-k$ and $w_i < w_{h\!+\!1} \leq w_{\ipo}$. By
  Definition~\ref{defn:ktableau}, there exists an integer $l$ such
  that $(l,i) \in \mathrm{Des}_k(w,c)$ and $(l,\ipo) \not\in
  \mathrm{Des}_k(w,c)$. Therefore $c_{l} = j-k$ and $w_{h} \leq w_{i}
  < w_{l} \leq w_{\ipo}$. The non-decreasing condition on $c$ implies
  that $c_{h\!+\!1} = j-k$, and so there exists an integer $m$ such
  that $(h\!+\!1,m) \in \mathrm{Des}_k(w,c)$ and $(h,m) \not\in
  \mathrm{Des}_k(w,c)$, i.e. $w_h \leq w_m < w_{h\!+\!1}$ with $c_m =
  j$. The only way to satisfy these two conditions is to have $m=i$
  and $l=h\!+\!1$.
\end{proof}

For $\mathbf{T}$ and $\mathbf{U}$ two $k$-tuples of semi-standard
tableaux, let $(w_{\mathbf{T}},c_{\mathbf{T}})$ and
$(w_{\mathbf{U}},c_{\mathbf{U}})$ denote the corresponding $k$-ribbon
words. Then $\mathbf{T}$ and $\mathbf{U}$ have the same shape if and
only if $\mathrm{Des}_k(w_{\mathbf{T}}) =
\mathrm{Des}_k(w_{\mathbf{U}})$ and $c_{\mathbf{T}} =
c_{\mathbf{U}}$. In particular, if we let $\mathrm{WRib}_k(c,D)$
denote the set of $k$-ribbon words with content vector $c$ and
$k$-descent set $D$, then we have established a bijective
correspondence
\begin{equation}
  \mathrm{WRib}_k(c,D) \; \stackrel{\sim}{\longleftrightarrow} \;
  \mathrm{SSYT}_k(\boldsymbol{\mu}) .
\label{eqn:correspondence}
\end{equation}

Define the \emph{set of $k$-inversions} and the \emph{$k$-inversion
  number} of a pair $(w,c)$ by
\begin{eqnarray*}
  \mathrm{Inv}_k(w,c) & = & \left\{ (i,j) \; | \; w_{i} > w_{j} \;
    \mbox{and} \; k > c_{j} - c_{i} > 0 \right\},
  \label{eqn:Invk-w}\\
  \mathrm{inv}_k(w,c) & = & \left| \mathrm{Inv}_k(w,c) \right| .
  \label{eqn:invk-w}
\end{eqnarray*}

Recalling \eqref{eqn:Invk-T}, we have
\begin{equation}
  \mathrm{Inv}_k(w_{\mathbf{T}},c_{\mathbf{T}}) \; = \;
  \mathrm{Inv}_k(\mathbf{T}) . 
\label{eqn:same-Invk}
\end{equation}
Therefore we may express LLT polynomials in terms of $k$-ribbon words
as follows.

\begin{corollary}
  Let $\boldsymbol{\mu}$ be a (skew) shape, and let $c,D$ be the
  content vector and $k$-descent set corresponding to
  $\boldsymbol{\mu}$ by \eqref{eqn:correspondence}. Then
  \begin{equation}
    \widetilde{G}^{(k)}_{\boldsymbol{\mu}}(x;q) \; = \; \sum_{(w,c)
      \in \mathrm{WRib}_k(c,D)} q^{\mathrm{inv}_k(w,c)} x^{w} \; = \;
    \sum_{\substack{(w,c) \in \mathrm{WRib}_k(c,D) \\ w \ \mathrm{a} \
        \mathrm{permutation}}} q^{\mathrm{inv}_k(w,c)} Q_{\sigma(w)}(x),
  \label{eqn:llt-words}    
  \end{equation}
  where $x^{w}$ is the monomial $x_1^{\pi_1} x_2^{\pi_2} \cdots$ when
  $w$ has weight $\pi$, and $\sigma(w)$ is defined as in
  \eqref{eqn:sigma}.
\label{cor:llt-words}
\end{corollary}

\subsection{Dual equivalence for tuples of tableaux}
\label{sec:llt-edges}

Let $V^{(k)}_{c,D}$ denote the set of permutations $w$ such that
$(w,c)$ is a standard $k$-ribbon word with $k$-descent set $D$, i.e.
\begin{equation}
  V^{(k)}_{c,D} = \left\{ w \; | \; \mbox{$(w,c)$ is a standard
      $k$-ribbon word with $\mathrm{Des}_k(w,c) = D$} \right\}.
\label{eqn:V-words}
\end{equation}
Define the distance between two letters $i$ and $j$ of $w \in
V^{(k)}_{c,D}$ by
\begin{equation}
  \mathrm{dist}(w_i,w_j) = \left| c_i - c_j \right|,
\end{equation}
with the obvious extension $\mathrm{dist}(a_1,\ldots,a_l) =
\max_{i,j}\{\mathrm{dist}(a_i,a_j)\}$. Note that if $(w,c)$ is a
standard $k$-ribbon word, then none of $\triple$ may occur with the
same content.

Similar to Definition~\ref{defn:ede}, define involutions $d_i$ and
$\widetilde{d}_i$ on permutations in which $i$ does not lie between
$\imo$ and $\ipo$ by
\begin{eqnarray}
  d_i    (\cdots\;   i  \;\cdots\;i\pm 1\;\cdots\;i\mp 1\;\cdots ) 
  & = &  \cdots\;i\mp 1\;\cdots\;i\pm 1\;\cdots\;  i   \;\cdots \; ,
  \label{eqn:d} \\
  \widetilde{d}_i(\cdots\;   i  \;\cdots\;i\pm 1\;\cdots\;i\mp 1\;\cdots ) 
  & = &  \cdots\;i\pm 1\;\cdots\;i\mp 1\;\cdots\;  i   \;\cdots \; ,
  \label{eqn:dwig}
\end{eqnarray}
where all other entries remain fixed. Note that the former involution
is precisely Haiman's dual equivalence on permutations. For fixed $k$,
combine these two maps into an involution $D^{(k)}_i$ by
\begin{equation}
  D^{(k)}_i(w) \; = \; \left\{
    \begin{array}{ll}
      d_i(w) & \mbox{if} \;\; \mathrm{dist}(\triple) > k \\
      \widetilde{d}_i(w) & \mbox{if} \;\; \mathrm{dist}(\triple) \leq k
    \end{array} \right. .
  \label{eqn:Dk}
\end{equation}

\begin{proposition}
  For $w$ a permutation, $c$ a content vector and $k>0$ an integer, we
  have
  \begin{eqnarray}
    \mathrm{Des}_k(w,c) & = & \mathrm{Des}_k(D^{(k)}_i(w),c),
    \label{eqn:same-Desk} \\
    \mathrm{inv}_k(w,c) & = & \mathrm{inv}_k(D^{(k)}_i(w),c) .
    \label{eqn:same-invk}
  \end{eqnarray}
  In particular, $D^{(k)}_i$ is a well-defined involution on
  $V^{(k)}_{c,D}$ that preserves the number of $k$-inversions.
\label{prop:preserve}
\end{proposition}

\begin{proof}
  If $i$ lies between $\imo$ and $\ipo$ in $w$, then the assertion is
  trivial. Assume then that $i$ does not lie between $\imo$ and $\ipo$
  in $w$. If $\mathrm{dist}(\triple) > k$ in $w$, then
  $\mathrm{Des}_k(w,c) = \mathrm{Des}_k(d_i(w),c)$ and
  $\mathrm{Inv}_k(w,c) = \mathrm{Inv}_k(d_i(w),c)$. Similarly, if
  $\mathrm{dist}(\triple) \leq k$ in $w$, then $\mathrm{Des}_k(w,c) =
  \mathrm{Des}_k(\widetilde{d}_i(w),c)$ and $\mathrm{inv}_k(w,c) =
  \mathrm{inv}_k(\widetilde{d}_i(w),c)$ (though $\mathrm{Inv}_k(w,c)
  \neq \mathrm{Inv}_k(\widetilde{d}_i(w),c)$). The result now follows.
\end{proof}

For each content vector $c$ of length $n$, and $k$-descent set $D$, we
construct a signed, colored graph $\G^{(k)}_{c,D}$ of type $(n,n)$ on
the vertex set $V^{(k)}_{c,D}$ as follows. Define the signature
function $\sigma: V^{(k)}_{c,D} \rightarrow \{\pm 1\}^{\nmo}$ by
\begin{equation}
  \sigma(w)_i \; = \; \left\{ 
    \begin{array}{ll}
      +1 & \; \mbox{if $i$ appears to the left of $\ipo$ in $w$} \\
      -1 & \; \mbox{if $\ipo$ appears to the left of $i$ in $w$}
    \end{array} \right. .
\label{eqn:sigma-words}
\end{equation}
By \eqref{eqn:same-Desk}, $D^{(k)}_i$ is an involution on vertices of
$V^{(k)}_{c,D}$ admitting an $i$-neighbor. Therefore for $1 < i < n$,
we may define the $i$-colored edges $E^{(k)}_{i}$ to be the set of
pairs $\{v,D_{i}^{(k)}(v)\}$ for each $v$ admitting an
$i$-neighbor. Finally, we define
\begin{equation}
  \G^{(k)}_{c,D} \; = \; \left( V^{(k)}_{c,D}, \sigma, E^{(k)} \right) .
\label{eqn:GVsE}
\end{equation}
An example of $\G^{(k)}_{c,D}$ is given in Figure~\ref{fig:LLT5}, and
additional examples may be found in Appendix~\ref{app:Dgraphs}.

\begin{figure}[ht]
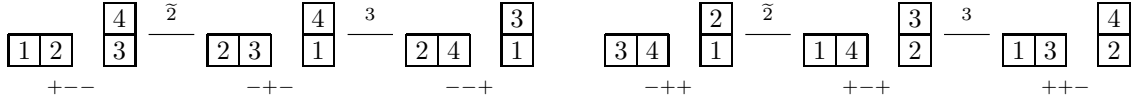

  \begin{displaymath}
    \begin{array}{\cs{4} \cs{4} \cs{4} \cs{4} \cs{4} c}
      \stab{a}{ & & & 4 \\ 1 & 2 & & 3}{+--} &
      \stab{b}{ & & & 4 \\ 2 & 3 & & 1}{-+-} &
      \stab{c}{ & & & 3 \\ 2 & 4 & & 1}{--+} &
      \stab{d}{ & & & 2 \\ 3 & 4 & & 1}{-++} &
      \stab{e}{ & & & 3 \\ 1 & 4 & & 2}{+-+} &
      \stab{f}{ & & & 4 \\ 1 & 3 & & 2}{++-} 
    \end{array}
    \psset{nodesep=5pt,linewidth=.1ex}
    \everypsbox{\scriptstyle}
    \ncline            {a}{b} \naput{\widetilde{2}}
    \ncline            {b}{c} \naput{3}
    \ncline            {d}{e} \naput{\widetilde{2}}
    \ncline            {e}{f} \naput{3}
  \end{displaymath}    
  \caption{\label{fig:LLT5}The graph
    $\G^{(2)}_{(-1,0,1,2),\{(-1,1)\}}$ on domino tableaux of shape $(
    \ (2), \ (1,1) \ )$.}
\end{figure}

By Corollary~\ref{cor:llt-words} and \eqref{eqn:llt-quasi}, the
generating function for $\G^{(k)}_{c,D}$ weighted by
$\mathrm{inv}_k(-,c)$ is given by
\begin{equation}
  \sum_{v \in V^{(k)}_{c,D}} q^{\mathrm{inv}_{k}(v,c)} Q_{\sigma(v)}(x) =
  \widetilde{G}^{(k)}_{\boldsymbol{\mu}}(x;q).
\label{eqn:llt-genfn}
\end{equation}
In particular, a formula for the Schur coefficients of the generating
function for $\G^{(k)}_{c,D}$ gives a formula for the Schur
coefficients of the LLT polynomial
$\widetilde{G}^{(k)}_{\boldsymbol{\mu}}(x;q)$. For example, since the
graph in Figure~\ref{fig:LLT5} is a dual equivalence graph, we have
\begin{displaymath}
  \widetilde{G}^{(2)}_{(2),(1,1)}(x;q) = q s_{3,1}(x) + q^2 s_{2,1,1}(x).
\end{displaymath}
In general, $\G^{(k)}_{c,D}$ does not satisfy dual equivalence axioms
$4$ or $6$; see Appendix~\ref{app:Dgraphs} for examples. These graphs
do, however, satisfy the other axioms as well as the following
weakened version of axiom $4$.

\begin{definition}
  A signed, colored graph $\G = (V,\sigma,E)$ of type $(n,N)$ is
  \emph{Schur positive for degree $m$}, denoted $\LSP_m$, if for every
  $m-2<i<n$ and every connected component $\C$ of
  $(V,\sigma,E_{i-(m-3)} \cup \cdots \cup E_i)$, the restricted degree
  $m$ generating function
  \begin{equation}
    \sum_{v \in \C} Q_{\sigma(v)_{i-(m-2),\ldots,i}}(x)
    \label{eqn:LSP}
  \end{equation}
  is symmetric and Schur positive. The graph $\G$ is \emph{locally
    Schur positive} if it is Schur positive for degrees $4$ and $5$.
\label{defn:LSP}
\end{definition}

Comparing Figures~\ref{fig:lambda4} and \ref{fig:lambda5} with the
standard dual equivalence graphs of sizes $4$ and $5$ (see
Figure~\ref{fig:G5}), dual equivalence graph axiom $4$ implies that
$\G_{\lambda}$ is locally Schur positive.

\begin{proposition}
  For each content vector $c$ and $k$-descent set $D$, the graph
  $\G^{(k)}_{c,D}$ satisfies dual equivalence graph axioms $1,2,3$ and
  $5$, is locally Schur positive, and the $k$-inversion number is
  constant on connected components.
\label{prop:ax1235}
\end{proposition}

\begin{proof}
  Axiom $1$ follows from the construction of $E^{(k)}$ using
  \eqref{eqn:Dk}, and axiom $2$ can easily be seen from equations
  (\ref{eqn:d}) and (\ref{eqn:dwig}). For axiom $3$, suppose $\{w,x\}
  \in E^{(k)}_{i}$ and $\sigma(w)_{\imt} = -\sigma(x)_{\imt}$. If $x =
  d_i(w)$, then both $\imt$ and $\ipo$ must lie between $\imo$ and
  $i$. In particular, $\sigma(w)_{\imt} = -\sigma(w)_{\imo}$. If $x =
  \widetilde{d}_i(w)$, then $\imt$ must lie between the position of
  $\imo$ in $w$ and the position of $\imo$ in $x$. In particular,
  $\imt$ must lie between $\imo$ and $i$ in both $w$ and $x$, and so
  again $\sigma(w)_{\imt} = -\sigma(w)_{\imo}$. The result for
  $\{w,x\} \in E^{(k)}_{i}$ with $\sigma(w)_{\ipo} =
  -\sigma(x)_{\ipo}$ is completely analogous. Axiom $5$ follows from
  the fact that if $w$ admits both an $i$-neighbor and a $j$-neighbor
  for some $|i-j| \geq 3$, then $D^{(k)}_{i}D^{(k)}_{j}(w) =
  D^{(k)}_{j}D^{(k)}_{i}(w)$.

  Notice that one may regard the data $k,c,D$ as specifying attacking
  positions in a permutation. That is, for a permutation $w$, say that
  $w_p$ attacks $w_q$ if $p<q$ and the difference in contents between
  $p$ and $q$ is at most $k$ when $(w,c)$ is regarded as a $k$-ribbon
  word. Therefore the structure of $\G^{(k)}_{c,D}$ is given by the
  graph on permutations where the edges are given by $\widetilde{d}_i$
  if $i$ attacks the rightmost of $i\pm 1$ or if the leftmost of $i
  \pm 1$ attacks $i$, and by $d_i$ otherwise. Since attacking
  positions are determined by distance, if $w_p$ attacks $w_r$ with
  $p<q<r$, then $w_p$ attacks $w_q$ as well. Therefore the graph on
  permutations of size $n$ is determined by $(a_1,\ldots,a_{n-1})$,
  where $a_j$ is the rightmost position that $w_j$ attacks. Since only
  triples of letters are of interest, we may assume that each position
  attacks its right neighbor, and so $j+1 \leq a_j \leq n$. Moreover,
  if $w_p$ attacks $w_r$ with $p<q<r$, then $w_q$ attacks $w_r$ as
  well. Therefore $a_p \leq a_{p+1}$. Hence the number of attacking
  vectors to consider for permutations of length $n$ is the $n-1$st
  Catalan number. In particular, this gives $5$ graph structures on
  permutations of $4$ and $14$ graph structures on permutations of
  $5$. These cases can be checked by hand or by computer (see
  Appendix~\ref{app:code}). This yields exactly $25$ possible
  non-isomorphic connected components of $(V,\sigma,E_{\imt} \cup
  E_{\imo} \cup E_i)$ in $\G^{(k)}_{c,D}$ for all possible $k,c,D$. Of
  these, $7$ correspond to the standard dual equivalence graphs of
  type $(5,5)$ and the remaining $18$ are locally Schur positive.

  Finally, the $k$-inversion number is constant on connected
  components of $\G^{(k)}_{c,D}$ by Proposition~\ref{prop:preserve}.
\end{proof}

As foreshadowed by Definition~\ref{defn:LSP}, the generating function
of a connected component of the signed, colored graph for LLT
polynomials is not, in general, a single Schur function, though it is
always Schur positive. In Section~\ref{sec:Dgraphs}, we describe an
algorithm by which the edges of every connected component of
$\G^{(k)}_{c,D}$ can be transformed so that the resulting graph is
indeed a dual equivalence graph. We do this inductively by
constructing a sequence of signed, colored graphs
\begin{displaymath}
  \G^{(k)}_{c,D} = \G_2 , \ldots , \G_{\nmo} = \widetilde{\G}^{(k)}_{c,D}
\end{displaymath}
on the same vertex set with the same signature function with the
following properties. For each $i=2,\ldots,\nmo$, the graph $\G_i$
satisfies dual equivalence graph axioms $1,2,3$ and $5$, and the
$(\ipo,N)$-restriction of $\G_i$ satisfies axioms $4$ and $6$ (and so
is a dual equivalence graph). Furthermore, vertices paired by $E_i$ in
$\G_i$ have the property that they lie on the same connected component
of $(V,\sigma,E_{2} \cup\cdots\cup E_{i})$ in $\G_{\imo}$. This
construction proves the following.

\begin{theorem}
  For $\boldsymbol{\mu}$ a $k$-tuple of (skew) shapes, let $c,D$ be the
  corresponding pair by \eqref{eqn:correspondence}, and let
  $\G^{(k)}_{c,D}$ be the signed, colored graph constructed
  above. Then for every connected component $\C$ of $\G^{(k)}_{c,D}$,
  the sum $\sum_{v \in V(\C)} Q_{\sigma(v)}(X)$ is symmetric and Schur
  positive.
\label{thm:Dgraph}
\end{theorem}

\begin{corollary}
  Let $\widetilde{\G}^{(k)}_{c,D}$ be the dual equivalence graph resulting
  from the transformation of the graph $\G^{(k)}_{c,D}$. Then for
  $\lambda$ a partition, we have
  \begin{equation}
    \widetilde{K}^{(k)}_{\lambda,\boldsymbol{\mu}}(q) \; = \; \sum_{\C \cong
      \G_{\lambda}} q^{\mathrm{inv}_{k}(\C)}, 
  \end{equation}
  where the sum is taken over all connected components $\C$ of
  $\widetilde{\G}^{(k)}_{c,D}$ that are isomorphic to
  $\G_{\lambda}$. In particular,
  $\widetilde{K}^{(k)}_{\lambda,\boldsymbol{\mu}}(q) \in
  \mathbb{N}[q]$, and, by \eqref{eqn:hag-llt-kostka},
  $\widetilde{K}_{\lambda,\mu}(q,t) \in \mathbb{N}[q,t]$.
\label{cor:Dgraph}
\end{corollary}

The proof of Theorem~\ref{thm:Dgraph} is the content of
Section~\ref{sec:Dgraphs}. Before delving into the proof, we consider
two extremal cases of $\G^{(k)}_{c,D}$ where the connected components
have particularly nice Schur expansions that can be proved by more
elementary means.

\subsection{Special cases}
\label{sec:llt-special}

Since $\mathrm{dist}(\triple) \geq 2$ for every $w \in V^{(k)}_{c,D}$,
$D^{(1)}_i$ is just the standard elementary dual equivalence on
$\triple$. Therefore $\G^{(1)}_{c,D}$ is isomorphic to the standard
dual equivalence graph $\G_{\lambda}$ for a unique partition
$\lambda$.

When $k \geq 3$, $E^{(k)}_i$ will not give the edges of a dual
equivalence graph. For instance, if $w$ has the pattern $2431$ with
$\mathrm{dist}(1,2,3) \leq k$, then $D_{2}^{(k)}(w)$ contains the
pattern $3412$. By axiom $4$, a dual equivalence graph must have
$\{w,D_{2}^{(k)}(w)\} \in E^{(k)}_2 \cap E^{(k)}_3$. However,
$D_{2}^{(k)}(w) \neq D_{3}^{(k)}(w)$, so this is not the case for
$\G^{(k)}_{c,D}$.  Therefore for $k \geq 3$, Theorem~\ref{thm:Dgraph}
is the best we can hope for. When $k=2$, however, this problematic
case does not arise, and we have the following result.

\begin{theorem}
  The graph $\G^{(2)}_{c,D}$ on $2$-ribbon words with content vector
  $c$ and $2$-descent set $D$ is a dual equivalence graph, and the
  $2$-inversion number is constant on connected components.
\label{thm:dominoes}
\end{theorem}

\begin{proof}
  By Proposition~\ref{prop:ax1235}, it suffices to show that dual
  equivalence axioms $4$ and $6$ hold. Since $k=2$, if $x =
  \widetilde{d}_i(w)$, then $\sigma(w)_{j} = \sigma(x)_{j}$ for all $j
  \neq \imo,i$. In particular, if $\{w,x\} \in E_{i}$ and
  $\sigma(w)_{\imt} = -\sigma(x)_{\imt}$, then $d_{i}(w) = x =
  d_{\imo}(w)$. This establishes axiom $4$ when $n \leq 4$.

  To prove that connected components of $(V^{(2)}_{c,D}, \sigma,
  E^{(2)}_{\imt} \cup E^{(2)}_{\imo} \cup E^{(2)}_{i})$ have the
  correct form, note that it suffices to show that if $x =
  D_{i}^{(2)}(w) = D_{\imo}^{(2)}(w)$ and $x$ admits an
  $\imt$-neighbor, then letting $y = D_{i}^{(2)}D_{\imt}^{(2)}(x)$, we
  have $D_{\imt}^{(2)}(y) = D_{\imo}^{(2)}(y)$. In this case, $x$ must
  have $\imt$ and $\ipo$ lying between $i$ and $\imo$ which have
  contents more than $2$ apart. Then in $D_{\imt}^{(2)}(x)$, $\imh,
  \imo$ and $\ipo$ will all lie between $i$ and $\imt$ which must also
  have contents more than $2$ apart. In
  $y=D_{i}^{(2)}D_{\imt}^{(2)}(x)$, $\imh$ and $i$ will both lie
  between $\imt$ and $\imo$ which must have contents more than $2$
  apart. Therefore $D_{\imt}^{(2)}(y) = d_{\imt}(y) = d_{\imo}(y) =
  D_{\imo}^{(2)}(y)$.

  Finally, to establish axiom $6$, it is helpful to have a
  characterization of the dual equivalence classes under
  $D_{i}^{(2)}$. It follows from \cite{A-mahonian08} that for a given
  dual equivalence class $\C$ there exists a partition $\lambda$ such
  that the Robinson-Schensted algorithm gives a bijection between $\C$
  and standard tableaux of shape $\lambda$ that preserves
  signatures. While this fact alone is enough to prove that
  $\G^{(2)}_{c,D}$ is a dual equivalence graph, it can also be used to
  give a direct description of the dual equivalence classes, from
  which a more direct proof of axiom $6$ follows.
\end{proof}

Since Theorem~\ref{thm:dominoes} does not use the transformations of
Section~\ref{sec:Dgraphs}, we obtain a simple proof of positivity of
LLT polynomials when $k=2$, and also of Macdonald polynomials indexed
by a partition with at most $2$ columns. For a bijective proof, see
also \cite{A-mahonian08}.

Next consider the case when $k \geq c_n - c_1$ and so $D_{i}^{(k)} =
\widetilde{d}_i$ for all $i$. Now there are no double edges in
$\G^{(k)}_{c,D}$. For the standard dual equivalence graphs,
$\G_{\lambda}$ has no double edges if and only if $\lambda$ is a hook,
i.e. $\lambda = (m,1^{n-m})$ for some $m$. Therefore the generating
function for a dual equivalence graph with no double edges is a sum of
Schur functions indexed by hooks. The analog of this fact for
$\G^{(k)}_{c,D}$ is that the generating function is a sum of skew
Schur functions indexed by ribbons.

Let $\nu$ be a ribbon of size $n$. Label the cells of $\nu$ from $1$
to $n$ in increasing order of content. Define the \emph{descent set of
  $\nu$}, denoted $\mathrm{Des}(\nu)$, to be the set of indices $i$
such that the cell labeled $\ipo$ lies south of the cell labeled
$i$.  Define the \emph{major index of a ribbon} by
\begin{equation}
  \mathrm{maj}(\nu) = \sum_{i \in \mathrm{Des}(\nu)} i.
\label{eqn:rib-maj}
\end{equation}
Notice that if $R$ is a filling of a column, and we reshape $R$ into a
semi-standard ribbon as described in Section~\ref{sec:pre-mac}, say of
shape $\nu$, then \eqref{eqn:rib-maj} agrees with \eqref{eqn:maj} in
the sense that $\mathrm{maj}(\nu) = \mathrm{maj}(R)$.

Any connected component of $\G^{(k)}_{c,D}$ such that $D_{i}^{(k)} =
\widetilde{d}_i$ on the entire component not only has constant
$k$-inversion number, but the relative ordering of the first and last
letters of each vertex is constant as well. That is, for $\C$ a
connected component of $\G^{(k)}_{c,D}$, $w_1 > w_n$ for some $w \in
V(\C)$ if and only if $w_1 > w_n$ for all $w \in V(\C)$. In the
affirmative case, say that \emph{$(1,n)$ is an inversion in $\C$}.

\begin{theorem}
  Let $\G^{(k)}_{c,D}$ be the signed, colored graph of type $(n,n)$ on
  $k$-ribbon words with contents $c$ and $k$-descent set $D$. Let $\C$
  be a connected component of $\G_{c,D}$ such that $D^{(k)}_{i}(v) =
  \widetilde{d}_i(v)$ for all $v \in V(\C)$. Then
  \begin{equation}
    \sum_{v \in V(\C)} Q_{\sigma(v)}(x) = \sum_{\nu \in \mathrm{Rib}(\C)} s_{\nu},
  \label{eqn:rib-expand}
  \end{equation}
  where $\mathrm{Rib}(\C)$ is the set of ribbons of length $n$ with major
  index equal to $\mathrm{inv}_k(\C)$ such that $\nmo$ is a descent if and
  only if $(1,n)$ is an inversion in $\C$.
\label{thm:k=n}
\end{theorem}

\begin{proof}
  From the hypotheses on $\C$, we may assume that $k=n$, $c =
  (1,\ldots,n)$ and $D = \emptyset$. Therefore $V^{(k)}_{c,D}$ is just
  the set of permutations of $[n]$ thought of as words. In this case,
  $k$-inversions are just the usual inversions for a permutation. By
  earlier remarks, for $w,v \in V(\C)$, $\mathrm{inv}(w) =
  \mathrm{inv}(v)$ and $(1,n) \in \mathrm{Inv}(w)$ if and only if
  $(1,n) \in \mathrm{Inv}(v)$. In fact, it is an exercise to show that
  this necessary condition for two vertices to coexist in $V(\C)$ is
  also sufficient. That is to say, $V(\C)$ is the set of words $w$
  with $\mathrm{inv}(w) = \mathrm{inv}(\C)$ and $(1,n) \in
  \mathrm{Inv}(w)$ if and only if $(1,n)$ is an inversion of $\C$.

  Recall Foata's bijection on words \cite{Foata1968}. For $w$ a word
  and $x$ a letter, $\phi$ is built recursively using an inner
  function $\gamma_x$ by $ \phi(wx) = \gamma_x \left( \phi(w) \right)
  x$. From this structure it follows that the last letter of $w$ is
  the same as the last letter of $\phi(w)$. Furthermore, $\gamma_x$ is
  defined so that the last letter of $w$ is greater than $x$ if and
  only if the first letter of $\gamma_x(w)$ is greater than $x$, and
  $\phi$ preserves the descent set of the inverse permutation,
  i.e. $\sigma(w) = \sigma(\phi(w))$.  Finally, the bijection
  satisfies $\mathrm{maj}(w) = \mathrm{inv}(\phi(w))$.  Summarizing
  these properties, $\phi$ is a $\sigma$-preserving bijection between
  the following sets:
  \begin{eqnarray*}
    \left\{ w \; | \; \mathrm{inv}(w)=j \; \mbox{and} \; (1,n) \in
      \mathrm{Inv}(w) \right\} & \stackrel{\sim}{\longleftrightarrow}
    & \left\{ w \; | \; \mathrm{maj}(w)=j \; \mbox{and} \; \nmo \in
      \mathrm{Des}(w) \right\}, \\ 
    \left\{ w \; | \; \mathrm{inv}(w)=j \; \mbox{and} \; (1,n) \not\in
      \mathrm{Inv}(w) \right\} & \stackrel{\sim}{\longleftrightarrow}
    & \left\{ w \; | \; \mathrm{maj}(w)=j \; \mbox{and} \; \nmo
      \not\in \mathrm{Des}(w) \right\}. 
  \end{eqnarray*}

  A standard filling of a ribbon $\nu$ is just a permutation $w$ such
  that $\mathrm{Des}(w) = \mathrm{Des}(\nu)$. Therefore by
  \eqref{eqn:quasi-s}, the Schur function $s_{\nu}$ may be expressed
  as
  \begin{equation}
    s_{\nu}(x) = \sum_{\mathrm{Des}(w) = \mathrm{Des}(\nu)} Q_{\sigma(w)}(x).
    \label{eqn:rib-s}
  \end{equation}
  Applying $\phi$ to this formula yields \eqref{eqn:rib-expand}.
\end{proof}

%
\section{Transformation into a dual equivalence graph }
%
\label{sec:Dgraphs}

\subsection{Packages and type}
\label{sec:D-type}

The algorithm used to transform $\G^{(k)}_{c,D}$ into a dual
equivalence graph primarily utilizes three transformations,
$\varphi_i$ $\psi_i$, and $\theta_i$, each of which identifies two
$i$-edges on the same connected component of $E_{2} \cup\cdots\cup
E_{i}$ and swaps the connections in the unique way that maintains the
reversal of $\sigma_{\imo}$ and $\sigma_i$. For example, in
Figure~\ref{fig:swap}, the $i$-edges given by solid lines are replaced
with $i$-edges given by the dashed lines.

\begin{figure}[ht]
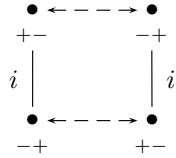

  \begin{displaymath}
    \begin{array}{\cs{7}c}
      \sbull{a}{+-} & \sbull{b}{-+} \\[6ex]
      \sbull{c}{-+} & \sbull{d}{+-}
    \end{array}
    \psset{nodesep=3pt,linewidth=.1ex}
    \ncline {aa}{c} \nbput{i}
    \ncline {bb}{d} \naput{i}
    \ncline[linestyle=dashed] {<->}{a}{b} 
    \ncline[linestyle=dashed] {<->}{c}{d} 
  \end{displaymath}
  \caption{\label{fig:swap} An illustration of how two $i$-edges are
    swapped in the transformation process.}
\end{figure}

The basic structure of these maps is depicted in
Figure~\ref{fig:involutions}. Axiom $4$ restricts the lengths of
$2$-color strings in the following way (see Definition~\ref{defn:type}
for the definition of $i$-type). Figure~\ref{fig:lambda4} forces the
number of edges of a nontrivial connected component of $E_{\imo} \cup
E_{i}$ to be two, either with three distinct vertices (in the cases
other than $i$-type W) or forming a cycle with two vertices (in the
case of $i$-type W). The map $\varphi_i$ swaps $i$-edges on connected
components of $E_{\imo} \cup E_i$ with more than two edges. Similarly,
Figure~\ref{fig:lambda5} forces the number of edges of a nontrivial
connected component of $E_{\imt} \cup E_{i}$ to be one (in the case of
$i$-type A) or four, where there are either five distinct vertices (in
the case of $i$-type B) or four vertices forming a cycle (in the case
of $i$-type C). The map $\psi_i$ swaps $i$-edges on connected
components of $E_{\imt} \cup E_i$ with more than four edges.

\begin{figure}[ht]
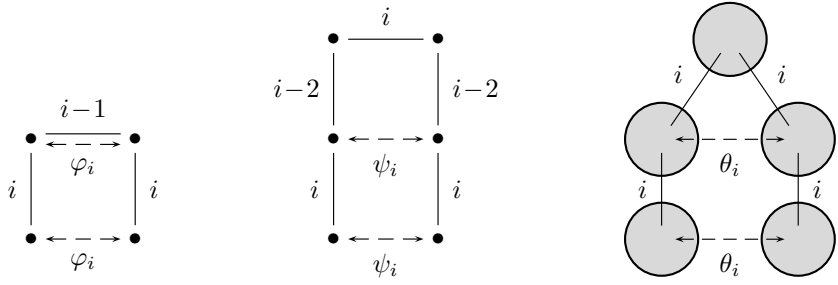

  \begin{displaymath}
    \begin{array}{\cs{8}c}
      & \\[6ex]
      \rnode{C}{\B} & \rnode{D}{\B}  \\[6ex] 
      \rnode{B}{\B} & \rnode{E}{\B}
    \end{array}
    \hspace{6em}
    \begin{array}{\cs{8}c}
      \rnode{a2}{\B} & \rnode{b2}{\B} \\[6ex] 
      \rnode{e2}{\B} & \rnode{f2}{\B} \\[6ex] 
      \rnode{h2}{\B} & \rnode{i2}{\B} 
    \end{array}
    \hspace{7em}
    \begin{array}{\cs{5}\cs{5}c}
      & \rnode{X}{%
        \psset{xunit=1ex}
        \psset{yunit=1ex}
        \pspicture(0,0)(1,1)
        \pscircle[fillstyle=solid,fillcolor=lightgray](0.5,0.5){0.5}
        \endpspicture} & \\[6ex]
      \rnode{Y1}{%
        \psset{xunit=1ex}
        \psset{yunit=1ex}
        \pspicture(0,0)(1,1)
        \pscircle[fillstyle=solid,fillcolor=lightgray](0.5,0.5){0.5}
        \endpspicture} & & \rnode{Y2}{%
        \psset{xunit=1ex}
        \psset{yunit=1ex}
        \pspicture(0,0)(1,1)
        \pscircle[fillstyle=solid,fillcolor=lightgray](0.5,0.5){0.5}
        \endpspicture} \\[6ex]
      \rnode{Z1}{%
        \psset{xunit=1ex}
        \psset{yunit=1ex}
        \pspicture(0,0)(1,1)
        \pscircle[fillstyle=solid,fillcolor=lightgray](0.5,0.5){0.5}
        \endpspicture} & & \rnode{Z2}{%
        \psset{xunit=1ex}
        \psset{yunit=1ex}
        \pspicture(0,0)(1,1)
        \pscircle[fillstyle=solid,fillcolor=lightgray](0.5,0.5){0.5}
        \endpspicture}
    \end{array}
    \psset{nodesep=3pt,linewidth=.1ex}
    \ncline[offset=2pt] {C}{D} \naput{\imo}
    \ncline {B}{C} \naput{i}
    \ncline {D}{E} \naput{i}
    \ncline[offset=-2pt,linestyle=dashed] {<->}{C}{D} \nbput{\varphi_i}
    \ncline[linestyle=dashed] {<->}{B}{E} \nbput{\varphi_i}
    \ncline {a2}{b2} \naput{i}
    \ncline {a2}{e2} \nbput{\imt}
    \ncline {b2}{f2} \naput{\imt}
    \ncline {e2}{h2} \nbput{i}
    \ncline {f2}{i2} \naput{i}
    \ncline[linestyle=dashed] {<->}{e2}{f2} \nbput{\psi_i}
    \ncline[linestyle=dashed] {<->}{h2}{i2} \nbput{\psi_i}
    \ncline {Y1}{X} \naput{i}
    \ncline {X}{Y2} \naput{i}
    \ncline {Y1}{Z1} \nbput{i}
    \ncline {Y2}{Z2} \naput{i}
    \ncline[linestyle=dashed]{<->} {Y1}{Y2} \nbput{\theta_i}
    \ncline[linestyle=dashed]{<->} {Z1}{Z2} \nbput{\theta_i}
  \end{displaymath}
  \caption{\label{fig:involutions} Illustrations of the involutions
    $\varphi_i$, $\psi_i$, and $\theta_i$ used to redefine $E_i$.}
\end{figure}

Axiom $6$ restricts the size of $E_2 \cup \cdots \cup E_{\imo}$
isomorphism classes of a connected component of $E_2 \cup \cdots \cup
E_i$ to be one. The map $\theta_i$ swaps $i$-edges on connected
components of $E_2 \cup \cdots \cup E_i$ with more than one member of
a given $E_2 \cup \cdots \cup E_{\imo}$ isomorphism class.

By construction, these transformations preserve axiom $1$. In order to
maintain axioms $2$ and $5$, we introduce the notion of the
$i$-package of a vertex admitting an $i$-neighbor. By axiom $5$, if
$\{w,x\} \in E_i$ and $\{x,y\} \in E_j$ for $|i-j| \geq 3$, then
$\{w,v\} \in E_j$ and $\{v,y\} \in E_i$ for some $v \in V$. Changing a
single $i$-edge may result in a violation of this condition. Therefore
when one $i$-edge is changed, all other $i$-edges that subsequently
violate axiom $5$ must also be changed, as illustrated in
Figure~\ref{fig:axiom5}.

\begin{figure}[ht]
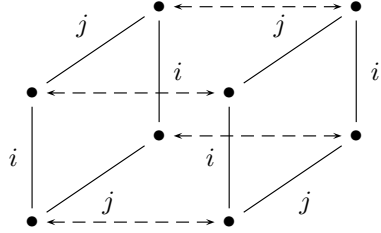

  \begin{displaymath}
    \begin{array}{\cs{5}\cs{5}\cs{5}\cs{5}\cs{5}c}
      & & \rnode{a3}{\bullet} & & & \rnode{a6}{\bullet} \\[1ex]
      & & & & & \\[1ex]
      \rnode{c1}{\bullet} & & & \rnode{c4}{\bullet} & & \\[1ex]
      & & \rnode{d3}{\bullet} & & & \rnode{d6}{\bullet} \\[1ex]
      & & & & & \\[1ex]
      \rnode{f1}{\bullet} & & & \rnode{f4}{\bullet} & & 
    \end{array}
    \psset{nodesep=3pt,linewidth=.1ex}
    \ncline {c1}{f1} \nbput{i}
    \ncline {c4}{f4} \nbput{i}
    \ncline {a3}{d3} \naput{i}
    \ncline {a6}{d6} \naput{i}
    \ncline {c1}{a3} \naput{j}
    \ncline {c4}{a6} \naput{j}
    \ncline {f1}{d3} \nbput{j}
    \ncline {f4}{d6} \nbput{j}
    \ncline[linestyle=dashed]{<->} {a3}{a6}
    \ncline[linestyle=dashed]{<->} {c1}{c4}
    \ncline[linestyle=dashed]{<->} {d3}{d6}
    \ncline[linestyle=dashed]{<->} {f1}{f4}
  \end{displaymath}
  \caption{\label{fig:axiom5} An illustration of how to maintain axiom
    $5$ when swapping $i$-edges.}
\end{figure}

\begin{definition}
  Let $(V,\sigma,E)$ be a signed, colored graph of type $(n,N)$
  satisfying axioms $1,2$ and $5$. For $w$ a vertex of $V$, the
  \emph{$i$-package of $w$} is the connected component containing $w$
  of
  \begin{displaymath}
    \left( V, (\sigma_1, \ldots, \sigma_{\imh}, \sigma_{\ipt}, \ldots,
      \sigma_{N-1}), E_2 \cup\cdots\cup E_{\imh} \cup E_{\iph}
      \cup\cdots\cup E_{\nmo} \right)
    \vspace{-\baselineskip}
  \end{displaymath}
  \label{defn:package}
\end{definition}

By axiom $2$, both $\sigma_{\imo}$ and $\sigma_{i}$ are constant on
$i$-packages. Therefore $w$ admits an $i$-neighbor if and only if
every vertex of the $i$-package of $w$ admits an $i$-neighbor. By
axiom $5$, knowing $E_i(w)$ determines $E_i$ on the entire $i$-package
of $w$. That is to say, $E_i$ may be regarded as an isomorphism
between the $i$-packages of $w$ and $E_i(w)$ that preserves $\sigma_1,
\ldots, \sigma_{\imh}, \sigma_{\ipt}, \ldots, \sigma_{N-1}$. If the
four vertices in Figure~\ref{fig:swap} have isomorphic $i$-packages,
we can swap all $i$-edges on the corresponding $i$-packages while
maintaining axioms $2$ and $5$.

By axioms $2$ and $5$, $E_h$ commutes with $E_j$ whenever $h\leq \imh$
and $j\geq\iph$. Bearing this in mind, the two halves of an
$i$-package, namely $E_2 \cup\cdots\cup E_{\imh}$ and $E_{\iph}
\cup\cdots\cup E_{\nmo}$, can be and often are handled separately in
the following sections. Most often, establishing results for $E_{\iph}
\cup\cdots\cup E_{\nmo}$ is straight-forward, though the same results
for $E_2 \cup\cdots\cup E_{\imh}$ may require considerable work.

To track axiom $3$ throughout the transformation process, it is
helpful to consider the following reformulation: For $\{w,x\} \in
E_{i}$, at least one of $w$ or $x$ admits an $i\pm 1$-neighbor. To be
more precise, if $i>2$, then at least one of $w$ or $x$ admits an
$\imo$-neighbor, and if $i<N\!-\!1$, then at least one of $w$ or $x$
admits an $\ipo$-neighbor. To see the equivalence, note that by axiom
$1$, neither $w$ nor $x$ will admit an $\imo$-neighbor if and only if
$\sigma(w)_{\imt} = \sigma(w)_{\imo}$ and $\sigma(x)_{\imt} =
\sigma(x)_{\imo}$. By axioms 1 and 2, this implies $\sigma(w)_{\imt} =
\sigma(w)_{\imo} = -\sigma(x)_{\imo} = - \sigma(x)_{\imt}$. The
analogous argument holds for $\ipo$. Therefore we will often prove
that axiom 3 holds by showing that at least one of $w$ and $E_i(w)$
admits an $\imo$-neighbor and at least one admits an
$\ipo$-neighbor. When axiom $3$ holds, we often utilize the
observation that both $w$ and $E_i(w)$ admit an $\imo$-neighbor if and
only if $\sigma(w)_{\imt} = - \sigma(E_i(w))_{\imt}$ and $w$ and
$E_i(w)$ admit an $\ipo$-neighbor if and only if $\sigma(w)_{\ipo} = -
\sigma(E_i(w))_{\ipo}$.

\begin{remark}
  For a signed, colored graph of type $(n,n)$ satisfying axiom $1$,
  axiom $3$ is implied by axiom $4$ and even by the weaker local Schur
  positivity condition. Indeed, if neither $w$ nor $E_i(w)$ admits an
  $\imo$-neighbor (resp. $\ipo$-neighbor) then the connected component
  of $E_{\imo} \cup E_{i}$ (resp. $E_{i} \cup E_{\ipo}$) containing
  $w$ consists solely of $w$ and $E_i(w)$ forcing the restricted
  degree 4 generating function to be $Q_{++-}+Q_{--+}$, which is not
  Schur positive. The requirement that the graph be of type $(n,n)$ is
  necessary in order to ensure that $E_{\ipo}$ edges exist in the
  graph. If the graph is of type $(n,N)$ with $n<N$, then neither
  local Schur positivity nor axiom $4$ is enough to ensure axiom $3$.
\end{remark}

To handle local Schur positivity, we introduce the notion of the
$i$-type of a vertex. In the case of a dual equivalence graph, a
vertex that is part of a double edge for $E_{\imo}$ and $E_i$ has
$i$-type W (compare Figure~\ref{fig:lambda4} with $i$-type W in
Figure~\ref{fig:type}), and otherwise the $i$-type of a vertex
determines the shape of the connected component of $(V, \sigma,
E_{\imt} \cup E_{\imo} \cup E_{i})$ containing the vertex (compare
Figure~\ref{fig:lambda5} with $i$-types A, B, and C in
Figure~\ref{fig:type}). More generally, we have the following.

\begin{definition}
  Let $\G$ be a signed, colored graph of type $(n,N)$ satisfying
  axioms $1, 2, 3$ and $5$. For $i \leq n$ with $i<N$, the
  \emph{$i$-type of a vertex $w$ of $\G$} admitting an $i$-neighbor is
    \begin{itemize}
    \item \emph{$i$-type W} if $\sigma(w)_{i} = -\sigma(E_{\imo}(w))_{i}$;
    \item \emph{$i$-type A} if $\sigma(w)_{i} =
      \sigma(E_{\imo}(w))_{i}$ and $w$ does not admit an $\imt$-neighbor;
    \item \emph{$i$-type B} if $\sigma(w)_{i} =
      \sigma(E_{\imo}(w))_{i}$ and $w$ admits an $\imt$-neighbor and
      if $w$ admits an $\imo$-neighbor, then $\sigma(w)_{\imo} =
      -\sigma(E_{\imt}(w))_{\imo}$; otherwise, $\sigma(w)_{i} =
      -\sigma(E_{\imo}E_{\imt}(w))_{i}$;
    \item \emph{$i$-type C} if $\sigma(w)_{i} =
      \sigma(E_{\imo}(w))_{i}$ and $w$ admits an $\imt$-neighbor and
      if $w$ admits an $\imo$-neighbor, then $\sigma(w)_{\imo} =
      \sigma(E_{\imt}(w))_{\imo}$; otherwise, $\sigma(w)_{i} =
      \sigma(E_{\imo}E_{\imt}(w))_{i}$.
    \end{itemize}
  \label{defn:type}
\end{definition}

\begin{figure}[ht]
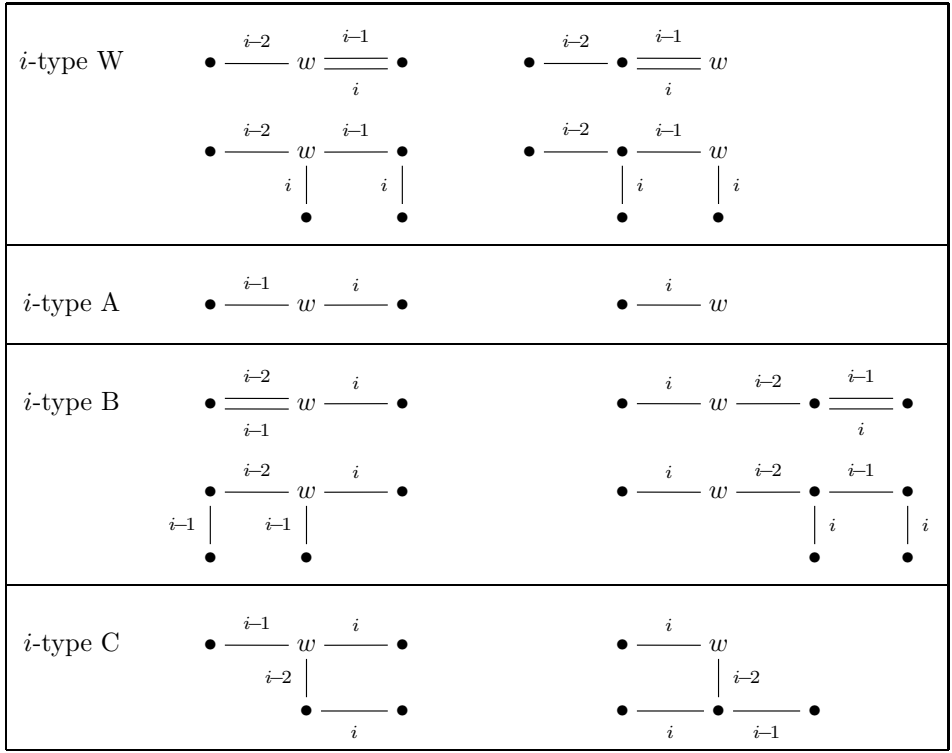

  \begin{displaymath}
    \begin{array}{|\cs{7}\cs{7}\cs{7}\cs{10}\cs{7}\cs{7}\cs{7}\cs{7}\cs{3}|}
      \hline 
      & & & & & & & & \\[1ex]
      \mbox{$i$-type W}
      & \rnode{d1}{\bullet} & \rnode{d2}{w} & \rnode{d3}{\bullet} &
      \rnode{d4}{\bullet} & \rnode{d5}{\bullet} & \rnode{d6}{w} & & \\[5ex]
      & \rnode{D1}{\bullet} & \rnode{D2}{w} & \rnode{D3}{\bullet} &
      \rnode{D4}{\bullet} & \rnode{D5}{\bullet} & \rnode{D6}{w} & & \\[3ex]
      & & \rnode{DD2}{\bullet} & \rnode{DD3}{\bullet} &
      & \rnode{DD5}{\bullet} & \rnode{DD6}{\bullet} & & \\[1ex]\hline
      & & & & & & & & \\[1ex]
      \mbox{$i$-type A} 
      & \rnode{a1}{\bullet} & \rnode{a2}{w} & \rnode{a3}{\bullet} &
      & \rnode{a4}{\bullet} & \rnode{a5}{w} & & \\[2ex]\hline
      & & & & & & & & \\[1ex]
      \mbox{$i$-type B} 
      & \rnode{b1}{\bullet} & \rnode{b2}{w} & \rnode{b3}{\bullet} &
      & \rnode{b4}{\bullet} & \rnode{b5}{w} & \rnode{b6}{\bullet} & \rnode{b7}{\bullet} \\[5ex]
      & \rnode{B1}{\bullet} & \rnode{B2}{w} & \rnode{B3}{\bullet} &
      & \rnode{B4}{\bullet} & \rnode{B5}{w} & \rnode{B6}{\bullet} & \rnode{B7}{\bullet} \\[3ex]
      & \rnode{BB1}{\bullet} & \rnode{BB2}{\bullet} & &
      & & & \rnode{BB6}{\bullet} & \rnode{BB7}{\bullet} \\[1ex]\hline
      & & & & & & & & \\[1ex]
      \mbox{$i$-type C} 
      & \rnode{c1}{\bullet} & \rnode{c2}{w} & \rnode{c3}{\bullet} &
      & \rnode{c4}{\bullet} & \rnode{c5}{w} & & \\[3ex]
      & & \rnode{cc2}{\bullet} & \rnode{cc1}{\bullet} &
      & \rnode{cc4}{\bullet} & \rnode{cc5}{\bullet} &
      \rnode{cc6}{\bullet} & \\[2ex]\hline
    \end{array}
    \psset{nodesep=3pt,linewidth=.1ex}
    \everypsbox{\scriptstyle}
    \ncline {d1}{d2} \naput{\imt}
    \ncline[offset=2pt] {d2}{d3} \naput{\imo}
    \ncline[offset=2pt] {d3}{d2} \naput{i}
    \ncline {d4}{d5} \naput{\imt}
    \ncline[offset=2pt] {d5}{d6} \naput{\imo}
    \ncline[offset=2pt] {d6}{d5} \naput{i}
    \ncline {D1}{D2} \naput{\imt}
    \ncline {D2}{D3} \naput{\imo}
    \ncline {D2}{DD2} \nbput{i}
    \ncline {D3}{DD3} \nbput{i}
    \ncline {D4}{D5} \naput{\imt}
    \ncline {D5}{D6} \naput{\imo}
    \ncline {D5}{DD5} \naput{i}
    \ncline {D6}{DD6} \naput{i}
    \ncline {a1}{a2} \naput{\imo}
    \ncline {a2}{a3} \naput{i}
    \ncline {a4}{a5} \naput{i}
    \ncline[offset=2pt] {b1}{b2} \naput{\imt}
    \ncline[offset=2pt] {b2}{b1} \naput{\imo}
    \ncline {b2}{b3} \naput{i}
    \ncline {b4}{b5} \naput{i}
    \ncline {b5}{b6} \naput{\imt}
    \ncline[offset=2pt] {b6}{b7} \naput{\imo}
    \ncline[offset=2pt] {b7}{b6} \naput{i}
    \ncline {B1}{B2} \naput{\imt}
    \ncline {B1}{BB1} \nbput{\imo}
    \ncline {B2}{BB2} \nbput{\imo}
    \ncline {B2}{B3} \naput{i}
    \ncline {B4}{B5} \naput{i}
    \ncline {B5}{B6} \naput{\imt}
    \ncline {B6}{B7} \naput{\imo}
    \ncline {B6}{BB6} \naput{i}
    \ncline {B7}{BB7} \naput{i}
    \ncline {c1}{c2} \naput{\imo}
    \ncline {c2}{c3} \naput{i}
    \ncline {c2}{cc2} \nbput{\imt}
    \ncline {cc2}{cc1} \nbput{i}
    \ncline {c4}{c5} \naput{i}
    \ncline {c5}{cc5} \naput{\imt}
    \ncline {cc4}{cc5} \nbput{i}
    \ncline {cc5}{cc6} \nbput{\imo}
  \end{displaymath}
  \caption{\label{fig:type} An illustration of $i$-type of $w$ based
    on neighboring $E_{\imt}$ and $E_{\imo}$ edges.}
\end{figure}

The $i$-type of $w$ is determined by the connected component of
$E_{\imt} \cup E_{\imo}$ containing $w$. For $i$-type W, if
$\sigma(w)_{i} = -\sigma(E_{\imo}(w))_{i}$, then certainly
$E_{\imo}(w) \neq w$ so $w$ does in fact have an $\imo$-neighbor. For
the other $i$-types, $w$ may or may not have an $\imo$-neighbor. For
$i$-types B and C, if $w$ admits an $\imt$-neighbor but not an
$\imo$-neighbor, then by axiom $3$, $E_{\imt}(w)$ admits an
$\imo$-neighbor.

Figure~\ref{fig:type} shows the $E_{\imt}$, $E_{\imo}$ and $E_{i}$
edges neighboring a vertex with a given $i$-type. If $E_i$ edges do
not exist in the graph, then the $i$-edges in Figure~\ref{fig:type}
indicate which vertices admit an $i$-neighbor. The top rows for
$i$-types W and B are the possibilities in a dual equivalence graph,
while the lower rows give the additional possibilities in the more
general setting when axiom $4$ does not hold.

\begin{remark}
  By axioms $1,2$ and $5$, edges $E_j$ with $j<\imf$ or $j \geq
  i\!+\!2$ do not change the $i$-type of a vertex, i.e. the $i$-type
  of $w$ is the $i$-type of $E_j(w)$. In contrast, $E_{\imh}$ often
  changes the $i$-type of a vertex as can $E_{\ipo}$, so these cases
  require some care. When axiom $4$ holds, $w$ and $E_i(w)$ have the
  same $i$-type, so much of the following sections is devoted to
  vertices for which this is not the case.
\end{remark}

By dual equivalence graph axioms $1$ and $2$, a connected component of
$(V,\sigma,E_{\imo} \cup E_{i})$ occurs in Figure~\ref{fig:lambda4} if
and only if it does not contain a vertex $w$ with $i$-type W for which
$E_{\imo}(w) \neq E_i(w)$. Define $W_i(\G)$ to be the set of all such
vertices bearing witness to the failure of Figure~\ref{fig:lambda4},
i.e.
\begin{equation}
  W_i(\G) \ = \ \left\{ w \in V \ | \ w \ \mbox{has $i$-type W but} \
    E_{\imo}(w) \neq E_{i}(w) \right\}.
\label{eqn:W}
\end{equation}

In the sections to follow, vertices of $i$-type W, especially those
vertices in $W_i(\G)$, play a crucial role in defining the
transformations that turn $\G_{c,D}^{(k)}$ into a dual equivalence
graph. Note that $w$ has $i$-type W if and only if $E_{\imo}(w)$ has
$i$-type W, so in some sense $i$-type W is a property of $\imo$-edges
rather than of vertices. It is helpful to have a dual property for
$i$-edges.

\begin{definition}
  Let $\G$ be a signed, colored graph of type $(n,N)$ satisfying
  axioms $1, 2, 3$ and $5$. For $i < n$, a vertex $w$ has a
  \emph{flat $i$-edge} if $\sigma(w)_{\imt} = \sigma(E_i(w))_{\imt}$.
  \label{defn:flat}
\end{definition}

In a signed, colored graph satisfying axioms $1, 2, 3$ and $5$, flat
$i$-edges relate to $i$-type W in the following way: a vertex $w$ has
a flat $i$-edge if and only if at most one of $w$ and $E_i(w)$ has
$i$-type W. By axiom $1$, a vertex $w$ has a flat $i$-edge if and only
if exactly one of $w$ and $E_i(w)$ has an $\imo$-neighbor. In a dual
equivalence graph, for a vertex $w$ that does not have $\imo$-type W,
$w$ has $i$-type C if and only if both $w \neq E_{\imt}(w)$ and both
admit flat $i$-edges. In a dual equivalence graph, a vertex $w$ has
$i$-type C if and only if $E_{\imt}(w)$ has $i$-type C.

For a graph $\G$ satisfying axioms $1,2,3$ and $5$, if all connected
components of $E_{\imo} \cup E_i$ all appear in
Figure~\ref{fig:lambda4}, axiom $4$ fails precisely when two vertices
with different $i$-types are paired by an $i$-edge or a pairing of
$i$-type C forms a cycle with more than four edges. Among $i$-types A,
B and C, $i$-type A vertices are distinguished by their lack of
$\imt$-neighbors, so, by axiom $2$, an $i$-type mismatch can only
occur between $i$-type B and $i$-type C. Therefore, when
Figure~\ref{fig:lambda4} holds, we focus on vertices of $i$-type C. In
the more general case when Figure~\ref{fig:lambda4} does not always
hold, instead of $i$-type C, we focus on vertices $x$ for which both
$x$ and $E_{\imt}(x)$ have flat $i$-edges. Define $C_i(\G)$ by
\begin{equation}
  C_i(\G) \ = \ \left\{ x \in V \ | \ x \mbox{ and } E_{\imt}(x) 
    \mbox{ have flat $i$-edges but } E_{\imt}E_i(x) \neq
    E_iE_{\imt}(x) \right\}. 
\label{eqn:X}
\end{equation}
Together, $W_i(\G)$ and $C_i(\G)$ measure how far $\G$ is from
satisfying dual equivalence axiom $4$, specifically, how many
connected components of $(V,\sigma,E_{\imt} \cup E_{\imo} \cup E_{i})$
do not appear in Figure~\ref{fig:lambda5}.

\begin{proposition}
  Let $\G$ be a locally Schur positive graph of type $(n,n)$
  satisfying dual equivalence axioms $1,2,3$ and $5$. Then $\G$
  satisfies dual equivalence axiom $4$ if and only if both $W_i(\G)$
  and $C_i(\G)$ are empty for all $1 < i < n$.
  \label{prop:empty}
\end{proposition}

\begin{proof}
  Suppose first that axiom $4$ holds for $\G$. If $w$ has $i$-type W,
  then from Figure~\ref{fig:lambda4} we must have $E_{\imo}(w) =
  E_i(w)$, and so $W_i(\G)$ is empty. If $x$ and $E_{\imt}(x)$ both
  have flat $i$-edges, then from the previous discussion both have
  $i$-type C. Therefore, from Figure~\ref{fig:lambda5}, we must have
  $E_{\imt}E_i(x) = E_iE_{\imt}(x)$, and so $C_i(\G)$ is empty.

  Now suppose that both $W_i(\G)$ and $C_i(\G)$ are empty for all
  $i$. Since $W_i(\G)$ is empty, a vertex $u$ has $i$-type W if and
  only if $E_{\imo}(u) = E_i(u)$. We claim that if $u$ has $i$-type A,
  B or C, then $u$ has a flat $i$-edge. This follows from axiom $3$
  when $u$ does not admit an $\imo$-neighbor, so assume $u$ admits an
  $\imo$-neighbor. Since, by assumption, $u$ does not have $i$-type W,
  $E_{\imo}(u)$ does not admit an $i$-neighbor. If the $i$-edge at $u$
  is not flat, then $E_i(u)$ admits an $\imo$-neighbor. If
  $E_{\imo}E_i(u)$ does not admit an $i$-neighbor, then the degree $4$
  generating function of the connected component of $E_{\imo} \cup
  E_i$ containing $u$ is $Q_{-++} + Q_{+-+} + Q_{-+-} + Q_{+--}$,
  which is not Schur positive. If $E_{\imo}E_i(u)$ admits an
  $i$-neighbor, then it has $i$-type W, and since $E_{\imo}E_i(u) \neq
  u$, this means $E_{\imo}E_i(u) \in W_i(\G)$. Since both scenarios
  lead to a contradiction, $E_i(u)$ does not admit an $\imo$-neighbor,
  and so the $i$-edge at $u$ is indeed flat. Therefore all connected
  components of $E_{\imo} \cup E_i$ appear in
  Figure~\ref{fig:lambda4}.

  As mentioned previously, among types A, B and C, a vertex has
  $i$-type A if and only if it does not admit an $\imt$-neighbor. If
  $u$ has a flat $i$-edge, then $u$ admits an $\imt$-neighbor if and
  only if $E_i(u)$ admits an $\imt$-neighbor. Therefore, if $u$ has
  $i$-type A so does $E_i(u)$, and the component of $E_{\imt} \cup
  E_{\imo} \cup E_i$ containing $u$ appears in
  Figure~\ref{fig:lambda5}. A vertex $u$ has $i$-type C if and only if
  $E_{\imt}(u)$ has $i$-type C, and, since $W_i(\G)$ is empty, both
  have flat $i$-edges. Since $C_i(\G)$ is empty, we must have
  $E_{\imt} E_i(u) = E_i E_{\imt}(u)$, and so the component of
  $E_{\imt} \cup E_{\imo} \cup E_i$ containing $u$ appears in
  Figure~\ref{fig:lambda5}. Finally, note that $i$-type B vertices may
  only appear as in the top row of Figure~\ref{fig:type} since
  connected components of $E_{\imo} \cup E_i$ all appear in
  Figure~\ref{fig:lambda4}. If $u$ has $i$-type B, then $E_i(u)$
  cannot have $i$-type W, A or C by the previous analysis. Therefore
  $E_i(u)$ must have $i$-type B, and once again the component of
  $E_{\imt} \cup E_{\imo} \cup E_i$ containing $u$ appears in
  Figure~\ref{fig:lambda5}.
\end{proof}

We conclude this section with the following result relating $i$-types
W and C on $i$-packages.

\begin{lemma}
  Let $\G$ be a dual equivalence graph of type $(n,N)$ with $i \leq
  n$. If a vertex $w$ of $\G$ has $i$-type W, then no vertex on the
  $i$-package of $w$ has $i$-type C.
  \label{lem:type-C}
\end{lemma}

\begin{proof}
  By Theorem~\ref{thm:isomorphic} and Lemma~\ref{lem:extend-signs}, we
  may assume $\G = \G_{\mu,A}$ for some partition $\mu$ of $i$ and
  some augmenting tableau $A$ containing entries $\ipo,\ldots,N$. Let
  $\lambda$ be the uniquely determined shape of $\mu$ together with
  the cell in $A$ containing $\ipo$. A tableau $T \in \G_{\lambda}$
  has $i$-type W if and only if both $\imt$ and $\ipo$ lie between
  $\imo$ and $i$ in the reading word of $T$. From the proof of
  Theorem~\ref{thm:cover}, a tableau $T \in \G_{\lambda}$ has $i$-type
  C if and only if $\imo$ lies between $i$ and $\ipo$ in the reading
  word of $T$. For $h \leq \imh$, an $E_h$ edge does not change the
  positions of entries greater than $\imt$, and for $h \geq \iph$, an
  $E_h$ edge does not change the positions of entries less than
  $\ipt$. In particular, the positions of $\imo,i,\ipo$ are constant
  on $i$-packages. The result now follows.
\end{proof}

\subsection{Involutions to resolve axiom $4$}
\label{sec:D-transformations}

By Proposition~\ref{prop:empty}, axiom $4$ holds if and only if $W_i$
and $C_i$ are empty. We construct two maps, $\varphi_i^w$ and
$\psi_i^x$, with the goal of reducing the cardinality of these sets.
We begin with the construction of the map $\varphi_i^w$, depicted in
Figure~\ref{fig:phi}, which takes as input an element $w \in
W_i(\G)$. We aim to use $\varphi_i^w$ to redefine $i$-edges so that
$E_{\imo}(w) = E_i(w)$.

\begin{figure}[ht]
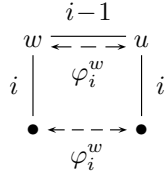

  \begin{displaymath}
    \begin{array}{\cs{8}c}
      \rnode{C}{w} & \rnode{D}{u}  \\[5ex] 
      \rnode{B}{\B} & \rnode{E}{\B}
    \end{array}
    \psset{nodesep=3pt,linewidth=.1ex}
    \ncline[offset=2pt] {C}{D} \naput{\imo}
    \ncline {B}{C} \naput{i}
    \ncline {D}{E} \naput{i}
    \ncline[offset=-2pt,linestyle=dashed] {<->}{C}{D} \nbput{\varphi_i^w}
    \ncline[linestyle=dashed] {<->}{B}{E} \nbput{\varphi_i^w}
  \end{displaymath}
  \caption{\label{fig:phi} An illustration of the involution
    $\varphi_i^w$, with $w \in W_i(\G)$ and $u = E_{\imo}(w)$.}
\end{figure}

How $\varphi_i^w$ acts on the connected component of $E_{\imo} \cup
E_i$ is straightforward. The difficulty lies in extending the map to
the $i$-package of $w$ as is necessary to maintain axiom $5$. The
following result characterizes when such an extension is possible.

\begin{lemma}
  Let $\G$ be a signed, colored graph of type $(n,N)$ satisfying dual
  equivalence axioms $1,2,3$ and $5$, and suppose that the
  $(\imt,N)$-restriction of $\G$ is a dual equivalence graph. Let $w$
  be a vertex of $i$-type W such that every vertex on the
  $\imo$-package of $w$ has a flat $\imo$-edge. Then there exists an
  isomorphism between the $i$-packages of $w$ and $E_{\imo}(w)$.
\label{lem:phi-compatible}
\end{lemma}

\begin{proof}
  If $E_{\imo}(w) = E_i(w)$, then the result follows immediately from
  axioms $2$ and $5$. Suppose then that $w \in W_i(\G)$, and set $u =
  E_{\imo}(w)$. Recall that $E_{\imo}$ may be regarded as an
  involution on vertices that admit an $\imo$-neighbor. Regarded as
  such, by axioms $1,2$ and $5$, $E_{\imo}$ gives an involution
  between $\imo$-packages of $w$ and $u$. Therefore we need only show
  that this isomorphism restricted to $E_{2} \cup \cdots \cup
  E_{\imf}$ extends to an isomorphism for $E_{2} \cup \cdots \cup
  E_{\imh}$, since the isomorphism for $E_{\iph} \cup \cdots \cup
  E_{\nmo}$ is already established.

  By the assumption that all vertices $v$ on the $\imo$-package of $w$
  have flat $\imo$-edges, we know $\sigma(v)_{\imh} =
  \sigma(E_{\imo}(v))_{\imh}$. Therefore $E_{\imo}$ gives an
  involution between the $(\imh,\imt)$-restrictions of the
  $i$-packages of $w$ and $u$. We extend this isomorphism as
  illustrated in Figure~\ref{fig:compatible}.

  \begin{figure}[ht]
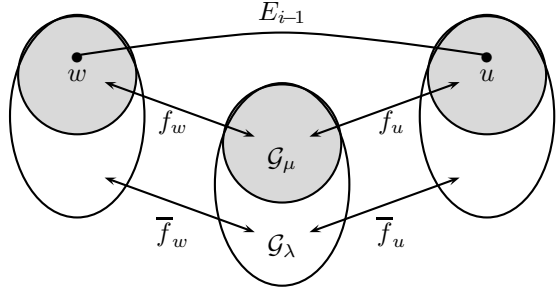

    \begin{center}
      \psset{xunit=3ex}
      \psset{yunit=3ex}
      \pspicture(0,-3)(14,5)
      \pscircle[fillstyle=solid,fillcolor=lightgray](1,3){.8}
      \psellipse(1,1.8)(2,3)
      \pscircle[fillstyle=solid,fillcolor=lightgray](7,1){.8}
      \psellipse(7,-.2)(2,3)
      \pscircle[fillstyle=solid,fillcolor=lightgray](13,3){.8}
      \psellipse(13,1.8)(2,3)
      \rput(1,3.5){$\bullet$}
      \rput(1,3){$w$}
      \rput(13,3.5){$\bullet$}
      \rput(13,3){$u$}
      \rput(7,.5){$\G_{\mu}$}
      \rput(7,-2){$\G_{\lambda}$}
      \pscurve(1.1,3.6)(4,4)(7,4.25)(10,4)(12.9,3.6)
      \rput(7,4.8){$E_{\imo}$}
      \psline{<->}(1.8,2.8)(6.2,1.2)
      \rput(3.8,1.6){$f_{w}$}
      \psline{<->}(7.8,1.2)(12.2,2.8)
      \rput(10.2,1.6){$f_{u}$}
      \psline{<->}(1.8,0)(6.2,-1.6)
      \rput(3.8,-1.6){$\overline{f}_{w}$}
      \psline{<->}(7.8,-1.6)(12.2,0)
      \rput(10.2,-1.6){$\overline{f}_{u}$}
      \endpspicture
      \caption{\label{fig:compatible}Extending the isomorphism of
        $\imo$-packages to an isomorphism of $i$-packages}
    \end{center}
  \end{figure}

  By Lemma~\ref{lem:extend-signs} and the hypothesis that the
  $(\imt,N)$-restriction of $\G$ is a dual equivalence graph, there
  exist isomorphisms, say $f_w$ and $f_u$, from the
  $(\imh,\imt)$-restrictions of the $i$-packages of $w$ and $u$ to the
  augmented dual equivalence graph $\G_{\mu,A}$ for a unique partition
  $\mu$ of $\imh$ and a unique single cell augmenting tableau $A$.  By
  Theorem~\ref{thm:cover}, the two isomorphism extend consistently
  across $E_{\imh}$ edges to give isomorphisms $\overline{f}_w$ and
  $\overline{f}_u$ from the $(\imt,\imt)$-restrictions of the
  connected components containing $w$ and $u$, respectively, to
  $\G_{\lambda}$ where $\lambda$ is the shape of $\mu$ augmented by
  $A$. In particular, the composition of these isomorphisms gives an
  isomorphism between the $(\imt,\imt)$-restrictions of the
  $i$-packages of $w$ and $u$.
\end{proof}

The hypotheses of Lemma~\ref{lem:phi-compatible} cannot be relaxed, so
these vertices are of particular importance. Therefore we define the
set $W_i^0(\G) \subseteq W_i(\G)$ by
\begin{equation}
  W_i^0(\G) = \{ w \in W_i(\G) \ | \ \mbox{every vertex on the
    $\imo$-package of $w$ has a flat $\imo$-edge} \}.
  \label{eqn:W-0}
\end{equation}

\begin{remark}
  If every connected component of $E_{\imt} \cup E_{\imo}$ appears in
  Figure~\ref{fig:lambda4}, then the $\imo$-edge at $v$ is not flat if
  and only if $E_{\imo}(v) = E_{\imt}(v)$. In this case, by axiom $2$,
  $\sigma(v)_{i} = \sigma(E_{\imt}(v))_{i} = \sigma(E_{\imo}(v))_{i}$,
  so $v$ does not have $i$-type W. Since $\sigma_i$ is constant on
  $i$-packages, no vertex on the $i$-package of $v$ has $i$-type
  W. Therefore, $W_i^0(\G) = W_i(\G)$ whenever the components of
  $E_{\imt} \cup E_{\imo}$ on the $i$-package of $w$ all appear in
  Figure~\ref{fig:lambda4}.
  \label{rmk:phi-lambda4}
\end{remark}

We use the isomorphism of Lemma~\ref{lem:phi-compatible} to define an
involution $\varphi^{w}_i$ on all vertices admitting an $i$-neighbor
as follows.

\begin{definition}
  For $w \in W_i^0(\G)$, let $u = E_{\imo}(w)$, and let $\phi$ the
  isomorphism of Lemma~\ref{lem:phi-compatible}. Define the involution
  $\varphi^{w}_i$ on all vertices admitting an $i$-neighbor by
  \begin{equation}
    \varphi^w_i(v) = \left\{ \begin{array}{rl}
        \phi(v) & \mbox{if $v$ lies on the $i$-package of $w$ or
          $u$,} \\ [1ex]
        E_{i} \phi E_{i}(v) & \mbox{if $E_{i}(v)$ lies on the
          $i$-package of $w$ or $u$,} \\ [1ex]
        E_i(v) & \mbox{otherwise.}
      \end{array} \right.
    \label{eqn:phi}
  \end{equation}
  Define $E_i'$ to be the set of pairs $\{v,\varphi_i^w(v)\}$ for each
  $v$ admitting an $i$-neighbor. Define a signed, colored graph
  $\varphi^w_i(\G)$ of type $(n,N)$ by
  \begin{equation}
    \varphi^w_i(\G) = (V, \sigma, E_2 \cup\cdots\cup E_{\imo} \cup E_i'
    \cup E_{\ipo} \cup\cdots\cup E_{\nmo}).
    \label{eqn:G'}
  \end{equation}
  \label{defn:phi}
\end{definition}

Since the isomorphisms from Lemma~\ref{lem:phi-compatible} for $w$ and
$u = E_{\imo}(w)$ are inverse to one another, we abuse notation in
Definition~\ref{defn:phi} by letting $\phi$ denote either. Note as
well that $\varphi_i^w = \varphi_i^u$.

The goal with $\varphi_i^w$ is to reduce the cardinality of
$W_i(\G)$. The following result shows that this happens provided the
$(i,N)$-restriction of $\G$ satisfies dual equivalence graph axiom
$4$.

\begin{proposition}
  Let $\G$ be a locally Schur positive graph of type $(n,N)$
  satisfying dual equivalence axioms $1,2,3$ and $5$, and suppose that
  the $(\imt,N)$-restriction of $\G$ is a dual equivalence graph and
  that the $(i,N)$-restriction of $\G$ satisfies dual equivalence
  axiom $4$. Then $W_i^0(\G) = W_i(\G)$, and for $w \in W_i(\G)$,
  $W_i(\varphi_i^w(\G))$ is a proper subset of $W_i(\G)$.
  \label{prop:phi-terminate}
\end{proposition}

\begin{proof}
  By Remark~\ref{rmk:phi-lambda4}, if connected components of
  $E_{\imt} \cup E_{\imo}$ all appear in Figure~\ref{fig:lambda4},
  then all vertices on the $\imo$-package of a vertex with $i$-type W
  have flat $\imo$-edges. Since the $(i,N)$-restriction of $\G$
  satisfies dual equivalence graph axiom $4$, this is the case, so
  $W_i^0(\G) = W_i(\G)$. As the $i$-type of a vertex is determined by
  the connected component of $E_{\imt} \cup E_{\imo}$ containing it,
  the $i$-type of a vertex is the same in $\G$ and
  $\varphi_i^w(\G)$. Therefore, to show that $v \not\in W_i(\G)$
  implies $v \not\in W_i(\varphi_i^w(\G))$, we must show that for $v$
  with $i$-type W such that $E_i(v) = E_{\imo}(v)$, we have
  $\varphi_i^w(v) = E_{\imo}(v)$ as well.

  It suffices to consider $v$ on the $i$-packages of $w$ and
  $E_{i}(w)$. We claim that for any $v$ on the $i$-package of $w$,
  $E_{\imo}(v) \neq E_i(v)$. By axiom $5$, both $E_{\imo}$ and $E_i$
  commute with $E_h$ for $h \leq \imf$ and $h \geq
  \iph$. Therefore, if the claim holds for some vertex $v$, then it
  holds for any vertex connected to $v$ by edges in $E_2 \cup \cdots
  \cup E_{\imf} \cup E_{\iph} \cup \cdots \cup E_{\nmo}$. By axiom $6$,
  it suffices to show the claim for $v = E_{\imh}(w)$ since any vertex
  on the $i$-package of $w$ can be reached by crossing at most one
  $E_{\imh}$ edge. 

  Let $v = E_{\imh}(w)$ and suppose that $E_{\imo}(v) = E_i(v)$. We
  claim that $v$ and $w$ must have $\imo$-type C. Since both admit an
  $\imh$-neighbor, they cannot have $\imo$-type A. Since $E_{\imh}
  \cup E_{\imt} \cup E_{\imo}$ satisfies axiom $4$, any vertex with
  $\imo$-type W has a double edge between $E_{\imt}$ and
  $E_{\imo}$. By axiom $2$, $E_{\imt}$ edges preserve $\sigma_i$, so
  such a vertex cannot have $i$-type W. Therefore neither $w$ nor $v$
  has $\imo$-type W. Since both $w$ and $v$ admit an $\imo$-neighbor,
  by axiom $3$ exactly one admits an $\imt$-neighbor, and so neither
  can have $\imo$-type B. All that remains must be the case, so both
  have $\imo$-type C as claimed. By axiom $4$, this means $E_{\imo}(v)
  = E_{\imh}E_{\imo}(w)$. Using this together with axiom $5$, we have
  $E_{\imh}E_i(w) = E_iE_{\imh}(w) = E_i(v) = E_{\imo}(v) =
  E_{\imh}E_{\imo}(w)$. By axiom $1$, this implies $E_{i}(w) =
  E_{\imo}(w)$, contradicting the assumption that $w \in
  W_i(\G)$. Thus for any $v$ on the $i$-package of $w$, $E_{\imo}(v)
  \neq E_i(v)$. By axiom $1$, the same now holds for vertices on the
  $i$-package of $E_i(w)$. Therefore $W_i(\varphi_i^w(\G))$ is indeed
  a proper subset of $W_i(\G)$.
\end{proof}

The second transformation, $\psi_i^x$, depicted in
Figure~\ref{fig:psi}, takes as input an element $x \in
C_i(\G)$. Similar to $\varphi_i^w$, the aim with $\psi_i^x$ is to
redefine $i$-edges so that $E_{\imt} E_i(x) = E_i E_{\imt}(x)$. 

\begin{figure}[ht]
  \begin{displaymath}
    \begin{array}{\cs{8}\cs{8}c}
      & \rnode{O1}{\B} & \\[4ex] 
      \rnode{a1}{u}  & \rnode{b1}{x}  & \rnode{c1}{\B} \\[5ex] 
      \rnode{e1}{\B} & \rnode{f1}{\B} & \\[5ex] 
      \rnode{h1}{\B} & \rnode{i1}{\B} & 
    \end{array}
    \hspace{4em}
    \begin{array}{\cs{8}\cs{8}\cs{8}c}
      & & \rnode{O}{\B} & \\[4ex] 
      \rnode{T}{\B} & \rnode{a2}{\B} & \rnode{b2}{x} & \rnode{c2}{\B} \\[5ex] 
      \rnode{B}{u} & \rnode{e2}{\B} & \rnode{f2}{\B} & \\[5ex] 
      \rnode{g2}{\B} & \rnode{h2}{\B} & \rnode{i2}{\B} & 
    \end{array}
    \psset{nodesep=3pt,linewidth=.1ex}
    \ncline {a1}{O1} \naput{\imo}
    \ncline {a1}{b1} \naput{i}
    \ncline {a1}{e1} \nbput{\imt}
    \ncline {b1}{f1} \naput{\imt}
    \ncline {f1}{c1} \nbput{\imo}
    \ncline {e1}{h1} \nbput{i}
    \ncline {f1}{i1} \naput{i}
    \ncline[linestyle=dashed] {<->}{e1}{f1} \nbput{\psi_i^x}
    \ncline[linestyle=dashed] {<->}{h1}{i1} \nbput{\psi_i^x}
    \ncline {O}{a2} \nbput{\imo}
    \ncline {T}{a2} \naput{\imt}
    \ncline {T}{B} \nbput{\imo}
    \ncline {B}{e2} \nbput{\imt}
    \ncline {a2}{b2} \naput{i}
    \ncline {b2}{f2} \naput{\imt}
    \ncline {f2}{c2} \nbput{\imo}
    \ncline {B}{g2} \nbput{i}
    \ncline {e2}{h2} \nbput{i}
    \ncline {f2}{i2} \naput{i}
    \ncline[linestyle=dashed] {<->}{e2}{f2} \nbput{\psi_i^x}
    \ncline[linestyle=dashed] {<->}{h2}{i2} \nbput{\psi_i^x}
  \end{displaymath}
  \caption{\label{fig:psi} An illustration of $\psi_i^x$ where $x \in
    C_i(\G)$ and $u = (E_{\imo}E_{\imt})^m E_{i}(x)$ does not have
    $\imo$-type W.}
\end{figure}

The definition of $\psi_i^x$ on the connected component of $E_{\imt}
\cup E_i$ containing $x$ is straightforward provided neither $x$ nor
$E_i(x)$ has $\imo$-type W. In general, $\psi_i^x$ will be defined
whenever some vertex on the connected component of $E_{\imt} \cup
E_{\imo}$ containing $E_i(x)$ does not have $\imo$-type W.  As before,
the first step in defining the transformation is to show that it can
be extended to $i$-packages.

\begin{lemma}
  Let $\G$ be a signed, colored graph of type $(n,N)$ satisfying dual
  equivalence axioms $1,2,3$ and $5$, and suppose that the
  $(\imt,N)$-restriction of $\G$ is a dual equivalence graph.  Let $x$
  not admit an $\imo$-neighbor but have a flat $i$-edge such that
  neither $x$ nor $\left( E_{\imo} E_{\imt} \right)^{m} E_i(x)$ has
  $\imo$-type W for some $m \geq 0$, and suppose all vertices between
  $x$ and $\left( E_{\imo} E_{\imt} \right)^{m} E_i(x)$ have flat
  $\imt$-edges throughout their $\imt$-packages. Then the $i$-package
  of $E_{\imt}(x)$ is isomorphic to the $i$-package of $E_{\imt}\left(
  E_{\imo} E_{\imt} \right)^{m} E_i(x)$.
\label{lem:psi-compatible}
\end{lemma}

\begin{proof}
  Let $u = \left( E_{\imo} E_{\imt} \right)^{m} E_i(x)$. By axioms
  $1,2$ and $5$, $E_{\imt}, E_{\imo}, E_i$ all commute with $E_h$ for
  $h \geq \iph$, so the restriction of the $i$-package of any vertex
  on the connected component of $E_{\imt} \cup E_{\imo} \cup E_i$
  containing $x$ to $E_{\iph} \cup \cdots \cup E_{\nmo}$ are
  isomorphic.  Therefore we focus our attention on extending the
  restriction to $E_{2} \cup \cdots \cup E_{\imh}$.

  Since all $E_{\imt}$ edges between $E_i(x)$ and $u$ are flat along
  their $\imt$-packages, Lemma~\ref{lem:phi-compatible} applies to
  each. Therefore, since $E_{\imo}$ always gives an isomorphism of
  $\imo$-packages, the $\imo$-package of $E_i(x)$ is isomorphic to the
  $\imo$-package of $u$. Further, each $E_{\imt}$ or $E_{\imo}$ edge
  changes $\sigma_j$ for $j = \imh,\imt,\imo$, and so $\sigma(u)_{j} =
  \sigma(E_i(x))_j$ for $j \leq \imo$ and $j \geq \ipo$. By axiom $2$,
  the $E_{\imt}$ edges preserve $\sigma_i$. Therefore, by axiom $1$,
  $E_{\imt} E_i(x)$ does not admit an $i$-neighbor, so neither
  $E_{\imt} E_i(x)$ nor $E_{\imo} E_{\imt} E_i(x)$ has $i$-type
  W. Continuing the argument along to $u$, no vertex of the form
  $E_{\imt} \left( E_{\imo} E_{\imt} \right)^{k} E_i(x)$ admits an
  $i$-neighbor for $0 \leq k < m$, and so none of the vertices after
  $E_i(x)$ can have $i$-type W. In particular, each $E_{\imo}$ edge
  from $E_i(x)$ to $u$ preserves $\sigma_i$ as well, and so $\sigma(u)
  = \sigma(E_i(x))$.
  
  Therefore we have an isomorphism between the
  $(\imh,\imt)$-restrictions of the $i$-packages of $E_i(x)$ and
  $u$. By the same argument used in the proof of
  Lemma~\ref{lem:phi-compatible}, we invoke
  Lemma~\ref{lem:extend-signs} and the hypothesis that the
  $(\imt,N)$-restriction of $\G$ is a dual equivalence graph to extend
  this to an isomorphism between the $i$-packages of $E_i(x)$ and $u$
  and $\sigma(u) = \sigma(E_i(x))$. Regarding $E_i$ as an isomorphism
  of $i$-packages, it follows that $x$ and $u$ also have isomorphic
  $i$-packages. Thus, by Theorem~\ref{thm:isomorphic} and
  Lemma~\ref{lem:extend-signs}, the connected components of the
  $(\imt,\imo)$-restriction of $\G$ containing $x$ and $u$ are both
  isomorphic to $\G_{\mu,A}$ for the same partition $\mu$ of $\imt$
  and the same augmenting tableau $A$ consisting of a single cell
  containing $\imo$. Denote these isomorphisms by $f_x$ and $f_u$,
  respectively, and let $\lambda$ be the shape of $\mu$ augmented by
  $A$. 

  Since the $(\imo,\imo)$-restriction of $\G$ satisfies the hypotheses
  of Theorem~\ref{thm:cover}, the isomorphisms $f_x$ and $f_u$ extend
  to morphisms $\overline{f}_x$ and $\overline{f}_u$ from the
  connected components of the $(\imo,\imo)$-restriction of $\G$
  containing $x$ and $u$ to $\G_{\lambda}$. The picture is very
  similar to Figure~\ref{fig:compatible}, though now the top map is
  $E_i$ and the extended maps are surjective though not necessarily
  injective. Despite the lack of injectivity, the uniqueness of
  $\lambda$ and the extended maps ensures that the
  $(\imt,\imo)$-restriction of $\G_{\lambda}$ containing $E_{\imt}(x)$
  is isomorphic to the $(\imt,\imo)$-restriction of $\G_{\lambda}$
  containing $E_{\imt}(u)$, thereby establishing the desired
  isomorphism of $i$-packages.
\end{proof}

As with Lemma~\ref{lem:phi-compatible}, the hypotheses of
Lemma~\ref{lem:psi-compatible} cannot be relaxed, so these vertices
are of particular importance. Therefore we define the set $C_i^0(\G)
\subseteq C_i(\G)$ by
\begin{equation}
  C_i^0(\G) = \{ x \in C_i(\G) \ | \ \begin{array}{ll}
    \mbox{neither $x$ nor $\left( E_{\imo} E_{\imt} \right)^{m}
      E_i(x)$ has $\imo$-type W for some $m \geq 0$}\\
    \mbox{all vertices between $x$ and $\left( E_{\imo}
        E_{\imt} \right)^{m} E_i(x)$ have flat $\imt$-edges}
  \end{array} \}.
  \label{eqn:X-0}
\end{equation}

\begin{remark}
  If every connected component of $E_{\imh} \cup E_{\imt}$ appears in
  Figure~\ref{fig:lambda4}, then if the $\imt$-edge at $x$ is not
  flat, by axiom $4$, $E_{\imt}(x) = E_{\imh}(x)$. By axiom $2$, this
  ensures that $E_{\imt}(x)$ does not have $\imo$-type W, so we are in
  the case where $m=0$. By axiom $5$, $E_{\imh}E_i(x) = E_iE_{\imh}(x)
  = E_iE_{\imt}(x)$. For $x \in C_i(\G)$, both $E_i(x)$ and
  $E_{\imh}E_i(x)$ admit $\imt$-neighbors, so by axiom $4$ we have
  $E_{\imt}(E_i(x)) = E_{\imh}(E_i(x))$. Therefore $E_{\imt}E_i(x) =
  E_iE_{\imt}(x)$, contradicting the assumption that $x \in
  C_i(\G)$. Therefore, $C_i^0(\G) = C_i(\G)$ whenever every connected
  component of $E_{\imh} \cup E_{\imt}$ appears in
  Figure~\ref{fig:lambda4}. 
  \label{rmk:psi-lambda4}
\end{remark}

As was the case with $w \in W^0_i(\G)$, given $x \in C^0_i(\G)$, we
use the isomorphism of Lemma~\ref{lem:psi-compatible} to define an
involution $\psi_i^x$ on all vertices admitting an $i$-neighbor.

\begin{definition}
  For $x \in C^0_i(\G)$, let $u = \left( E_{\imo} E_{\imt} \right)^{m}
  E_i(x)$ be the first vertex on the connected component of $E_{\imt}
  \cup E_{\imo}$ containing $E_i(x)$ not having $\imo$-type W. Let
  $\phi$ denote the isomorphism of
  Lemma~\ref{lem:psi-compatible}. Define the involution $\psi_i^x$ on
  all vertices admitting an $i$-neighbor as follows.
  \begin{equation}
    \psi_i^x(v) = \left\{ \begin{array}{rl}
        \phi(v) & \mbox{if $v$ lies on the $i$-package of
          $E_{\imt}(x)$ or $E_{\imt} (u)$,} \\[1ex]  
        E_i \phi E_i(v) & \mbox{if $E_i(v)$ lies on the
          $i$-package of $E_{\imt}(x)$ or $E_{\imt} (u)$,}\\[1ex]
        E_i(v) & \mbox{otherwise.}
      \end{array} \right.
    \label{eqn:psi}
  \end{equation}
  Define $E'_i$ to be the set of pairs $\{v,\psi_i^x(v)\}$ for each
  $v$ admitting an $i$-neighbor. Define a signed, colored graph
  $\psi_i^x(\G)$ of type $(n,N)$ by
  \begin{equation}
    \psi_i^x(\G) = (V, \sigma, E_2 \cup\cdots\cup E_{\imo} \cup E'_i
    \cup E_{\ipo} \cup\cdots\cup E_{\nmo}). 
  \end{equation}
\label{defn:psi}
\end{definition}

We again abuse notation by letting $\phi$ denote both the isomorphism
from the $i$-package of $x$ to the $i$-package of $u$ and its
inverse. Note that $\psi_i^x = \psi_i^u$ when $m = 0$.

The goal with $\psi_i^x$ is to reduce $C_i(\G)$ without increasing
$W_i(\G)$. The following result shows that this happens provided the
$(i,N)$-restriction of $\G$ is a dual equivalence graph, and more
generally as well.

\begin{proposition}
  Let $\G$ be a locally Schur positive graph of type $(n,N)$
  satisfying dual equivalence axioms $1,2,3$ and $5$, and suppose that
  $(\imt,N)$-restriction of $\G$ is a dual equivalence graph and that
  the $(i,N)$-restriction of $\G$ satisfies dual equivalence axiom
  $4$. Then $C_i^0(\G) = C_i(\G)$, and for $x \in C_i(\G)$ such that
  $E_i(x) \in C_i(\G)$, $C_i(\psi_i^x(\G))$ is a proper subset of
  $C_i(\G)$ and $W_i(\psi_i^x(\G)) = W_i(\G)$.
  \label{prop:psi-terminate}
\end{proposition}

\begin{proof}
  By Remark~\ref{rmk:psi-lambda4}, since connected components of
  $E_{\imh} \cup E_{\imt}$ all appear in Figure~\ref{fig:lambda4},
  $C_i^0(\G) = C_i(\G)$. The $i$-type of a vertex is determined by the
  connected component of $E_{\imt} \cup E_{\imo}$ containing it, so
  the $i$-type of a vertex remains unchanged by $\psi_i^w$. By the
  previous discussion, no $E_{\imh}$ edge on the $i$-package of $x$ or
  $E_i(x)$ is part of a double edge with $E_{\imt}$, so whether or not
  the vertex admits an $\imo$-neighbor is preserved. Therefore to show
  that $v \not\in C_i(\G)$ implies $v \not\in C_i(\psi_i^x(\G))$, we
  must show that if $E_{\imt}E_i(v) = E_iE_{\imt}(v)$, then
  $E_{\imt}\psi_i^x(v) = \psi_i^xE_{\imt}(v)$.

  Again, it suffices to consider $v$ on the $i$-packages of $x$ and
  $E_{i}(x)$. We claim that for any $v$ on the $i$-package of $x$, if
  $E_{\imt}E_i(v) = E_iE_{\imt}(v)$, then $E_{\imt}\psi_i^x(v) =
  \psi_i^xE_{\imt}(v)$.  By axiom $5$, $E_{\imt}$ and $E_i$ all
  commute with $E_h$ for $h \leq i\!-\!5$ and $h \geq \iph$. Therefore
  if the claim holds for some vertex $v$, the it holds for every
  vertex on the connected component of $E_2 \cup \cdots \cup
  E_{i\!-\!5} \cup E_{\iph} \cup \cdots \cup E_{\nmo}$ containing
  $v$. Since $\sigma(v)_{\imf} = \sigma(E_{\imt}(v))_{\imf}$ for $v =
  x, E_i(x)$, by axiom $3$ neither $E_{\imh}(x)$ nor $E_{\imh}E_i(x)$
  admits an $\imt$-neighbor, so neither can have $i$-type
  C. Therefore, by axiom $6$, it suffices to show the claim for $v =
  E_{\imf}(x)$.

  \begin{figure}[ht]
    \begin{displaymath}
      \begin{array}{\cs{8}\cs{8}\cs{8}\cs{8}\cs{8}\cs{8}\cs{8}c}
        \rnode{a1}{\B} & \rnode{b1}{\B} & \rnode{c1}{\B} & \rnode{d1}{u}  &
        \rnode{e1}{x}  & \rnode{f1}{\B} & \rnode{g1}{\B} & \rnode{h1}{\B} \\[5ex]
        \rnode{a2}{\B} & \rnode{b2}{\B} & \rnode{c2}{\B} & \rnode{d2}{\B} &
        \rnode{e2}{\B} & \rnode{f2}{\B} & \rnode{g2}{\B} & \rnode{h2}{\B} \\[5ex]
        & \rnode{b3}{\B} & \rnode{c3}{\B} & & & \rnode{f3}{\B} & \rnode{g3}{\B} & 
      \end{array}
      \psset{nodesep=3pt,linewidth=.1ex}
      \everypsbox{\scriptstyle}
      \ncline {a1}{b1} \naput{\imt}
      \ncline {b1}{c1} \naput{i}
      \ncline {c1}{d1} \naput{\imt}
      \ncline {d1}{e1} \naput{i}
      \ncline {e1}{f1} \naput{\imt}
      \ncline {f1}{g1} \naput{i}
      \ncline {g1}{h1} \naput{\imt}
      \ncline {b1}{b2} \nbput{\imf}
      \ncline {c1}{c2} \nbput{\imf}
      \ncline[offset=2pt] {a1}{a2} \nbput{\imf}
      \ncline[offset=2pt] {a2}{a1} \nbput{\imh}
      \ncline[offset=2pt] {d1}{d2} \nbput{\imf}
      \ncline[offset=2pt] {d2}{d1} \nbput{\imh}
      \ncline[offset=2pt] {e1}{e2} \nbput{\imf}
      \ncline[offset=2pt] {e2}{e1} \nbput{\imh}
      \ncline[offset=2pt] {h1}{h2} \nbput{\imf}
      \ncline[offset=2pt] {h2}{h1} \nbput{\imh}
      \ncline {f1}{f2} \naput{\imf}
      \ncline {g1}{g2} \naput{\imf}
      \ncline {b2}{c2} \naput{i}
      \ncline {d2}{e2} \naput{i}
      \ncline {f2}{g2} \naput{i}
      \ncline[offset=2pt] {b2}{b3} \nbput{\imh}
      \ncline[offset=2pt] {b3}{b2} \nbput{\imt}
      \ncline[offset=2pt] {c2}{c3} \nbput{\imh}
      \ncline[offset=2pt] {c3}{c2} \nbput{\imt}
      \ncline {b3}{c3} \naput{i}
      \ncline {f3}{g3} \naput{i}
      \ncline[offset=2pt] {f2}{f3} \nbput{\imh}
      \ncline[offset=2pt] {f3}{f2} \nbput{\imt}
      \ncline[offset=2pt] {g2}{g3} \nbput{\imh}
      \ncline[offset=2pt] {g3}{g2} \nbput{\imt}
    \end{displaymath}
    \caption{\label{fig:psi-X} Components of $E_{\imf} \cup
      E_{\imh} \cup E_{\imt}$ when $x$ has $\imt$-type B.}
  \end{figure}
  
  Consider the $\imt$-type of $x$ and $E_i(x)$, which must be the same
  since, by axiom $5$, $E_i$ commutes with both $E_{\imf}$ and
  $E_{\imh}$. From before, both $x$ and $E_i(x)$ have flat
  $\imt$-edge, and so, by axiom $4$, they cannot have $\imt$-type
  W. Since by assumption both admit an $\imf$-neighbor, they cannot
  have $\imt$-type A. If they have $\imt$-type C, then the top row of
  Figure~\ref{fig:psi-X} commutes with $E_{\imf}$, so $E_{\imf}(x) \in
  C_i(\G)$ and $E_{\imf}(x) \not\in C_i(\psi_i^x(\G))$. If they have
  $\imt$-type B, then, by axioms $4$ and $5$, the situation is as
  depicted in Figure~\ref{fig:psi-X} since none of the endpoints of
  the $i$-edges has $i$-type W by Lemma~\ref{lem:type-C}. From the
  figure, it is clear that applying $\psi_i^x$ adds no vertices to
  $C_i(\G)$, and so $C_i(\psi_i^x(\G))$ is indeed a proper subset of
  $C_i(\G)$.  Moreover, by Lemma~\ref{lem:type-C}, none of the
  vertices involved has $i$-type W, and so $W_i(\psi_i^x(\G)) =
  W_i(\G)$.
\end{proof}

For $\G$ a locally Schur positive graph of type $(n,N)$ satisfying
dual equivalence axioms $1, 2, 3$ and $5$ such that the
$(\imt,N)$-restriction of $\G$ is a dual equivalence graph, both
$\varphi^w_i(\G)$ and $\psi_i^x(\G)$ also satisfy axioms $1, 2$ and
$5$. This follows immediately from Lemmas~\ref{lem:phi-compatible} and
\ref{lem:psi-compatible} and the definitions of the maps on
$i$-packages. It turns out that $\varphi_i^w$ and $\psi_i^x$ also
preserve axiom $3$, but this requires considerably more work to prove
in general. However, when restricting to edges $E_i$ and lower, not
only does axiom $3$ hold, but $\LSP_4$ does as well.

\begin{lemma}
  Let $\G$ be a locally Schur positive graph of type $(\ipo,\ipo)$
  satisfying dual equivalence axioms $1,2,3$ and $5$, and suppose that
  the $(\imt,\ipo)$-restriction of $\G$ is a dual equivalence graph
  and that the $(i,\ipo)$-restriction of $\G$ satisfies dual
  equivalence axiom $4$. Then $W_i^0(\G) = W_i(\G)$ and $C_i^0(\G) =
  C_i(\G)$, and both $\varphi_i^w(\G)$ and $\psi_i^x(\G)$ are $\LSP_4$
  for any $w \in W_i(\G)$ and any $x \in C_i(\G)$.
  \label{lem:LSP4}
\end{lemma}

\begin{proof}
  Since the $(i,\ipo)$-restriction of $\G$ satisfies dual equivalence
  axiom $4$, by Remark~\ref{rmk:phi-lambda4}, $\varphi_i^w$ may be
  applied for any $w \in W_i(\G)$ and by Remark~\ref{rmk:psi-lambda4},
  $\psi_i^x$ may be applied for any $x \in C_i(\G)$ and, in this case,
  $E_i(x) \in C_i(\G)$ with $\psi_i^{E_i(x)} = \psi_i^x$. 

  First consider $\varphi_i^w(\G)$. The component containing $w$ and
  $E_{\imo}(w)$ has degree $4$ generating function $s_{(2,2)}$, which
  is Schur positive, and so the positivity for $E_i(w)$ and $E_i
  E_{\imo}(w)$ follows since the component was Schur positive in
  $\G$. By axiom $5$, if the component containing $v$ is $\LSP_4$,
  then so are the components on any vertex of the connected component
  of $E_2 \cup \cdots \cup E_{\imf}$ containing $v$. By axiom $6$, it
  suffices to consider the positivity across a single $E_{\imh}$ edge
  from $E_i(w), w, E_{\imo}(w)$ and $E_i E_{\imo}(w)$. Since all four
  vertices have isomorphic $i$-packages by
  Lemma~\ref{lem:psi-compatible}, if one of the four admits an
  $\imh$-neighbor, then they all do. Assuming this is the case,
  consider the $\imo$-type of $w$. By the symmetry between $w$ and
  $E_{\imo}(w)$ and the fact that the $\imo$-edge between them is
  flat, we may assume $w$ admits an $\imt$-neighbor and $E_i(w)$ does
  not, as depicted in Figure~\ref{fig:phi-degree4}. The $\imo$-type of
  $w$ cannot be W, since $w$ has a flat $\imo$-edge, nor can it be A,
  since $w$ admits an $\imh$-neighbor. If $w$ has $\imo$-type C, then
  by axiom $4$, $E_{\imo} E_{\imh}(w) = E_{\imh} E_{\imo}(w)$, and so
  $E_{\imh}(w) \in W_i(\G)$ and $\varphi_i^{E_{\imh}(w)} =
  \varphi_i^w$, in which case the previous argument shows that
  $\LSP_4$ is maintained. If $w$ has $\imo$-type B, then $E_{\imh}(w)
  = E_{\imt}(w)$, and so $E_{\imh}(w)$ has no $\imo$-neighbor. On the
  other side, $E_{\imo} E_{\imh} E_{\imo}(w) = E_{\imt} E_{\imh}
  E_{\imo}(w)$ which does not admit an $i$-neighbor. Therefore the
  connected component $E_{\imo} \cup E_i$ containing $E_{\imh}(w)$
  will have generating function $s_{(3,1)}$ or $s_{(2,1,1)}$, thus
  establishing $\LSP_4$ in this case as well. Therefore $\LSP_4$ holds
  in $\varphi_i^w(\G)$ for any $w \in W_i(\G)$.

  \begin{figure}[ht]
    \begin{displaymath}
      \begin{array}{\cs{8}\cs{8}\cs{8}c}
        & & \rnode{uu}{\B} & \\[3ex]
        \rnode{x}{\B} & \rnode{w}{w} & \rnode{u}{\B} & \rnode{v}{\B} \\[4ex]
        \rnode{X}{\B} & \rnode{W}{\B} & \rnode{U}{\B} & \rnode{V}{\B} \\[3ex]
        & \rnode{WW}{\B} & &       
      \end{array}
      \hspace{5em}
      \begin{array}{\cs{8}\cs{8}\cs{8}c}
        & & & \\[3ex]
        \rnode{x1}{\B} & \rnode{w1}{w} & \rnode{u1}{\B} & \rnode{v1}{\B} \\[4ex]
        \rnode{X1}{\B} & \rnode{W1}{\B} & \rnode{U1}{\B} & \rnode{V1}{\B} \\[3ex]
        & \rnode{WW1}{\B} & &
      \end{array}
      \psset{nodesep=3pt,linewidth=.1ex}
      \everypsbox{\scriptstyle}
      \ncline {w}{uu} \naput{\imt}
      \ncline {x}{w} \naput{i}
      \ncline {w}{u} \naput{\imo}
      \ncline {u}{v} \naput{i}
      \ncline {x}{X} \nbput{\imh}
      \ncline {w}{W} \nbput{\imh}
      \ncline {u}{U} \naput{\imh}
      \ncline {v}{V} \naput{\imh}
      \ncline {X}{W} \nbput{i}
      \ncline {U}{V} \nbput{i}
      \ncline {W}{U} \nbput{\imo}
      \ncline {WW}{U} \nbput{\imt}
      \ncline {x1}{w1} \naput{i}
      \ncline {w1}{u1} \naput{\imo}
      \ncline {u1}{v1} \naput{i}
      \ncline {x1}{X1} \nbput{\imh}
      \ncline[offset=2pt] {w1}{W1} \naput{\imt}
      \ncline[offset=2pt] {W1}{w1} \naput{\imh}
      \ncline {u1}{U1} \naput{\imh}
      \ncline {v1}{V1} \naput{\imh}
      \ncline {X1}{W1} \nbput{i}
      \ncline {U1}{V1} \nbput{i}
      \ncline[offset=2pt] {WW1}{U1} \nbput{\imo}
      \ncline[offset=2pt] {U1}{WW1} \nbput{\imt}
    \end{displaymath}
    \caption{\label{fig:phi-degree4} The two possible $\imo$-types for
      $w \in W_i(\G)$ admitting an $\imh$-neighbor.}
  \end{figure}

  Second consider $\psi_i^x(\G)$. Since the $i$-edge between $x$ and
  $E_i(x)$ is flat, exactly one of these vertices admits an
  $\imo$-neighbor, say $E_i(x)$ does and $x$ does not. Since $x$ does
  not admit an $\imo$-neighbor, by axiom $3$, $E_{\imt}(x)$
  does. Since $x \in C_i(\G)$, $E_{\imt}(x)$ has a flat $i$-edge, and
  so since $E_{\imt}(x)$ admits an $\imo$-neighbor, $E_i E_{\imt}(x)$
  does not. On the other side, since $E_i(x)$ admits an
  $\imo$-neighbor but does not have $\imo$-type W, $E_{\imt} E_i(x)$
  does not admit an $\imo$-neighbor. Therefore neither
  $E_i(E_{\imt}(x))$ nor $\psi_i^x(E_{\imt}(x)) = E_{\imt} E_i(x)$
  admits an $\imo$-neighbor, so $\LSP_4$ of the connected component
  containing $E_{\imt}(x)$ is preserved. Similarly, since neither
  $E_i( E_i E_{\imt} E_i (x)) = E_{\imt} E_i(x)$ nor $\psi_i^x( E_i
  E_{\imt} E_i (x) ) = E_i E_{\imt}(x)$ admit an $\imo$-neighbor,
  $\LSP_4$ for the connected component containing $E_{\imt} E_i (x)$
  is preserved.  By axioms $2$ and $5$, both $E_{\imo}$ and $E_i$
  commute with $E_h$ for $h \leq \imf$, so $\LSP_4$ is maintained on
  $E_2 \cup \cdots \cup E_{\imf}$. Since both $E_{\imt}(x)$ and
  $E_{\imt} E_i(x)$ have flat $\imt$-edges, neither has $\imt$-type W,
  and so if either, and hence both, admits an $\imh$-neighbor, the
  $\imh$-edge must preserve $\sigma_{\imo}$. Thus $\LSP_4$ extends
  across a single $E_{\imh}$ edge as well. By axiom $6$, the claim
  follows.
\end{proof}

Unfortunately, neither $\varphi_i^w(\G)$ nor $\psi_i^x(\G)$ always has
$\LSP_5$. If $\varphi_i^w(\G)$ or $\psi_i^x(\G)$ has $\LSP_5$ for at
least one $w \in W_i(\G)$ or at least one $x \in C_i(\G)$, then by
Propositions~\ref{prop:phi-terminate} and \ref{prop:psi-terminate}, we
could always apply one of the maps until both $W_i$ and $C_i$ were
both empty. With this idea in mind, define a set $U_i(\G) \subset
W_i(\G) \cup C_i(\G)$ by
\begin{equation}
  U_i(\G) = \left\{ 
    \begin{array}{c}
      w \in W_i(\G) \\
      x \in C_i(\G)
    \end{array}
    \big| \ \mbox{components of
      $E_{\imt} \cup E_{\imo} \cup E_i$ are Schur positive in} 
    \begin{array}{c}
      \varphi_i^w(\G) \\
      \psi_i^x(\G)
    \end{array} 
  \right\} .
  \label{eqn:U}
\end{equation}

When $W_i(\G) \cup C_i(\G)$ is nonempty but $U_i(\G)$ is empty, we can
apply a slight variant of the map $\psi_i$, denoted $\gamma_i$, to
$\G$ so that $U_i(\G)$ is nonempty. The map $\gamma_i$ is depicted in
Figure~\ref{fig:gamma}. As usual, we begin by establishing the
necessary isomorphism of $i$-packages.

\begin{lemma}
  Let $\G$ be a signed, colored graph of type $(n,N)$ satisfying dual
  equivalence axioms $1,2,3$ and $5$, and suppose that the
  $(\imt,N)$-restriction of $\G$ is a dual equivalence graph.  Let $z$
  have a non-flat $i$-edge such that no vertex between $z$ and $\left(
  E_{\imo} E_{i} \right)^{m} (z)$ has $\imo$-type W, and suppose $z$
  and $\left( E_{\imo} E_{i} \right)^{m} (z)$ admit an $\imt$-neighbor
  and that both or neither have flat $\imt$-edges. Then the
  $i$-package of $E_{\imt}(z)$ is isomorphic to the $i$-package of
  $E_{\imt}\left( E_{\imo} E_{i} \right)^{m} (z)$.
  \label{lem:gamma-compatible}
\end{lemma}

\begin{proof}
  Since $z$ has a non-flat $i$-edge, $E_{i}(z)$ must have $i$-type W,
  and so, too, must $E_{\imo} E_{i}(z)$. If $E_{\imo} E_{i}(z)$ has a
  non-flat $i$-edge, then the pattern persists so that all vertices
  between $z$ and $u = \left( E_{\imo} E_{i} \right)^{m} (z)$ have
  $i$-type W, and all $E_i$ edges between them are non-flat. Thus each
  $E_{\imo}$ edge between $z$ and $u$ toggles $\sigma_{\imt,\imo}$ by
  axiom $2$ and toggles $\sigma_i$ since it has $\imo$-type
  W. Similarly, each $E_i$ edge between $z$ and $u$ toggles
  $\sigma_{\imo,i}$ by axiom $2$ and toggles $\sigma_{\imt}$ since it
  is non-flat. Finally, since there is an even number of edges between
  $u$ and $z$ each of which toggles $\sigma_{\imt,\imo,i}$, we have
  $\sigma(z)_{\imt,\imo,i} = \sigma(u)_{\imt,\imo,i}$. In particular,
  both or neither admit an $\imt$-neighbor. If neither does, the
  result is clearly true, so assume both do. By
  Lemma~\ref{lem:phi-compatible} and the fact that $E_i$ gives an
  isomorphism of $i$-packages, the $i$-package of $z$ is isomorphic to
  the $i$-package of $u$. Now the same argument in
  Lemma~\ref{lem:psi-compatible} applies to extend the $i$-package
  isomorphisms across the flat $E_{\imt}$ edges since neither has
  $\imo$-type W.
\end{proof}

\begin{figure}[ht]
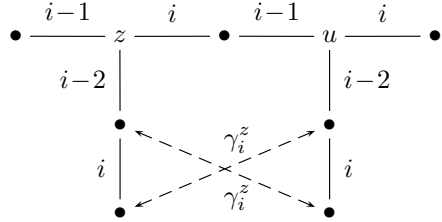

  \begin{displaymath}
    \begin{array}{\cs{8}\cs{8}\cs{8}\cs{8}c}
      \rnode{a0}{\B} & \rnode{b0}{z} & \rnode{c0}{\B} &
      \rnode{d0}{u} & \rnode{e0}{\B} \\[5ex] 
      & \rnode{b1}{\B} & & \rnode{d1}{\B} & \\[5ex]
      & \rnode{b2}{\B} & & \rnode{d2}{\B} & 
    \end{array}
    \psset{nodesep=3pt,linewidth=.1ex}
    \ncline {a0}{b0} \naput{\imo}
    \ncline {b0}{c0} \naput{i}
    \ncline {c0}{d0} \naput{\imo}
    \ncline {d0}{e0} \naput{i}
    \ncline {b0}{b1} \nbput{\imt}
    \ncline {d0}{d1} \naput{\imt}
    \ncline {b1}{b2} \nbput{i}
    \ncline {d1}{d2} \naput{i}
    \ncline[linestyle=dashed] {<->}{b1}{d2} \naput{\gamma_i^z}
    \ncline[linestyle=dashed] {<->}{b2}{d1} \nbput{\gamma_i^z}
  \end{displaymath}
  \caption{\label{fig:gamma} An illustration of
    $\gamma_i^z$.}
\end{figure}

Following the familiar pattern, we use the isomorphism of
Lemma~\ref{lem:gamma-compatible} to define an involution
$\gamma_i^z$ on all vertices admitting an $i$-neighbor.

\begin{definition}
  For $z$ not $\imo$-type W with a non-flat $i$-edge and a flat
  $\imt$-edge, let $u = \left( E_{\imo} E_{i} \right)^{m} (z)$, $m>0$,
  such that $u$ does not have $\imo$-type W and has a flat
  $\imt$-edge. Let $\phi$ denote the isomorphism of
  Lemma~\ref{lem:gamma-compatible}. Define the involution
  $\gamma_i^z$ on all vertices admitting an $i$-neighbor as
  follows.
  \begin{equation}
    \gamma_i^z(v) = \left\{ \begin{array}{rl}
        E_i\phi(v) & \mbox{if $v$ lies on the $i$-package of
          $E_{\imt}(z)$ or $E_{\imt} (u)$,} \\[1ex]  
        \phi E_i(v) & \mbox{if $E_i(v)$ lies on the
          $i$-package of $E_{\imt}(z)$ or $E_{\imt} (u)$,}\\[1ex]
        E_i(v) & \mbox{otherwise.}
      \end{array} \right.
    \label{eqn:gamma}
  \end{equation}
  Define $E'_i$ to be the set of pairs $\{v,\gamma_i^x(v)\}$
  for each $v$ admitting an $i$-neighbor. Define a signed, colored
  graph $\gamma_i^x(\G)$ of type $(n,N)$ by
  \begin{equation}
    \gamma_i^x(\G) = (V, \sigma, E_2 \cup\cdots\cup E_{\imo}
    \cup E'_i \cup E_{\ipo} \cup\cdots\cup E_{\nmo}). 
  \end{equation}
\label{defn:gamma}
\end{definition}

\begin{remark}
  Note that the $m>0$ case for $\psi_i$ handles the situation where
  vertices have $\imo$-type W, that is, components of $E_{\imt} \cup
  E_{\imo}$ that do not appear in Figure~\ref{fig:lambda4}. The map
  $\gamma_i$ is similar, but it handles the situation where vertices
  have $i$-type W, that is, components of $E_{\imo} \cup E_{i}$ do not
  appear in Figure~\ref{fig:lambda4}. Therefore while applying
  $\psi_i$ for $m>0$ is relatively rare (e.g. does not arise when
  axiom $4$ holds for the $(i,N)$-restriction), $\gamma_i$ is often
  indispensable.
  \label{rmk:psi-v-gamma}
\end{remark}

The map $\gamma_i^z(\G)$ maintains $\LSP_4$ for the same reasons that
$\psi_i^x$ does. Unlike $\varphi_i^w$ and $\psi_i^x$, the map
$\gamma_i^z$ does not separate connected components of $E_{\imt} \cup
E_{\imo} \cup E_i$, so $\LSP_5$ is trivially maintained. While
$\gamma_i^z(\G)$ does not decrease $W_i$ or $C_i$, neither does it
increase them. Its usefulness lies in the fact that it allows
$\varphi_i^w(\G)$ or $\psi_i^x(\G)$ to be applied while maintaining
$\LSP_5$. 

We begin our study of $\gamma_i$ by observing that when $W_i(\G) \cup
C_i(\G)$ is nonempty and $U_i(\G)$ is empty, the structure of
connected components of $E_{\imt} \cup E_{\imo} \cup E_i$ of the graph
is that of a rooted tree.

\begin{lemma}
  Let $\G$ be a locally Schur positive graph of type $(\ipo,\ipo)$
  satisfying dual equivalence axioms $1,2,3$ and $5$, and suppose that
  the $(\imt,\ipo)$-restriction of $\G$ is a dual equivalence graph
  and that the $(i,\ipo)$-restriction of $\G$ satisfies dual
  equivalence axiom $4$. If a connected component of $E_{\imt} \cup
  E_{\imo} \cup E_i$ not appearing in Figure~\ref{fig:lambda5} has no
  element in $U_i(\G)$, then, treating double edges as single edges,
  the component is a tree consisting of vertices of $i$-types W, B and
  C. Moreover, the component contains a unique vertex $w$ not
  admitting an $\imt$-neighbor such that $E_{\imo}(w) = E_i(w)$.
  \label{lem:tree}
\end{lemma}

\begin{proof}
  If there is a vertex of $i$-type A, then we claim there exists a
  vertex $u$ admitting an $\imt$-neighbor but not an $\imo$-neighbor
  nor an $i$-neighbor. Suppose $w$ has $i$-type A. Then either there
  exists $w$ admitting neither an $\imo$-neighbor nor an
  $\imt$-neighbor, or there exists $w$ admitting an $\imo$-neighbor
  but not an $\imt$-neighbor such that $E_{\imo}(w)$ does not admit an
  $i$-neighbor. In the former case, $\sigma_{\imh,\imt,\imo,i}(w) =
  +++-$ or $---+$, which, by $\LSP_5$, must be contribute to the Schur
  function $s_{(4,1)}$ or $s_{2,1,1,1}$, respectively. Therefore there
  must be a vertex $u$ with signature $\sigma_{\imh,\imt,\imo,i}(u) =
  -+++$ or $+---$, respectively, and so $u$ admits an $\imt$-neighbor
  but not an $\imo$-neighbor nor an $i$-neighbor. In the latter case,
  since the $(i,\ipo)$-restriction satisfies axiom $4$, $E_{\imo}(w)$
  cannot have $\imo$-type W, and so $u = E_{\imt} E_{\imo} (w)$ will
  admit neither an $\imo$-neighbor nor an $i$-neighbor, thereby
  establishing the claim. 

  If $E_{\imo}E_{\imt}(u)$ has a flat $i$-edge, then the component
  appears in Figure~\ref{fig:lambda5} after all, so assume it has a
  non-flat $i$-edge. Since the component of $E_{\imo} \cup E_i$ begins
  at $E_{\imt}(u)$ with $\sigma(E_{\imt}(u))_{\imh,\imt,\imo,i} =
  +-++$ or $-+--$, $\LSP_4$ ensures that after an even number of
  alternating $E_{\imo}$ and $E_i$ edges, the components ends after a
  flat $i$-edge at a vertex $v$ with $\sigma(v)_{\imt,\imo,i} = ++-$
  or $--+$, respectively. Each $E_{\imo}$ edge on the component must
  be flat, since otherwise by Figure~\ref{fig:lambda4} it would be a
  double edge with $E_{\imt}$, and by axiom $4$ $E_i$ fixes
  $\sigma_{\imh}$, so $\sigma(v)_{\imh} =
  \sigma(E_{\imt}(u))_{\imh}$. Therefore applying $\varphi_i^{w}$ for
  any $w$ between $E_{\imt}(u)$ and $E_i(v)$ removes a component of
  $E_{\imt} \cup E_{\imo} \cup E_i$ with generating function
  $s_{(4,1)}$ or $s_{(2,1,1,1)}$, respectively, contradicting the
  assumption that $U_i(\G)$ is empty. Hence no vertex on the component
  has $i$-type A.

  Suppose there is a sequence of at least three edges forming a
  loop. If the loop does not contain a vertex of $i$-type W, then it
  must consist entirely of $E_{\imt}$ and $E_i$ edges. In this case,
  any vertex $x$ on the loop lies in $C_i(\G)$. Applying $\psi_i^x$
  removes a component of $E_{\imt} \cup E_{\imo} \cup E_i$ with
  generating function $s_{(3,1,1)}$, contradicting the assumption that
  $U_i(\G)$ is empty. If the loop contains a vertex $w$ of $i$-type W,
  then applying $\varphi_i^w$ does not split the component, once again
  contradicting the assumption that $U_i(\G)$ is empty. Therefore
  there is no closed loop apart from double edges.

  We next claim that there is a vertex $w$ such that $E_{\imo}(w) =
  E_i(w)$. If no vertex has left $i$-type B, meaning the component of
  $i$-type B on the left side of Figure~\ref{fig:type}, then $\LSP_5$
  dictates that no vertex can have right $i$-type B either since both
  vertices can only contribute to the Schur function $s_{(3,2)}$ or
  $s_{(2,2,1)}$, and so no vertex can have $i$-type W. Thus all
  vertices have $i$-type C, in which case the finiteness of the graph
  ensures there is a closed loop, contradicting the previous
  result. Therefore there must be a vertex with left $i$-type
  B. Starting from this vertex, we can follow the graph outwards never
  looping back. If we reach a vertex with $i$-type C, then the
  $E_{\imo}$ leads to a leaf and the other edge continues on. If we
  reach a vertex with left $i$-type B, then we reach a leaf since we
  must follow the double edge between $E_{\imt}$ and $E_{\imo}$. If we
  reach a vertex with right $i$-type B, then we reach a vertex with
  $i$-type W. At this point, if $E_{\imo}$ and $E_i$ form a double
  edge, then we have reached a leaf. Otherwise, we branch in two
  directions. Therefore every path must end in a double edge. If all
  endings are at vertices with a double edge between $E_{\imt}$ and
  $E_{\imo}$, then there will be one more left $i$-type B component
  than right $i$-type B component, contradicting that $\G$ is
  $\LSP_5$.

  Follow edges from $w$, say starting with $E_{\imt}$. Each time we
  reach a vertex $v$ admitting an $i$-edge, either $v$ does not admit
  an $\imo$-neighbor, thereby forcing the $i$-edge to be flat, or $v
  \in W_i(\G)$. In either case, we cannot have $E_{\imo}(v) = E_i(v)$,
  so the vertex $w$ is unique up to interchanging $w$ and $E_{\imo}(w)
  = E_i(w)$. Since exactly one of the two admits an $\imt$-neighbor,
  the lemma follows.
\end{proof}

\begin{theorem}
  Let $\G$ be a locally Schur positive graph of type $(\ipo,\ipo)$
  satisfying dual equivalence axioms $1,2,3$ and $5$, and suppose that
  the $(\imt,\ipo)$-restriction of $\G$ is a dual equivalence graph
  and that the $(i,\ipo)$-restriction of $\G$ satisfies dual
  equivalence axiom $4$. If $W_i(\G) \cup C_i(\G)$ is nonempty and
  $U_i(\G)$ is empty, then there exists $z$ such that
  $U_i(\gamma_i^z(\G))$ is nonempty.
  \label{thm:nonempty}
\end{theorem}

\begin{proof}
  Fix a connected component of $E_{\imt} \cup E_{\imo} \cup E_i$ not
  appearing in Figure~\ref{fig:lambda5}. Recall that under the
  hypothesis that axiom $4$ holds for the $(i,\ipo)$-restriction of
  $\G$, if $w \in W_i(\G)$ then $\varphi_i^w$ applies and if $x \in
  C_i(\G)$ then $\psi_i^x$ applies, though neither necessarily
  preserves $\LSP_5$.  By Lemma~\ref{lem:tree}, the component is a
  rooted tree consisting of vertices of $i$-types W, B and C, with the
  root being the unique vertex not admitting an $\imt$-neighbor with a
  double edge for $\imo$ and $i$.
  
  Identify each connected component of $E_{\imt} \cup E_{\imo}$ as
  $i$-type C ($C$-node), left $i$-type B ($L$-node), or right $i$-type
  B and $i$-type W ($R$-node), where this last case has a vertex $v$
  of right $i$-type B and a vertex $E_{\imt}(v)$ with $i$-type
  W. Consider the graph with nodes given by these components and
  directed edges given by $i$-edges directed away from the root. Since
  the graph is a tree, every node has a unique incoming
  $i$-edge. Furthermore, $L$-nodes correspond precisely to leaves, a
  $C$-node has one outgoing flat $i$-edge, and an $R$-node has one
  outgoing flat $i$-edge and one outgoing non-flat $i$-edge. In the
  case of the root, an $R$-node, the outgoing non-flat $i$-edge is a
  double edge with $E_{\imo}$, i.e. a loop back to the root, and this
  is the only loop in the graph.

  Figure~\ref{fig:peel} illustrates three situations that cannot arise
  in this graph when $U_i(\G)$ is empty. First, if an $R$-node goes to
  an $L$-node by a flat $i$-edge, then $\varphi_i$ preserves local
  Schur positivity as depicted in the left case of
  Figure~\ref{fig:peel}. Second, if a $C$-node goes to another
  $C$-node (necessarily by a flat $i$-edge), then $\psi_i$ preserves
  local Schur positivity as depicted in the middle case of
  Figure~\ref{fig:peel}. The third case is more complicated. If an
  $R$-node goes to another $R$-node by a flat $i$-edge and each of
  these $R$-nodes goes to an $L$-node by a nonflat $i$-edge, then
  $\psi_i$ preserves local Schur positivity. This is the rightmost
  case depicted in Figure~\ref{fig:peel}. The assumption that
  $U_i(\G)$ is empty forbids these cases.

  \begin{figure}[ht]
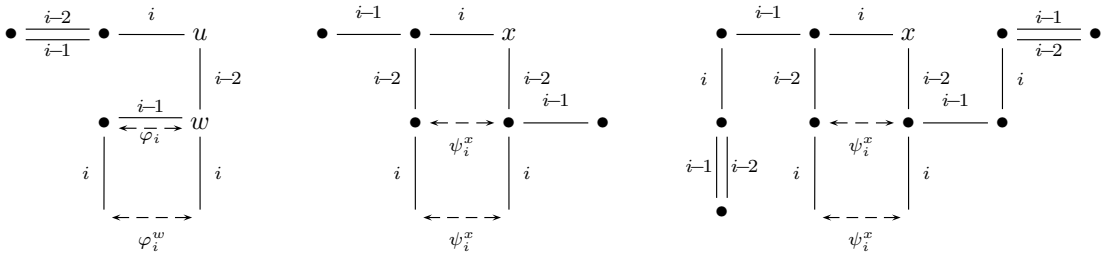

    \begin{displaymath}
      \begin{array}{\cs{7}\cs{7}c}
        \rnode{a3}{\B} & \rnode{a2}{\B} & \rnode{a1}{u} \\[5ex]
        & \rnode{b2}{\B} & \rnode{b1}{w}  \\[5ex]
        & \rnode{c2}{} & \rnode{c1}{}
      \end{array}
      \hspace{3em}
      \begin{array}{\cs{7}\cs{7}\cs{7}c}
        \rnode{A1}{\B} & \rnode{A2}{\B} & \rnode{A3}{x} &  \\[5ex]
        & \rnode{B2}{\B} & \rnode{B3}{\B} & \rnode{B4}{\B} \\[5ex]
        & \rnode{C2}{} & \rnode{C3}{} &
      \end{array}
      \hspace{3em}
      \begin{array}{\cs{7}\cs{7}\cs{7}\cs{7}c}
        \rnode{AA1}{\B} & \rnode{AA2}{\B} & \rnode{AA3}{x} & \rnode{AA4}{\B} & \rnode{ZZ4}{\B}  \\[5ex]
        \rnode{BB1}{\B} & \rnode{BB2}{\B} & \rnode{BB3}{\B} & \rnode{BB4}{\B}&  \\[5ex]
        \rnode{CC1}{\B} & \rnode{CC2}{} & \rnode{CC3}{} & & 
      \end{array}
      \psset{nodesep=3pt,linewidth=.1ex}
      \everypsbox{\scriptstyle}
      \ncline{a2}{a1} \naput{i}
      \ncline[offset=2pt]{a3}{a2} \nbput{\imo}
      \ncline[offset=2pt]{a2}{a3} \nbput{\imt}
      \ncline{b1}{a1} \nbput{\imt}
      \ncline[offset=2pt,linestyle=dashed]{<->}{b1}{b2} \nbput{\imo}
      \ncline[offset=2pt]{b2}{b1} \nbput{\varphi_i}
      \ncline{c1}{b1} \nbput{i}
      \ncline{c2}{b2} \naput{i}
      \ncline[linestyle=dashed]{<->}{c2}{c1} \nbput{\varphi^w_i}
      \ncline{A1}{A2} \naput{\imo}
      \ncline{A2}{A3} \naput{i}
      \ncline{A2}{B2} \nbput{\imt}
      \ncline{A3}{B3} \naput{\imt}
      \ncline[linestyle=dashed]{<->}{B2}{B3} \nbput{\psi^x_i}
      \ncline{B3}{B4} \naput{\imo}
      \ncline{B2}{C2} \nbput{i}
      \ncline{B3}{C3} \naput{i}
      \ncline[linestyle=dashed]{<->}{C2}{C3} \nbput{\psi^x_i}
      \ncline[offset=2pt]{AA4}{ZZ4} \nbput{\imt}
      \ncline[offset=2pt]{ZZ4}{AA4} \nbput{\imo}
      \ncline{AA1}{AA2} \naput{\imo}
      \ncline{AA2}{AA3} \naput{i}
      \ncline{AA1}{BB1} \nbput{i}
      \ncline{AA2}{BB2} \nbput{\imt}
      \ncline{AA3}{BB3} \naput{\imt}
      \ncline{AA4}{BB4} \naput{i}
      \ncline[linestyle=dashed]{<->}{BB2}{BB3} \nbput{\psi^x_i}
      \ncline{BB3}{BB4} \naput{\imo}
      \ncline[offset=2pt]{BB1}{CC1} \nbput{\imo}
      \ncline[offset=2pt]{CC1}{BB1} \nbput{\imt}
      \ncline{BB2}{CC2} \nbput{i}
      \ncline{BB3}{CC3} \naput{i}
      \ncline[linestyle=dashed]{<->}{CC2}{CC3} \nbput{\psi^x_i}
    \end{displaymath}
    \caption{\label{fig:peel} Three cases where $\varphi_i^w$ (left)
      or $\psi_i^x$ (middle and right) preserve $\LSP_5$.}
  \end{figure}

  Figure~\ref{fig:bar} depicts two cases that are easily resolved with
  $\gamma_i$. The lefthand case depicts the situation when an $R$-node
  goes to a $C$-node by a nonflat $i$-edge and that $C$-node goes to
  an $L$-node (necessarily by a flat $i$-edge). In this case, applying
  $\gamma_i$ interchanges the subtree below the $R$-node with that
  subtree below the $C$-node, and the result is an instance of the
  leftmost case of Figure~\ref{fig:peel} where $\varphi_i$ can by
  applied. The righthand case of Figure~\ref{fig:bar} depicts that
  situation when both the flat and nonflat $i$-edges from an $R$-node
  go to $C$-nodes. Once again, applying $\gamma_i$ interchanges the
  subtrees, now resulting in an instance of the middle case of
  Figure~\ref{fig:peel} where $\psi_i$ can by applied.

  \begin{figure}[ht]
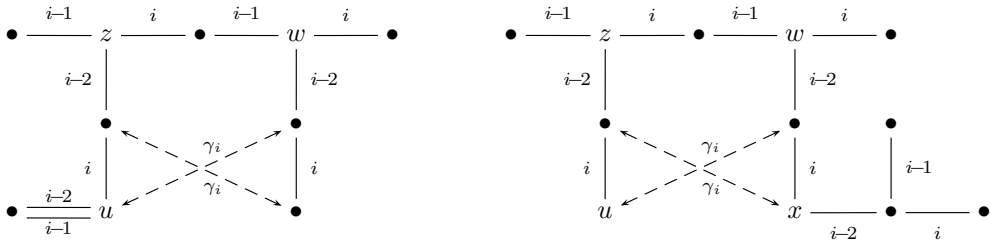

    \begin{displaymath}
      \begin{array}{\cs{7}\cs{7}\cs{7}\cs{7}c}
        \rnode{a1}{\B} & \rnode{a2}{z} & \rnode{a3}{\B} & \rnode{a4}{w} & \rnode{a5}{\B} \\[5ex]
        & \rnode{b2}{\B} & & \rnode{b4}{\B} & \\[5ex]
        \rnode{c1}{\B} & \rnode{c2}{u} & & \rnode{c4}{\B} &
      \end{array}
      \hspace{3em}
      \begin{array}{\cs{7}\cs{7}\cs{7}\cs{7}\cs{7}c}
        \rnode{A1}{\B} & \rnode{A2}{z} & \rnode{A3}{\B} & \rnode{A4}{w} & \rnode{A5}{\B} & \\[5ex]
        & \rnode{B2}{\B} & & \rnode{B4}{\B} & \rnode{B5}{\B} & \\[5ex]
        & \rnode{C2}{u} & & \rnode{C4}{x} & \rnode{C5}{\B} & \rnode{C6}{\B}
      \end{array}
      \psset{nodesep=3pt,linewidth=.1ex}
      \everypsbox{\scriptstyle}
      \ncline{a1}{a2} \naput{\imo}
      \ncline{a2}{a3} \naput{i}
      \ncline{a3}{a4} \naput{\imo}
      \ncline{a4}{a5} \naput{i}
      \ncline{a2}{b2} \nbput{\imt}
      \ncline{a4}{b4} \naput{\imt}
      \ncline{b2}{c2} \nbput{i}
      \ncline{b4}{c4} \naput{i}
      \ncline[offset=2pt]{c1}{c2} \nbput{\imo}
      \ncline[offset=2pt]{c2}{c1} \nbput{\imt}
      \ncline[linestyle=dashed]{<->}{b2}{c4} \naput{\gamma_i}
      \ncline[linestyle=dashed]{<->}{b4}{c2} \naput{\gamma_i}
      \ncline{A1}{A2} \naput{\imo}
      \ncline{A2}{A3} \naput{i}
      \ncline{A3}{A4} \naput{\imo}
      \ncline{A4}{A5} \naput{i}
      \ncline{A2}{B2} \nbput{\imt}
      \ncline{A4}{B4} \naput{\imt}
      \ncline{B2}{C2} \nbput{i}
      \ncline{B4}{C4} \naput{i}
      \ncline[linestyle=dashed]{<->}{B2}{C4} \naput{\gamma_i}
      \ncline[linestyle=dashed]{<->}{B4}{C2} \naput{\gamma_i}
      \ncline{B5}{C5} \naput{\imo}
      \ncline{C4}{C5} \nbput{\imt}
      \ncline{C5}{C6} \nbput{i}
    \end{displaymath}
    \caption{\label{fig:bar} Two cases where $\gamma_i^z$
      allows $\varphi_i^w$ (left) or $\psi_i^x$ (right) to preserve
      $\LSP_5$.}
  \end{figure}
  
  We claim that this analysis resolves all configurations for edges
  coming from an $R$-node, except for the four shown in
  Figure~\ref{fig:node} or the case where the non-flat edge connected
  to another $R$-node. For the figures, we draw flat $i$-edges
  vertically and nonflat $i$-edges horizontally. If the nonflat
  $i$-edge of the $R$-node goes to an $L$-node, then the flat $i$-edge
  must either go to a $C$-node or another $R$-node, since the left
  side of Figure~\ref{fig:peel} precludes an $L$-node. In the former
  case, the (necessarily flat) $i$-edge from the $C$-node must go
  either to an $L$-node or an $R$-node, the left two cases of
  Figure~\ref{fig:node}, since the middle case of
  Figure~\ref{fig:peel} precludes another $C$-node. In the latter
  case, the nonflat $i$-edge of the second $R$-node must go to a
  $C$-node, since the right case of Figure~\ref{fig:peel} precludes
  another $L$-node. The flat $i$-edge from the second $R$-node cannot
  go to an $L$-node (by the left case of Figure~\ref{fig:peel}) nor to
  a $C$-node (by the right case of Figure~\ref{fig:bar}), so it must
  go to another $R$-node. Similarly, the (necessarily flat) $i$-edge
  from the $C$-node cannot go to another $C$-node (by the middle case
  of Figure~\ref{fig:peel}) nor to an $L$-node (by the left case of
  Figure~\ref{fig:bar}), so it must go to yet another $R$-node. The
  resulting case is the third of Figure~\ref{fig:node}. This handles
  all cases where the nonflat $i$-edge of an $R$-node goes to an
  $L$-node, so consider the alternative case in which the nonflat
  $i$-edge must go to a $C$-node. The analysis here is identical to
  the previous case, resulting in the rightmost case in
  Figure~\ref{fig:node}. Thus the claim is proved, and
  Figure~\ref{fig:node} contains all the remaining cases. Moreover,
  the root, necessarily an $R$-node, must be one of the middle two
  cases but with the non-flat edge looping instead of going to an
  $L$-node.

  \begin{figure}[ht]
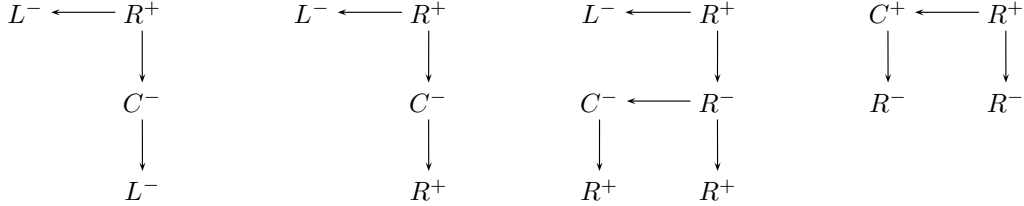

    \begin{displaymath}
      \begin{array}{\cs{7}c}
        \rnode{a1}{L^-} & \rnode{x1}{R^+} \\[5ex]
        \rnode{b1}{ } & \rnode{y1}{C^-} \\[5ex]
        \rnode{c1}{ } & \rnode{z1}{L^-}
      \end{array}
      \hspace{4em}
      \begin{array}{\cs{7}c}
        \rnode{a2}{L^-} & \rnode{x2}{R^+} \\[5ex]
        \rnode{b2}{ } & \rnode{y2}{C^-} \\[5ex]
        \rnode{c2}{ } & \rnode{z2}{R^+}
      \end{array}
      \hspace{4em}
      \begin{array}{\cs{7}c}
        \rnode{a3}{L^-} & \rnode{x3}{R^+} \\[5ex]
        \rnode{b3}{C^-} & \rnode{y3}{R^-} \\[5ex]
        \rnode{c3}{R^+} & \rnode{z3}{R^+}
      \end{array}
      \hspace{4em}
      \begin{array}{\cs{7}c}
        \rnode{a4}{C^+} & \rnode{x4}{R^+} \\[5ex]
        \rnode{b4}{R^-} & \rnode{y4}{R^-} \\[5ex]
        &
      \end{array}
      \psset{nodesep=3pt,linewidth=.1ex}
      \everypsbox{\scriptstyle}
      \ncline{->}{x1}{a1}
      \ncline{->}{x1}{y1}
      \ncline{->}{y1}{z1}
      \ncline{->}{x2}{a2}
      \ncline{->}{x2}{y2}
      \ncline{->}{y2}{z2}
      \ncline{->}{x3}{a3}
      \ncline{->}{x3}{y3}
      \ncline{->}{y3}{b3}
      \ncline{->}{b3}{c3}
      \ncline{->}{y3}{z3}
      \ncline{->}{x4}{a4}
      \ncline{->}{a4}{b4}
      \ncline{->}{x4}{y4}
    \end{displaymath}
    \caption{\label{fig:node} The four possible scenarios for edges
      emanating from an $R^+$-node, where horizontal edges are flat
      and vertical edges are non-flat.}
  \end{figure}

  For the case where long chains of $R$-nodes are connected by
  non-flat edges, the finiteness of the graph ensures that eventually
  one of these $R$-nodes must connect to either an $L$-node or a
  $C$-node, so this last $R$-node will also fall into one of the four
  cases depicted in Figure~\ref{fig:node}. 

  Associate a sign to each node as follows. For $C$-nodes, the sign is
  positive if $\sigma_{\imh}(v)=++--$ where $v$ is the vertex
  admitting neither an $\imt$-neighbor nor an $i$-neighbor and
  negative otherwise. For $L$-nodes and $R$-nodes, the sign is
  positive if the component belongs to $\G_{(3,2)}$ and negative if it
  belongs to $\G_{(2,2,1)}$. Then the graph described in this way has
  $\LSP_5$ if and only if
  \begin{equation}
    \# C^+ = \# C^- \mbox{ and } \# L^+ = \# R^+ \mbox{ and } \# L^- = \# R^-. 
    \label{e:LSP5}
  \end{equation}
  Note that a flat edge changes the sign except for leaves, and a
  non-flat edge preserves the sign except for leaves.

  Note that if the leaf reached from the root using only flat edges
  has the same sign as the root, then $\psi_i$ may be applied to
  remove this leaf and the root, which has generating function
  $s_{(3,2)}$ if positive or $s_{(2,2,1)}$ if negative. Given the four
  possibilities in Figure~\ref{fig:node}, the only terminal case is
  the leftmost. Since the graph is locally Schur positive,
  \eqref{e:LSP5} ensures that there must be some leaf with the same
  sign as the root and a flat incoming edge. In the two rightmost
  cases in Figure~\ref{fig:node}, the map $\gamma_i$ may always be
  applied and doing so swaps the subtrees from the lower two
  $R$-nodes, similar to the scenarios in Figure~\ref{fig:bar}.
  Therefore we may use $\gamma_i$ to swap subtrees until this leaf
  lies on the flat path from the root.
\end{proof}

\begin{theorem}
  Let $\G$ be a locally Schur positive graph of type $(n,N)$
  satisfying dual equivalence axioms $1,2,3$ and $5$, and suppose that
  the $(\imt,N)$-restriction of $\G$ is a dual equivalence graph and
  that the $(i,N)$-restriction of $\G$ satisfies dual equivalence
  axiom $4$. 

  Then we can apply $\varphi_i, \psi_i$ and $\gamma_i$ in such a way
  that the resulting graph still satisfies axioms $1,2$ and $5$, the
  $(\ipo,N)$-restriction satisfies axioms $3$ and $4$, and the
  $(i,N)$-restriction remains a dual equivalence graph.
  \label{thm:axiom4}
\end{theorem}

\begin{proof}
  By Theorem~\ref{thm:nonempty}, we may always apply either
  $\varphi_i$ or $\psi_i$, perhaps with an intermediate application of
  $\gamma_i$. By Proposition~\ref{prop:phi-terminate}, each
  application of $\varphi_i$ strictly decreases $|W_i|$, and by
  Proposition~\ref{prop:psi-terminate} applying $\psi_i$ does not
  increase $|W_i|$, so eventually $W_i$ will be empty. By
  Proposition~\ref{prop:psi-terminate}, applying $\psi_i$ strictly
  decreases $|C_i|$, so once $W_i$ is empty, $\varphi_i$ will no
  longer be applied, and repeated applications of $\psi_i$ will result
  in $C_i$ being empty as well. At this point, by
  Proposition~\ref{prop:empty}, axiom $4$ holds for the
  $(\ipo,N)$-restriction. By construction, these maps maintain axioms
  $1,2$ and $5$. Finally, axiom $4$ implies axiom $3$ for the
  $(\ipo,N)$-restriction. 
\end{proof}

\subsection{An involution to resolve axiom $6$}
\label{sec:D-axiom6}

Let $\G$ be a locally Schur positive graph of type $(n,N)$ satisfying
dual equivalence graph axioms $1,2,3$ and $5$ such that the
$(i,N)$-restriction of $\G$ is a dual equivalence graph. Using the
results of the previous section, we can alter $i$-edges of $\G$,
without losing local Schur positivity for the $(\ipo,N)$-restriction,
until axiom $4$ holds for the $(\ipo,N)$-restriction of $\G$. However,
if we wish to continue further to establish axiom $4$ for higher
edges, the hypotheses of those results require that a suitable
restriction of $\G$ also satisfies axiom $6$.

By Theorem~\ref{thm:cover}, for each connected component $\mathcal{H}$
of the $(\ipo,\ipo)$-restriction of $\G$, there exists a (surjective)
morphism $\phi$ from $\mathcal{H}$ to $\G_{\lambda}$ for a unique
partition $\lambda$ of $\ipo$, and, by Corollary~\ref{cor:fibers}, the
fiber over each vertex of $\G_{\lambda}$ has the same cardinality. By
Proposition~\ref{prop:good-defn} and Theorem~\ref{thm:isomorphic},
$\mathcal{H}$ satisfies axiom $6$ if and only if $\phi$ is an
isomorphism.

Similar to the previous transformations, we define an involution
$\theta_i$ on vertices of $\mathcal{H}$ admitting an $i$-neighbor as
indicated in Figure~\ref{fig:theta} and use it to redefine $i$-edges
that are in violation of axiom $6$.

\begin{figure}[ht]
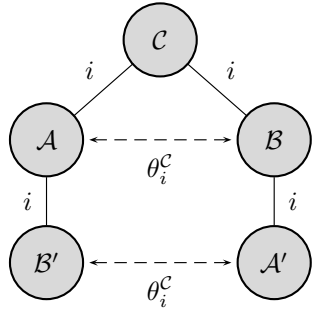

  \begin{displaymath}
    \begin{array}{\cs{9}\cs{9}c}
      & \rnode{X}{%
        \psset{xunit=1ex}
        \psset{yunit=1ex}
        \pspicture(0,0)(1,1)
        \pscircle[fillstyle=solid,fillcolor=lightgray](0.5,0.5){0.5}
        \rput(0.5,0.5){$\C$}
        \endpspicture} & \\[6ex]
      \rnode{Y1}{%
        \psset{xunit=1ex}
        \psset{yunit=1ex}
        \pspicture(0,0)(1,1)
        \pscircle[fillstyle=solid,fillcolor=lightgray](0.5,0.5){0.5}
        \rput(0.5,0.5){$\mathcal{A}$}
        \endpspicture} & & \rnode{Y2}{%
        \psset{xunit=1ex}
        \psset{yunit=1ex}
        \pspicture(0,0)(1,1)
        \pscircle[fillstyle=solid,fillcolor=lightgray](0.5,0.5){0.5}
        \rput(0.5,0.5){$\mathcal{B}$}
        \endpspicture} \\[8ex]
      \rnode{Z1}{%
        \psset{xunit=1ex}
        \psset{yunit=1ex}
        \pspicture(0,0)(1,1)
        \pscircle[fillstyle=solid,fillcolor=lightgray](0.5,0.5){0.5}
        \rput(0.5,0.5){$\mathcal{B}^{\prime}$}
        \endpspicture} & & \rnode{Z2}{%
        \psset{xunit=1ex}
        \psset{yunit=1ex}
        \pspicture(0,0)(1,1)
        \pscircle[fillstyle=solid,fillcolor=lightgray](0.5,0.5){0.5}
        \rput(0.5,0.5){$\mathcal{A}^{\prime}$}
        \endpspicture}
    \end{array}
    \psset{nodesep=12pt,linewidth=.1ex}
    \ncline[nodesep=8pt] {Y1}{X} \naput{i}
    \ncline[nodesep=8pt] {X}{Y2} \naput{i}
    \ncline {Y1}{Z1} \nbput{i}
    \ncline {Y2}{Z2} \naput{i}
    \ncline[nodesep=14pt,linestyle=dashed]{<->} {Y1}{Y2} \nbput{\theta_i^{\C}}
    \ncline[nodesep=14pt,linestyle=dashed]{<->} {Z1}{Z2} \nbput{\theta_i^{\C}}
  \end{displaymath}
  \caption{\label{fig:theta} An illustration of the involution
    $\theta_i^{\C}$ where $\mathcal{A} \cong \mathcal{A}^{\prime}$ and
    $\mathcal{B} \cong \mathcal{B}^{\prime}$.}
\end{figure}

\begin{definition}
  Let $\mathcal{H}$ be a connected component of the
  $(\ipo,\ipo)$-restriction of $\G$ and let $\C$ be a connected
  component of the $(i,i)$-restriction of $\mathcal{H}$. Let $E_i(\C)$
  be the union of all connected components $\mathcal{B}$ of the
  $(i,i)$-restriction of $\mathcal{H}$ such that $\mathcal{B} \neq \C$
  and $\{w,u\} \in E_{i}$ for some $w \in \C$ and some $u \in
  \mathcal{B}$. For each connected component $\mathcal{B}^{\prime}$ of
  the $(i,i)$-restriction of $\mathcal{H}$, let
  $\phi_{\mathcal{B}^{\prime}}$ be the (unique) isomorphism from
  $\mathcal{B}^{\prime}$ to some (unique) $\mathcal{B} \subset
  E_i(\C)$. Define the involution $\theta_i^{\C}$ by
  \begin{equation}
    \theta_i^{\C}(u) = \left\{ \begin{array}{rl}
        \phi_{\mathcal{B}^{\prime}}(E_{i}(u))
        & \mbox{if} \ u \in E_i(\C) \ \mbox{and} \ E_i(u) \in
        \mathcal{B}^{\prime}, \\  
        E_i(\phi_{\mathcal{B}^{\prime}}(u))
        & \mbox{if} \ E_i(u) \in E_i(\C) \ \mbox{and} \ u \in
        \mathcal{B}^{\prime}, \\  
        E_{i}(u) & \mbox{otherwise.}
      \end{array} \right.
    \label{eqn:theta}
  \end{equation}
  Define $E_i'$ to be the set of pairs $\{v,\theta_i^{\C}(v)\}$ for
  all vertices $v$ admitting an $i$-neighbor. Define a signed, colored
  graph $\theta_i^{\C}(\G)$ by
  \begin{equation}
    \theta_i^{\C}(\G) = (V, \sigma, E_2 \cup\cdots\cup E_{\imo} \cup E'_i
    \cup E_{\ipo} \cup\cdots\cup E_{\nmo}). 
  \end{equation}
  \label{defn:theta}
\end{definition}

Note that lower $i$-packages are implicitly preserved for the
definition of $\theta_i$ since all $i$-edges on a connected component
of $E_2 \cup \cdots \cup E_{\imo}$ are redefined together. In order to
ensure that axiom $3$ is maintained, one must be careful in the choice
of $\C$.

\begin{definition}
  Let $\mathcal{H}$ be a connected component of the
  $(\ipo,\ipo)$-restriction of $\G$, and let $\lambda$ be the unique
  partition of $\ipo$ such that there is a surjective morphism from
  $\mathcal{H}$ to $\G_{\lambda}$. A connected component $\C$ of the
  $(i,i)$-restriction of $\mathcal{H}$ is \emph{negatively dominant}
  if one of the following holds:
  \begin{itemize}
  \item $\sigma_{\ipo}(\C) \equiv -1$ and for every connected
    component $\mathcal{B}$ of the $(i,i)$-restriction of
    $\mathcal{H}$ such that $\sigma_{\ipo}(\mathcal{B}) \equiv -1$, if
    $\C \cong \G_{\mu}$ and $\mathcal{B} \cong \G_{\nu}$ for $\mu,\nu
    \subset \lambda$, then $\mu \geq \nu$ in dominance order;
  \item $\sigma_{\ipo}(\mathcal{B}) \equiv +1$ for every connected
    component $\mathcal{B}$ of the $(i,i)$-restriction of
    $\mathcal{H}$, and if $\C \cong \G_{\mu}$ and $\mathcal{B} \cong
    \G_{\nu}$ for $\mu,\nu \subset \lambda$, then $\mu \geq \nu$ in
    dominance order.
  \end{itemize}
  \label{defn:hindge}
\end{definition}

\begin{proposition}
  Let $\G$ be a locally Schur positive graph of type $(n,N)$
  satisfying dual equivalence axioms $1,2,3$ and $5$ such that the
  $(i,N)$-restriction is a dual equivalence graph and the
  $(\ipo,N)$-restriction satisfies dual equivalence axiom $4$. For
  $\C$ a negatively dominant $(i,i)$-restricted component of $\G$, the
  graph $\theta_i^{\C}(\G)$ also satisfies dual equivalence axioms
  $1,2,3$ and $5$ and the $(\ipo,N)$-restriction of
  $\theta_i^{\C}(\G)$ also satisfies dual equivalence axiom
  $4$. Moreover, if $\mathcal{H}$ is the connected component of the
  $(\ipo,N)$-restriction of $\G$ containing $\C$, then
  $\theta_i^{\C}(\mathcal{H})$ has two connected components.
  \label{prop:theta-reasonable}
\end{proposition}

\begin{proof}
  The assertion that $\mathcal{H}$ has two connected components is
  obvious from the definition of $\theta_i^{\C}$. Axioms $1,2$ and $5$
  follow from the definition of $\theta_i^{\C}$, and axiom $4$ follows
  from the fact that edges are swapped only between isomorphic
  components, so the local structure of the $E_{\imt} \cup E_{\imo}
  \cup E_i$ remains unchanged by $\theta_i^{\C}$. Therefore we need
  only address axiom $3$.

  Let $\mathcal{A}$ and $\mathcal{B}$ be two connected components of
  the $(i,i)$-restriction of $\mathcal{H}$, and suppose $\mathcal{A}
  \cong \G_{\alpha}$ and $\mathcal{B} \cong \G_{\beta}$ with $\alpha >
  \beta$ in dominance order. Let $a \in \mathcal{A}$ and $b \in
  \mathcal{B}$ and suppose $\{a,b\} \in E_i$. Similar to the proof of
  Lemma~\ref{lem:extend-signs}, we have $\sigma(w)_{\imo,i} =+-$ and
  $\sigma(v)_{\imo,i} = -+$. Therefore axiom $3$ fails for this edge
  if and only if $\sigma(w)_{\imo,i,\ipo} =+--$ and
  $\sigma(v)_{\imo,i,\ipo} = -++$ if and only if
  $\sigma(\mathcal{A})_{\ipo} = -1$ and $\sigma(\mathcal{B})_{\ipo} =
  +1$. With this characterization in mind, suppose now that
  $\mathcal{A},\mathcal{B}$ and $\mathcal{B}'$ are restricted
  components of $\mathcal{H}$, with $\mathcal{A},\mathcal{B} \in
  E_i(\C)$, $\mathcal{B}' \cong \mathcal{B}$, and $a \in \mathcal{A},
  b \in \mathcal{B}$ and $b' \in \mathcal{B}'$ such that $\{a,b'\} \in
  E_i$ and $b = \theta_i^{\C}(a)$. As before, let $\mathcal{A} \cong
  \G_{\alpha}$ and $\mathcal{B}' \cong \mathcal{B} \cong
  \G_{\beta}$. Suppose $\sigma(\mathcal{A})_{\ipo} = -1$ and
  $\sigma(\mathcal{B})_{\ipo} = +1$. Let $\C \cong \G_{\mu}$.  Then
  the choice of $\C$ as negatively dominant ensures that
  $\sigma(\C)_{\ipo} = -1$ and that $\mu > \alpha$. Further, since
  axiom $3$ holds for $\G$, the preceding characterization ensures
  that $\beta > \mu$. Therefore $\beta > \alpha$ and axiom $3$ holds
  for $\theta_i^{\C}(\G)$ as well.
\end{proof}

Finally, we must consider the local Schur positivity. This is not
difficult to show when the $(\iph,N)$-restriction of $\G$ satisfies
dual equivalence axiom $4$.

\begin{theorem}
  Let $\G$ be a locally Schur positive graph of type $(n,N)$
  satisfying dual equivalence axioms $1,2,3$ and $5$ such that the
  $(i,N)$-restriction is a dual equivalence graph and the
  $(\iph,N)$-restriction satisfies dual equivalence axiom $4$. For
  $\C$ a negatively dominant restricted component, $\theta_i^{\C}(\G)$
  is locally Schur positive.
  \label{thm:theta-LSP}
\end{theorem}

\begin{proof}
  Let $\{w,x\}, \{u,v\} \in E_i(\G)$ with $\theta_i^{\C}(w) = u$ and
  $\theta_i^{\C}(x) = v$. By the definition of $\theta_i^{\C}$, there
  is an isomorphism from the connected component of the
  $(i,N)$-restriction of $\G$ containing $w$ to the connected
  component of the $(i,N)$-restriction of $\G$ containing $v$ that
  sends $w$ to $v$, and similarly for the pair $x$ and
  $u$. Furthermore, by axiom $2$, $\sigma(u)_{\ipt} = \sigma(w)_{\ipt}
  = \sigma(x)_{\ipt} = \sigma(v)_{\ipt}$.

  By Proposition~\ref{prop:theta-reasonable}, if local Schur
  positivity fails, then it must do so for some connected component of
  $E_i \cup E_{\ipo}$, $E_{\imo} \cup E_i \cup E_{\ipo}$, or $E_i \cup
  E_{\ipo} \cup E_{\ipt}$. If either $\{w,x\} \in E_{\ipo}(\G)$ or
  $\{u,v\} \in E_{\ipo}(\G)$, then applying $\theta_i^{\C}$ results in
  the two components being joined for all three cases, thereby
  ensuring that local Schur positivity is preserved. Therefore assume
  none of $u,w,x,v$ has $\ipo$-type W, so that $\sigma(w)_{\ipo} =
  \sigma(x)_{\ipo} = \sigma(u)_{\ipo} = \sigma(v)_{\ipo}$ by axiom
  $3$. For $E_i \cup E_{\ipo}$, since these components appear in
  Figure~\ref{fig:lambda4}, each chain must have one $E_i$ edge and
  one $E_{\ipo}$ edge. Therefore if some component is not locally
  Schur positive in $\theta_i^{\C}(\G)$, then it must be that one
  chain has two $E_{\ipo}$ edges while the other has none, violating
  axiom $3$ contradicting
  Proposition~\ref{prop:theta-reasonable}. Thus $\theta_i^{\C}(\G)$
  maintains $\LSP_4$.

  Consider now components of $E_{\imo} \cup E_i \cup E_{\ipo}$. As
  just observed, either $u$ and $x$ admit $\ipo$-neighbors or $v$ and
  $w$ admit $\ipo$-neighbors, but not both. By symmetry, we may assume
  $u$ and $x$ admit $\ipo$-neighbors, and $w$ and $v$ do not.  By the
  isomorphisms mentioned earlier and the preservation of
  $\sigma_{\ipo}$, $u$ and $x$ must have the same
  $\ipo$-type. Furthermore, there are only two possibilities for the
  $\ipo$-types, namely A or C, since they cannot have $\ipo$-type B
  since $\theta_i^{\C}$ doesn't alter $i$-edges within a restricted
  component, i.e. when $E_{\imo}$ and $E_i$ coincide. If $u$ and $x$
  have $\ipo$-type C, then both admit an $\imo$-neighbor as well and
  neither $v$ nor $w$ admits an $\imo$-neighbor. Thus the components
  are isomorphic after applying $\theta_i^{\C}$. Similarly, if $u$ and
  $x$ have $\ipo$-type A, then they do not admit an $\imo$-neighbor
  but both $v$ and $w$ do, and neither $E_{\imo}(v)$ nor $E_{\imo}(w)$
  admits an $\ipo$-neighbor. Once again, the components are isomorphic
  after applying $\theta_i^{\C}$. Thus connected components of
  $E_{\imo} \cup E_i \cup E_{\ipo}$ remain locally Schur positive.
 
  The case of $E_i \cup E_{\ipo} \cup E_{\ipt}$ is similarly resolved
  by considering the $\ipt$-types of $u,w,x,v$. Since none of these
  vertices has $\ipo$-type W, all or none of them admit an
  $\ipt$-neighbor. If none does, then the preservation of local Schur
  positivity, in fact of axiom $4$, is clear. If any of them has
  $\ipt$-type C, then, since axiom $4$ holds for $\G$, the two
  components become one in $\theta_i^{\C}(\G)$, thus preserving local
  Schur positivity. The only remaining options are for $x$ and $u$ to
  have $\ipt$-type W or $\ipt$-type B, in which case $w$ and $v$ have
  the $\ipt$-type B or $\ipt$-type W, respectively, and axiom $4$ is
  preserved by $\theta_i^{\C}$. Therefore $\theta_i^{\C}(\G)$ has
  $\LSP_5$.
\end{proof}

\begin{theorem}
  Let $\G$ be a signed, colored graph of type $(n,N)$ satisfying dual
  equivalence axioms $1,2,3$ and $5$ such that the $(i,N)$-restriction
  is a dual equivalence graph and the $(\iph,N)$-restriction satisfies
  dual equivalence axiom $4$.

  Then we can apply $\theta_i$ together with $\varphi_{\ipo}$ and
  $\psi_{\ipt}$ in such a way that the resulting graph still satisfies
  axioms $1,2$ and $5$, the $(\iph,N)$-restriction satisfies axioms
  $3$ and $4$, and the $(\ipo,N)$-restriction remains a dual
  equivalence graph.
  \label{thm:axiom6}
\end{theorem}

\begin{proof}
  By Proposition~\ref{prop:theta-reasonable}, if axiom $6$ fails for
  the $(\ipo,N)$-restriction, then we may choose a negatively dominant
  $(i,i)$-restricted component $\C$ and apply $\theta_i^{\C}$ while
  maintaining axioms $1,2,3$ and $5$. By Theorem~\ref{thm:theta-LSP},
  the resulting graph remains locally Schur positive. Moreover, there
  are only two cases where the the $(\iph,N)$-restriction of
  $\theta_i^{\C}(\G)$ fails to satisfy axiom $4$: if one of $u,w,x,v$
  has $\ipo$-type W or if one of $u,w,x,v$ has $\ipt$-type C. Suppose
  $w$ is the offending vertex. Then $w \in W_{\ipo}(\G)$ in the former
  case and $z = E_{\ipt}(w) \in C_{\ipt}(\G)$ in the latter. Given the
  strong hypotheses of the Proposition, $W_{\ipo}^0(\G) =
  W_{\ipo}(\G)$ in the former case and $C_{\ipt}^0(\G) = C_{\ipt}(\G)$
  in the latter. Therefore $\varphi_{\ipo}$ or $\psi_{\ipt}$ may be
  used to restore axiom $4$ for the $(\iph,N)$-restriction. 
  Therefore we may choose another negatively dominant component and
  continue thus. By Proposition~\ref{prop:theta-reasonable}, this
  process terminates exactly when axiom $6$ is satisfied for the
  $(\ipo,N)$-restriction, thus completing step 6. The result satisfies
  axioms $1, 2$ and $5$ by construction, and once again axiom $4$
  implies axiom $3$ for the $(\iph,N)$-restriction. 
\end{proof}

\subsection{Transforming a D graph into a dual equivalence graph}
\label{sec:D-transformation}

We now outline the algorithm for transforming $\G = \G^{(k)}_{c,D}$
into a dual equivalence graph. We proceed by constructing a sequence
of signed, colored graphs $\G = \G_2 , \ldots, \G_{\nmo} =
\widetilde{\G}$ such that $\G_{\imo}$ is a locally Schur positive
graph satisfying dual equivalence axioms $1,2,3$ and $5$, and the
$(i,N)$-restriction of $\G_{\imo}$ is a dual equivalence graph. 

The result for $\G_2 = \G$ is trivial, so we consider how to construct
$\G_i$ from $\G_{\imo}$. By Theorem~\ref{thm:axiom4}, we can apply
$\varphi_i, \psi_i$ and $\gamma_i$ until axiom $4$ holds for the
$(\ipo,N)$-restriction while preserving axioms $1,2$ and $5$. If they
also maintain local Schur positivity (which implies axiom $3$ for the
$(n,n)$-restriction), then we may apply $\varphi_{\ipo}$,
$\psi_{\ipo}$, and $\gamma_{\ipo}$ until axiom $4$ holds for the
$(\ipt,N)$-restriction and then apply $\varphi_{\ipt}$, $\psi_{\ipt}$,
and $\gamma_{\ipt}$ until axiom $4$ holds for the
$(\iph,N)$-restriction. At this point, by Theorem~\ref{thm:axiom6},
we may apply $\theta_i$ together with $\varphi_{\ipo}$ and
$\psi_{\ipt}$ as needed until the $(\ipo,N)$-restriction satisfies
axiom $6$, all while maintaining axioms $1,2$ and $5$ as well as
axioms $3$ and $4$ for the $(\iph,N)$-restriction.

Therefore, assuming that $\varphi_i, \psi_i$ and $\gamma_i$ maintain
local Schur positivity, we may set $\G_i$ to be the resulting graph,
and the induction may proceed. However, this assumption is not
necessarily true for edges higher than $E_i$. Therefore before this
algorithm can be put into effect, we require a more careful analysis
of these maps and how they affect local Schur positivity.

The following two conditions are essential properties for ensuring
that there is a way to maintain $\LSP_4$ for edges greater than $i$.

\begin{definition}
  A signed, colored graph $\G$ satisfying dual equivalence graph
  axioms $1,2,3$ and $5$ \emph{satisfies axiom $4'$} if the following
  conditions hold:
  \begin{itemize}
  \item[(ax$4'a$)] if $w \in W_i(\G)$ has a non-flat $\imo$-edge, then
    the components of $E_{\imt} \cup E_{\imo}$ and $E_{\imo} \cup E_i$
    containing $w$ have the same quasisymmetric functions in their
    degree $4$ generating functions;
  \item[(ax$4'b$)] if $x \in C_i(\G)$ has $\ipo$-type W and
    $E_{\imt}(x)$ does not, then $\left(E_{\imt}E_i\right)^m(x)$ 
    has $\ipo$-type W for all $m \geq 1$ for which
    $\left(E_{\imt}E_i\right)^{m-1}(x)$ admits an $i$-neighbor.
  \end{itemize}
  \label{defn:axiom4p}
\end{definition}

\begin{figure}[ht]
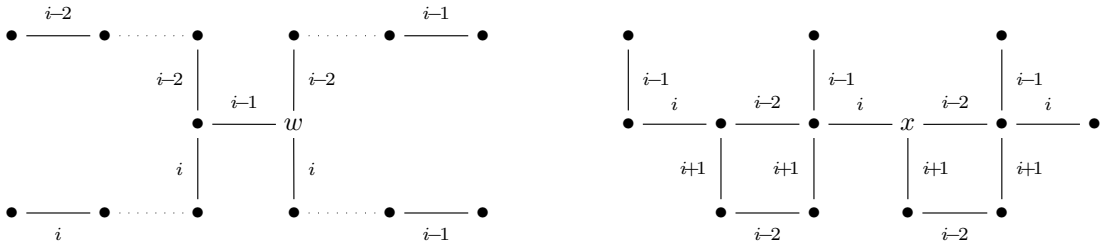

  \begin{displaymath}
    \begin{array}{\cs{7}\cs{7}\cs{7}\cs{7}\cs{7}c}
      \rnode{A0}{\B} & \rnode{B0}{\B} & \rnode{C0}{\B} & \rnode{D0}{\B} & \rnode{E0}{\B} & \rnode{F0}{\B} \\[5ex]
      & & \rnode{C1}{\B} & \rnode{D1}{w} & & \\[5ex]      
      \rnode{A2}{\B} & \rnode{B2}{\B} & \rnode{C2}{\B} & \rnode{D2}{\B} & \rnode{E2}{\B} & \rnode{F2}{\B}
    \end{array}
    \hspace{4em}
    \begin{array}{\cs{7}\cs{7}\cs{7}\cs{7}\cs{7}c}
      \rnode{a0}{\B} & \rnode{b0}{  } & \rnode{c0}{\B} &
      \rnode{d0}{  } & \rnode{e0}{\B} & \rnode{f0}{  } \\[5ex]
      \rnode{a1}{\B} & \rnode{b1}{\B} & \rnode{c1}{\B} &
      \rnode{d1}{x}  & \rnode{e1}{\B} & \rnode{f1}{\B} \\[5ex]
      \rnode{a2}{  } & \rnode{b2}{\B} & \rnode{c2}{\B} &
      \rnode{d2}{\B} & \rnode{e2}{\B} & \rnode{f2}{  }
    \end{array}
    \psset{nodesep=3pt,linewidth=.1ex}
    \everypsbox{\scriptstyle}
    \ncline {A0}{B0} \naput{\imt}
    \ncline[linestyle=dotted] {B0}{C0} 
    \ncline[linestyle=dotted] {D0}{E0}
    \ncline {E0}{F0} \naput{\imo}
    \ncline {C0}{C1} \nbput{\imt}
    \ncline {D0}{D1} \naput{\imt}
    \ncline {C1}{D1} \naput{\imo}
    \ncline {C1}{C2} \nbput{i}
    \ncline {D1}{D2} \naput{i}
    \ncline {A2}{B2} \nbput{i}
    \ncline[linestyle=dotted] {B2}{C2} 
    \ncline[linestyle=dotted] {D2}{E2}
    \ncline {E2}{F2} \nbput{\imo}
    \ncline {a1}{b1} \naput{i}
    \ncline {b1}{c1} \naput{\imt}
    \ncline {c1}{d1} \naput{i}
    \ncline {d1}{e1} \naput{\imt}
    \ncline {e1}{f1} \naput{i}
    \ncline {a1}{a0} \nbput{\imo}
    \ncline {c1}{c0} \nbput{\imo}
    \ncline {e1}{e0} \nbput{\imo}
    \ncline {b1}{b2} \nbput{\ipo}
    \ncline {c1}{c2} \nbput{\ipo}
    \ncline {d1}{d2} \naput{\ipo}
    \ncline {e1}{e2} \naput{\ipo}
    \ncline {b2}{c2} \nbput{\imt}
    \ncline {d2}{e2} \nbput{\imt}
  \end{displaymath}
  \caption{\label{fig:axiom4p} The cases forbidden by axiom $4'a$
    (left) and $4'b$ (right).}
\end{figure}

The hypotheses of axiom $4'a$ ensure that both $w$ and $E_{\imo}(w)$
admit an $\imt$, an $\imo$ and an $i$-neighbor. If $E_{\imo} \cup E_i$
forms a closed loop through $w$, then each edge toggles
$\sigma_{\imt}$. By axiom $2$, $E_i$ preserves $\sigma_{\imh}$. Since
$w$ admits an $\imt$-neighbor, $E_i(w)$ does not. By axiom $2$,
$E_{\imo}E_i(w)$ therefore admits an $\imt$-neighbor and the cycle
continues so that $(E_{\imo}E_i)^m (w)$ admits an $\imt$-neighbor and
$E_i (E_{\imo}E_i)^m (w)$ does not for all $m > 0$. By the assumption
that the component is a loop, we must have $w = (E_{\imo}E_i)^m (w)$
for some $m > 0$, so then $E_i(w) = E_i (E_{\imo}E_i)^m (w)$ does not
admit an $\imt$-neighbor. This contradiction works for $E_{\imt} \cup
E_{\imo}$ as well, therefore neither can be a loop. This leaves two
ways to align the two-color strings sharing an $E_{\imo}$ edge. One
way results in the same degree $4$ signatures while the other is given
on the left side of Figure~\ref{fig:axiom4p}. Note that applying
$\varphi_{\imo}^w$ in this case breaks $\LSP_4$, if it held for the
graph, for $E_{\imo} \cup E_{i}$ and $\varphi_{i}^w$ cannot be applied
since the $\imo$-edge at $w$ is not flat. Moreover, this is the only
case where both maps fail.

The hypotheses of axiom $4'b$ ensure that both $x$ and $E_i(x)$ admit
an $\ipo$-neighbor (though $E_{\ipo}$ edges need not exist). By axioms
$2$ and $1$, both $E_{\imt}(x)$ and $E_{\imt}(E_i(x))$ must admit an
$\ipo$-neighbor (again, $E_{\ipo}$ edges need not exist, though if
they do, axiom $5$ ensures the shown commutativity). The forbidden
conclusion is that neither $E_iE_{\imt}(x)$ nor $E_iE_{\imt}(E_i(x))$
admits an $\ipo$-neighbor as depicted on the right side of
Figure~\ref{fig:axiom4p}. Note that applying $\psi_{i}^x$ in this case
breaks $\LSP_4$, if it held or the graph, for $E_{i} \cup E_{\ipo}$
and, even if $E_{\ipo}$ edges do not exist, fails axiom $3$. Assuming
$E_{\ipo}$ edges exist, applying $\varphi_{\ipo}^x$ breaks $\LSP_4$,
if it held for the graph, for $E_{i} \cup E_{\ipo}$ across the
$E_{\imt}$ edges, which are part of the $\ipo$-package of $w$. Again,
this is the only case where both maps fail. See
Appendix~\ref{app:bad-graphs} for examples of locally Schur positive
graphs satisfying axioms $1,2,3,5$ and exactly half of $4'$, neither
of which has a Schur positive generating function.

\begin{definition}
  A \emph{D graph} is a locally Schur positive graph satisfying dual
  equivalence axioms $1,2,3$ and $5$ as well as axiom $4'$. 
\label{def:Dgraph}
\end{definition}

Since axiom $4$ implies axiom $4'$ (vacuously) and local Schur
positivity, D graphs are a generalization of dual equivalence
graphs. The following result show that the $\G^{(k)}_{c,D}$ are
further examples of D graphs. Two non-examples are given in
Appendix~\ref{app:bad-graphs}.

\begin{proposition}
  For any content vector $c$ and $k$-descent set $D$, $\G^{(k)}_{c,D}$
  is a D graph.
  \label{prop:ax4p}
\end{proposition}

\begin{proof}
  By Proposition~\ref{prop:ax1235}, we need only show that axiom $4'$
  holds. Note that the conditions of axiom $4'$ are local;
  specifically they need only be tested for connected components of
  $E_{\imt} \cup E_{\imo} \cup E_{i}$ with signatures
  $\sigma_{\imh,\ldots,\ipo}$.  Recall from the proof of
  Proposition~\ref{prop:ax1235} that there are $25$ non-isomorphic
  connected components of $(V,\sigma,E_{\imt} \cup E_{\imo} \cup E_i)$
  in $\G^{(k)}_{c,D}$, of which 18 are not dual equivalence
  graphs. Axiom $4'a$ can be directly checked on these cases. For
  axiom $4'b$, we must first add all possible signatures
  $\sigma_{\ipo}$. To do this, we may construct the $14$ graph
  structures on permutations of length $5$, as discussed in the proof
  of Proposition~\ref{prop:ax1235}, and then check each $6$ times
  corresponding to the six possible locations to insert a $6$ in the
  permutations, thereby producing the possible signatures. Again,
  these computations can be done by hand or by computer (see
  Appendix~\ref{app:code}).
\end{proof}

Note that we now have the additional task of ensuring that all of the
maps maintain axiom $4'$. 

\begin{lemma}
  Let $\G$ be a D graph of type $(n,N)$ such that the
  $(i,N)$-restriction is a dual equivalence graph and the
  $(\iph,N)$-restriction satisfies dual equivalence axiom $4$. For
  $\C$ a negatively dominant restricted component, $\theta_i^{\C}(\G)$
  vacuously satisfies axiom $4'$, and so is a D graph.
  \label{lem:theta-4p}
\end{lemma}

\begin{proof}
  By Theorem~\ref{thm:theta-LSP}, it suffices to show that axiom $4'$
  is maintained.  Since axiom $4$ holds for the $(\iph,N)$-restriction
  of $\G$, even after applying $\theta_i^{\C}$, axiom $4'a$ is
  vacuously satisfied since only $E_i \cup E_{\ipo}$ strings have the
  potential not to appear in Figure~\ref{fig:lambda4}. Similarly for
  axiom $4'b$, the result for $E_{\imt} \cup E_i$ vacuously follows
  from Proposition~\ref{prop:theta-reasonable}. For $E_i \cup
  E_{\ipt}$, as before let $\{w,x\}, \{u,v\} \in E_i(\G)$ with
  $\theta_i^{\C}(w) = u$ and $\theta_i^{\C}(x) = v$.  Note that if
  $w,x$ have $\ipt$-type C, then any violation of axiom $4'b$ that
  occurs in $\theta_i^{\C}(\G)$ must also have existed in $\G$, which
  is not possible. Alternatively, if none of $u,w,x,v$ has $\ipt$-type
  C, then axiom $4'b$ is vacuously satisfied.
\end{proof}

\begin{lemma}
  Let $\G$ be a D graph such that the $(\imt,N)$-restriction of $\G$
  is a dual equivalence graph. For any $w \in W_i^0(\G)$ or $x \in
  C_i^0(\G)$, if $\varphi_i^w(\G)$ or $\psi_i^x(\G)$ or
  $\gamma_i^z(\G)$ has $\LSP_4$, then $\varphi_i^w(\G)$ or
  $\psi_i^x(\G)$ or $\gamma_i^z(\G)$ maintains axiom $4'$,
  respectively.
\label{lem:axiom4p}
\end{lemma}

\begin{proof}
  For each map, there are three cases to consider for axiom $4'a$:
  $E_{\imt} \cup E_{\imo} \cup E_i$, $E_{\imo} \cup E_i \cup
  E_{\ipo}$, and $E_i \cup E_{\ipo} \cup E_{\ipt}$. For the outer two
  cases, the middle edge of the left side of Figure~\ref{fig:axiom4p}
  connecting to $w$ remains unchanged along with all edges to the
  right or left of it. Thus, since $\G$ is assumed to be locally Schur
  positive, the violation must have existed in $\G$ as well,
  contradicting the assertion that $\G$ is a D graph. For $E_{\imo}
  \cup E_i \cup E_{\ipo}$, the same argument applies if the map
  creates any but the middle $i$-edge. If the middle $i$-edge is newly
  created, then by local Schur positivity of $\G$, $w$ must still have
  $i$-type W and a nonflat $i$-edge, thus ensuring a violation of
  axiom $4'a$ already existed in $\G$. Therefore none of the maps can
  create a new violation of axiom $4'a$.

  There are two cases to consider for axiom $4'b$: $E_{\imt} \cup E_i$
  and $E_i \cup E_{\ipt}$. Similar to the analysis for axiom $4'a$,
  the latter case is easily resolved by axiom $2$ since all edges
  involved in an application of $\varphi_i^w$ or $\psi_i^x$ or
  $\gamma_i^z$ preserve $\sigma_{\ipt,\iph}$. Thus any violation after
  applying either map must have already existed in $\G$. Any
  violation of axiom $4'b$ for $E_{\imt} \cup E_i$ created by
  $\psi_i^x$ or $\gamma_i^z$ must have existed already in $\G$, so we
  consider how $\varphi_i^w$ might result in a component as depicted
  in the right side of Figure~\ref{fig:axiom4p}. There are three
  $i$-edges that could have resulted from $\varphi_i^w$.

  \begin{figure}[ht]
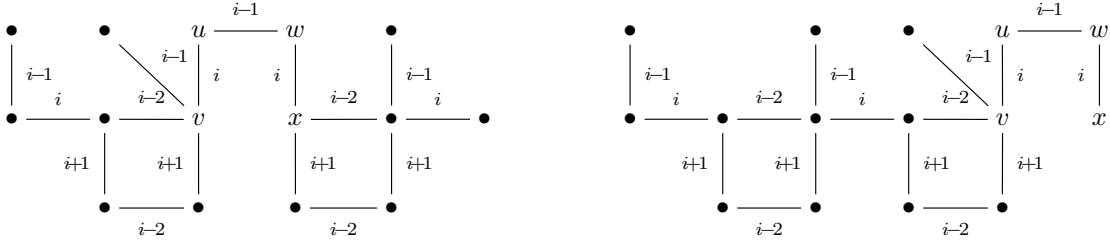

    \begin{displaymath}
      \begin{array}{\cs{7}\cs{7}\cs{7}\cs{7}\cs{7}c}
        \rnode{a0}{\B} & \rnode{b0}{\B} & \rnode{c0}{u}  &
        \rnode{d0}{w}  & \rnode{e0}{\B} & \rnode{f0}{  } \\[5ex]
        \rnode{a1}{\B} & \rnode{b1}{\B} & \rnode{c1}{v}  &
        \rnode{d1}{x}  & \rnode{e1}{\B} & \rnode{f1}{\B} \\[5ex]
        \rnode{a2}{  } & \rnode{b2}{\B} & \rnode{c2}{\B} &
        \rnode{d2}{\B} & \rnode{e2}{\B} & \rnode{f2}{  }
      \end{array}
      \hspace{4em}
      \begin{array}{\cs{7}\cs{7}\cs{7}\cs{7}\cs{7}c}
        \rnode{A0}{\B} & \rnode{B0}{  } & \rnode{C0}{\B} &
        \rnode{D0}{\B} & \rnode{E0}{u}  & \rnode{F0}{w}  \\[5ex]
        \rnode{A1}{\B} & \rnode{B1}{\B} & \rnode{C1}{\B} &
        \rnode{D1}{\B} & \rnode{E1}{v}  & \rnode{F1}{x}  \\[5ex]
        \rnode{A2}{  } & \rnode{B2}{\B} & \rnode{C2}{\B} &
        \rnode{D2}{\B} & \rnode{E2}{\B} & \rnode{F2}{  }
      \end{array}
      \psset{nodesep=3pt,linewidth=.1ex}
      \everypsbox{\scriptstyle}
      \ncline {a1}{b1} \naput{i}
      \ncline {b1}{c1} \naput{\imt}
      \ncline {c1}{c0} \nbput{i}
      \ncline {c0}{d0} \naput{\imo}
      \ncline {d0}{d1} \nbput{i}
      \ncline {d1}{e1} \naput{\imt}
      \ncline {e1}{f1} \naput{i}
      \ncline {a1}{a0} \nbput{\imo}
      \ncline {c1}{b0} \nbput{\imo}
      \ncline {e1}{e0} \nbput{\imo}
      \ncline {b1}{b2} \nbput{\ipo}
      \ncline {c1}{c2} \nbput{\ipo}
      \ncline {d1}{d2} \naput{\ipo}
      \ncline {e1}{e2} \naput{\ipo}
      \ncline {b2}{c2} \nbput{\imt}
      \ncline {d2}{e2} \nbput{\imt}
      \ncline {A1}{B1} \naput{i}
      \ncline {B1}{C1} \naput{\imt}
      \ncline {C1}{D1} \naput{i}
      \ncline {D1}{E1} \naput{\imt}
      \ncline {E1}{E0} \nbput{i}
      \ncline {E0}{F0} \naput{\imo}
      \ncline {F0}{F1} \nbput{i}
      \ncline {A1}{A0} \nbput{\imo}
      \ncline {C1}{C0} \nbput{\imo}
      \ncline {E1}{D0} \nbput{\imo}
      \ncline {B1}{B2} \nbput{\ipo}
      \ncline {C1}{C2} \nbput{\ipo}
      \ncline {D1}{D2} \naput{\ipo}
      \ncline {E1}{E2} \naput{\ipo}
      \ncline {B2}{C2} \nbput{\imt}
      \ncline {D2}{E2} \nbput{\imt}
    \end{displaymath}
    \caption{\label{fig:phi-ax4p} An illustration of how $\varphi_i^w$
      might result in a violation of axiom $4'b$.}
  \end{figure}

  For the middle edge, let $x,w,u,v$ be as depicted in
  Figure~\ref{fig:phi-ax4p}. By axiom $2$ and the fact that the
  $i$-edge between $u$ and $v$ is not flat, $u$ has no $\imt$-neighbor
  since $v$ does. By axiom $2$ again, $w$ must admit an
  $\imt$-neighbor. Since $x$ also admits an $\imt$-neighbor, the
  $i$-edge between $w$ and $x$ must be flat by axiom $3$. Now for
  $\ipo$-neighbors, by axiom $2$, $\sigma(u)_{\ipo} =
  \sigma(w)_{\ipo}$ and, by axiom $1$, $\sigma(u)_{i} =
  -\sigma(w)_{i}$. Therefore exactly one of $u$ and $w$ admits an
  $\ipo$-neighbor, so we consider each case in turn. If $w$ admits an
  $\ipo$-neighbor, then the $i$-edge between $w$ and $x$ is the middle
  edge in a violation of axiom $4'b$ for $E_{\imt} \cup E_i$ in $\G$,
  a contradiction. Alternatively, if $u$ admits an $\ipo$-neighbor,
  then by axiom $4'a$ for $E_{\imo} \cup E_i \cup E_{\ipo}$ in $\G$,
  following the $E_i \cup E_{\ipo}$ string from $u$ through $v$ and
  onwards must terminate in an $i$-edge. In particular, $E_{\ipo}(v)$
  admits an $i$-neighbor, say $z = E_i E_{\ipo}(v)$. Axioms $2$ and
  $3$ ensure that both $v$ and $z$ admit an $\imo$-neighbor while
  $E_{\ipo}(v) = E_i(z)$ does not. Moreover, all of $v, E_{\ipo}(v) =
  E_i(z)$ and $z$ admit an $\imt$-neighbor. In particular, the
  $i$-edge between $E_i(z)$ and $z$ is flat, and, since $E_i(z)$ does
  not admit an $\imo$-neighbor, neither $E_i(z)$ nor $E_{\imt}E_i(z)$
  has $\imo$-type W. Therefore by axioms $2$ and $3$, $E_{\imt}E_i(z)$
  admits an $i$-neighbor. By axiom $5$, $E_{\imt}E_{\ipo}(v) =
  E_{\ipo}E_{\imt}(v)$, and the assumption on $\G$ is that $E_i
  E_{\imt}(v)$ does not admit an $\ipo$-neighbor. Therefore by local
  Schur positivity of $E_i \cup E_{\ipo}$ in $\G$, the $E_i \cup
  E_{\ipo}$ string beginning at $E_i E_{\imt}(v)$ ends with an
  $\ipo$-edge. In particular, this implies $E_i E_{\imt} E_i(z)$ admit
  an $\ipo$-neighbor. At long last, this creates a violation of axiom
  $4'b$ in $\G$, another contradiction. 

  For the right edge, again let $x,w,u,v$ be as depicted. By
  assumption $x$ does not admit an $\ipo$-neighbor, so by axiom $2$,
  $w$ must. By axioms $1$ and $2$, this ensures that $u$ does not
  admit an $\ipo$-neighbor. Therefore the $i$-edge between $v$ and $u$
  is the right edge in a violation of axiom $4'b$ in $\G$. The case
  for left edge is similarly resolved.
\end{proof}

With axiom $4'$ resolved, we now aim to understand when $\varphi_i$ or
$\psi_i$ breaks local Schur positivity.

Note that the proof of Lemma~\ref{lem:LSP4} for $\psi_i^x$ did not use
the stronger hypothesis that the $(i,\ipo)$-restriction of the graph
satisfied axiom $4$, only that it has $\LSP_4$. However, the proof for
$\varphi_i^w$ relied strongly on the stronger hypothesis. While
$\varphi_i^w$ will not always preserve $\LSP_4$ in this more general
setting, there is always a choice of $w$ such that it does.

\begin{lemma}
  Let $\G$ be a D graph such that the $(\imt,N)$-restriction of $\G$
  is a dual equivalence graph. For any $w \in W_i^0(\G)$, if some
  connected component of $E_{\imo} \cup E_{i}$ in $\varphi_i^w(\G)$ is
  not locally Schur positive, then $z \in W_i(\G)$ for $z =
  E_{\imh}(w)$ or $E_{\imh}E_{\imo}(w)$ and connected components of
  $E_{\imo} \cup E_{i}$ are locally Schur positive in both
  $\varphi_i^z(\G)$ and $\varphi_i^w(\varphi_i^z(\G))$.
  \label{lem:phi-2-down}
\end{lemma}

\begin{proof}
  For vertices $v$ such that neither $v$ nor $\varphi_{i}(v)$ lies on
  the $i$-package of $w$ or $E_i(w)$, the result follows from the
  hypotheses on $\G$. By the symmetry between $w$ and $E_{\imo}(w)$,
  we consider only those vertices on the $i$-packages of $w$ and
  $x = E_i(w)$.

  The result is immediate for $w$ and $u = \varphi_i^w(w)$, since the
  degree $4$ generating function is $s_{(2,2)}$, and so the result
  also follows for $x$ and $v = \varphi_i^w(x)$ since $\G$ is locally
  Schur positive. By axiom $2$, both $\sigma_{\imt}$ and
  $\sigma_{\imo}$ are constant on $E_2 \cup \cdots \cup E_{\imf} \cup
  E_{\iph} \cup \cdots \cup E_{\nmo}$, and by axiom $5$, those edges
  all commute with $E_{\imo}$ and $E_i$. Therefore we need only show
  the result across $E_{\imh}$ edges. By dual equivalence axiom $6$,
  we can reach any vertex on the $i$-package of a given vertex by
  crossing at most one $\imh$-edge, so we need only prove the result
  for vertices on the connected component of $E_{\imh} \cup E_{\imo}
  \cup E_i$ containing $w$.

  \begin{figure}[ht]
    \begin{displaymath}
      \begin{array}{\cs{8}\cs{8}\cs{8}\cs{8}\cs{8}c}
        & & & \rnode{uu}{\B} & & \\[3ex]
        \rnode{l}{} & \rnode{x}{x} & \rnode{w}{w} & \rnode{u}{u} &
        \rnode{v}{v} & \rnode{r}{} \\[4ex]
        & \rnode{X}{\B} & \rnode{W}{\B} & \rnode{U}{\B} & \rnode{V}{\B}
        & \\[3ex]
        & & \rnode{WW}{\B} & & &
      \end{array}
      \psset{nodesep=3pt,linewidth=.1ex}
      \everypsbox{\scriptstyle}
      \ncline {w}{uu} \naput{\imt}
      \ncline[linestyle=dotted] {l}{x} \naput{\imo}
      \ncline {x}{w} \naput{i}
      \ncline {w}{u} \naput{\imo}
      \ncline {u}{v} \naput{i}
      \ncline[linestyle=dotted] {v}{r} \naput{\imo}
      \ncline {x}{X} \nbput{\imh}
      \ncline {w}{W} \nbput{\imh}
      \ncline {u}{U} \naput{\imh}
      \ncline {v}{V} \naput{\imh}
      \ncline {X}{W} \nbput{i}
      \ncline {U}{V} \nbput{i}
      \ncline {WW}{U} \nbput{\imo}
    \end{displaymath}
    \caption{\label{fig:phi-3-down} Forced $E_{\imo}$ edges on a
      component of $W_i(\G) \cap (E_{\imh} \cup E_{\imo} \cup E_i)$.}
  \end{figure}
  
  By Lemma~\ref{lem:phi-compatible}, the $i$-packages of $x,w,u$ and
  $v$ are all isomorphic, and so one of $x,w,u,v$ admits an 
  $\imh$-neighbor if and only if they all do. Assume they all admit an
  $\imh$-neighbor. By axiom $5$, $E_{\imh}(x) = E_iE_{\imh}(w)$ and
  $E_{\imh}(v) = E_iE_{\imh}(u)$, as shown in
  Figure~\ref{fig:phi-3-down}. Since $\sigma_i(w)_{\imh} =
  \sigma_i(u)_{\imh}$, by axioms $1$ and $2$, exactly one of $w$ and
  $u$ admits an $\imt$-neighbor. Since both $w,u \in W_i(\G)$, we may
  assume $w$ admits an $\imt$-neighbor and $u$ does not. By axiom $1$,
  this means $\sigma(u)_{\imh} = \sigma(u)_{\imt}$, and so by axiom
  $3$ for $\G$, $\sigma(u)_{\imt} = \sigma(E_{\imh}(u))_{\imt}$. By
  axiom $2$, $\sigma(u)_{\imo} = \sigma(E_{\imh}(u))_{\imo}$, and so
  $E_{\imh}(u)$ must also admit an $\imo$-neighbor. Therefore we have
  the situation depicted in Figure~\ref{fig:phi-3-down}.
  
  We claim that if neither $E_{\imh}(w)$ nor $E_{\imh}(u)$ lies in
  $W_i(\G)$, then connected components of $E_{\imo} \cup E_i$ in
  $\varphi_i^w(\G)$ are locally Schur positive. If either
  $E_{\imo}(E_{\imh}(w)) = E_{\imh}(x)$ or $E_{\imo}(E_{\imh}(u)) =
  E_{\imh}(v)$, then applying $\varphi_i^w$ results in all four of
  $E_{\imh}(x), E_{\imh}(w), E_{\imh}(u)$ and $E_{\imh}(v)$ lying on
  the same connected component of $E_{\imo} \cup E_i$, making the
  component the union of two locally Schur positive components. This
  proves the latter assertion of the proposition: once
  $\varphi_i^{E_{\imh}(w)}$ and $\varphi_i^{E_{\imh}(u)}$ have been
  applied, $\varphi_i^w$ preserves local Schur positivity.
 
  Note that $E_{\imh}(w)$ admits an $\imo$-neighbor if and only if
  $E_{\imh}(w)$ has $\imo$-type C and lies in $C_{\imo}(\G)$. In this
  case, $E_{\imh}(w)$ must have $i$-type W by axiom $4'b$. Then local
  Schur positivity would fail only in the situation depicted on the
  left in Figure~\ref{fig:phi-bad-2}. In the alternative case,
  $E_{\imh}(w)$ does not admit an $\imo$-neighbor, so by axiom $2$,
  $E_{\imh}(x)$ does. Since $u$ has no $\imt$-neighbor, axiom $3$
  ensures that $\sigma(u)_{\imt} = \sigma(E_{\imh}(u))_{\imt}$, so
  $E_{\imh}(w)$ admits an $\imo$-neighbor. If $E_{\imh}(u)$ does not
  have $i$-type W, then local Schur positivity is preserved. In fact,
  local Schur positivity would fail only in the situation depicted on
  the right in Figure~\ref{fig:phi-bad-2}. The claim having been
  proved, we now address these two cases.

  \begin{figure}[ht]
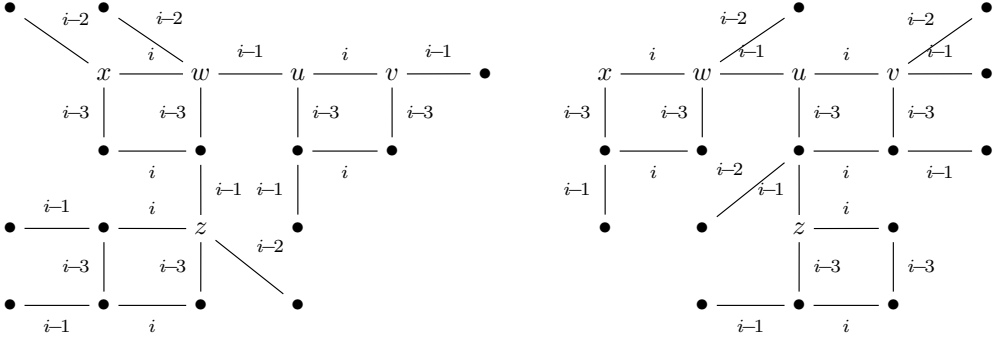

    \begin{displaymath}
      \begin{array}{\cs{7}\cs{7}\cs{7}\cs{7}\cs{7}c}
        \rnode{ww}{\B} & \rnode{uu}{\B} & & \\[3ex]
        & \rnode{x}{x} & \rnode{w}{w} & \rnode{u}{u} & \rnode{v}{v} & \rnode{r}{\B} \\[4ex]
        & \rnode{X}{\B} & \rnode{W}{\B} & \rnode{U}{\B} & \rnode{V}{\B} & \\[4ex]
        \rnode{LL}{\B} & \rnode{XX}{\B} & \rnode{WW}{z} & \rnode{UU}{\B} & & \\[4ex]
        \rnode{LLL}{\B} & \rnode{XXX}{\B} & \rnode{WWW}{\B} & \rnode{UUU}{\B} & & 
      \end{array}
      \hspace{3em}
      \begin{array}{\cs{7}\cs{7}\cs{7}\cs{7}c}
        & & \rnode{ww2}{\B} & & \rnode{uu2}{\B} \\[3ex]
        \rnode{x2}{x} & \rnode{w2}{w} & \rnode{u2}{u} & \rnode{v2}{v} & \rnode{r2}{\B} \\[4ex]
        \rnode{X2}{\B} & \rnode{W2}{\B} & \rnode{U2}{\B} & \rnode{V2}{\B} & \rnode{R2}{\B} \\[4ex]
        \rnode{XX2}{\B} & \rnode{WW2}{\B} & \rnode{UU2}{z} & \rnode{VV2}{\B} & \\[4ex]
        & \rnode{WWW2}{\B} & \rnode{UUU2}{\B} & \rnode{VVV2}{\B} & 
      \end{array}
      \psset{nodesep=3pt,linewidth=.1ex}
      \everypsbox{\scriptstyle}
      \ncline {x}{ww} \nbput{\imt}
      \ncline {w}{uu} \nbput{\imt}
      \ncline {x}{w} \naput{i}
      \ncline {w}{u} \naput{\imo}
      \ncline {u}{v} \naput{i}
      \ncline {v}{r} \naput{\imo}
      \ncline {x}{X} \nbput{\imh}
      \ncline {w}{W} \nbput{\imh}
      \ncline {u}{U} \naput{\imh}
      \ncline {v}{V} \naput{\imh}
      \ncline {X}{W} \nbput{i}
      \ncline {U}{V} \nbput{i}
      \ncline {W}{WW} \naput{\imo}
      \ncline {U}{UU} \nbput{\imo}
      \ncline {LL}{XX} \naput{\imo}
      \ncline {XX}{WW} \naput{i}
      \ncline {XX}{XXX} \nbput{\imh}
      \ncline {WW}{WWW} \nbput{\imh}
      \ncline {WW}{UUU} \naput{\imt}
      \ncline {LLL}{XXX} \nbput{\imo}
      \ncline {XXX}{WWW} \nbput{i}
      \ncline {w2}{ww2} \naput{\imt}
      \ncline {v2}{uu2} \naput{\imt}
      \ncline {x2}{w2} \naput{i}
      \ncline {w2}{u2} \naput{\imo}
      \ncline {u2}{v2} \naput{i}
      \ncline {v2}{r2} \naput{\imo}
      \ncline {x2}{X2} \nbput{\imh}
      \ncline {w2}{W2} \nbput{\imh}
      \ncline {u2}{U2} \naput{\imh}
      \ncline {v2}{V2} \naput{\imh}
      \ncline {X2}{W2} \nbput{i}
      \ncline {U2}{V2} \nbput{i}
      \ncline {V2}{R2} \nbput{\imo}
      \ncline {U2}{UU2} \nbput{\imo}
      \ncline {U2}{WW2} \nbput{\imt}
      \ncline {X2}{XX2} \nbput{\imo}
      \ncline {UU2}{VV2} \naput{i}
      \ncline {UU2}{UUU2} \naput{\imh}
      \ncline {VV2}{VVV2} \naput{\imh}
      \ncline {WWW2}{UUU2} \nbput{\imo}
      \ncline {UUU2}{VVV2} \nbput{i}
    \end{displaymath}
    \caption{\label{fig:phi-bad-2} The two possible scenarios where
      $\varphi_i^w$ breaks the local Schur positivity of $E_{\imo} \cup
      E_i$. The resolution is to apply $\varphi_i^{z}$ first.}
  \end{figure}
  
  For the left hand side, we assume that $E_{\imh}(x)$ does not admit an
  $\imo$-neighbor, and so, by axiom $3$, $E_{\imh}(w)$ must. By earlier
  remarks, this implies that $E_{\imh}(w)$ must have $i$-type W, and by
  the local Schur positivity of $\G$, the connected component of
  $E_{\imo} \cup E_i$ containing $E_{\imh}(x)$ and $E_{\imh}(w)$ must be
  as depicted. Since $w$ admits an $\imt$-neighbor and has $\imo$-type
  C, $E_{\imh}(w)$ does not admit an $\imt$-neighbor. By axiom $3$ and
  the fact that $E_{\imh}(x)$ does not admit an $\imo$-neighbor,
  $\sigma(E_{\imh}(w))_{\imt} = \sigma(E_{\imh}(x))_{\imt}$, and so, by
  axiom $1$, $E_{\imh}(x)$ must also not admit an $\imt$-neighbor. By
  axiom $3$, this means $x$ must admit an $\imt$-neighbor. Now since
  both $x$ and $w$ admit $\imt$-neighbors, by axiom $1$, we have
  $\sigma(w)_{\imt} = \sigma(x)_{\imt}$, and so $x$ does not admit an
  $\imo$-neighbor. Moving down the diagram, since $E_{\imh}(w)$ does not
  admit an $\imt$-neighbor, $z$ must admit an $\imt$-neighbor and an
  $\imh$-neighbor by axiom $3$. By axiom $2$, $E_i(z)$ must also admit
  an $\imt$-neighbor, and by axiom $5$, $E_{\imh}E_i(z) =
  E_iE_{\imh}(z)$. Since both $z$ and $E_i(z)$ admit an $\imo$-neighbor,
  $E_i(z)$ cannot admit an $\imt$-neighbor. Therefore axiom $3$ ensures
  that since $E_i(z)$ admits an $\imo$-neighbor, so does
  $E_{\imh}E_i(z)$. Finally, if $E_{\imh}(z)$ admits an $\imo$-neighbor,
  then both $z$ and $E_{\imh}(z)$ must have $\imo$-type C, and so by
  axiom $4'b$, $E_{\imh}(z)$ must have $i$-type W. Therefore if
  $E_{\imh}(z)$ admits an $\imo$-neighbor, then $E_{\imo}E_{\imh}(z)$
  admits an $i$-neighbor. Whether this is the case or not, applying
  $\varphi_i^{E_{\imh}(w)} = \varphi_i^z$ is seen to preserve local
  Schur positivity across the $E_{\imh}$ edges, thereby resolving this
  case. A similar analysis and diagram chase resolves the righthand side.
\end{proof}

We now conclude that $\varphi^w_i$, $\psi^x_i$ and $\gamma^z_i$ all
preserve the local Schur positivity of connected components of
$E_{\imo} \cup E_i$. However, this is not enough to conclude that
$\LSP_4$ is preserved since connected components of $E_i \cup
E_{\ipo}$ might be, and often are, disconnected in non-Schur positive
ways. Therefore we must resolve local Schur positivity for $E_i \cup
E_{\ipo}$, $E_{\imo} \cup E_i \cup E_{\ipo}$ and $E_i \cup E_{\ipo}
\cup E_{\ipt}$ when applying these maps.

We first address the case of $E_i \cup E_{\ipo}$. While neither
$\varphi^w_i$ nor $\psi^x_i$ necessarily maintains local Schur
positivity for this case, by first applying $\varphi^z_{\ipo}$ for a
cleverly selected $z \in W_{\ipo}^0(\G)$, connected components of $E_i
\cup E_{\ipo}$ do remain locally Schur positive.

\begin{lemma}
  Let $\G$ be a D graph and the $(\imo,N)$-restriction of $\G$ a
  dual equivalence graph. 
  \begin{enumerate}
  \item For any $w \in W_i^0(\G)$, if some connected component of
    $E_{i} \cup E_{\ipo}$ in $\varphi_i^w(\G)$ is not locally Schur
    positive, then $w \in W_{i+1}^0(\G)$ and connected components of
    $E_{i} \cup E_{\ipo}$ are locally Schur positive in both
    $\varphi_{i+1}^w(\G)$ and $\varphi_i^w(\varphi_{i+1}^w(\G))$.
  \item For any $x \in C_i^0(\G)$, if there is a connected component
    of $E_{i} \cup E_{\ipo}$ in $\psi_i^x(\G)$ that is not locally
    Schur positive, then for $z = E_{\imt}(x)$ or $E_{\imt}E_{i}(x)$,
    we have $z \in W_{\ipo}^0(\G)$ and connected components of $E_{i}
    \cup E_{\ipo}$ are locally Schur positive in both
    $\varphi_{\ipo}^z(\G)$ and $\psi_i^x(\varphi_{\ipo}^z(\G))$.
  \end{enumerate}
  In particular, if $\varphi_i^w$ or $\psi_i^x$ breaks the local Schur
  positivity of $E_{i} \cup E_{\ipo}$, then it can be restored with
  $\varphi^u_{\ipo}$ for some $u \in W_{\ipo}^0(\G)$.
  \label{lem:phi-2-up}
\end{lemma}

\begin{proof}
  We begin with $\varphi_i^w$. We need only be concerned with
  vertices on the $i$-packages of $w$ and $E_i(w)$. Keeping the
  notation from before, let $x = E_i(w)$, $u = E_{\imo}(w)$, and $v =
  E_i(u) = E_i(E_{\imo}(w))$.

  By axioms $2$ and $5$, both $E_{\imo}$ and $E_i$ commute with $E_h$
  for $h \geq \iph$. Therefore any vertex connected component of
  $E_{\iph} \cup \cdots \cup E_{\nmo}$ containing $w$ or $u$ also lies
  in $W_i^0(\G)$. Similarly, both $E_{i}$ and $E_{\ipo}$ commute with
  $E_{h}$ for $h \leq \imh$, and $\sigma_{i}$ and $\sigma_{\ipo}$ are
  constant on $E_2 \cup \cdots \cup E_{\imh}$. Therefore if the result
  holds for some $x,w,u$ and $v$, then it holds for any vertex on the
  connected component of $E_{2} \cup \cdots \cup E_{\imh}$ containing
  those vertices. Therefore it suffices to prove the result for
  $x,w,u$ and $v$.

  Since $w \in W^0_i(\G)$, $\sigma(w)_{i} = -\sigma(u)_{i}$ and, by
  axiom 2, $\sigma(w)_{\ipo} = \sigma(u)_{\ipo}$, so exactly one of
  $w$ and $u = \varphi_i(w)$ admits an $\ipo$-neighbor. Since both $w,
  u \in W_i^0(\G)$, we may assume $w$ admits an $\ipo$-neighbor and
  $u$ does not. By axiom $3$ for $\G$, $v$ must admit an
  $\ipo$-neighbor since $u$ does not. We now have the situation
  depicted in Figure~\ref{fig:phi-3-up}.

  \begin{figure}[ht]
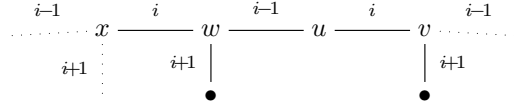

    \begin{displaymath}
      \begin{array}{\cs{8}\cs{8}\cs{8}\cs{8}\cs{8}c}
        \rnode{l}{} & \rnode{x}{x} & \rnode{w}{w} & \rnode{u}{u} &
        \rnode{v}{v} & \rnode{r}{} \\[3ex]
        & \rnode{X}{} & \rnode{W}{\B} & & \rnode{V}{\B} & 
      \end{array}
      \psset{nodesep=3pt,linewidth=.1ex}
      \everypsbox{\scriptstyle}
      \ncline[linestyle=dotted] {l}{x} \naput{\imo}
      \ncline {x}{w} \naput{i}
      \ncline {w}{u} \naput{\imo}
      \ncline {u}{v} \naput{i}
      \ncline[linestyle=dotted] {v}{r} \naput{\imo}
      \ncline[linestyle=dotted] {x}{X} \nbput{\ipo}
      \ncline {w}{W} \nbput{\ipo}
      \ncline {v}{V} \naput{\ipo}
    \end{displaymath}
    \caption{\label{fig:phi-3-up} Forced $E_{\ipo}$ edges on a
      component of $W_i(\G) \cap (E_{\imo} \cup E_i)$, with possible
      edges indicated by dotted lines.}
  \end{figure}
  
  If $x$ does not admit an $\ipo$-neighbor, then the connected
  components of $E_i \cup E_{\ipo}$ remain locally Schur positive in
  $\varphi_i^w(\G)$, so assume it does. In particular, this implies $w
  \in W_{\ipo}(\G)$. If $x$ admits an $\imo$-neighbor, then by axiom
  2, $E_{\imo}(x)$ does not admit an $\ipo$-neighbor. If $E_{\imo}(x)$
  admits an $i$-neighbor, then by axiom 2, $E_i E_{\imo}(x)$ admits an
  $\ipo$-neighbor. Continuing along an alternating chain of $E_{\imo}$
  followed by $E_i$ edges from $x$, every other vertex admits an
  $\ipo$-neighbor. Similarly, continuing along an alternating chain of
  $E_{\imo}$ followed by $E_{i}$ edges from $v$, every other vertex
  admits an $\ipo$-neighbor. If these two alternating chains form a
  closed loop, then $x = E_{\imo}(E_iE_{\imo})^m(v)$ does not admit an
  $\ipo$-neighbor. A similar contraction arises if the connected
  component of $E_i \cup E_{\ipo}$ containing $w$ and $x$ is a loop,
  so assume that neither is.

  By local Schur positivity, since $u$ does not admit an
  $\ipo$-neighbor, the $E_i \cup E_{\ipo}$ chain starting at $u$ must
  end with an $E_{\ipo}$ edge. If the $E_i \cup E_{\ipo}$ passing from
  $w$ through $x$ and continuing on ends in an $E_i$ edge, then local
  Schur positivity is maintained in $\varphi_i^w(\G)$. In the
  alternative case, axiom $4'a$ ensures that the $E_{\imo} \cup E_i$
  chain passing from $w$ through $x$ and continuing on must end in an
  $E_{i}$ edge. Moreover, since the vertices along this chain
  alternate in whether or not they admit an $\ipo$-neighbor, for any
  vertex $z$ along this chain, we may apply $\varphi_i^z(\G)$ while
  maintaining local Schur positivity for $E_i \cup
  E_{\ipo}$. Therefore we may assume that $x$ is the endpoint of this
  chain, i.e. $x$ does not admit an $\imo$-neighbor. This means the
  $i$-edge between $x$ and $w$ is flat, so $w \in W_{\ipo}^0(\G)$ and
  $\varphi_{\ipo}^w$ may be applied after which $E_{\ipo}(x) = E_i(x)
  = w$, so $\varphi_i^{w}$ will then preserve local Schur positivity
  of $E_i \cup E_{\ipo}$.

  Next we consider $\psi_i^x$. For vertices $v$ such that neither $v$
  nor $\psi_i^x(v)$ lies on the $i$-package of $E_{\imt}(x)$ or
  $E_iE_{\imt}(x)$, all results follow from the hypotheses on $\G$. To
  ease notation, let $u = (E_{\imo}E_{\imt})^m E_i(x)$, $w =
  E_{\imt}(x)$, and $v = E_{\imt}(u)$. By axiom $2$, both $E_{\imt}$
  and $E_{\imo}$ preserve $\sigma_{\ipo}$, so $\sigma(x)_{\ipo} =
  -\sigma(u)_{\ipo}$ if and only if $x$ has $\ipo$-type W. By axiom
  $4'b$, this implies that $w$ also has $\ipo$-type W, in which case
  both $w$ and $E_i(w)$ admit an $\ipo$-neighbor thereby ensuring
  axiom $3$ is preserved. To see the ways in which local Schur
  positivity may fail, we revisit Figure~\ref{fig:phi-bad-2}, and note
  that this figure is precisely an instance where $\varphi_{\imo}$ is
  applicable. The resolution therefore is to apply $\varphi_{\ipo}^v$
  and $\varphi_{\ipo}^w$ as needed before proceeding with
  $\psi_i^x$. This result extends along $E_{\iph} \cup \cdots \cup
  E_{\nmo}$ by the commutativity of $E_{\imt} \cup E_i$ ensured by
  axioms $2$ and $5$, and it extends along $E_{2} \cup \cdots \cup
  E_{\imh}$ by the commutativity of $E_{i} \cup E_{\ipo}$ ensured by
  axioms $2$ and $5$.
\end{proof}

Note that even when $\varphi_i^w$ or $\psi_i^x$ breaks local Schur
positivity, axiom $3$ is still maintained.

Next, we consider the cases necessary to establish degree $5$ local
Schur positivity. The key idea is that if there is an alternative path
between two vertices connected by an $i$-edge that gets deleted by the
map, then the component remains connected after applying the map.

\begin{lemma}
  Let $\G$ be a D graph and the $(\imo,N)$-restriction of $\G$
  a dual equivalence graph. 
  \begin{enumerate}
  \item For any $w \in W_i^0(\G)$, if there is a connected component
    of $E_{\imo} \cup E_{i} \cup E_{\ipo}$ in $\varphi_i^w(\G)$ that
    is not locally Schur positive, then either $u \in W_{\ipo}^0(\G)$
    for $u = w$ or $E_{\imo}(w)$ such that connected components of
    $E_{\imo} \cup E_{i} \cup E_{\ipo}$ are locally Schur positive in
    both $\varphi_{\ipo}^u(\G)$ and
    $\varphi_i^w(\varphi_{\ipo}^u(\G))$, or there exists $z \in
    C_{\ipo}^0(\G)$ such that connected components of $E_{\imo} \cup
    E_{i} \cup E_{\ipo}$ are locally Schur positive in both
    $\psi_{\ipo}^z(\G)$ and $\varphi_i^w(\psi_{\ipo}^z(\G))$.
  \item For any $x \in C_i^0(\G)$, if there is a connected component
    of $E_{\imo} \cup E_{i} \cup E_{\ipo}$ in $\psi_i^x(\G)$ that is
    not locally Schur positive, then there exists $z \in
    W_{\ipo}^0(\G)$ such that connected components of $E_{\imo} \cup
    E_{i} \cup E_{\ipo}$ are locally Schur positive in both
    $\varphi_{\ipo}^z(\G)$ and $\psi_i^x(\varphi_{\ipo}^z(\G))$.
  \end{enumerate}
  In particular, if $\varphi_i^w$ or $\psi_i^x$ breaks the local Schur
  positivity of $E_{\imo} \cup E_{i} \cup E_{\ipo}$, then it can be
  restored with either $\varphi^u_{\ipo}$ or $\psi^z_{\ipo}$ for some
  $u \in W_{\ipo}^0(\G)$ or $z \in C_{\ipo}^0(\G)$.
  \label{lem:phi-3-middle}
\end{lemma}

\begin{proof}
  We begin with $\varphi_i^w$. As depicted in
  Figure~\ref{fig:phi-3-up}, let $x = E_i(w)$ and $u =
  E_{\imo}(w)$. Similar to the discussion in Lemma~\ref{lem:phi-2-up},
  exactly one of $w$ and $u$ admits an $\ipo$-neighbor, so without
  loss of generality assume $w$ does and $u$ does not. If $w \in
  W_{\ipo}(\G)$, then by the same analysis as in
  Lemma~\ref{lem:phi-2-up}, $w \in W_{\ipo}^0(\G)$ and applying
  $\varphi_{\ipo}^w$ results in an alternative path from $w$ to $x$
  via $E_{\ipo}$. Therefore a subsequent application of $\varphi_i^w$
  does not disconnect the component of $E_{\imo} \cup E_{i} \cup
  E_{\ipo}$, thereby ensuring it remains locally Schur
  positive. Alternatively, if $w \not\in W_{\ipo}(\G)$, then letting
  $z = E_{\ipo}(w)$, $z$ must admit an $\imo$-neighbor and no
  $i$-neighbor. Following the $E_{\imo} \cup E_i$ string from $x$
  through $w$ and onwards results in a vertex $v$ admitting both an
  $i$-neighbor and an $\imo$-neighbor but not having $i$-type
  W. Therefore $E_{\imo}(v)$ will not admit an $i$-neighbor. Since the
  edges along the string toggle $E_{\ipo}$ by axioms $2$ and $3$, $v$
  must admit an $\ipo$-neighbor and so, too, must
  $E_{\imo}(v)$. Therefore $z \in C_{\ipo}^0(\G)$ with $m>0$, and
  applying $\psi_{\ipo}^z$ results in an alternative path from $w$ to
  $v$. See, e.g. Figure~\ref{fig:phi-bad-3b}. Once again, a subsequent
  application of $\varphi_i^w$ does not disconnect the component of
  $E_{\imo} \cup E_{i} \cup E_{\ipo}$, thereby ensuring it remains
  locally Schur positive.

  \begin{figure}[ht]
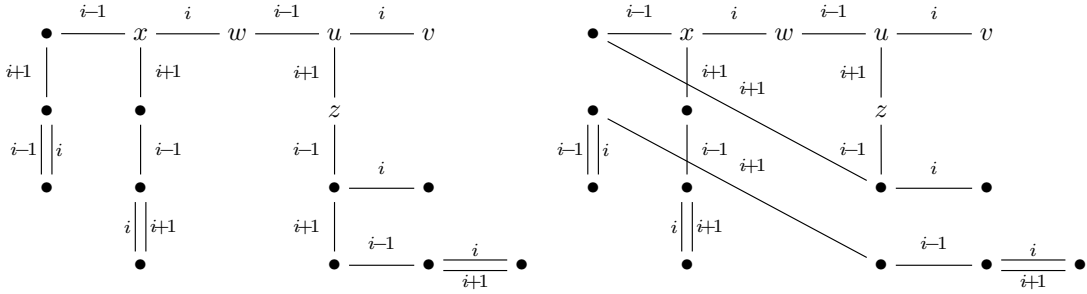

    \begin{displaymath}
      \begin{array}{\cs{7}\cs{7}\cs{7}\cs{7}\cs{7} c}
        \rnode{a1}{\B} & \rnode{b1}{x} & \rnode{c1}{w} & \rnode{d1}{u} & \rnode{e1}{v} & \\[4ex]
        \rnode{a2}{\B} & \rnode{b2}{\B} & & \rnode{d2}{z} & & \\[4ex]
        \rnode{a3}{\B} & \rnode{b3}{\B} & & \rnode{d3}{\B} & \rnode{e3}{\B} & \\[4ex]
        & \rnode{b4}{\B} & & \rnode{d4}{\B} & \rnode{e4}{\B} & \rnode{f4}{\B} 
      \end{array}
      \hspace{\cellsize}
      \begin{array}{\cs{7}\cs{7}\cs{7}\cs{8}\cs{7} c}
        \rnode{A1}{\B} & \rnode{B1}{x} & \rnode{C1}{w} & \rnode{D1}{u} & \rnode{E1}{v} & \\[4ex]
        \rnode{A2}{\B} & \rnode{B2}{\B} & & \rnode{D2}{z} & & \\[4ex]
        \rnode{A3}{\B} & \rnode{B3}{\B} & & \rnode{D3}{\B} & \rnode{E3}{\B} & \\[4ex]
        & \rnode{B4}{\B} & & \rnode{D4}{\B} & \rnode{E4}{\B} & \rnode{F4}{\B} 
      \end{array}
      \psset{linewidth=.1ex,nodesep=3pt}
      \everypsbox{\scriptstyle}
      \ncline {a1}{b1} \naput{\imo}
      \ncline {b1}{c1} \naput{i}
      \ncline {c1}{d1} \naput{\imo}
      \ncline {d1}{e1} \naput{i}
      \ncline {a1}{a2} \nbput{\ipo}
      \ncline {b1}{b2} \naput{\ipo}
      \ncline {d1}{d2} \nbput{\ipo}
      \ncline[offset=2pt] {a2}{a3} \nbput{\imo}
      \ncline[offset=2pt] {a3}{a2} \nbput{i}
      \ncline {b2}{b3} \naput{\imo}
      \ncline {d2}{d3} \nbput{\imo}
      \ncline {d3}{e3} \naput{i}
      \ncline[offset=2pt] {b3}{b4} \nbput{i}
      \ncline[offset=2pt] {b4}{b3} \nbput{\ipo}
      \ncline {d3}{d4} \nbput{\ipo}
      \ncline {d4}{e4} \naput{\imo}
      \ncline[offset=2pt] {e4}{f4} \nbput{\ipo}
      \ncline[offset=2pt] {f4}{e4} \nbput{i}
      \ncline {A1}{B1} \naput{\imo}
      \ncline {B1}{C1} \naput{i}
      \ncline {C1}{D1} \naput{\imo}
      \ncline {D1}{E1} \naput{i}
      \ncline {B1}{B2} \naput{\ipo}
      \ncline {D1}{D2} \nbput{\ipo}
      \ncline[offset=2pt] {A2}{A3} \nbput{\imo}
      \ncline[offset=2pt] {A3}{A2} \nbput{i}
      \ncline {B2}{B3} \naput{\imo}
      \ncline {D2}{D3} \nbput{\imo}
      \ncline {D3}{E3} \naput{i}
      \ncline[offset=2pt] {B3}{B4} \nbput{i}
      \ncline[offset=2pt] {B4}{B3} \nbput{\ipo}
      \ncline {D4}{E4} \naput{\imo}
      \ncline[offset=2pt] {E4}{F4} \nbput{\ipo}
      \ncline[offset=2pt] {F4}{E4} \nbput{i}
      \ncline {A1}{D3} \naput{\ipo}
      \ncline {A2}{D4} \naput{\ipo}
    \end{displaymath}
    \caption{\label{fig:phi-bad-3b}Scenario where $\varphi_i^w$ breaks
      local Schur positivity of $E_{\imo} \cup E_{i} \cup E_{\ipo}$
      (left), and the solution by applying $\psi_{\ipo}^{z}$ first
      (right).}
  \end{figure}

  For $\psi_i^x$, the situation is simpler. As usual, let $u =
  (E_{\imo} E_{\imt})^m E_i(x)$, and let $w = E_{\imt}(x)$ and $v =
  E_{\imt}(u)$. If either $w$ or $v$ has $\ipo$-type W, then applying
  $\varphi_{\ipo}$ at that vertex creates an alternative path via
  $E_{\ipo}$ for the $i$-edge at that vertex, and so a subsequent
  application of $\psi_i^x$ does not disconnect the component of
  $E_{\imo} \cup E_{i} \cup E_{\ipo}$. If $x$ has $\ipo$-type W, then
  by axiom $4'b$, so must at least one of $w$ or $v$, so the same
  solution applies. Therefore we may assume that none of $w,v,x$ has
  $\ipo$-type W. In this case, exactly one of $x$ and $u$ admits an
  $\ipo$-neighbor, so assume $x$ does and $u$ does not. By axiom $2$,
  this implies that $w$ admits an $\ipo$-neighbor and $v$ does
  not. Since neither has $\ipo$-type W, we conclude that $E_i(w)$ does
  not admit an $\ipo$-neighbor while $E_i(v)$ does. Following
  $\imo$-neighbors in the same way, both $v$ and $E_i(w)$ admit an
  $i$-edge but not $\imo$-neighbor nor an $\ipo$-neighbor. Therefore
  applying $\psi_i^x$ does not change the quasisymmetric functions
  associated to either component of $E_{\imo} \cup E_{i} \cup
  E_{\ipo}$ containing $v$ or $w$. Hence local Schur positivity of
  $E_{\imo} \cup E_{i} \cup E_{\ipo}$ is once again maintained.
\end{proof}

\begin{lemma}
  Let $\G$ be a D graph and the $(\imo,N)$-restriction of $\G$
  a dual equivalence graph. 
  \begin{enumerate}
  \item For any $w \in W_i^0(\G)$, if there is a connected component
    of $E_{i} \cup E_{\ipo} \cup E_{\ipt}$ in $\varphi_i^w(\G)$ that
    is not locally Schur positive, then either there exists $u \in
    W_{\ipo}^0(\G)$ such that connected components of $E_{i} \cup
    E_{\ipo} \cup E_{\ipt}$ are locally Schur positive in both
    $\varphi_{\ipo}^u(\G)$ and $\varphi_i^w(\varphi_{\ipo}^u(\G))$, or
    there exists $z \in C_{\ipt}^0(\G)$ such that connected components
    of $E_{i} \cup E_{\ipo} \cup E_{\ipt}$ are locally Schur positive
    in both $\psi_{\ipt}^z(\G)$ and $\varphi_i^w(\psi_{\ipt}^z(\G))$.
  \item For any $x \in C_i^0(\G)$, if there is a connected component
    of $E_{i} \cup E_{\ipo} \cup E_{\ipt}$ in $\psi_i^x(\G)$ that is
    not locally Schur positive, then either there exists $u \in
    W_{\ipo}^0(\G)$ such that connected components of $E_{i} \cup
    E_{\ipo} \cup E_{\ipt}$ are locally Schur positive in both
    $\varphi_{\ipo}^u(\G)$ and $\psi_i^x(\varphi_{\ipo}^u(\G))$, or
    there exists $z \in C_{\ipt}^0(\G)$ such that connected components
    of $E_{i} \cup E_{\ipo} \cup E_{\ipt}$ are locally Schur positive
    in both $\psi_{\ipt}^z(\G)$ and $\psi_i^x(\psi_{\ipt}^z(\G))$.
  \end{enumerate}
  In particular, if $\varphi_i^w$ or $\psi_i^x$ breaks the local Schur
  positivity of $E_{i} \cup E_{\ipo} \cup E_{\ipt}$, then it can be
  restored with either $\varphi^u_{\ipo}$ or $\psi^z_{\ipt}$ for some
  $u \in W_{\ipo}^0(\G)$ or $z \in C_{\ipt}^0(\G)$.
  \label{lem:phi-3-up}
\end{lemma}

\begin{proof}
  We begin with $\varphi_i^w$. Keeping notation from before, exactly
  one of $w$ or $u$ admits an $\ipo$-neighbor, so assume $u$ does and
  $w$ does not. By axiom $3$, $x$ admits an $\ipo$-neighbor since $w$
  does not. By axiom $2$, $\sigma(w)_{\ipo,\ipt} =
  \sigma(u)_{\ipo,\ipt}$, so $w$ admits an $\ipt$-neighbor if and only
  if $u$ admits an $\ipt$-neighbor, and, if so, by axiom $5$,
  $E_{\imo}E_{\ipt}(w) = E_{\ipt}(u)$. Since $w$ does not admit an
  $\ipo$-neighbor, by axioms $1$ and $3$, $\sigma(w)_{\ipo} =
  \sigma(x)_{\ipo}$, and by axiom $2$, $\sigma(w)_{\ipt} =
  \sigma(x)_{\ipt}$. In particular, $x$ admits an $\ipt$-neighbor if
  and only if $w$ admits an $\ipt$-neighbor. As before, if $u$ has
  $\ipo$-type W, then we may apply $\varphi_{\ipo}^w$, thereby
  creating an alternative path. Therefore assume $u$ does not have
  $\ipo$-type W, and so by axiom $2$, $v$ does not admit an
  $\ipo$-neighbor and $v$ admits an $\ipt$-neighbor if and only if $u$
  admits an $\ipt$-neighbor. Thus consider the two cases based on
  whether all or none of $x,w,u,v$ admit an $\ipt$-neighbor.

  \begin{figure}[ht]
    \begin{displaymath}
      \begin{array}{\cs{7}\cs{7}\cs{7} c}
        & \rnode{a0}{\B} & & \rnode{c0}{\B} \\[4ex]
        \rnode{a1}{x} & \rnode{b1}{w} & \rnode{c1}{u} & \rnode{d1}{v} \\[4ex]
        \rnode{a2}{\B} & \rnode{b2}{\B} & \rnode{c2}{\B} & \rnode{d2}{\B} \\[4ex]
        \rnode{a3}{\B} & \rnode{b3}{\B} & \rnode{c3}{\B} & \rnode{d3}{\B}
      \end{array}
      \hspace{4em}
      \begin{array}{\cs{7}\cs{7}\cs{7} c}
        & \rnode{A0}{\B} & & \rnode{C0}{\B} \\[4ex]
        \rnode{A1}{x} & \rnode{B1}{w} & \rnode{C1}{u} & \rnode{D1}{v} \\[4ex]
        & & & \\[4ex]
        & & &
      \end{array}
      \psset{linewidth=.1ex,nodesep=3pt}
      \everypsbox{\scriptstyle}
      \ncline {a0}{a1} \nbput{\ipo}
      \ncline {c0}{c1} \nbput{\ipo}
      \ncline {a1}{b1} \naput{i}
      \ncline {b1}{c1} \naput{\imo}
      \ncline {c1}{d1} \naput{i}
      \ncline {a1}{a2} \nbput{\ipt}
      \ncline {b1}{b2} \nbput{\ipt}
      \ncline {c1}{c2} \naput{\ipt}
      \ncline {d1}{d2} \naput{\ipt}
      \ncline {b2}{c2} \naput{\imo}
      \ncline {b2}{b3} \nbput{i}
      \ncline {b2}{a3} \nbput{\ipo}
      \ncline {d2}{d3} \naput{i}
      \ncline {d2}{c3} \nbput{\ipo}
      \ncline {A0}{A1} \nbput{\ipo}
      \ncline {C0}{C1} \nbput{\ipo}
      \ncline {A1}{B1} \naput{i}
      \ncline {B1}{C1} \naput{\imo}
      \ncline {C1}{D1} \naput{i}
    \end{displaymath}
    \caption{\label{fig:phi-bad-3c}The two possibilities of $E_{\ipo}
      \cup E_{\ipt}$ on a component of $W_i(\G) \cap (E_{\imo} \cup
      E_i)$.}
  \end{figure}

  If $x,w,u,v$ all admit an $\ipt$-neighbor, then we have the case
  depicted in the left side of Figure~\ref{fig:phi-bad-3c}. If both
  $E_{\ipt}(w)$ and $E_{\ipt}(v)$ have $\ipo$-type W, then applying
  $\varphi_{\ipo}$ to both results in terminal pieces of $E_i \cup
  E_{\ipo} \cup E_{\ipt}$. In this case, the quasisymmetric expansion
  of the components remains unchanged after applying $\varphi_i^w$. On
  the other hand, if one of them does not have $\ipo$-type W, then
  $\psi_{\ipt}$ applies, creating an alternative path so that a
  subsequent application of $\varphi_i^w$ does not separate the
  components of $E_i \cup E_{\ipo} \cup E_{\ipt}$. Hence local Schur
  positivity can always be maintained in this case. Alternately, if
  none of $x,w,u,v$ admits an $\ipt$-neighbor, then we have the case
  depicted in the right side of Figure~\ref{fig:phi-bad-3c}. In this
  case, both $w$ and $v$ are $\ipt$-type A extremal points of a
  component of $E_i \cup E_{\ipo} \cup E_{\ipt}$, so the
  quasisymmetric expansion of the components remains unchanged after
  applying $\varphi_i^w$. Therefore local Schur positivity is again
  maintained.

  The case for $\psi_i^x$ is not dissimilar. As usual, let $u =
  (E_{\imo} E_{\imt})^m E_i(x)$, and set $w = E_{\imt}(x)$ and $v =
  E_{\imt}(u)$. As in the previous lemma, if any of $v,x,v$ has
  $\ipo$-type W, then $\varphi_{\ipo}$ creates the desired alternative
  path. Therefore assume none does, and so $x$ admits an
  $\ipt$-neighbor if and only if $u$ does, and by axiom $2$, $w$
  admits an $\ipt$-neighbor if and only if $x$ does and similarly for
  $v$ and $u$. Therefore there are two cases to consider, either they
  all admit $\ipt$-neighbors or none of them does. If none does, then
  both $v$ and $E_{i}(w)$ are $\ipt$-type A terminal vertices for the
  component of $E_i \cup E_{\ipo} \cup E_{\ipt}$, so the
  quasisymmetric expansion of the components remains unchanged after
  applying $\psi_i^x$. If they all admit $\ipt$-neighbors, then $v \in
  C_{\ipt}^0(\G)$, and applying $\psi_{\ipt}^v$ creates the desired
  alternative path so that a subsequent application of $\psi_i^x$ does
  not separate components of $E_i \cup E_{\ipo} \cup E_{\ipt}$. Once
  again, local Schur positivity of $E_i \cup E_{\ipo} \cup E_{\ipt}$
  can always be maintained.
\end{proof}

Lemmas~\ref{lem:phi-2-down}, \ref{lem:phi-2-up},
\ref{lem:phi-3-middle} and \ref{lem:phi-3-up} determine when local
Schur positivity is maintained by $\varphi$ and $\psi$. For our
purposes, however, it is more prudent to use them to understand how
local Schur positivity can be \emph{restored} after applying these
maps..

\begin{theorem}
  Let $\mathcal{H}$ be a D graph. Let $\G$ be a signed, colored graph
  obtained from a $\mathcal{H}$ by applying the maps $\varphi_j, \psi_j,
  \gamma_j, \theta_j$ for $j < i$ in such a way that the
  $(i,N)$-restriction of $\G$ is a D graph satisfying dual equivalence
  axiom $4$ and the $(\imt,N)$-restriction of $\G$ is a dual
  equivalence graph.

  Then there exists a signed, colored graph $\widetilde{\G}$ obtained
  from $\mathcal{H}$ by applying the maps $\varphi_j, \psi_j,
  \gamma_j, \theta_j$ for $j \leq i$ such that $\widetilde{\G}$
  satisfies axioms $1, 2, 3$ and $5$, the $(\ipo,N)$-restriction of
  $\widetilde{\G}$ is a D graph, the $(i,N)$-restriction of
  $\widetilde{\G}$ satisfies dual equivalence axiom $4$, and the
  $(\imt,N)$-restriction of $\widetilde{\G}$ is a dual equivalence
  graph.
  \label{thm:restore}
\end{theorem}

\begin{proof}
  Since the original graph is locally Schur positive, if at some point
  during the construction of $\G$ the $(\ipo,N)$-restriction is not
  locally Schur positive, then the failure must be with some component
  of $E_{\imo} \cup E_{i}$ failing $\LSP_4$ or some component of
  $E_{\imt} \cup E_{\imo} \cup E_{i}$ failing $\LSP_5$, or both. 

  If some component of $E_{\imo} \cup E_{i}$ fails $\LSP_4$, then this
  must have resulted from an alteration of an $\imo$-edge by either
  $\varphi_{\imo}$ or $\psi_{\imo}$. Consider the graph obtained just
  before the offending map was applied. By Lemma~\ref{lem:phi-2-up},
  $\LSP_4$ can be maintained by first applying $\varphi_{i}^u$ for a
  suitable $u \in W_{i}^0(\G)$. Since the action of $\varphi_{\imo}$
  is independent of the $j$ edges of the graph for $j \geq i$, this
  does not change the $(i,N)$-restriction of the final
  graph. Therefore $\LSP_4$ can always be maintained without changing
  the $(i,N)$-restriction. By Lemmas~\ref{lem:axiom4p} and
  \ref{lem:theta-4p}, this also ensures that the maps maintain axiom
  $4'$ for the $(\ipo,N)$-restriction.
  
  Similarly, if at some point during the transformation $E_{\imt} \cup
  E_{\imo} \cup E_{i}$ fails $\LSP_5$, then this must have resulted
  from an altered $E_{\imt}$ or $E_{\imo}$ edge after an application
  of $\varphi_{\imt}, \psi_{\imt}$ or $\varphi_{\imo}, \psi_{\imo}$,
  respectively.  Lemma~\ref{lem:phi-3-up} ensures that the first two
  cases can be avoided by first applying $\varphi_{\imo}^u$ or
  $\psi_{i}^z$ for a suitable $u \in W_{\imo}^0(\G)$ or $z \in
  C_{i}^0(\G)$. Since the action of $\varphi_{\imt}$ is independent of
  the $j$ edges of the graph for $j \geq \imo$, this does not change
  the $(\imo,N)$-restriction of the final graph. Since the
  $(i,N)$-restriction of $\G$ is a D graph satisfying axiom $4$, the
  required application of $\varphi_{\imo}^u$ to ensure $\LSP_5$ was
  also necessary to achieve axiom $4$, so this does not change the
  $(i,N)$-restriction either. By Lemma~\ref{lem:phi-3-middle}, the
  latter two cases can be avoided by first applying $\varphi_{i}^u$ or
  $\psi_{i}^z$ for a suitable $u \in W_{i}^0(\G)$ or $z \in
  C_{i}^0(\G)$. Again, the actions of the maps are independent of the
  $i$ edges of the graph, so the $(i,N)$-restriction of the graph
  remains unchanged. Therefore $\LSP_5$ can be maintained as
  well. Thus the $(\ipo,N)$-restriction of the resulting graph is a D
  graph, and, since the $(i,N)$-restriction of $\widetilde{\G}$ is
  that same as the $(i,N)$-restriction of $\G$, it satisfies axiom $4$
  and the $(\imt,N)$-restriction is a dual equivalence graph.
\end{proof}

Finally, we are ready to show that we can apply the maps $\varphi,
\psi, \gamma$ and $\theta$ repeatedly to $\G^{(k)}_{c,D}$, or, more
generally, to any D graph, until dual equivalence axioms $4$ and $6$
hold while maintaining axioms $1, 2, 3$ and $5$. The following theorem
is the final ingredient to the proof of Theorem~\ref{thm:Dgraph}, and,
as a corollary, to LLT and Macdonald positivity.

\begin{theorem}
  Let $\G = (V,\sigma,E)$ be a D graph of type $(n,N)$. Then there
  exists a dual equivalence graph $\widetilde{\G} =
  (V,\sigma,\widetilde{E})$ of type $(n,N)$ with the same vertex set
  and signature function. In particular, for $n=N$, the sum $\sum_{v
    \in V} Q_{\sigma(v)}(X)$ is symmetric and Schur positive.
  \label{thm:D}
\end{theorem}

\begin{proof}
  We proceed by modifying the edges of $\G$ iteratively to construct a
  sequence of signed, colored graphs $\G = \G_2 , \ldots , \G_{\nmo} =
  \widetilde{\G}$ such that, for each $i>2$, $\G_{\imo}$ is
  constructed from $\G$ by applying the maps $\varphi_j, \psi_j,
  \gamma_j, \theta_j$ for $j < \ipt$, $\G_{\imo}$ satisfies dual
  equivalence axioms $1,2,3$ and $5$, and the $(i,N)$-restriction of
  $\G_{\imo}$ is a dual equivalence graph. Note that the intermediate
  graphs $\G_{\imo}$ might fail local Schur positivity, and so they
  are not necessarily D graphs.

  Since axioms $4$ and $6$ are vacuously satisfied for a graph of type
  $(3,N)$, the base case $\G_2 = \G$ is proved. Therefore we proceed
  by constructing $\G_i$ from $\G_{\imo}$ as follows. Since the
  $(i,N)$-restriction of $\G_{\imo}$ is a dual equivalence graph, by
  Theorem~\ref{thm:restore}, we can $\varphi_i$ and $\psi_i$ until the
  $(\ipo,N)$-restriction is once again a D graph while maintaining
  axioms $1,2,3$ and $5$. Then by Theorem~\ref{thm:axiom4}, we can
  apply $\varphi_i, \psi_i$ and $\gamma_i$ until axiom $4$ holds for
  the $(\ipo,N)$-restriction while preserving axioms $1,2$ and $5$. By
  Lemmas~\ref{lem:phi-2-down} and \ref{lem:phi-2-up}, axiom $3$ is
  also maintained. At this point the $(\ipo,N)$-restriction is a D
  graph satisfying axiom $4$ and the $(i,N)$-restriction remains a
  dual equivalence graph, so by Theorem~\ref{thm:restore}, we can
  $\varphi_{\ipo}$ and $\psi_{\ipo}$ until the $(\ipt,N)$-restriction
  is once again a D graph while maintaining axioms $1,2,3$ and
  $5$. Next, by Theorem~\ref{thm:axiom4}, we can apply
  $\varphi_{\ipo}, \psi_{\ipo}$ and $\gamma_{\ipo}$ until axiom $4$
  holds for the $(\ipt,N)$-restriction while preserving axioms $1,2$
  and $5$ as well as axiom $3$, by Lemmas~\ref{lem:phi-2-down} and
  \ref{lem:phi-2-up}. Continuing on, since the the
  $(\ipt,N)$-restriction is now a D graph satisfying axiom $4$ and the
  $(i,N)$-restriction remains a dual equivalence graph, by
  Theorem~\ref{thm:restore}, we can $\varphi_{\ipt}$ and $\psi_{\ipt}$
  until the $(\iph,N)$-restriction is once again a D graph while
  maintaining axioms $1,2,3$ and $5$. Again, by
  Theorem~\ref{thm:axiom4}, we can apply $\varphi_{\ipt}, \psi_{\ipt}$
  and $\gamma_{\ipt}$ until axiom $4$ holds for the
  $(\iph,N)$-restriction while preserving axioms $1,2$ and $5$ as well
  as axiom $3$, by Lemmas~\ref{lem:phi-2-down} and
  \ref{lem:phi-2-up}. The result is a graph satisfying axioms $1,2,3$
  and $5$ for which the $(\iph,N)$-restriction satisfies axiom $4$ and
  the $(i,N)$-restriction is a dual equivalence graph. 

  By Theorem~\ref{thm:axiom6}, we may apply $\theta_i$ together with
  $\varphi_{\ipo}$ and $\psi_{\ipt}$ as needed until the
  $(\ipo,N)$-restriction satisfies axiom $6$, all while maintaining
  axioms $1,2$ and $5$ as well as axiom and $4$ for the
  $(\iph,N)$-restriction. Call this resulting graph $\G_i$ and notice
  that it satisfies dual equivalence axioms $1,2,3$ and $5$, the
  $(\iph,N)$-restriction is a D graph, and the $(\ipo,N)$-restriction
  is a dual equivalence graph. Further, $\G_i$ was constructed from
  $\G$ using the maps $\varphi_j, \psi_j, \gamma_j, \theta_j$ for $j
  \leq \ipt$. Therefore we may proceed with the construction until
  $\G_{\nmo} = \widetilde{\G}$, which is its own $(n,N)$-restriction
  and, as such, is a dual equivalence graph.
\end{proof}

While transforming a D graph into a dual equivalence graph is quite
complicated, it is not necessary to carry out explicitly for any given
application. Once a D graph structure is established, the generating
function is proved to be Schur positive by Theorem~\ref{thm:D} and
Corollary~\ref{cor:schurpos}. Therefore we hope that there will be
many further applications of this theory to other classes of symmetric
functions beyond the immediate application to LLT and Macdonald
polynomials.

%
%

\bibliographystyle{amsalpha} 
\bibliography{../../references}

%
\appendix
%

\clearpage
\section{Standard dual equivalence graphs}
\label{app:DEGs}

Below we give the dual equivalence graphs of type $(6,6)$. The graphs
for the conjugate shapes may be obtained by transposing each tableau
and multiplying the signature coordinate-wise by $-1$.

\begin{figure}[ht]
  \begin{displaymath}
    \begin{array}{c}
      \tableau{1 & 2 & 3 & 4 & 5 & 6} \\  _{+++++} 
    \end{array}
\end{displaymath}
\caption{\label{fig:G6}The standard dual equivalence graph $\G_{6}$.}
\end{figure}

\begin{figure}[ht]
  \begin{displaymath}
    \begin{array}{\cs{3} \cs{3} \cs{3} \cs{3} c}
      \stab{a}{2 \\ 1 & 3 & 4 & 5 & 6}{-++++} &
      \stab{b}{3 \\ 1 & 2 & 4 & 5 & 6}{+-+++} &
      \stab{c}{4 \\ 1 & 2 & 3 & 5 & 6}{++-++} &
      \stab{d}{5 \\ 1 & 2 & 3 & 4 & 6}{+++-+} &
      \stab{e}{6 \\ 1 & 2 & 3 & 4 & 5}{++++-} 
    \end{array}
    \psset{nodesep=3pt,linewidth=.1ex}
    \everypsbox{\scriptstyle}
    \ncline            {a}{b} \naput{2}
    \ncline            {b}{c} \naput{3}
    \ncline            {c}{d} \naput{4}
    \ncline            {d}{e} \naput{5}
  \end{displaymath}
\caption{\label{fig:G51}The standard dual equivalence graph $\G_{5,1}$.}
\end{figure}

\begin{figure}[ht]
  \begin{displaymath}
    \begin{array}{\cs{2} \cs{2} \cs{2} \cs{4} \cs{4} \cs{2} \cs{2} c}
      \stab{a}{3 & 4 \\ 1 & 2 & 5 & 6}{+-+++} & &     
      \stab{b}{2 & 4 \\ 1 & 3 & 5 & 6}{-+-++} & &     
      \stab{c}{2 & 5 \\ 1 & 3 & 4 & 6}{-++-+} & & & \\[1.5\cellsize]
      & & & \stab{i}{2 & 6 \\ 1 & 3 & 4 & 5}{-+++-} & &
      \stab{j}{3 & 5 \\ 1 & 2 & 4 & 6}{+-+-+} & &
      \stab{k}{4 & 5 \\ 1 & 2 & 3 & 6}{++-++} \\[1.5\cellsize]
      \stab{z}{5 & 6 \\ 1 & 2 & 3 & 4}{+++-+} & &     
      \stab{y}{4 & 6 \\ 1 & 2 & 3 & 5}{++-+-} & &     
      \stab{x}{3 & 6 \\ 1 & 2 & 4 & 5}{+-++-} & & & 
    \end{array}
    \psset{nodesep=3pt,linewidth=.1ex}
    \everypsbox{\scriptstyle}
    \ncline            {cc}{i} \naput{5}
    \ncline            {ii}{x} \naput{2}
    \ncline            {c}{b} \nbput{4}
    \ncline            {y}{x} \nbput{3}
    \ncline[offset=2pt]{b}{a}\naput{3}
    \ncline[offset=2pt]{a}{b}\naput{2}
    \ncline            {cc}{j} \nbput{2}
    \ncline            {jj}{x} \nbput{5}
    \ncline[offset=2pt]{y}{z}\naput{5}
    \ncline[offset=2pt]{z}{y}\naput{4}
    \ncline[offset=2pt]{j}{k}\naput{3}
    \ncline[offset=2pt]{k}{j}\naput{4}
  \end{displaymath}
\caption{\label{fig:G42}The standard dual equivalence graph $\G_{4,2}$.}
\end{figure}

\begin{figure}[ht]
  \begin{displaymath}
    \begin{array}{\cs{5} \cs{5} \cs{5} c}
      & & \stab{w}{3 & 4 & 6 \\ 1 & 2 & 5}{+-++-} & \\[.5\cellsize]
      \stab{z}{4 & 5 & 6 \\ 1 & 2 & 3}{++-++} &
      \stab{y}{3 & 5 & 6 \\ 1 & 2 & 4}{+-+-+} & &
      \stab{v}{2 & 4 & 6 \\ 1 & 3 & 5}{-+-+-} \\[.5\cellsize]
      & & \stab{x}{2 & 5 & 6 \\ 1 & 3 & 4}{-++-+} & 
    \end{array}
    \psset{nodesep=3pt,linewidth=.1ex}
    \everypsbox{\scriptstyle}
    \ncline[offset=2pt]{v}{w} \naput{3}
    \ncline[offset=2pt]{w}{v} \naput{2}
    \ncline            {x}{y} \naput{2}
    \ncline[offset=2pt]{y}{z} \naput{4}
    \ncline[offset=2pt]{z}{y} \naput{3}
    \ncline[offset=2pt]{v}{x} \nbput{4}
    \ncline[offset=2pt]{x}{v} \nbput{5}
    \ncline            {w}{y} \nbput{5}
  \end{displaymath}  
\caption{\label{fig:G33}The standard dual equivalence graph $\G_{3,3}$.}
\end{figure}

\begin{figure}[ht]
  \begin{displaymath}
    \begin{array}{\cs{1} \cs{1} \cs{1} \cs{1} \cs{1} \cs{1} c}
      \stab{a}{3 \\ 2 \\ 1 & 4 & 5 & 6}{--+++} & &
      \stab{c}{4 \\ 3 \\ 1 & 2 & 5 & 6}{+--++} & &
      \stab{f}{5 \\ 4 \\ 1 & 2 & 3 & 6}{++--+} & &
      \stab{j}{6 \\ 5 \\ 1 & 2 & 3 & 4}{+++--} \\[2\cellsize]
      &
      \stab{b}{4 \\ 2 \\ 1 & 3 & 5 & 6}{-+-++} & &
      \stab{e}{5 \\ 3 \\ 1 & 2 & 4 & 6}{+-+-+} & &
      \stab{i}{6 \\ 4 \\ 1 & 2 & 3 & 5}{++-+-} &  \\[2\cellsize]
      & & 
      \stab{d}{5 \\ 2 \\ 1 & 3 & 4 & 6}{-++-+} & &
      \stab{h}{6 \\ 3 \\ 1 & 2 & 4 & 5}{+-++-} & & \\[2\cellsize]
      & & & \stab{g}{6 \\ 2 \\ 1 & 3 & 4 & 5}{-+++-} & & & 
    \end{array}
    \psset{nodesep=3pt,linewidth=.1ex}
    \everypsbox{\scriptstyle}
    \ncline            {a}{b} \naput{3}
    \ncline            {b}{d} \naput{4}
    \ncline            {c}{e} \naput{4}
    \ncline            {d}{g} \naput{5}
    \ncline            {e}{h} \naput{5}
    \ncline            {f}{i} \naput{5}
    \ncline            {b}{c} \naput{2}
    \ncline            {d}{e} \naput{2}
    \ncline            {e}{f} \naput{3}
    \ncline            {g}{h} \naput{2}
    \ncline            {h}{i} \naput{3}
    \ncline            {i}{j} \naput{4}
  \end{displaymath}  
\caption{\label{fig:G411}The standard dual equivalence graph $\G_{4,1,1}$.}
\end{figure}

\begin{figure}[ht]
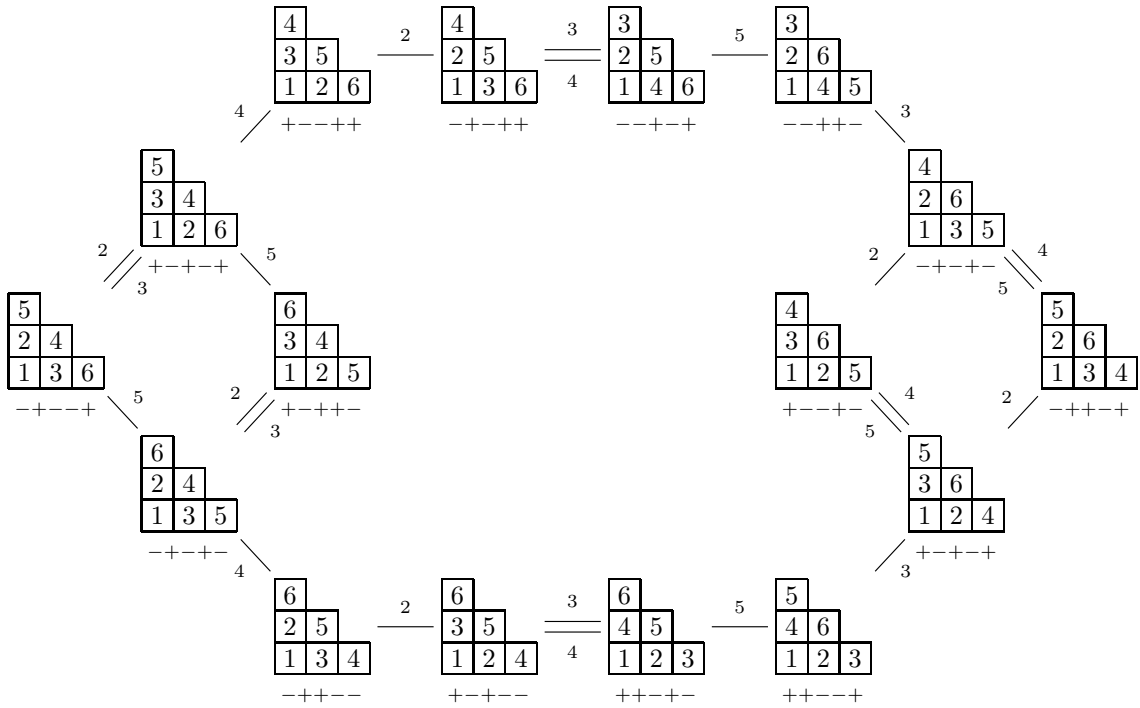

  \begin{displaymath}
    \begin{array}{\cs{1} \cs{1} \cs{4} \cs{4} \cs{4} \cs{1} \cs{1} c}
      & &
      \stab{b2}{4 \\ 3 & 5 \\ 1 & 2 & 6}{+--++} &
      \stab{a3}{4 \\ 2 & 5 \\ 1 & 3 & 6}{-+-++} &
      \stab{a4}{3 \\ 2 & 5 \\ 1 & 4 & 6}{--+-+} &
      \stab{b5}{3 \\ 2 & 6 \\ 1 & 4 & 5}{--++-} & & \\[2\cellsize]
      &
      \stab{c2}{5 \\ 3 & 4 \\ 1 & 2 & 6}{+-+-+} & & & & &
      \stab{c5}{4 \\ 2 & 6 \\ 1 & 3 & 5}{-+-+-} & \\[2\cellsize]
      \stab{d1}{5 \\ 2 & 4 \\ 1 & 3 & 6}{-+--+} & &
      \stab{d3}{6 \\ 3 & 4 \\ 1 & 2 & 5}{+-++-} & & &
      \stab{d4}{4 \\ 3 & 6 \\ 1 & 2 & 5}{+--+-} & &
      \stab{d6}{5 \\ 2 & 6 \\ 1 & 3 & 4}{-++-+}  \\ [2\cellsize]
      &
      \stab{e2}{6 \\ 2 & 4 \\ 1 & 3 & 5}{-+-+-} & & & & &
      \stab{e5}{5 \\ 3 & 6 \\ 1 & 2 & 4}{+-+-+} & \\ [2\cellsize]
      & &
      \stab{f2}{6 \\ 2 & 5 \\ 1 & 3 & 4}{-++--} &
      \stab{g3}{6 \\ 3 & 5 \\ 1 & 2 & 4}{+-+--} &
      \stab{g4}{6 \\ 4 & 5 \\ 1 & 2 & 3}{++-+-} &
      \stab{f5}{5 \\ 4 & 6 \\ 1 & 2 & 3}{++--+} & & 
    \end{array}
    \psset{nodesep=3pt,linewidth=.1ex}
    \everypsbox{\scriptstyle}
    \ncline[offset=2pt]{a3}{a4} \naput{3}
    \ncline[offset=2pt]{a4}{a3} \naput{4}
    \ncline            {b2}{a3} \naput{2}
    \ncline            {a4}{b5} \naput{5}
    \ncline            {b2}{c2} \nbput{4}
    \ncline            {b5}{c5} \naput{3}
    \ncline[offset=2pt]{d1}{c2} \naput{2}
    \ncline[offset=2pt]{c2}{d1} \naput{3}
    \ncline            {c2}{d3} \naput{5}
    \ncline            {d4}{c5} \naput{2}
    \ncline[offset=2pt]{c5}{d6} \naput{4}
    \ncline[offset=2pt]{d6}{c5} \naput{5}
    \ncline            {d1}{e2} \naput{5}
    \ncline[offset=2pt]{e2}{d3} \naput{2}
    \ncline[offset=2pt]{d3}{e2} \naput{3}
    \ncline[offset=2pt]{d4}{e5} \naput{4}
    \ncline[offset=2pt]{e5}{d4} \naput{5}
    \ncline            {e5}{d6} \naput{2}
    \ncline            {e2}{f2} \nbput{4}
    \ncline            {e5}{f5} \naput{3}
    \ncline            {f2}{g3} \naput{2}
    \ncline            {g4}{f5} \naput{5}
    \ncline[offset=2pt]{g3}{g4} \naput{3}
    \ncline[offset=2pt]{g4}{g3} \naput{4}
  \end{displaymath}
  \caption{\label{fig:G321}The standard dual equivalence graph $\G_{3,2,1}$.}
\end{figure}

\clearpage
\section{Graphs for tuples of tableaux}
\label{app:Dgraphs}

This appendix gives examples of connected components of the graphs
$\G^{(k)}_{c,D}$ constructed in Section~\ref{sec:llt}. The graph in
Figure~\ref{fig:domino} come from domino tableaux of shape $((3),
(2,1))$. Comparing this graph with the examples above, it is
isomorphic to $\G_{(4,2)}$. This demonstrates
Theorem~\ref{thm:dominoes}, which states that the graph on domino
tableaux is always a dual equivalence graph.

\begin{figure}[ht]
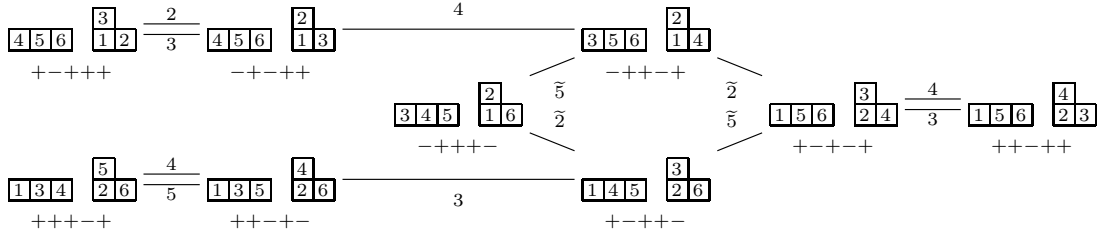

  \begin{displaymath}
    \begin{array}{\cs{4} \cs{3} \cs{3} \cs{3} \cs{4} c}
      \smstab{I}{ & & & & _3 \\ _4 & _5 & _6 & & _1 & _2}{+-+++} & 
      \smstab{G}{ & & & & _2 \\ _4 & _5 & _6 & & _1 & _3}{-+-++} & &
      \smstab{E}{ & & & & _2 \\ _3 & _5 & _6 & & _1 & _4}{-++-+} & & \\ & &
      \smstab{A}{ & & & & _2 \\ _3 & _4 & _5 & & _1 & _6}{-+++-} & &
      \smstab{F}{ & & & & _3 \\ _1 & _5 & _6 & & _2 & _4}{+-+-+} &
      \smstab{H}{ & & & & _4 \\ _1 & _5 & _6 & & _2 & _3}{++-++} \\
      \smstab{D}{ & & & & _5 \\ _1 & _3 & _4 & & _2 & _6}{+++-+} &
      \smstab{C}{ & & & & _4 \\ _1 & _3 & _5 & & _2 & _6}{++-+-} & &
      \smstab{B}{ & & & & _3 \\ _1 & _4 & _5 & & _2 & _6}{+-++-} & &
    \end{array}
    \psset{nodesep=3pt,linewidth=.1ex}
    \everypsbox{\scriptstyle}
    \ncline            {A}{B} \naput{\widetilde{2}}
    \ncline            {B}{C} \naput{3}
    \ncline[offset=2pt]{C}{D} \nbput{4}
    \ncline[offset=2pt]{D}{C} \nbput{5}
    \ncline            {A}{E} \nbput{\widetilde{5}}
    \ncline            {B}{F} \naput{\widetilde{5}}
    \ncline            {E}{F} \nbput{\widetilde{2}}
    \ncline            {E}{G} \nbput{4}
    \ncline[offset=2pt]{F}{H} \nbput{3}
    \ncline[offset=2pt]{H}{F} \nbput{4}
    \ncline[offset=2pt]{G}{I} \nbput{2}
    \ncline[offset=2pt]{I}{G} \nbput{3}
  \end{displaymath}
  \caption{\label{fig:domino}A connected component of the graph for
    domino tableaux of shape $((3),(2,1))$.}
\end{figure}

The graph in Figure~\ref{fig:non-DEG} comes from the graph for the
Macdonald polynomial $\widetilde{H}_{(4,1)}(X;q,t)$. Note that while
the generating function of the graph is $s_{(3,2)}+s_{(4,1)}$ which
indeed is Schur positive, the graph itself is not a dual equivalence
graph.

\begin{figure}[ht]
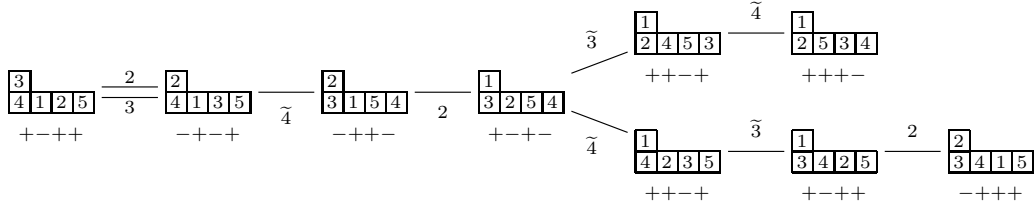

  \begin{displaymath}
    \begin{array}{\cs{4} \cs{4} \cs{4} \cs{4} \cs{4} \cs{4} c}
      & & & & 
      \smstab{b2}{_1 \\ _2 & _4 & _5 & _3}{++-+} &
      \smstab{a2}{_1 \\ _2 & _5 & _3 & _4}{+++-} & \\[-.5\cellsize]
      \smstab{a1}{_3 \\ _4 & _1 & _2 & _5}{+-++} &
      \smstab{b1}{_2 \\ _4 & _1 & _3 & _5}{-+-+} &
      \smstab{c1}{_2 \\ _3 & _1 & _5 & _4}{-++-} &
      \smstab{c2}{_1 \\ _3 & _2 & _5 & _4}{+-+-} & & & \\[-.5\cellsize]
      & & & & 
      \smstab{c3}{_1 \\ _4 & _2 & _3 & _5}{++-+} &
      \smstab{b3}{_1 \\ _3 & _4 & _2 & _5}{+-++} &
      \smstab{a3}{_2 \\ _3 & _4 & _1 & _5}{-+++} 
    \end{array}
    \psset{linewidth=.1ex,nodesep=3pt}
    \everypsbox{\scriptstyle}
    \ncline[offset=2pt]{a1}{b1} \nbput{3}
    \ncline[offset=2pt]{b1}{a1} \nbput{2}
    \ncline            {c1}{b1} \naput{\widetilde{4}}
    \ncline            {c1}{c2} \nbput{2}
    \ncline            {b2}{a2} \naput{\widetilde{4}}
    \ncline            {b2}{c2} \nbput{\widetilde{3}}
    \ncline            {c2}{c3} \nbput{\widetilde{4}}
    \ncline            {a3}{b3} \nbput{2}
    \ncline            {b3}{c3} \nbput{\widetilde{3}}
  \end{displaymath}
  \caption{\label{fig:non-DEG}A connected component of the graph for
    standard fillings of shape $(4,1)$.}
\end{figure}

\section{Resolution of axiom $4$}
\label{app:axiom4}

The graph in Figure~\ref{fig:box} arises from the graph for the
Macdonald polynomial $\widetilde{H}_{(5)}(X;q,t)$. The transformation
of this graph into a dual equivalence graph requires only $\varphi_3$
and $\varphi_4$. The result is the dual equivalence graph given in
Figure~\ref{fig:open-box}. For this example, axiom $6$ is immediate
from axiom $4$ given the size of the graph, and it is mere coincidence
that $\psi_4$ was not needed to resolve axiom $4$.

\begin{figure}[ht]
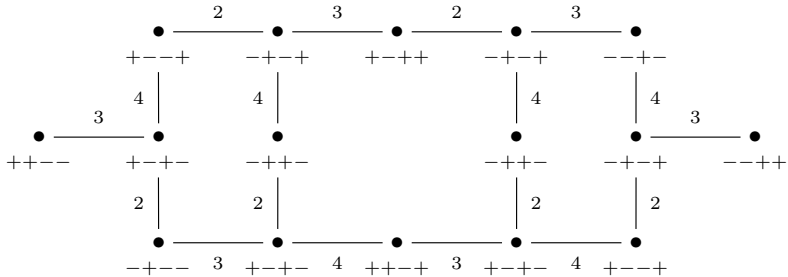

  \begin{displaymath}
    \begin{array}{\cs{7}\cs{7}\cs{7}\cs{7}\cs{7}\cs{7}c}
      & \sbull{t1}{+--+} & \sbull{t2}{-+-+} & \sbull{t3}{+-++} 
      & \sbull{t4}{-+-+} & \sbull{t5}{--+-} & \\[2\cellsize]
      \sbull{m0}{++--} & \sbull{m1}{+-+-} & \sbull{m2}{-++-} & 
      & \sbull{m4}{-++-} & \sbull{m5}{-+-+} & \sbull{m6}{--++} \\[2\cellsize]
      & \sbull{b1}{-+--} & \sbull{b2}{+-+-} & \sbull{b3}{++-+} 
      & \sbull{b4}{+-+-} & \sbull{b5}{+--+} &
    \end{array}
    \psset{linewidth=.1ex,nodesep=3pt}
    \everypsbox{\scriptstyle}
    \ncline {t1}{t2} \naput{2}
    \ncline {t2}{t3} \naput{3}
    \ncline {t3}{t4} \naput{2}
    \ncline {t4}{t5} \naput{3}
    \ncline {t1t1}{m1} \nbput{4}
    \ncline {t2t2}{m2} \nbput{4}
    \ncline {t4t4}{m4} \naput{4}
    \ncline {t5t5}{m5} \naput{4}
    \ncline {m0}{m1} \naput{3}
    \ncline {m5}{m6} \naput{3}
    \ncline {m1m1}{b1} \nbput{2}
    \ncline {m2m2}{b2} \nbput{2}
    \ncline {m4m4}{b4} \naput{2}
    \ncline {m5m5}{b5} \naput{2}
    \ncline {b1}{b2} \nbput{3}
    \ncline {b2}{b3} \nbput{4}
    \ncline {b3}{b4} \nbput{3}
    \ncline {b4}{b5} \nbput{4}
  \end{displaymath}
  \caption{\label{fig:box}A connected component for the graph for the
    $5$-tuple $((1),(1),(1),(1),(1))$ with generating function
    $s_{3,2} + s_{3,1,1} + s_{2,2,1}$.}
\end{figure}  

\begin{figure}[ht]
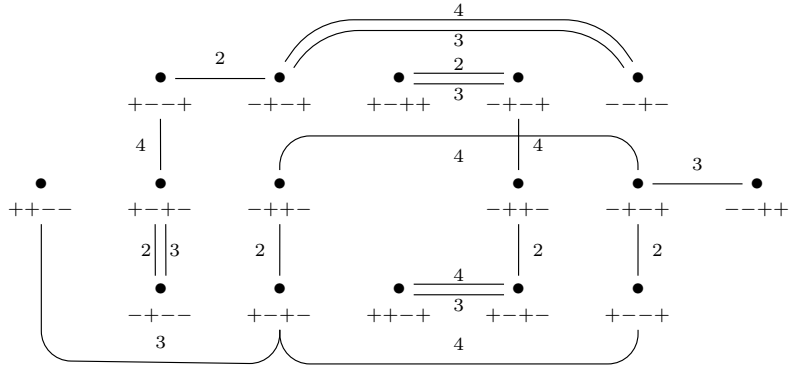

  \begin{displaymath}
    \begin{array}{\cs{7}\cs{7}\cs{7}\cs{7}\cs{7}\cs{7}c} \\[\cellsize]
      & \sbull{t1}{+--+} & \sbull{t2}{-+-+} & \sbull{t3}{+-++} 
      & \sbull{t4}{-+-+} & \sbull{t5}{--+-} & \\[2\cellsize]
      \sbull{m0}{++--} & \sbull{m1}{+-+-} & \sbull{m2}{-++-} & 
      & \sbull{m4}{-++-} & \sbull{m5}{-+-+} & \sbull{m6}{--++} \\[2\cellsize]
      & \sbull{b1}{-+--} & \sbull{b2}{+-+-} & \sbull{b3}{++-+} 
      & \sbull{b4}{+-+-} & \sbull{b5}{+--+} & \\[\cellsize]
    \end{array}
    \psset{linewidth=.1ex,nodesep=3pt}
    \everypsbox{\scriptstyle}
    \ncline {t1}{t2} \naput{2}
    \ncline[offset=2pt] {t3}{t4} \nbput{3}
    \ncline[offset=2pt] {t4}{t3} \nbput{2}
    \ncdiag[offset=2pt,angleA=55,angleB=125,arm=4.5ex,linearc=1] {t2}{t5} \nbput{3}
    \ncdiag[offset=2pt,angleB=55,angleA=125,arm=4ex,linearc=1] {t5}{t2} \nbput{4}
    \ncline {t1t1}{m1} \nbput{4}
    \ncline {t4t4}{m4} \naput{4}
    \ncdiag[angleA=90,angleB=90,arm=3ex,linearc=.4] {m2}{m5} \nbput{4}
    \ncdiag[angleA=-90,armA=12ex,angleB=-90,armB=3ex,linearc=.4] {m0m0}{b2b2} \naput{3}
    \ncline {m5}{m6} \naput{3}
    \ncline[offset=2pt] {m1m1}{b1} \nbput{2}
    \ncline[offset=2pt] {b1}{m1m1} \nbput{3}
    \ncline {m2m2}{b2} \nbput{2}
    \ncline {m4m4}{b4} \naput{2}
    \ncline {m5m5}{b5} \naput{2}
    \ncdiag[angleA=-90,angleB=-90,arm=3ex,linearc=.4] {b5b5}{b2b2} \nbput{4}
    \ncline[offset=2pt] {b3}{b4} \nbput{3}
    \ncline[offset=2pt] {b4}{b3} \nbput{4}
  \end{displaymath}
  \caption{\label{fig:open-box}The transformation of the graph in
    Figure~\ref{fig:box} using $\varphi_3$ and $\varphi_4$.}
\end{figure}  

The graph in Figure~\ref{fig:frog} is also not a dual equivalence
graph and also arises as a connected component of the graph for the
Macdonald polynomial
$\widetilde{H}_{(5)}(X;q,t)$. Figure~\ref{fig:frog} shows the
resulting dual equivalence graph after implementing the algorithms of
Section~\ref{sec:Dgraphs}, this time requiring $\psi_4$ as well as
$\varphi_3$ and $\varphi_4$. Again, axiom $6$ is immediate from axiom
$4$ given the size of the graph.

\begin{figure}[ht]
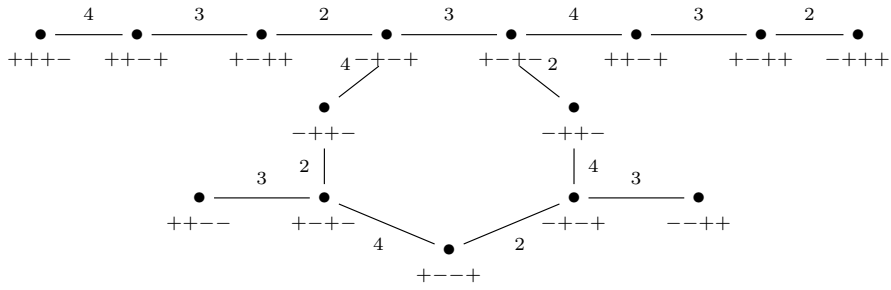

  \begin{displaymath}
    \begin{array}{\cs{3} \cs{2} \cs{2} \cs{2} \cs{2} \cs{2} \cs{2} \cs{2} \cs{2} \cs{2} \cs{2} \cs{2} \cs{2} \cs{3} c}
      \\[\cellsize]
      \sbull{A0}{+++-} & &
      \sbull{A1}{++-+} & & 
      \sbull{B1}{+-++} & & 
      \sbull{C1}{-+-+} & & 
      \sbull{D1}{+-+-} & & 
      \sbull{E1}{++-+} & & 
      \sbull{F1}{+-++} & &
      \sbull{F0}{-+++} \\[\cellsize] & & & & &
      \sbull{C2}{-++-} & & & & 
      \sbull{D2}{-++-} & & & & & \\[1.5\cellsize] & & &
      \sbull{B3}{++--} & & 
      \sbull{C3}{+-+-} & &  & & 
      \sbull{D3}{-+-+} & & 
      \sbull{E3}{--++} & & & \\ & & & & & & &
      \sbull{CD}{+--+} & & & & & & &
    \end{array}
    \psset{linewidth=.1ex,nodesep=3pt}
    \everypsbox{\scriptstyle}
    \ncline {A0}{A1} \naput{4}
    \ncline {F1}{F0} \naput{2}
    \ncline {A1}{B1} \naput{3}
    \ncline {B1}{C1} \naput{2}
    \ncline {C1}{D1} \naput{3}
    \ncline {D1}{E1} \naput{4}
    \ncline {E1}{F1} \naput{3}
    \ncline {C2}{C1C1} \naput{4}
    \ncline {D1D1}{D2} \naput{2}
    \ncline {C3}{C2C2} \naput{2}
    \ncline {D2D2}{D3} \naput{4}
    \ncline {B3}{C3} \naput{3}
    \ncline {C3}{CD} \nbput{4}
    \ncline {CD}{D3} \nbput{2}
    \ncline {D3}{E3} \naput{3}
  \end{displaymath}
  \caption{\label{fig:frog}A connected component of the graph for the
    $5$-tuple $((1),(1),(1),(1),(1))$ with generating function
    $s_{4,1} + s_{3,2} + s_{3,1,1}$.}
\end{figure}

\begin{figure}[ht]
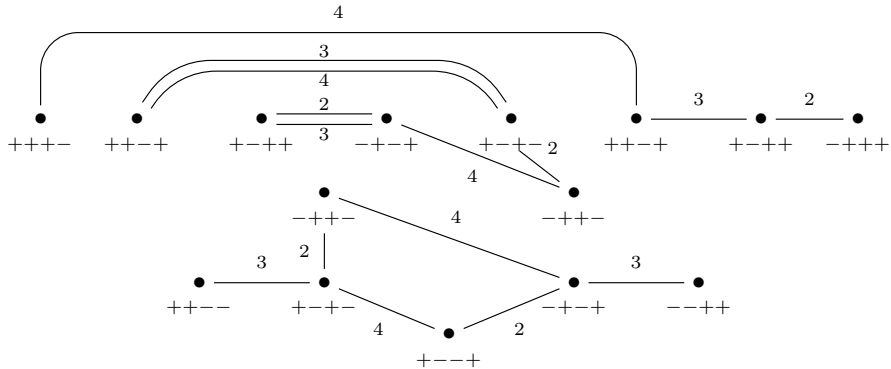

  \begin{displaymath}
    \begin{array}{\cs{3} \cs{2} \cs{2} \cs{2} \cs{2} \cs{2} \cs{2} \cs{2} \cs{2} \cs{2} \cs{2} \cs{2} \cs{2} \cs{3} c}
      \\[2\cellsize]
      \sbull{xA0}{+++-} & &
      \sbull{xA1}{++-+} & & 
      \sbull{xB1}{+-++} & & 
      \sbull{xC1}{-+-+} & & 
      \sbull{xD1}{+-+-} & & 
      \sbull{xE1}{++-+} & & 
      \sbull{xF1}{+-++} & &
      \sbull{xF0}{-+++} \\[\cellsize] & & & & &
      \sbull{xC2}{-++-} & & & & 
      \sbull{xD2}{-++-} & & & & & \\[1.5\cellsize] & & &
      \sbull{xB3}{++--} & & 
      \sbull{xC3}{+-+-} & &  & & 
      \sbull{xD3}{-+-+} & & 
      \sbull{xE3}{--++} & & & \\ & & & & & & &
      \sbull{xCD}{+--+} & & & & & & &
    \end{array}
    \psset{linewidth=.1ex,nodesep=3pt}
    \everypsbox{\scriptstyle}
    \ncdiag[angleA=90,angleB=90,arm=6.4ex,linearc=.5] {xA0}{xE1} \naput{4}
    \ncline[offset=2pt] {xB1}{xC1} \nbput{3}
    \ncline[offset=2pt] {xC1}{xB1} \nbput{2}
    \ncline {xE1}{xF1} \naput{3}
    \ncdiag[offset=2pt,angleA=55,angleB=125,arm=4.5ex,linearc=1]%
    {xA1}{xD1} \nbput{4}
    \ncdiag[offset=2pt,angleB=55,angleA=125,arm=4ex,linearc=1]%
    {xD1}{xA1} \nbput{3}
    \ncline {xF1}{xF0} \naput{2}
    \ncline {xD1xD1}{xD2} \naput{2}
    \ncline {xC2}{xD3} \naput{4}
    \ncline {xD2}{xC1} \naput{4}
    \ncline {xC3}{xC2xC2} \naput{2}
    \ncline {xCD}{xD3} \nbput{2}
    \ncline {xB3}{xC3} \naput{3}
    \ncline {xC3}{xCD} \nbput{4}
    \ncline {xD3}{xE3} \naput{3}
  \end{displaymath}
  \caption{\label{fig:dissect}The transformation of the graph in
    Figure~\ref{fig:frog} using $\varphi_3, \varphi_4$ and $\psi_4$.}
\end{figure}

\clearpage
\section{Resolution of axiom $6$}
\label{app:axiom6}

The example in Figure~\ref{fig:gregg}, first observed by Gregg
Musiker, demonstrates the necessity of axiom $6$. This graph arises
when transforming the graph for the Macdonald polynomial
$\widetilde{H}_{(6)}(X;q,t)$. It satisfies axioms $1$ through $5$, but
fails axiom $6$. Comparing with the standard dual equivalence graphs
in Appendix~\ref{app:DEGs}, this graph is a two-fold cover of
$\G_{(3,2,1)}$ as expected from its generating function $2
s_{(3,2,1)}(X)$. Figure~\ref{fig:tab-gregg} gives the isomorphism
classes of the $(5,6)$-restriction of this graph.

\begin{figure}[ht]
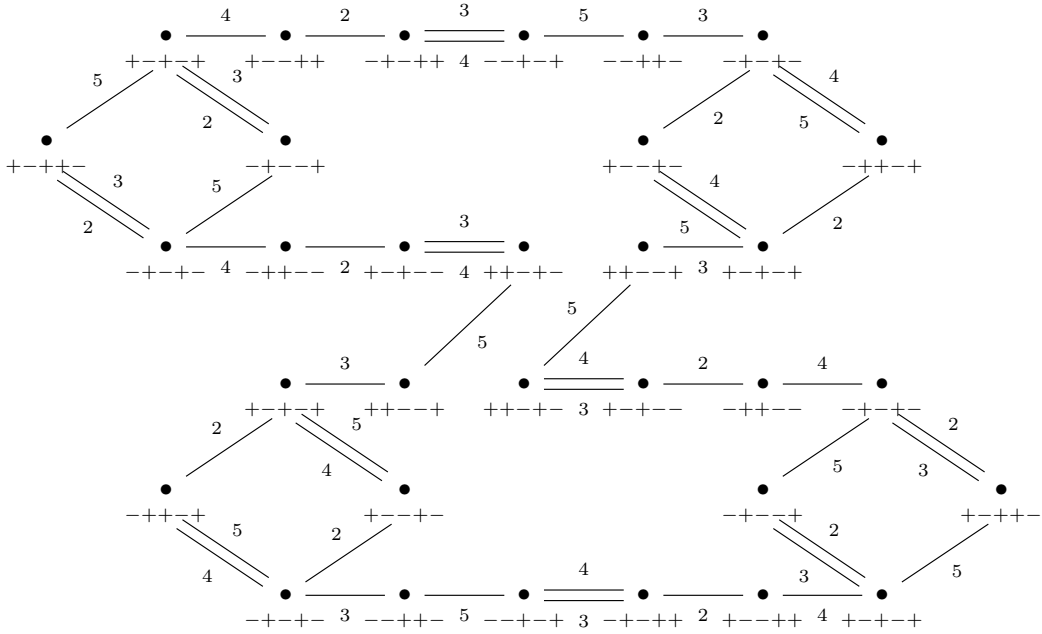

  \begin{displaymath}
    \begin{array}{\cs{7} \cs{7} \cs{7} \cs{7} \cs{7} \cs{7} \cs{7} \cs{7} c}
      &
      \sbull{c2}{+-+-+} &
      \sbull{b2}{+--++} &
      \sbull{a3}{-+-++} &
      \sbull{a4}{--+-+} &
      \sbull{b5}{--++-} &
      \sbull{c5}{-+-+-} &
      & \\[2\cellsize]
      \sbull{d3}{+-++-} & &
      \sbull{d1}{-+--+} & & &
      \sbull{d4}{+--+-} & &
      \sbull{d6}{-++-+} 
      & \\[2\cellsize]
      &
      \sbull{e2}{-+-+-} &
      \sbull{f2}{-++--} &
      \sbull{g3}{+-+--} &
      \sbull{g4}{++-+-} &
      \sbull{f5}{++--+} &
      \sbull{e5}{+-+-+} & 
      & \\[3\cellsize]
      & &
      \sbull{xe5}{+-+-+} & 
      \sbull{xf5}{++--+} &
      \sbull{xg4}{++-+-} &
      \sbull{xg3}{+-+--} &
      \sbull{xf2}{-++--} &
      \sbull{xe2}{-+-+-} & 
      \\[2\cellsize]
      &
      \sbull{xd6}{-++-+} & &
      \sbull{xd4}{+--+-} & & &
      \sbull{xd1}{-+--+}  & &
      \sbull{xd3}{+-++-}
      \\ [2\cellsize]
      & &
      \sbull{xc5}{-+-+-} &
      \sbull{xb5}{--++-} &
      \sbull{xa4}{--+-+} &
      \sbull{xa3}{-+-++} &
      \sbull{xb2}{+--++} &
      \sbull{xc2}{+-+-+} & 
    \end{array}
    \psset{nodesep=5pt,linewidth=.1ex}
    \everypsbox{\scriptstyle}
    \ncline[offset=2pt]{a3}{a4} \naput{3}
    \ncline[offset=2pt]{a4}{a3} \naput{4}
    \ncline            {b2}{a3} \naput{2}
    \ncline            {a4}{b5} \naput{5}
    \ncline            {c2}{b2} \naput{4}
    \ncline            {b5}{c5} \naput{3}
    \ncline[offset=2pt]{d1}{c2c2} \naput{2}
    \ncline[offset=2pt]{c2c2}{d1} \naput{3}
    \ncline            {c2c2}{d3} \nbput{5}
    \ncline            {d4}{c5c5} \nbput{2}
    \ncline[offset=2pt]{c5c5}{d6} \naput{4}
    \ncline[offset=2pt]{d6}{c5c5} \naput{5}
    \ncline            {d1d1}{e2} \nbput{5}
    \ncline[offset=2pt]{e2}{d3d3} \naput{2}
    \ncline[offset=2pt]{d3d3}{e2} \naput{3}
    \ncline[offset=2pt]{d4d4}{e5} \naput{4}
    \ncline[offset=2pt]{e5}{d4d4} \naput{5}
    \ncline            {e5}{d6d6} \nbput{2}
    \ncline            {e2}{f2} \nbput{4}
    \ncline            {f5}{e5} \nbput{3}
    \ncline            {f2}{g3} \nbput{2}
    \ncline            {g4g4}{xf5} \naput{5}
    \ncline[offset=2pt]{g3}{g4} \naput{3}
    \ncline[offset=2pt]{g4}{g3} \naput{4}
    \ncline[offset=2pt]{xa3}{xa4} \naput{3}
    \ncline[offset=2pt]{xa4}{xa3} \naput{4}
    \ncline            {xb2}{xa3} \naput{2}
    \ncline            {xa4}{xb5} \naput{5}
    \ncline            {xc2}{xb2} \naput{4}
    \ncline            {xb5}{xc5} \naput{3}
    \ncline[offset=2pt]{xd1xd1}{xc2} \naput{2}
    \ncline[offset=2pt]{xc2}{xd1xd1} \naput{3}
    \ncline            {xc2}{xd3xd3} \nbput{5}
    \ncline            {xd4xd4}{xc5} \nbput{2}
    \ncline[offset=2pt]{xc5}{xd6xd6} \naput{4}
    \ncline[offset=2pt]{xd6xd6}{xc5} \naput{5}
    \ncline            {xd1}{xe2xe2} \nbput{5}
    \ncline[offset=2pt]{xe2xe2}{xd3} \naput{2}
    \ncline[offset=2pt]{xd3}{xe2xe2} \naput{3}
    \ncline[offset=2pt]{xd4}{xe5xe5} \naput{4}
    \ncline[offset=2pt]{xe5xe5}{xd4} \naput{5}
    \ncline            {xe5xe5}{xd6} \nbput{2}
    \ncline            {xe2}{xf2} \nbput{4}
    \ncline            {xf5}{xe5} \nbput{3}
    \ncline            {xf2}{xg3} \nbput{2}
    \ncline            {xg4}{f5f5} \naput{5}
    \ncline[offset=2pt]{xg3}{xg4} \naput{3}
    \ncline[offset=2pt]{xg4}{xg3} \naput{4}
  \end{displaymath}
  \caption{\label{fig:gregg}The smallest graph satisfying dual
    equivalence graph axioms $1-5$ but not $6$.}
\end{figure}

\begin{figure}[ht]
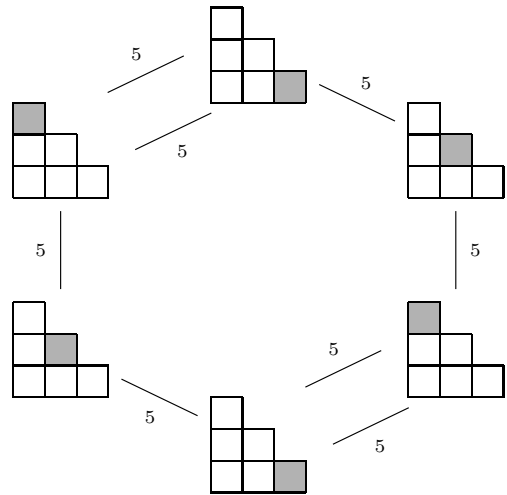

  \begin{displaymath}
    \begin{array}{\cs{9}\cs{9}c}
      & \rnode{a}{\tableau{\e \\ \e & \e \\ \e & \e & \cb}} & \\[\cellsize]
      \rnode{b}{\tableau{\cb \\ \e & \e \\ \e & \e & \e}} & 
      & \rnode{c}{\tableau{\e \\ \e & \cb \\ \e & \e & \e}} \\[5\cellsize]
      \rnode{C}{\tableau{\e \\ \e & \cb \\ \e & \e & \e}} & 
      & \rnode{B}{\tableau{\cb \\ \e & \e \\ \e & \e & \e}} \\[\cellsize]
      & \rnode{A}{\tableau{\e \\ \e & \e \\ \e & \e & \cb}} & 
    \end{array}
    \psset{nodesep=5pt,linewidth=.1ex}
    \everypsbox{\scriptstyle}
    \ncline[offset=12pt] {a}{b} \naput{5}
    \ncline[offset=12pt] {b}{a} \naput{5}
    \ncline {a}{c} \naput{5}
    \ncline {b}{C} \nbput{5}
    \ncline {c}{B} \naput{5}
    \ncline {C}{A} \nbput{5}
    \ncline[offset=12pt] {B}{A} \naput{5}
    \ncline[offset=12pt] {A}{B} \naput{5}
  \end{displaymath}
  \caption{\label{fig:tab-gregg}The $(5,6)$-restriction of
    Figure~\ref{fig:gregg} highlighting the two-fold cover of
    $\G_{3,2,1}$.}
\end{figure}

\clearpage
\section{Graphs failing axiom $4'$}
\label{app:bad-graphs}

In this final appendix, we give examples of locally Schur positive
graphs satisfying dual equivalence axioms $1,2,3$ and $5$ but failing
axiom $4'$. Not coincidentally, the transformations presented in
Section~\ref{sec:Dgraphs} cannot be applied to transform these graphs
into dual equivalence graphs.

Figure~\ref{fig:4'a} shows a graph violating only axiom $4'a$. The
generating function is not Schur positive. Here $\varphi_4$ is needed
in two places, and in both instances breaks local Schur
positivity. There are two places requiring $\varphi_5$, however
neither satisfies the hypotheses necessary to apply the map.

\begin{figure}[ht]
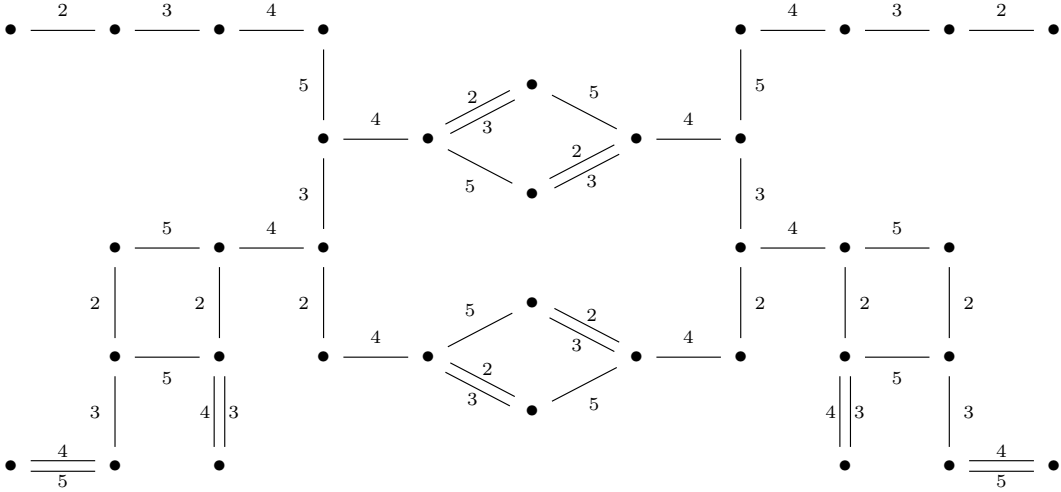

  \begin{displaymath}
    \begin{array}{\cs{8}\cs{8}\cs{8}\cs{8}\cs{8}\cs{8}\cs{8}\cs{8}\cs{8}\cs{8}c}
      \rnode{z1}{\B} & \rnode{a0}{\B} & \rnode{b0}{\B} & \rnode{c0}{\B} & \rnode{d0}{  } & \rnode{e0}{  } & \rnode{f0}{  } & \rnode{g0}{\B} & \rnode{h0}{\B} & \rnode{i0}{\B} & \rnode{y1}{\B} \\[2ex]
      & \rnode{a1}{  } & \rnode{b1}{  } & \rnode{c1}{  } & \rnode{d1}{  } & \rnode{e1}{\B} & \rnode{f1}{  } & \rnode{g1}{  } & \rnode{h1}{  } & \rnode{i1}{  } & \\[2ex]
      & \rnode{a2}{  } & \rnode{b2}{  } & \rnode{c2}{\B} & \rnode{d2}{\B} & \rnode{e2}{  } & \rnode{f2}{\B} & \rnode{g2}{\B} & \rnode{h2}{  } & \rnode{i2}{  } & \\[2ex]
      & \rnode{a3}{  } & \rnode{b3}{  } & \rnode{c3}{  } & \rnode{d3}{  } & \rnode{e3}{\B} & \rnode{f3}{  } & \rnode{g3}{  } & \rnode{h3}{  } & \rnode{i3}{  } & \\[2ex]
      & \rnode{a4}{\B} & \rnode{b4}{\B} & \rnode{c4}{\B} & \rnode{d4}{  } & \rnode{e4}{  } & \rnode{f4}{  } & \rnode{g4}{\B} & \rnode{h4}{\B} & \rnode{i4}{\B} & \\[2ex]
      & \rnode{a5}{  } & \rnode{b5}{  } & \rnode{c5}{  } & \rnode{d5}{  } & \rnode{e5}{\B} & \rnode{f5}{  } & \rnode{g5}{  } & \rnode{h5}{  } & \rnode{i5}{  } & \\[2ex]
      & \rnode{a6}{\B} & \rnode{b6}{\B} & \rnode{c6}{\B} & \rnode{d6}{\B} & \rnode{e6}{  } & \rnode{f6}{\B} & \rnode{g6}{\B} & \rnode{h6}{\B} & \rnode{i6}{\B} & \\[2ex]
      & \rnode{a7}{  } & \rnode{b7}{  } & \rnode{c7}{  } & \rnode{d7}{  } & \rnode{e7}{\B} & \rnode{f7}{  } & \rnode{g7}{  } & \rnode{h7}{  } & \rnode{i7}{  } & \\[2ex]
      \rnode{z8}{\B} & \rnode{a8}{\B} & \rnode{b8}{\B} & \rnode{c8}{  } & \rnode{d8}{  } & \rnode{e8}{  } & \rnode{f8}{  } & \rnode{g8}{  } & \rnode{h8}{\B} & \rnode{i8}{\B} & \rnode{y8}{\B} 
    \end{array}
    \psset{nodesep=5pt,linewidth=.1ex}
    \everypsbox{\scriptstyle}
    \ncline {z1}{a0} \naput{2}
    \ncline {a0}{b0} \naput{3}
    \ncline {b0}{c0} \naput{4}
    \ncline {g0}{h0} \naput{4}
    \ncline {h0}{i0} \naput{3}
    \ncline {i0}{y1} \naput{2}
    \ncline {c0}{c2} \nbput{5}
    \ncline {g0}{g2} \naput{5}
    \ncline {c2}{d2} \naput{4}
    \ncline[offset=2pt] {d2}{e1} \nbput{3}
    \ncline[offset=2pt] {e1}{d2} \nbput{2}
    \ncline {d2}{e3} \nbput{5}
    \ncline[offset=2pt] {e3}{f2} \nbput{3}
    \ncline[offset=2pt] {f2}{e3} \nbput{2}
    \ncline {e1}{f2} \naput{5}
    \ncline {f2}{g2} \naput{4}
    \ncline {c2}{c4} \nbput{3}
    \ncline {g2}{g4} \naput{3}
    \ncline {a4}{b4} \naput{5}
    \ncline {b4}{c4} \naput{4}
    \ncline {g4}{h4} \naput{4}
    \ncline {h4}{i4} \naput{5}
    \ncline {a4}{a6} \nbput{2}
    \ncline {b4}{b6} \nbput{2}
    \ncline {c4}{c6} \nbput{2}
    \ncline {g4}{g6} \naput{2}
    \ncline {h4}{h6} \naput{2}
    \ncline {i4}{i6} \naput{2}
    \ncline {a6}{b6} \nbput{5}
    \ncline {c6}{d6} \naput{4}
    \ncline[offset=2pt] {d6}{e7} \nbput{3}
    \ncline[offset=2pt] {e7}{d6} \nbput{2}
    \ncline {d6}{e5} \naput{5}
    \ncline[offset=2pt] {e5}{f6} \nbput{3}
    \ncline[offset=2pt] {f6}{e5} \nbput{2}
    \ncline {e7}{f6} \nbput{5}
    \ncline {f6}{g6} \naput{4}
    \ncline {h6}{i6} \nbput{5}
    \ncline {a6}{a8} \nbput{3}
    \ncline[offset=2pt] {b6}{b8} \nbput{4}
    \ncline[offset=2pt] {b8}{b6} \nbput{3}
    \ncline[offset=2pt] {h6}{h8} \nbput{4}
    \ncline[offset=2pt] {h8}{h6} \nbput{3}
    \ncline {i6}{i8} \naput{3}
    \ncline[offset=2pt] {z8}{a8} \nbput{5}
    \ncline[offset=2pt] {a8}{z8} \nbput{4}
    \ncline[offset=2pt] {i8}{y8} \nbput{5}
    \ncline[offset=2pt] {y8}{i8} \nbput{4}
  \end{displaymath}
  \caption{\label{fig:4'a}A locally Schur positive graph satisfying
    axioms $1,2,3$ and $5$ along with axiom $4'b$ but not $4'a$.}
\end{figure}

Figure~\ref{fig:4'b} shows a graph that violates only axiom $4'b$. The
generating function is not Schur positive. Neither $\varphi_3$ nor
$\varphi_4$ is needed. Each of $\varphi_5, \psi_4$ and $\psi_5$ can be
applied in exactly one place, and none of these preserves local Schur
positivity. In fact, both $\varphi_5$ and $\psi_4$ violate axiom $3$.

\begin{figure}[ht]
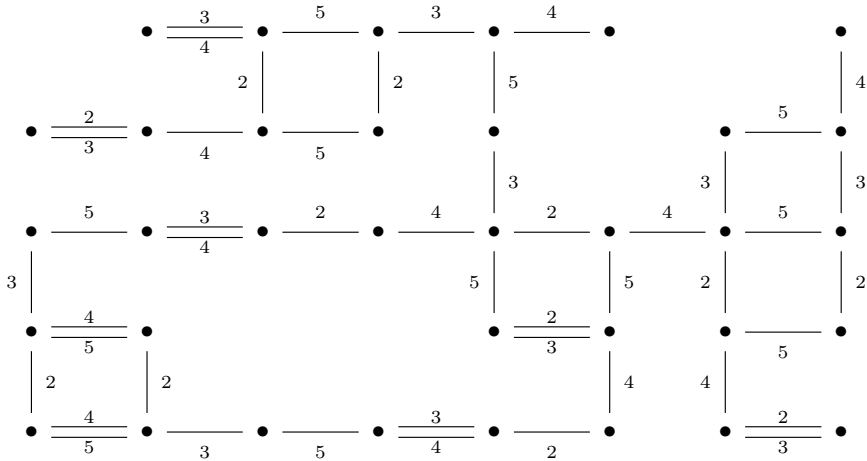

  \begin{displaymath}
    \begin{array}{\cs{9}\cs{9}\cs{9}\cs{9}\cs{9}\cs{9}\cs{9}c}
      \rnode{a1}{  } & \rnode{b1}{\B} & \rnode{c1}{\B} & \rnode{d1}{\B} & \rnode{e1}{\B} & \rnode{f1}{\B} & \rnode{g1}{  } & \rnode{h1}{\B} \\[6ex]
      \rnode{a2}{\B} & \rnode{b2}{\B} & \rnode{c2}{\B} & \rnode{d2}{\B} & \rnode{e2}{\B} & \rnode{f2}{  } & \rnode{g2}{\B} & \rnode{h2}{\B} \\[6ex]
      \rnode{b4}{\B} & \rnode{b3}{\B} & \rnode{c3}{\B} & \rnode{d3}{\B} & \rnode{e3}{\B} & \rnode{f3}{\B} & \rnode{g3}{\B} & \rnode{h3}{\B} \\[6ex]
      \rnode{b5}{\B} & \rnode{c5}{\B} & \rnode{c4}{  } & \rnode{d4}{  } & \rnode{e4}{\B} & \rnode{f4}{\B} & \rnode{g4}{\B} & \rnode{h4}{\B} \\[6ex]
      \rnode{b6}{\B} & \rnode{c6}{\B} & \rnode{d6}{\B} & \rnode{e6}{\B} & \rnode{f6}{\B} & \rnode{f5}{\B} & \rnode{g5}{\B} & \rnode{g6}{\B}
    \end{array}
    \psset{nodesep=5pt,linewidth=.1ex}
    \everypsbox{\scriptstyle}
    \ncline[offset=2pt] {b1}{c1} \nbput{4}
    \ncline[offset=2pt] {c1}{b1} \nbput{3}
    \ncline {c1}{d1} \naput{5}
    \ncline {d1}{e1} \naput{3}
    \ncline {e1}{f1} \naput{4}
    \ncline {c1}{c2} \nbput{2}
    \ncline {d1}{d2} \naput{2}
    \ncline {e1}{e2} \naput{5}
    \ncline {h1}{h2} \naput{4}
    \ncline[offset=2pt] {a2}{b2} \nbput{3}
    \ncline[offset=2pt] {b2}{a2} \nbput{2}
    \ncline {b2}{c2} \nbput{4}
    \ncline {c2}{d2} \nbput{5}
    \ncline {g2}{h2} \naput{5}
    \ncline {e2}{e3} \naput{3}
    \ncline {g2}{g3} \nbput{3}
    \ncline {h2}{h3} \naput{3}
    \ncline[offset=2pt] {b3}{c3} \nbput{4}
    \ncline[offset=2pt] {c3}{b3} \nbput{3}
    \ncline {c3}{d3} \naput{2}
    \ncline {d3}{e3} \naput{4}
    \ncline {e3}{f3} \naput{2}
    \ncline {f3}{g3} \naput{4}
    \ncline {g3}{h3} \naput{5}
    \ncline {b3}{b4} \nbput{5}
    \ncline {e3}{e4} \nbput{5}
    \ncline {f3}{f4} \naput{5}
    \ncline {g3}{g4} \nbput{2}
    \ncline {h3}{h4} \naput{2}
    \ncline[offset=2pt] {e4}{f4} \nbput{3}
    \ncline[offset=2pt] {f4}{e4} \nbput{2}
    \ncline {g4}{h4} \nbput{5}
    \ncline {b4}{b5} \nbput{3}
    \ncline {f4}{f5} \naput{4}
    \ncline {g4}{g5} \nbput{4}
    \ncline[offset=2pt] {b5}{c5} \nbput{5}
    \ncline[offset=2pt] {c5}{b5} \nbput{4}
    \ncline {b5}{b6} \naput{2}
    \ncline {c5}{c6} \naput{2}
    \ncline {f5}{f6} \naput{2}
    \ncline[offset=2pt] {g5}{g6} \nbput{3}
    \ncline[offset=2pt] {g6}{g5} \nbput{2}
    \ncline[offset=2pt] {b6}{c6} \nbput{5}
    \ncline[offset=2pt] {c6}{b6} \nbput{4}
    \ncline {c6}{d6} \nbput{3}
    \ncline {d6}{e6} \nbput{5}
    \ncline[offset=2pt] {e6}{f6} \nbput{4}
    \ncline[offset=2pt] {f6}{e6} \nbput{3}
  \end{displaymath}
  \caption{\label{fig:4'b}A locally Schur positive graph satisfying
    axioms $1,2,3$ and $5$ along with axiom $4'a$ but not
    $4'b$.}
\end{figure}

\clearpage
\section{Computer code to verify local Schur positivity and axiom $4'$}
\label{app:code}

In Proposition~\ref{prop:ax1235}, we claim that the graphs
$\G^{(k)}_{c,D}$ are locally Schur positive for any content vector $c$
and any $k$-descent set $D$. As mentioned in the proof, it suffices to
check graphs of type $(5,5)$. Similarly, in
Proposition~\ref{prop:ax4p}, we claim that these graphs also satisfy
axiom $4'$. To show this, it suffices to check graphs of type
$(5,6)$. In this appendix, we provide the computer code used to verify
these cases. The following code is written in Maple. 

\subsection{Basic combinatorial objects}

The function \texttt{nextPerm()} takes as input a permutation (as a
single line array) and returns the next permutation in lexicographic
order or a special character (\texttt{NULL}) after the last
permutation. 

\begin{verbatim}

nextPerm := proc(word) local N, i, j, left, right, new;
  N:=nops(word);
  # READING RIGHT TO LEFT, FIND FIRST INSTANCE OF A DECENT
  i:=N;
  while i>1 do
    if word[i-1] < word[i] then break; fi;
    i:=i-1; od;
  # IF NO SUCH INSTANCE EXISTS, THIS IS THE LAST PERMUTATION
  if i=1 then RETURN(NULL); fi;
  # OTHERWISE FIND THE POSITION WITH WHICH TO SWAP 
  left:=i; right:=N;
  if word[N]>word[i-1] then left:=N;
  else while left+1<>right do
    j:=ceil((left+right)/2);
    if word[i-1]>word[j] then right:=j;
    else left:=j; fi; od; fi;
  # BUILD NEXT PERMUTATION
  new:=word;
  new[i-1]:=word[left];
  new[left]:=word[i-1];
  for j from i to floor((N+i)/2) do
    right:=new[N-(j-i)];
    new[N-(j-i)]:=new[j];
    new[j]:= right; od;
  return new;
end:

\end{verbatim}

Recall from the proof of Proposition~\ref{prop:ax1235} that we may
encode content vectors by recording the last position with which $w_j$
can form the first member of an inversion pair. The function
\texttt{nextContent()} takes as input a content vector and returns the
next content vector in lexicographic order or a special character
(\texttt{NULL}) after the last content vector.

\begin{verbatim}

nextContent := proc(convec) local C, R, N, i;
  N:=nops(convec); 
  ## i <= convec[i] <= convec[i+1] <= N+1
  for i from N to 1 by -1 do
    if convec[i] < N+1 then break; fi; od;
  ## NO INDEX CAN INCREASE
  if i=0 then return NULL;
  ## INCREASE AND RESET convec
  else return [op(convec[1..i-1]), seq(max(convec[i]+1,j+1),j=i..N)]; fi;
end:

\end{verbatim}

\subsection{Generating D graphs}

Given the above method for encoding content vectors, the conditions in
\eqref{eqn:Dk} may be reinterpretted as apply $d_i$ if $i$ and the
farther away of $\imo,\ipo$ are not attacking; otherwise apply
$\widetilde{d}_i$. The function \texttt{iMove()} takes a permutation
and a content vector and returns the permutation resulting from
applying $D_i^{(k)}$.

\begin{verbatim}

iMove := proc(word,convec,i)
local j, L, C, R, result;
  result:=word;
  ## INITIALIZE LOCATIONS TO 0, SEARCH FOR i-1,i,i+1
  L:=0; C:=0; R:=0;
  for j from 1 to nops(word) do
    if word[j]=i-1 then
      if L=0 then L:=j;
      elif C=0 then C:=j;
      else R:=j; fi;
    elif word[j]=i then
      if L=0 then L:=j;
      elif C=0 then C:=j;
      else R:=j; fi;
    elif word[j]=i+1 then
      if L=0 then L:=j;
      elif C=0 then C:=j;
      else R:=j; fi; fi; od;
  ## IF AN i IS NOT THE MIDDLE LETTER
  if word[C]<>i then
    ## DO AN i-SWITCH
    if convec[L] < R then result[L]:=word[R]; result[R]:=word[L];
    ## DO AN i-SHIFT
    elif word[L]=i then result[L]:=word[C]; result[C]:=word[R]; result[R]:=word[L];
    else result[L]:=word[R]; result[C]:=word[L]; result[R]:=word[C]; fi;
    return result;
  else return 0; fi;
end:

\end{verbatim}

The function \texttt{findIt()} is a simple routine to find an element
in an ordered list. It takes as input an element and a list, and it
returns the (first) index of the element, if found, otherwise it
returns $0$. (Note that Maple lists begin at index $1$.)

\begin{verbatim}

findIt := proc(elem,items) local i;
  for i from 1 to nops(items) do
    if elem=items[i] then return i; fi;
  od;
  return 0;
end:

\end{verbatim}

The function \texttt{connComp()} takes as input a permutation and a
content vector and returns the connected component of graph
$\G^{(k)}_{c,D}$ corresponding to the given content vector that
contains the given permutation. The graph is encoded as vertices and
edges. The vertices are an ordered list of permutations. The edges are
encoded as a set of pairs where the first entry is the set of indices
corresponding to the paired vertices, and the second entry is the
color of the edge. Double edges are encoded by digits in base $10$,
e.g. a double edge between $2$ and $3$ is encoded by the number $23$.

\begin{verbatim}

connComp := proc(word,convec) 
  local N, i, vertices, edges, w, u, j, m, top, mid, bot;
  ## INITIALIZE LOCAL VARIABLES
  N:=nops(word);
  j := binarySig(word);
  vertices:=[word]; edges:={};
  top := []; mid:=[word]; bot:=[];
  ## GENERATE VERTICES AND EDGES IN THREE LAYERS
  while mid <> [] do
    ## RUN THROUGH ALL VERTICES IN MIDDLE LAYER
    for w from 1 to nops(mid) do
      ## GENERATE ALL POSSIBLE EDGES
      for i from 2 to N-1 do
        u := iMove(mid[w],convec,i);
        ## IF i-EDGE EXISTS AND NOT ALREADY INCLUDED
        if u<>0 and not(u in top) then
          ## CHECK IF ALREADY IN BOTTOM LAYER
          j := findIt(u,bot);
          m := nops(vertices);
          if j=0 then
            j:=findIt(u,mid);
            ## ENTIRELY NEW VERTEX
            if j=0 then
              bot := [op(bot),u];
              ## CHECK FOR DOUBLE EDGE WITH i+1
              if i+1<N and u = iMove(mid[w],convec,i+1) then
                edges := edges union {[{m-nops(mid)+w,m+nops(bot)},i*11+1]};
                i:=i+1;
              else
                edges := edges union {[{m-nops(mid)+w,m+nops(bot)},i]}; fi;
            ## LATER VERTEX IN MIDDLE LAYER
            elif w<j then
              ## CHECK FOR DOUBLE EDGE WITH i+1
              if i+1<N and u = iMove(mid[w],convec,i+1) then
                edges := edges union {[{m-nops(mid)+w,m-nops(mid)+j},i*11+1]};
                i:=i+1;
              else
                edges := edges union {[{m-nops(mid)+w,m-nops(mid)+j},i]}; fi; fi;
          ## EXISTING VERTEX IN BOTTOM LAYER
          else
            ## CHECK FOR DOUBLE EDGE WITH i+1
            if i+1<N and u = iMove(mid[w],convec,i+1) then
              edges := edges union {[{m-nops(mid)+w, m+j},i*11+1]};
              i:=i+1;
            else
              edges := edges union {[{m-nops(mid)+w, m+j},i]}; fi; fi; fi;
      od; od;
    ## ADD NEW VERTICES TO LIST
    vertices := [op(vertices), op(bot)];
    ## RESET LAYERS
    top:=mid; mid:=bot; bot:=[]; od;
  return [vertices,edges];
end:

\end{verbatim}

\subsection{Verifying local Schur positivity}

The following routines compute the generating function of the
components.

The function \texttt{binarySig()} computes the signature of a
permutation encoded as a number by interpretting $\sigma$ to be a
binary expansion where $+1$ represents $0$ and $-1$ represents
$1$. For a partition, \texttt{domSig()} computes the signature for the
superstandard tableau of shape of the given shape.

\begin{verbatim}

binarySig := proc(word) local i, j, N, sig;
  N:=nops(word); sig:=0;
  for i from 1 to N-1 do
    for j from 1 to N do
      if word[j]=i+1 then sig:=sig+2^(i-1); break;
      elif word[j]=i then break; fi; od; od;
  return sig;
end:

domSig := mu -> convert([seq(2^(convert(mu[1..i],`+`)-1),i=1..nops(mu)-1)],`+`):

\end{verbatim}

The function \texttt{QtoS()} takes a sum of fundamental quasisymmetric
functions (indexed by integers taken from signatures as mentioned
above) and returns the Schur expansion or $0$ if the sum is not Schur
positive. This is the key ingredient in the function
\texttt{genFun()}, which takes a graph and returns its generating
function. 

\begin{verbatim}

if not assigned(`StoQ`) then StoQ:=table():
  ## N = 4
  StoQ[4] := Q[0]:
  StoQ[3,1] := Q[1]+Q[2]+Q[4]:
  StoQ[2,2] := Q[5]+Q[2]:
  StoQ[2,1,1] := Q[3]+Q[5]+Q[6]:
  StoQ[1,1,1,1] := Q[7]:
  ## N = 5
  StoQ[5] := Q[0]:
  StoQ[4,1] := Q[1]+Q[2]+Q[4]+Q[8]:
  StoQ[3,2] := Q[5]+Q[2]+Q[9]+Q[10]+Q[4]:
  StoQ[3,1,1] := Q[3]+Q[5]+Q[6]+Q[9]+Q[10]+Q[12]:
  StoQ[2,2,1] := Q[11]+Q[5]+Q[6]+Q[13]+Q[10]:
  StoQ[2,1,1,1] := Q[7]+Q[11]+Q[13]+Q[14]:
  StoQ[1,1,1,1,1] := Q[15]: fi:
if not assigned(StoQ[N]) then ERROR(`insufficient data to convert`); fi:

QtoS := proc(Qsum,N) local mu, m, r, sam; global StoQ;
  r:=Qsum; sam:=0; mu:=[N];
  while mu<>NULL do
    m := coeff(r,Q[domSig(mu)]);
    sam := sam + m*s[op(mu)];
    r := r - m*StoQ[op(mu)];
    mu:=nextPar(mu); od;
  if r=0 then return sam;
  else return 0; fi;
end:

genFun := GG -> QtoS(convert([seq(Q[binarySig(GG[1][i])],i=1..nops(GG[1]))],`+`),
                     nops(GG[1][1])):

\end{verbatim}

The above uses the function \texttt{nextPar()}, which a partition and
returns the next partition in lexicographic order. This procedure was
copied from John Stembridge's SF package.

\begin{verbatim}

nextPar := proc(mu) local i,k,m,r;
  if member(1,mu,'i') then i:=i-1 else i:=nops(mu) fi;
  if i=0 then NULL else
    k:=mu[i]-1; m:=iquo(nops(mu)-i+mu[i],k,'r');
    if r=0 then r:=NULL fi;
    [op(1..i-1,mu),k$m,r]; fi;
end:

\end{verbatim}

One can check $\LSP_4$ and $\LSP_5$ for all $\G^{(k)}_{c,D}$ with the
function \texttt{checkLSP}.

\begin{verbatim}

checkLSP := proc() local word, convec, G;
  print("CHECKING LSP FOR N=4");
  convec := [2,3,4];
  while convec <> NULL do 
    word := [1,2,3,4];
    while word <> NULL do
      G := connComp(word, convec);
      if genFun(G) = 0 then return false; fi;
      word := nextPerm(word); od;
    convec := nextContent(convec); od;
  print(true);
  print("CHECKING LSP FOR N=5");
  convec := [2,3,4,5];
  while convec <> NULL do 
    word := [1,2,3,4,5];
    while word <> NULL do
      G := connComp(word, convec);
      if genFun(G) = 0 then return false; fi;
      word := nextPerm(word); od;
    convec := nextContent(convec); od;
  return true;
end:

\end{verbatim}

\subsection{Verifying axiom $4'$}

The function \texttt{twoString()} takes a permutation and content
vector along with two edge colors $i$ and $j$, and it returns the
connected component of $\G^{(k)}_{c,D}$ restricted to $E_i \cup E_j$
containing the given permutation. Since edges are involutions, the
returned data is an list of vertices that alternate edges together
with an index $k$ specifying the location in the list of the given
permutation. 

\begin{verbatim}

twoString := proc(word, convec, i, j) local str, u, k;
  str := [word];
  u := iMove(word,convec,i);
  while u <> 0 do
    if u = str[1] or u = str[-1] then break; fi;
    str := [u,op(str)];
    u := iMove(u,convec,j);
    if u = 0 or u = str[1] or u = str[-1] then break; fi;
    str := [u,op(str)];
    u := iMove(u,convec,i); od;
  k := nops(str);
  u := iMove(word,convec,j);
  while u <> 0 do
    if u = str[1] or u = str[-1] then break; fi;
    str := [op(str),u];
    u := iMove(u,convec,i);
    if u = 0 or u = str[1] or u = str[-1] then break; fi;
    str := [u,op(str)];
    u := iMove(u,convec,j); od;
  return [str,k];
end:

\end{verbatim}

The function \texttt{twoString} is used in the following two functions
that check axiom $4'$.

\begin{verbatim}

axiom4pa := proc(word, convec) local N, i, u, mStr, pStr, mSet, pSet;
  N := nops(word);
  for i from 3 to N-2 do
    u := iMove(word,convec,i);
    if u = 0 then next; fi;
    if iMove(word,convec,i-1) <> 0 and iMove(word,convec,i+1) <> 0
    and iMove(u,convec,i-1) <> 0 and iMove(u,convec,i+1) <> 0 then
      mStr := twoString(word,convec,i-1,i)[1];
      mSet := {op(map(binarySig,mStr))};
      pStr := twoString(word,convec,i,i+1)[1];
      pSet := {op(map(binarySig,pStr))};
      if mSet <> pSet then return false; fi; fi; od;
  return true;
end:

axiom4pb := proc(word, convec) local N, M, i, str, k;
  N := nops(word);
  for i from 4 to N-2 do
    str := twoString(word,convec,i-2,i);
    k := str[2];
    str := str[1];
    M := nops(str);
    ## FINE FOR LOOPS
    if iMove(str[1],convec,i) = str[M] or iMove(str[1],convec,i-2) = str[M] 
    ## MUST NOT BE FIXED POINT FOR phi_i-2 OR phi_i
    or k = 1 or k = M then
        next; fi;
    ## CHECK NEITHER word NOR phi_i(word) IS A FIXED POINT FOR phi_i+1
    if iMove(str[k],convec,i+1) <> 0 and iMove(str[k+1],convec,i+1)<> 0 then
      ## phi_i phi_i-2(word) IS A FIXED POINT FOR phi_i+1
      if (k-2 >= 1 and iMove(str[k-2],convec,i+1) = 0) then
        ## CHECK NOTHING ELSE ON THE TWO STRING IS A FIXED POINT FOR phi_i+1
        k := k+1;
        while k <= M-2 do
          if iMove(str[k+2],convec,i+1) = 0 then return false; fi;
          k := k+2; od; fi; fi; od;
  return true;
end:

\end{verbatim}

Finally, one can check axiom $4'$ for all $\G^{(k)}_{c,D}$ with the
function \texttt{checkAx4p}.

\begin{verbatim} 

checkAx4p := proc() local word, convec;
  print("CHECKING AXIOM 4'a FOR N=5");
  convec := [2,3,4,5];
  while convec <> NULL do 
    word := [1,2,3,4,5];
    while word <> NULL do
        if not(axiom4pa(word,convec)) then return false; fi;
        word := nextPerm(word); od;
    convec := nextContent(convec); od;
  print(true);
  print("CHECKING AXIOM 4'b FOR N=6");
  convec := [2,3,4,5,6];
  while convec <> NULL do 
    word := [1,2,3,4,5,6];
    while word <> NULL do
        if not(axiom4pb(word,convec)) then return false; fi;
        word := nextPerm(word); od;
    convec := nextContent(convec); od;
  return true;
end:

\end{verbatim}

\end{document}